\documentclass[10pt]{article} 
\usepackage{graphicx} 
\usepackage[top=30truemm,bottom=30truemm,left=25truemm,right=25truemm]{geometry}
\usepackage{color}
\usepackage{amsmath, amssymb, ascmac, amsthm}  
\usepackage{float}  
\usepackage{mathrsfs}
\usepackage{tikz}   
\usetikzlibrary{intersections, calc, arrows.meta} 
\usepackage{mathtools}
\usepackage{physics}
\usepackage{here}
\usepackage{enumerate}
\usepackage{bm}
\usepackage{appendix}

\usepackage{stackengine}
\stackMath

\usepackage{hyperref} 

\hypersetup{%
 colorlinks=true,
 linkcolor=blue,
 citecolor=blue
}

\usepackage{enumitem}
\setlist[itemize]{leftmargin=6mm, itemsep=0pt, topsep=2pt}

\theoremstyle{plain}
\newtheorem{theorem}{Theorem}
\newtheorem*{theorem*}{Theorem} 
\theoremstyle{definition}
\newtheorem{definition}[theorem]{Definition}
\newtheorem*{definition*}{Definition} 
\newtheorem{lemma}[theorem]{Lemma}
\newtheorem*{lemma*}{Lemma} 
\newtheorem{proposition}[theorem]{Proposition}
\newtheorem*{proposition*}{Proposition} 
\newtheorem{corollary}[theorem]{Corollary}
\newtheorem*{corollary*}{Corollary} 

\newtheorem*{example*}{Example} 
\newtheorem{remark}[theorem]{Remark}
\newtheorem*{remark*}{Remark} 

\newtheorem*{guide*}{Confer} 

\numberwithin{theorem}{section}  
\numberwithin{equation}{section} 

\newcommand{\cA}{\mathcal{A}}

\newcommand{\DD}{\mathcal{D}}

\newcommand{\FFF}{\mathscr{F}}
\newcommand{\cF}{\mathcal{F}}

\newcommand{\LL}{\mathscr{L}}
\newcommand{\cL}{\mathcal{L}}
\newcommand{\MMM}{\mathcal{M}}
\newcommand{\NNN}{\mathcal{N}}
\newcommand{\OO}{\mathcal{O}}
\newcommand{\PPP}{\mathcal{P}}

\newcommand{\XX}{\mathscr{X}}
\newcommand{\GGG}{\mathscr{G}}
\newcommand{\RRR}{\mathcal{R}}

\newcommand{\cS}{\mathcal{S}}

\newcommand{\RR}{\mathbb{R}}
\newcommand{\NN}{\mathbb{N}}
\newcommand{\PP}{\mathbb{P}}

\newcommand{\EE}{\mathbb{E}}
\newcommand{\FF}{\mathbb{F}}
\newcommand{\SSS}{\mathbb{S}}
\newcommand{\TT}{\mathbb{T}}
\newcommand{\II}{\mathbb{I}}
\newcommand{\GG}{\mathbb{G}}

\newcommand{\fU}{\mathfrak{U}}
\newcommand{\fM}{\mathfrak{M}}
\newcommand{\fB}{\mathfrak{B}}
\newcommand{\fL}{\mathfrak{L}}

\newcommand{\cptarrow}{\overset{c}{\hookrightarrow}}
\newcommand{\weakarrow}{\overset{w}{\longrightarrow}}
\newcommand{\weakstararrow}{\overset{*}{\longrightarrow}}

\newcommand{\rdiv}{\operatorname{div}}

\newcommand{\pp}{\partial}

\newcommand{\Bog}{\mathcal{B}_D}
\newcommand{\invdiv}[1]{\nabla \Delta^{-1}\left(#1\right)}
\newcommand{\Riesz}[1]{\RRR\left(#1\right)}

\newcommand{\intT}{\int_{0}^T}
\newcommand{\inttau}{\int_{0}^\tau}

\newcommand{\rn}{\mathrm{n}}
\newcommand{\pD}{\partial D}
\newcommand{\bphi}{\boldsymbol{\varphi}}
\newcommand{\urho}{\underline{\rho}}
\newcommand{\orho}{\overline{\rho}}
\newcommand{\oD}{\overline{D}}
\newcommand{\stochbasis}{(\Omega,\FFF,(\FFF_t)_{t\geq 0},\PP)}
\newcommand{\probsp}{(\Omega,\FFF,\PP)}
\newcommand{\tstochbasis}{(\tOmega,\tFFF,(\tFFF_t)_{t\geq 0},\tPP)}
\newcommand{\tprobsp}{(\tOmega,\tFFF,\tPP)}

\newcommand{\ico}[2]{[#1,#2)}
\newcommand{\ioc}[2]{(#1,#2]}
\newcommand{\evm}[2]{\ev{#1,#2}_{V_m^\ast,V_m}}
\newcommand{\inner}[2]{\left(#1,#2\right)_{L^2_x}}

\newcommand{\esssup}{\operatorname{esssup}}

\newcommand{\tOmega}{\tilde{\Omega}}
\newcommand{\tPP}{\tilde{\PP}}
\newcommand{\tFFF}{\tilde{\FFF}}
\newcommand{\trho}{\tilde{\rho}} 
\newcommand{\tu}{\tilde{u}}
\newcommand{\tW}{\tilde{W}}
\newcommand{\tEE}{\tilde{\EE}}
\newcommand{\tq}{\tilde{q}}
\newcommand{\tk}{\tilde{k}}
\newcommand{\tE}{\tilde{E}}
\newcommand{\tnu}{\tilde{\nu}}
\newcommand{\tY}{\tilde{Y}}

\newcommand{\osc}{\textbf{osc}}
\newcommand{\ov}[1]{\overline{#1}}

\title{Stochastically forced compressible Navier--Stokes equations with slip boundary conditions of friction type}
\author{Reo Tsuboya\thanks{Department of Mathematics, Graduate School of Science, Kyoto University, Kitashirakawa Oiwake-cho, Sakyo-ku,  Kyoto 606-8502, Japan. Email: tsuboya.reo.56m@st.kyoto-u.ac.jp.}}
\date{}

\begin{document}
\maketitle
\begin{abstract}
	We study a mathematical model of a compressible viscous fluid driven by stochastic forces under slip boundary conditions of friction type.
	We introduce a notion of a weak solution that is analytically and probabilistically consistent with this model.
	Our main result establishes the existence of such weak solutions under slip boundary conditions on bounded domains with $C^{2+\nu}$-boundary ($\nu>0$).
	The proof of this result combines an extended version of the four-layer approximation scheme on the torus by Breit/Feireisl/Hofmanov\'{a} \cite{BFHbook18} with the convex approximation method for absolute value functions studied by Ne\v{c}asov\'{a}/Ogorzaly/Scherz \cite{NOS23}.
\end{abstract}


\tableofcontents
\section{Introduction}
In this paper, we consider the stochastically forced compressible Navier--Stokes equations   
\begin{align}
	&d\rho + \rdiv(\rho u)dt = 0 & &\mathrm{in}\  (0,T)\times D, \label{eq:1.1}\\
	&d(\rho u) + \rdiv (\rho u \otimes u)dt = \rdiv \TT dt + \GG(\rho,\rho u)dW & &\mathrm{in}\ (0,T)\times D,  \label{eq:1.2}\\ 
	&(\rho(0),(\rho u)(0)) = (\rho_0,q_0) & &\mathrm{in}\ D, \label{eq:1.3}\\ 
	&u\cdot \rn = 0 & & \mathrm{on} \  (0,T)\times \pp D, \label{eq:1.4}\\ 
	&|(\TT \rn)_\tau| \leq g,\quad (\TT \rn)_\tau \cdot u_\tau +g |u_\tau| = 0 & & \mathrm{on}\ (0,T)\times \pp D, \label{eq:1.5}
\end{align}
where $D\subset \RR^3$ is an open and bounded domain with locally Lipschitz boundary $\pp D$ and outward unit normal vector $\rn$ on $\pD$, and $T>0$ is fixed. 
$\rho:(0,T)\times D\to \RR$ and $u:(0,T)\times D\to \RR^3$ denote the unknown density of the fluid and the unknown velocity vector of the fluid, respectively, and $\TT$ is the Cauchy stress tensor of the viscous Newtonian fluid with isentropic pressure law
\begin{align*}
	\TT &= \underbrace{\SSS(\nabla u)}_{\text{viscous stress}} - \underbrace{p(\rho)\II}_{\text{pressure}}, \\ 
	\SSS(\nabla u) &= \mu \left(\nabla u + (\nabla u)^T -\frac{2}{3}\rdiv u \II \right) + \lambda \rdiv u \II,\quad p(\rho) = a \rho^\gamma,\quad a>0,  
\end{align*}
where $\mu>0$ and $\lambda\geq 0$ are viscous coefficients, $\gamma>1$ is the adiabatic exponent, and $\II$ is the identity matrix on $\RR^3$. 
The driving process $W$ is a cylindrical Wiener process defined on some probability space $(\Omega,\FFF,\PP)$ and the coefficient $\GG$ is generally nonlinear and satisfies suitable growth conditions.
The function $g$, given on $\pD$ and assumed to be positive, is called a modulus of friction and we use the notation $u_{\rn}= (u\cdot \rn)\rn,\ u_\tau = u- u_{\rn}$.
In the considered model, \eqref{eq:1.1} and \eqref{eq:1.2} represent the equation of continuity and the momentum equation, respectively. The initial conditions are given by \eqref{eq:1.3}. Finally, \eqref{eq:1.4} and \eqref{eq:1.5} represent the slip boundary condition of friction type. 

Throughout this paper, we fix a separable Hilbert space $\fU$ and its complete orthonormal system $(e_k)_{k\in \NN}$, and the diffusion coefficient $\GG=\GG(\rho,\rho u)$ is defined as a superposition operator from $\fU$ to functions on $D$ as follows:
$$
	\GG(\rho,q)e_k = G_k(\cdot,\rho(\cdot),q(\cdot)),
$$
where the coefficients $G_k=G_k(x,\rho,q):D\times [0,\infty)\times \RR^3\to \RR^3$ are $C^1$-functions such that there exists constants $(g_k)_{k\in \NN}\subset [0,\infty)$ with $\sum_{k=1}^{\infty}g_k^2<\infty$ and uniformly in $x\in D$,
\begin{align}
   \left|G_k(x,\rho,q) \right|&\leq g_k(\rho+\left|q \right|),  \label{eq:1.6}\\
   \left|\nabla_{\rho,q} G_k(x,\rho,q) \right|&\leq g_k. \label{eq:1.7}
\end{align}
Moreover, vector notation is omitted when there is no risk of confusion. For example, we use the notation $L^p(D)$ instead of $L^p(D;\RR^3)$.

Now, let us consider the case where a stochastic basis $(\Omega,\FFF,(\FFF_t)_{t\geq 0},\PP)$ and a cylindrical $(\FFF_t)$-Wiener process $W=\sum_{k=1}^{\infty}e_k W_k$ on $\fU$ are given, where $(W_k)_{k\in\NN}$ is a sequence of mutually independent real-valued $(\FFF_t)$-Wiener processes. 
If $\rho$ and $\sqrt{\rho}u$ are $(\FFF_t)$-progressively measurable processes belonging to 
\begin{align*}
   \rho &\in L^{2}(\Omega;L^\infty(0,T;L^\gamma(D))),\quad \gamma\geq 1,  \\
   \sqrt{\rho}u&\in L^4(\Omega; L^\infty(0,T;L^2(D))), 
\end{align*} 
respectively, then the operator $\GG(\rho,\rho u)$ is an $(\FFF_t)$-progressively measurable process belonging to 
\begin{align*}
	&\GG(\rho,\rho u)\in L^2(\Omega\times (0,T);L_2(\fU;(W^{b,2}(D))^\ast)),\quad b>\tfrac{3}{2}.  
\end{align*}
Indeed, using the embedding $L^1(D)\hookrightarrow (W^{b,2}(D))^\ast,\ b>\frac32$ and \eqref{eq:1.6}, we have 
\begin{align*}
	\norm{\GG(\rho,\rho u)}_{L_2(\fU,(W_x^{b,2})^\ast)}^2 &= \sum_{k=1}^{\infty} \norm{G_k(\rho,\rho u)}_{(W_x^{b,2})^\ast}^2 \leq C \sum_{k=1}^{\infty}\norm{G_k(\rho,\rho u)}_{L_x^1}^2 \\ 
	&\lesssim \sum_{k=1}^{\infty}g_k^2 \norm{\rho + \rho |u|}_{L_x^1}^2 \lesssim \norm{\rho + \rho |u|^2}_{L_x^1}\norm{\rho}_{L_x^1} \\ 
	&\lesssim \norm{\rho}_{L_x^\gamma}^2 + \norm{\rho |u|^2}_{L_x^1}^2.
\end{align*}
Therefore, the stochastic integral in \eqref{eq:1.2} is well-defined as an $(\FFF_t)$-martingale taking values in $(W^{b,2}(D))^\ast$, $b>\frac32,$ and the stochastic driving force is represented by the stochastic differential of the form 
$$
	\GG(\rho,\rho u)dW = \sum_{k=1}^{\infty}G_k(x,\rho,\rho u)dW_k.
$$
Note that the integrability assumptions on $\rho$ and $\sqrt{\rho}|u|$ follow from the energy inequality discussed later.
\subsection{Motivation and previous results}
The Navier--Stokes equations have been studied under various types of boundary conditions with slip phenomena. In the deterministic setting, the existence theory for weak solutions to the Navier--Stokes equations is now well-established. For an introduction to the incompressible case we refer to \cite{Lio96}, and for the compressible case to \cite{Lio98}.
In the compressible framework, the focus is on isentropic fluids, where the pressure is assumed to be proportional to the density raised to the power of the adiabatic exponent $\gamma > 9/5$. These monographs mainly deal with the cases of the \textit{no-slip boundary condition}, the \textit{periodic boundary condition}, and the flow in the whole space.
For the compressible Navier--Stokes equations, the existence of weak solutions under the no-slip boundary condition was extended to the more physically relevant case $\gamma >3/2$ in \cite{FNP01}.
Furthermore, the existence theory for weak solutions to the full compressible Navier--Stokes--Fourier system, including heat conduction, was developed in \cite{FN09}. In that work, both the no-slip boundary condition and the \textit{complete slip boundary condition} in which the normal component of the velocity vanishes on the boundary are treated.
The no-slip boundary condition has long been the most standard and widely used condition for both incompressible and compressible flows. However, several paradoxical phenomena have been pointed out \cite{Hes04,Hil07}, and consequently, nonstandard boundary conditions have recently attracted growing interest.
Among these, the \textit{Navier slip boundary condition} is often regarded as a physically reasonable model, allowing for tangential slip proportional to the tangential stress.
For the incompressible Navier--Stokes equations, the literature is already vast; see, for example, \cite{ATACG21}. In the compressible case, the existence of weak solutions under the Navier slip boundary condition in time-dependent domains was obtained in \cite{FKNN13}, while the full system was treated in \cite{KMNWK18}. 
Other types of boundary conditions allowing for tangential slip when the tangential component of the stress is sufficiently large include the \textit{slip boundary condition of friction type} and the \textit{Coulomb friction law boundary condition}. 
The former was introduced in \cite{Fuj94, FKS95} for the stationary Stokes and Navier--Stokes equations. For the incompressible Navier--Stokes equations, the global-in-time existence of strong solutions in two dimensions and the local-in-time existence in three dimensions were shown in \cite{Kas12}. In the compressible setting, the global-in-time existence of weak solutions was recently established in \cite{NOS23}.
For the latter type, the existence of weak solutions in the incompressible case was proven in \cite{BST17}, and in the compressible case, the existence of dissipative solutions was obtained in \cite{NOS25}.

The theory of weak solutions for compressible viscous fluids in probabilistic settings was first established in \cite{BH16}, where the existence of finite energy weak martingale solutions was shown in a three-dimensional torus with periodic boundary conditions. See also \cite{BFHbook18}.
Subsequently, \cite{Smi17} independently established the existence of weak martingale solutions under the no-slip boundary condition on smooth bounded domains. 
For extensions to the full Navier--Stokes--Fourier system we refer to \cite{BF20}, and for the case of the whole space to \cite{DMM21}.
In contrast, to the best of our knowledge, no general results are available concerning weak solutions under slip-type boundary conditions in the probabilistic setting. This gap in the literature motivates the present study. 

The main purpose of this paper is to develop the weak solution theory for a compressible viscous fluid subject to stochastic forces under a \textit{slip boundary condition of friction type} on bounded domains with $C^{2+\nu}$-boundary.
We establish the existence of a global-in-time finite energy weak solution, which is both analytically and probabilistically weak. Moreover, as discussed in \cite{NOS23}, we prove that any sufficiently regular weak solution become a solution in the strong sense.
As an application of the main result, we further introduce a suitable weak formulation for the \textit{Navier slip boundary condition} and prove the existence of weak solutions in this setting as well.

\subsection{Notation}
\begin{description}[labelwidth=7em,itemsep=0pt]
	\item[$A:B$]  Scalar product $\sum_{ij}A_{ij}B_{ij}$ between two matrices $A,B$
	\item[$W_{\rn}^{k,p}(D)$] Sobolev functions with zero normal trace in $D$
	\item[{$[X,w]$}] Banach space $X$ equipped with the weak topology 
	\item[{$C_w([0,T];X)$}] Continuous functions with values in $[X,w]$ 
	\item[{$C^{k+\alpha}([0,T];X)$}] $X$-valued $k$-times continuously differentiable functions with $\alpha$-H\"older continuous derivatives  
	\item[$L_t^pL_x^q$] Abbreviation for $L^p(0,T;L^q(D))$ 
	\item[$(\cdot,\cdot)_{H}$] Inner product on $H$
	\item[$\ev{\cdot,\cdot}_{X^\ast,X}$] Duality pairing between $X^\ast$ and $X$  
\end{description}

\subsection{Weak formulation and main result}
Weak solutions to the Navier--Stokes equations under the boundary condition \eqref{eq:1.4}--\eqref{eq:1.5} have traditionally been formulated as satisfying a certain variational inequality, which arises from the nonlinearity of the boundary condition. Following the spirit of \cite{NOS23}, we incorporate the boundary condition into a variational inequality that combines the momentum equation and the energy inequality. 
However, in a probabilistic setting, this formulation is not sufficient. The difficulty stems from the low regularity of stochastic integrals and the fact that such a formulation does not provide a representation of the momentum $\rho u$ as a semimartingale. Therefore, we additionally include in the definition a form of the momentum equation that is unaffected by the boundary condition.
The precise formulation is as follows.
\begin{definition}[Dissipative martingale solution]
	\label{Def:1.1}
	Let $\Lambda$ be a Borel probability measure on $L^1(D)\times L^1(D)$ and let $g\in L^2((0,T)\times \pD)$ satisfies $g\geq 0$.
	The quantity $((\Omega,\FFF,(\FFF_t)_{t\geq 0},\PP),\rho,u,W)$ is called a \textit{dissipative martingale solution} to the Navier--Stokes system \eqref{eq:1.1}--\eqref{eq:1.5} with the initial law $\Lambda$ if:
	\begin{itemize}
		\item[(1)] $(\Omega,\FFF,(\FFF_t)_{t\geq 0},\PP)$ is a stochastic basis with a complete right-continuous filtration;
		\item[(2)] $W$ is a cylindrical $(\FFF_t)$-Wiener process on $\fU$;
		\item[(3)] the density $\rho$ is an $(\FFF_t)$-progressively measurable stochastic process such that 
		$$
			\rho \geq 0,\quad \rho\in C_w([0,T];L^\gamma(D))\quad \PP\text{-a.s.;}
		$$
		\item[(4)] the velocity $u$ is an $(\FFF_t)$-adapted random distribution such that  
		$$
			u\in L^2(0,T;W_{\rn }^{1,2}(D))\quad \PP\text{-a.s.;}
		$$ 
		\item[(5)] the momentum $\rho u$ is an $(\FFF_t)$-progressively measurable stochastic process such that 
		$$
			\rho u\in C_w([0,T];L^{\frac{2\gamma}{\gamma + 1}}(D))	\quad \PP\text{-a.s.;}
		$$ 
		\item[(6)] there exists an $L^1(D)\times L^1(D)$-valued $\FFF_0$-measurable random variable $(\rho_0,u_0)$ such that $\rho_0 u_0\in L^1(D)$ $\PP$-a.s., $\Lambda = \LL[\rho_0,\rho_0 u_0]$ and $(\rho_0,\rho_0 u_0) = (\rho(0,\cdot),\rho u(0,\cdot))$ $\PP$-a.s.;
		\item[(7)] the equation of continuity 
		\begin{equation}
			\label{eq:1.8}
			-\intT \pp_t\phi \int_D \rho \psi dxdt = \phi(0)\int_D \rho_0\psi dx + \intT \phi \int_D \rho u\cdot \nabla\psi dxdt
		\end{equation}
		holds for all $\phi\in C_c^\infty(\ico{0}{T})$ and all $\psi\in C_c^\infty\left(\RR^3\right)$ $\PP$-a.s.;
		\item[(8)] the interior momentum equation 
		\begin{align}
			&-\intT \pp_t \phi\int_D \rho u\cdot \bphi dxdt-\phi(0)\int_D \rho_0 u_0 \cdot \bphi dx  \notag\\
			&\quad = \intT \phi\int_D[\rho u\otimes u:\nabla\bphi + p(\rho)\rdiv\bphi]dxdt - \intT \phi\int_D \SSS(\nabla u):\nabla\bphi dxdt + \intT \phi\int_D \GG(\rho,\rho u)\cdot \bphi dx dW \label{eq:1.9}
		\end{align}
		holds for all $\phi\in C_c^\infty \left(\ico{0}{T}\right)$ and all $\bphi\in C_c^\infty(D)$ $\PP$-a.s.;
		\item[(9)] the momentum and energy inequality 
		\begin{align}
		   &\int_D \left[\frac12\rho_0 |u_0|^2 + P(\rho_0)\right]dx - \int_D\left[\frac12 \rho \left|u \right|^2 + P(\rho)\right](\tau)dx -\inttau \pp_t\phi \int_D\rho u\cdot \bphi dxdt - \phi(0)\int_D \rho_0 u_0 \cdot \bphi dx  \notag\\
		   &\quad -\inttau \phi \int_D(\rho u\otimes u):\nabla \bphi dx dt- \inttau \phi\int_D p(\rho)\rdiv \bphi dx dt + \inttau \int_D \SSS(\nabla u):\nabla(\phi \bphi-u)dx dt   \notag\\ 
		   &\quad +\frac12 \sum_{k=1}^{\infty}\inttau \int_D \rho^{-1}\left|G_k(\rho,\rho u) \right|^2dx dt - \inttau \int_D \GG(\rho,\rho u)\cdot (\phi\bphi - u)dxdW \notag\\ 
		   &\quad + \inttau \int_{\pD} g\left|\phi\bphi \right|-g\left|u \right|d\Gamma dt\geq 0 \label{eq:1.10}
		\end{align}
		holds for all $\tau\in [0,T],$ for all $ \phi\in C_c^\infty(\ico{0}{\tau})$ and for all $\bphi\in C^\infty(\overline{D})$ with $\bphi\cdot \rn|_{\pD}=0$ $\PP$-a.s., where the pressure potential $P$ is given by 
		$$
			P(\rho) = \rho \int_1^\rho \frac{p(z)}{z^2}dz;
		$$ 
		\item[(10)] if $b\in C^1(\RR)$ such that $b'(z)=0$ for all $z\geq M_b$, then, for all $\phi\in C_c^\infty(\ico{0}{T})$ and all $\psi\in C_c^\infty(\RR^3)$, we have $\PP$-a.s. 
		\begin{align*}
			-\intT \pp_t \phi \int_D b(\rho)\psi dxdt &= \phi(0) \int_D b(\rho_0)\psi dx + \intT \phi\int_D b(\rho)u\cdot \nabla \psi dxdt \\ 
			&\quad - \intT \phi \int_D \left(b'(\rho) \rho - b(\rho) \right) \rdiv u \psi dx dt.
		\end{align*}
	\end{itemize}
\end{definition}
\begin{remark}
	We implicitly assume that all stochastic integrals appearing in Definition~\ref{Def:1.1} are well-defined. If $\gamma>\frac32$, they are well-defined thanks to the regularity specified in Definition \ref{Def:1.1} (3)--(5). Indeed, since 
	$$
		\norm{G_k(\rho,\rho u)}_{L^2_t(W^{b,2}_x)^\ast}^2 \lesssim g_k^2 \norm{G_k(\rho,\rho u)}_{L^2_tL^1_x}^2  \lesssim g_k^2\sup_{\tau\in [0,T]}\left(\norm{\rho}_{L^1_x}^2 + \norm{\rho u}_{L^1_x}^2\right)\quad \PP\text{-a.s,}
	$$
	for $b>\frac32$, $\GG$ is an $L_2(\fU,(W^{b,2}(D))^\ast)$-valued stochastically integrable process. Moreover, Lemma \ref{lem:prog_m'ble_r.d.} guarantees the existence of a progressively measurable stochastic process belonging to the same equivalence class as $u$, and using the embedding $W^{1,2}(D)\hookrightarrow L^6(D)$, we have $\PP$-a.s.,
	\begin{align*}
		\intT \left|\int_D G_k(\rho,\rho u)\cdot u dx \right|^2 dt &\lesssim g_k^2 \left(\sup_{\tau\in [0,T]}\norm{\rho u}_{L^1_x}^2 + \intT \norm{\rho u}_{L_x^{\frac{2\gamma }{\gamma + 1}}}^2\norm{u}_{W_x^{1,2}}^2dt \right) \\ 
		&\leq  g_k^2 \left(\sup_{\tau\in [0,T]}\norm{\rho u}_{L^1_x}^2 + \sup_{\tau\in [0,T]}\norm{\rho u}_{L_x^{\frac{2\gamma }{\gamma + 1}}}^2 \intT \norm{u}_{W_x^{1,2}}^2dt \right).
	\end{align*}
	Thus, for a progressively measurable representative, $\int_D \GG(\rho,\rho u)\cdot udx$ is an $L_2(\fU,\RR)$-valued stochastically integrable process, and the stochastic integral $\inttau \int_D \GG(\rho,\rho u)\cdot u dx dW$ is given by 
	\begin{align*}
		&\inttau \int_D \GG(\rho,\rho u)\cdot u dx dW = \sum_{k=1}^{\infty} \inttau \int_D G_k(\rho,\rho u)\cdot u dx dW_k.
	\end{align*}
\end{remark}
The main result of this paper is as follows.
\begin{theorem}
	\label{thm:1.3}
	Let $T>0$ and let $D\subset \RR^3$ be a bounded domain of class $C^{2+\nu}$ for some $\nu>0$. Let $\gamma>\frac32$ and let $\Lambda$ be a Borel probability measure defined on $L^1(D)\times L^1(D)$ such that 
	\begin{equation*}
		\Lambda\left\{ \rho\geq 0 \right\}=1,\quad \Lambda\left\{ \underline{\rho}\leq \int_D \rho dx \leq \overline{\rho} \right\}=1,
	\end{equation*}
	for some deterministic constants $\urho,\orho>0$ and 
	\begin{equation*}
		\int_{L^1_x\times L^1_x}\left|\int_D \left[\frac12\frac{|q|^2}{\rho}+P(\rho)\right]dx \right|^rd\Lambda(\rho,q) <\infty,
	\end{equation*}
	for some $r\geq 4$.
	Assume that the diffusion coefficients $\GG=\GG(x,\rho,q)$ are continuously differentiable and satisfy \eqref{eq:1.6}--\eqref{eq:1.7}, and the modulus of friction $g\in L^2((0,T)\times \pD)$ satisfies the conditions 
	\begin{align*}
		&g\geq 0,\quad \intT \int_{\pD}  g d\Gamma dt >0.
	\end{align*}
	Then there exists a dissipative martingale solution to \eqref{eq:1.1}--\eqref{eq:1.5} in the sense of Definition \ref{Def:1.1}.
	Furthermore, this solution satisfies the following:
	\begin{align*}
		&\EE\left[\sup_{t\in [0,T]}\norm{\rho(t)}_{L^\gamma(D)}^{\gamma r}\right] + \EE\left[\norm{u}_{L^2(0,T;W^{1,2}(D))}^r \right]< \infty,\\ 
		&\EE\left[\sup_{t\in [0,T]}\norm{\rho |u|^2(t)}_{L^1(D)}^r\right] + \EE\left[\sup_{t\in [0,T]}\norm{\rho u (t)}_{L^{\frac{2\gamma}{\gamma +1}}(D)}^{\frac{2\gamma}{\gamma +1}r}\right]<\infty.
	\end{align*}
\end{theorem}
The proof of the main theorem is presented in Sections \ref{sec:Approximation system}--\ref{sec:Vanishing artificial pressure limit}.
Our approach in the proof is based on the four-layer approximation scheme from \cite{BFHbook18} for periodic boundary conditions, and combines convex approximation techniques used in \cite{NOS23} to approximate the variational inequality containing the barrier term for the boundary conditions.
As will be described in Section \ref{sec:Outline of the proof of the main result}, the convex approximation layer is introduced after the loss of deterministic boundedness of the velocity. Consequently, the convergence of the velocity in the convex approximation is ensured only in the weak sense, and arguments similar to those in \cite{NOS23} cannot be applied. 
Therefore, we need to reconsider the appropriate path space.

Next, when allowing degrees of freedom for the boundary trace of the velocity in the probabilistic setting, several difficulties arise.
The first is that when considering finite-dimensional approximations of the momentum equation under slip boundary conditions, the available properties of the orthogonal projection are significantly weakened.
In particular, dimension-independent $L^p$ continuity and convergence in Sobolev spaces become non-trivial. Existing approximation methods in the probabilistic setting rely heavily on these properties.
Therefore, we consider refining the conventional approximation of the noise coefficient, and, by combining the ideas used in \cite{Smi17} when passing to the infinite-dimensional limit, we establish a new approach that avoids the use of projection operators.
More precisely, we approximate the diffusion coefficient $\GG(\rho,\rho u)$ by a uniformly bounded regularized family $(\GG_\varepsilon(\rho,\rho u))_{\varepsilon\in (0,1)}$ and by its finite-dimensional approximations $([\GG_\varepsilon(\rho,\rho u)]_m^\ast)_{(\varepsilon,m)\in (0,1)\times \NN}$, using only the density of $\bigcup_{m\in\NN}V_m$ in $W^{1,p}_{\rn}(D)$, where $[\cdot]_m^\ast$ denotes the orthogonal projection onto $V_m^\ast$, and all notation appearing here is introduced in Section \ref{sec:Outline of the proof of the main result}.   
One of our main novelties is that we establish the well-posedness of the first approximation problem in this setting and derive effective energy estimates (see Sections \ref{sec:Outline of the proof of the main result}--\ref{sec:Well-posedness of the basic approximate problem}).
In particular, although the It\^o correction term appearing in the energy equality takes a complicated form involving operators, we verify that it does not cause any difficulty in inequality estimates (see Proposition \ref{prop:energy_balance} and Remark \ref{rem:quad_var_est}).

The second difficulty arises when deriving a priori estimates for the velocity. Usually, a priori estimates for the velocity are derived from the boundedness of the energy via a generalized Korn--Poincar\'e inequality.
In the case of no-slip boundary conditions, this becomes a straightforward problem, since the viscous stress term in the energy can completely control the velocity. In the case of slip boundary conditions, in deterministic settings, the density can be used as an effective weight, making the approach of combining the viscous stress term with the momentum to control velocity the standard method (see \cite{FN09}). 
However, in probabilistic settings, due to the randomness of the density, such an approach cannot be directly applied. Therefore, we address this issue by further generalizing the Korn--Poincar\'e inequality and treating the modulus of friction as the weight.
The conditions on the modulus of friction $g$ in Theorem \ref{thm:1.3} are technical assumptions required for this purpose (see Corollary \ref{cor:est_u_R}, Lemma \ref{lem:korn_poincare}).

Finally, for the artificial viscosity approximation and the artificial pressure approximation, it is necessary to verify that the method based on regularity of the \textit{effective viscous flux}, when deriving strong convergence of the density, is consistent with the probabilistic setting even under slip boundary conditions (see Sections \ref{sec:Vanishing viscosity limit}--\ref{sec:Vanishing artificial pressure limit}). 
We follow the method presented in \cite{FN09} for deterministic settings under complete slip boundary conditions and establish the appropriate effective viscous flux identity.

Further discussion of the ideas of the proof is provided in Section \ref{sec:Outline of the proof of the main result}.

\subsection{Equivalence of definitions}
In this section, we prove that a sufficiently regular solution to the variational formulation of \eqref{eq:1.1}--\eqref{eq:1.5} in the sense of Definition \ref{Def:1.1} coincides with an analytically strong solution to \eqref{eq:1.1}--\eqref{eq:1.5}. 
Before proceeding to the main topic, we briefly discuss the existence of such solutions.
For the Navier--Stokes equations, it is well-known that strong solutions exist globally in time for sufficiently small data, or locally in time for general data. 
For recent developments, in the deterministic setting, we refer to, for instance, \cite{KNP20}. In the probabilistic framework, the local-in-time existence and uniqueness of strong solutions were established in \cite{BFH18}. 
On the other hand, under slip boundary conditions, no general results are currently available in the probabilistic case, and the existence of strong solutions under slip boundary condition of friction type remains an open problem in both deterministic and probabilistic settings.

First, we introduce the definition of a solution to \eqref{eq:1.1}--\eqref{eq:1.5} that is analytically strong but probabilistically weak.
\begin{definition}
	\label{def:strong martingale solution}
	Let $\Lambda$ be a Borel probability measure on $W^{s,2}(D)\times W^{s,2}(D)$, let $s\geq 1$ and let $g\in L^2((0,T)\times \pD)$ with $g\geq 0$. 
	A multiplet 
	$$
		(\stochbasis, \rho,u,W)
	$$
	is called a \textit{strong martingale solution} to the system \eqref{eq:1.1}--\eqref{eq:1.5} with the initial law $\Lambda$, if:
	\begin{itemize}
		\item[(1)] $(\Omega,\FFF,(\FFF_t)_{t\geq 0},\PP)$ is a stochastic basis with a complete right-continuous filtration;
		\item[(2)] $W$ is a cylindrical $(\FFF_t)$-Wiener process on $\fU$;
		\item[(3)] the density is a $W^{s,2}(D)$-valued $(\FFF_t)$-progressively measurable stochastic process such that
		$$
			\rho\geq 0,\quad \rho\in C([0,T];W^{s,2}(D))\quad \PP\text{-a.s.;}
		$$
		\item[(4)] the velocity $u$ is a $W^{s,2}(D)$-valued $(\FFF_t)$-progressively measurable stochastic process such that 
		$$
			u\in C([0,T];W^{s,2}(D))\cap L^2(0,T;W^{s+1,2}(D))\quad u\cdot \rn |_{\pD} = 0\quad  \PP\text{-a.s.;}
		$$
		\item[(5)] there exists a $W^{s,2}(D)\times W^{s,2}(D)$-valued $\FFF_0$-measurable random variable $(\rho_0,u_0)$ such that $\Lambda = \LL[\rho_0, u_0]$;
		\item[(6)] the equation of continuity 
		\begin{equation}
			\rho(\tau) = \rho_0 - \inttau \rdiv (\rho u)dt \label{eq:1.11}
		\end{equation}
		holds for all $\tau\in [0,T]$ $\PP$-a.s.;
		\item[(7)] the momentum equation 
		\begin{align}
			\rho u(\tau) = \rho_0 u_0 - \inttau \rdiv (\rho u\otimes u)dt + \inttau \rdiv \SSS(\nabla u)dt - \inttau \nabla p(\rho)dt + \inttau \GG(\rho,\rho u)dW \label{eq:1.12}
		\end{align}
		holds for all $\tau\in [0,T]$ $\PP$-a.s.;
		\item[(8)] the boundary condition of friction type 
		\begin{align}
			\left|\left(\TT \rn\right)_\tau \right| \leq g,\quad \left(\TT \rn\right)_\tau \cdot u_\tau + g\left|u_\tau \right| = 0 \label{eq:1.13}
		\end{align}
		is satisfied a.e. on $(0,T)\times \pD$ $\PP$-a.s.
	\end{itemize}
\end{definition}
\begin{proposition}
	\label{prop:1.5}
	Let $((\Omega,\FFF,(\FFF_t)_{t\geq 0},\PP),\rho,u,W)$ be a solution in the sense of Definition \ref{def:strong martingale solution} with $s>3/2$, which satisfies
	$$
		\rho\geq \urho\quad \PP\text{-a.s.,}
	$$
	for some $\urho>0$. Then the multiplet is also a solution in the sense of Definition \ref{Def:1.1}. 
	Conversely, if $((\Omega,\FFF,(\FFF_t)_{t\geq 0},\PP),\rho,u,W)$ is a solution in the sense of Definition \ref{Def:1.1} such that  
	\begin{align}
		&\rho\geq\urho,\quad \rho\in C([0,T];W^{s,2}(D))\quad \PP\text{-a.s.,} \label{eq:1.14}\\ 
		&u\in C([0,T];W^{s,2}(D))\cap L^2(0,T;W^{s+1,2}(D))\quad  \PP\text{-a.s.,}\label{eq:1.15}
	\end{align}
	for some $\urho>0$ and $s>3/2$, 
	then it is also a solution in the sense of Definition \ref{def:strong martingale solution}. 
\end{proposition}
\begin{remark}
	\begin{itemize}
		\item[1.] The essential part of the proof lies in deriving the energy equality using the lower bound of the density $\rho$. Proposition \ref{prop:1.5} asserts the equivalence between Definitions \ref{Def:1.1} and \ref{def:strong martingale solution} for sufficiently regular $(\rho,u)$ under this assumption.
		\item[2.] 
		If $\rho,\rho u$ are $L^2(D)$-valued $(\FFF_t)$-progressively measurable stochastic processes such that 
		$$
			\rho\in L^2(0,T;L^2(D)),\quad \rho u\in L^2(0,T;L^2(D))\quad \PP\text{-a.s.,}
		$$
		and $\GG$ satisfies \eqref{eq:1.6}--\eqref{eq:1.7}, then the stochastic integral 
		$$
			\inttau \GG(\rho,\rho u)dW = \sum_{k=1}^{\infty}\inttau G_k(\cdot,\rho,\rho u)dW_k
		$$
		is well-defined as an $(\FFF_t)$-local martingale in $L^2(D)$. 
		\item[3.] The assumption $s>3/2$ is required to derive the energy equality. Under this assumption, the momentum equation \eqref{eq:1.12} holds on $L^2(D)$. 
	\end{itemize}
\end{remark}
\begin{proof}[Proof of Proposition \ref{prop:1.5}]
	Let $(\rho,u)$ be a solution in the sense of Definition \ref{def:strong martingale solution}. By the momentum equation \eqref{eq:1.12}, we have 
	\begin{align}
		&-\inttau \pp_t \phi\int_D \rho u\cdot \bphi dxdt-\phi(0)\int_D \rho_0 u_0 \cdot \bphi dx  \notag\\
		&\quad = \inttau \phi\int_D[\rho u\otimes u:\nabla\bphi + p(\rho)\rdiv\bphi]dxdt - \inttau \phi\int_D \SSS(\nabla u):\nabla\bphi dxdt + \inttau \phi\int_D \GG(\rho,\rho u)\cdot \bphi dx dW  \notag \\ 
		&\quad \quad + \inttau \phi \int_{\pD} \left(\TT\rn\right)_\tau \cdot \bphi d\Gamma dt, \label{eq:1.16}
	\end{align}
	for all $\phi\in C_c^\infty \left(\ico{0}{\tau}\right)$ and all $\bphi\in C^\infty(\oD)$ with $\bphi\cdot \rn|_{\pD} = 0$ $\PP$-a.s. Next, to derive the energy balance, we introduce the mapping
	\begin{align*}
		&\MMM[\rho]:L^2(D)\to L^2(D), \quad \MMM[\rho](u) := \rho u  
	\end{align*}
	for $\rho\in W^{s,2}(D)$ with $\rho\geq \urho$, which is invertible and possesses the following properties:
	\begin{align*}
		&\norm{\MMM[\rho_1] - \MMM[\rho_2]}_{\LL(L^2(D))} \leq \norm{\rho_1 - \rho_2}_{L^\infty(D)},\\ 
		&\norm{\MMM^{-1}[\rho]}_{\LL(L^2(D))}\leq \urho^{-1}, \\ 
		&\norm{\MMM^{-1}[\rho_1]-\MMM^{-1}[\rho_2]}_{\LL(L^2(D))}\leq c(\urho) \norm{\rho_1-\rho_2}_{L^\infty(D)},  
	\end{align*}	
	for any $\rho,\rho_1,\rho_2\geq \urho$. Note that $W^{s,2}(D)\hookrightarrow L^\infty(D)$ for $s>3/2$. Fixing a cut-off function $\chi_c\in C^\infty(\RR)$ such that 
    \begin{align*}
        &\chi_c(x) = x\quad \text{for } c\leq x\leq \tfrac 1 c, \\ 
        &\chi_c \subset \left[\tfrac c 2, \tfrac{2}{c}\right], \\
        &\chi_c \ \text{is monotonically increasing},
    \end{align*}
    we define the functional $\tilde{f}_c:W^{s,2}(D)\times L^2(D)\to \RR$, 
    $$
        \tilde{f}_c[\rho,q] := \frac 1 2  \int_D q \MMM^{-1}[\chi_c(\rho)]q dx = \frac 1 2 \inner{q}{\MMM^{-1}[\chi_c(\rho)]q}.
    $$
	Applying It\^{o}'s formula to $\tilde{f}_c[\rho, \rho u]$ using \eqref{eq:1.11} and \eqref{eq:1.12} and then letting $c\to 0$ yield the energy balance 
    \begin{align}
        &\int_D \left[\frac 1 2 \rho |u|^2 + P(\rho)\right](\tau)dx + \inttau \int_D \SSS(\nabla u):\nabla u dxdt  \notag\\ 
        &\quad = \int_D \left[\frac 1 2 \rho_0 |u_0|^2 + P(\rho_0)\right]dx +\frac 1 2 \sum_{k=1}^{\infty} \inttau \int_D \rho^{-1}\left|G_k(\rho,\rho u) \right|^2 dxdt \notag \\ 
        &\quad \quad + \inttau \int_D \GG(\rho,\rho u) \cdot u dx dW + \inttau \int_{\pD} \left(\TT \rn \right)_\tau \cdot u d\Gamma dt, \label{eq:1.17}
    \end{align}
	for all $\tau\in [0,T]$ $\PP$-a.s. For the detailed proof, see Proposition \ref{prop:energy_balance}. 
	Note that, in contrast to the proof of Proposition \ref{prop:energy_balance}, $L^2(D)$ is replaced by $W^{s,2}(D)$ to ensure the existence of $\pp_\rho \tilde{f}_c$.

	On the other hand, using the boundary condition \eqref{eq:1.13} and $u\cdot \rn|_{\pD} = 0$ yield
	\begin{align*}
		&\inttau \int_{\pD}  \left(\TT \rn \right)_\tau \cdot [\phi\bphi - u] d\Gamma dt + \inttau \int_{\pD} g\left|\phi\bphi \right| - g\left|u \right| d\Gamma dt \\ 
		&\quad = \inttau \int_{\pD}  \left(\TT \rn \right)_\tau \cdot \phi\bphi d\Gamma dt + \inttau \int_{\pD} g\left|\phi\bphi \right|d\Gamma dt \\ 
		&\quad \geq 0.
	\end{align*}
	Therefore, subtracting \eqref{eq:1.17} from \eqref{eq:1.16} yields the variational inequality \eqref{eq:1.10}. 
	
	Conversely, if $(\rho,u)$ is a solution in the sense of Definition \ref{Def:1.1} satisfying \eqref{eq:1.14}--\eqref{eq:1.15} for some $\urho>0$ and $s>3/2$, then $(\rho,u)$ satisfies \eqref{eq:1.11} and \eqref{eq:1.12} in the strong sense. 
	Thus, by the same argument as above, $(\rho,u)$ satisfies the momentum equation \eqref{eq:1.16} and the energy balance \eqref{eq:1.17}. Subtracting \eqref{eq:1.17} from \eqref{eq:1.16} and then subtracting this result from the momentum and energy inequality \eqref{eq:1.10} yields the following inequality:
	\begin{align*}
		\inttau \int_{\pD} g\left|\phi\bphi \right| -g\left|u \right|d\Gamma dt \geq -\inttau \int_{\pD} \left(\TT \rn \right)_\tau \cdot [\phi\bphi - u] d\Gamma dt
	\end{align*}
	for all $\tau\in [0,T]$, all $\phi\in C_c(\ico{0}{\tau})$, and all $\bphi\in C^\infty(\oD)$ with $\bphi\cdot \rn |_{\pD} = 0$ $\PP$-a.s. By a density argument, we also have 
	\begin{align}
		\inttau \int_{\pD} g\left|\psi \right| -g\left|u \right|d\Gamma dt \geq -\inttau \int_{\pD} \left(\TT \rn \right)_\tau \cdot [\psi - u] d\Gamma dt \label{eq:1.18}
	\end{align}
	for all $\tau\in [0,T]$, and all $\psi\in L^2(0,\tau;W_{\rn}^{1,2}(D))$ $\PP$-a.s. 
	Choosing $\tau =T$, and testing \eqref{eq:1.18} by $u\pm \psi$, we obtain that 
	\begin{align*}
		\left|\intT\int_{\pD} \left(\TT \rn \right)_\tau \cdot\psi d\Gamma dt \right|\leq \intT \int_{\pD} g\left|\psi \right|d\Gamma dt.
	\end{align*}
	This implies 
	\begin{align}
		\left|\left(\TT \rn \right)_\tau \right| \leq g\quad \text{a.e. on } (0,T)\times \pD\quad \PP\text{-a.s.}\label{eq:1.19}
	\end{align}
	Choosing $\psi= 0$ in \eqref{eq:1.18}, together with \eqref{eq:1.19} yields 
	$$
		\left(\TT \rn \right)_\tau\cdot u + g\left|u \right| = 0\quad \text{a.e. on } (0,T)\times \pD\quad \PP\text{-a.s.}
	$$
	This completes the proof of Proposition~\ref{prop:1.5}.
\end{proof}

\subsection{Application to other slip boundary conditions}
In this section, we consider the slip boundary condition 
\begin{align}
	\label{eq:1.20}
	&(\SSS(\nabla u) \rn)_\tau  +g u_\tau = 0,\quad g\geq 0\quad \mathrm{on}\ (0,T)\times \pp D
\end{align}
proposed by Navier (see \cite{Nav23}). By slightly modifying the proof of Theorem \ref{thm:1.3}, we can prove the existence of dissipative martingale solutions for this boundary condition.
First, we provide the precise weak formulation to \eqref{eq:1.1}--\eqref{eq:1.4} and \eqref{eq:1.20}.
\begin{definition}
	\label{def:1.7}
	Let $\Lambda$ be a Borel probability measure on $L^1(D)\times L^1(D)$ and let $g\in L^\infty((0,T)\times \pD)$ satisfies $g\geq 0$.
	The quantity $((\Omega,\FFF,(\FFF_t)_{t\geq 0},\PP),\rho,u,W)$ is called a \textit{dissipative martingale solution} to the Navier--Stokes system \eqref{eq:1.1}--\eqref{eq:1.4} and \eqref{eq:1.20} with the initial law $\Lambda$ if:
	\begin{itemize}
		\item[(1)] $(\Omega,\FFF,(\FFF_t)_{t\geq 0},\PP)$ is a stochastic basis with a complete right-continuous filtration;
		\item[(2)] $W$ is a cylindrical $(\FFF_t)$-Wiener process on $\fU$;
		\item[(3)] the density $\rho$ is an $(\FFF_t)$-progressively measurable stochastic process such that 
		$$
			\rho \geq 0,\quad \rho\in C_w([0,T];L^\gamma(D))\quad \PP\text{-a.s.;}
		$$
		\item[(4)] the velocity $u$ is an $(\FFF_t)$-adapted random distribution such that  
		$$
			u\in L^2(0,T;W_{\rn }^{1,2}(D))\quad \PP\text{-a.s.;}
		$$ 
		\item[(5)] the momentum $\rho u$ is an $(\FFF_t)$-progressively measurable stochastic process such that 
		$$
			\rho u\in C_w([0,T];L^{\frac{2\gamma}{\gamma + 1}}(D))	\quad \PP\text{-a.s.;}
		$$ 
		\item[(6)] there exists an $L^1(D)\times L^1(D)$-valued $\FFF_0$-measurable random variable $(\rho_0,u_0)$ such that $\rho_0 u_0\in L^1(D)$ $\PP$-a.s., $\Lambda = \LL[\rho_0,\rho_0 u_0]$ and $(\rho_0,\rho_0 u_0) = (\rho(0,\cdot),\rho u(0,\cdot))$ $\PP$-a.s.;
		\item[(7)] the equation of continuity 
		\begin{equation}
			-\intT \pp_t\phi \int_D \rho \psi dxdt = \phi(0)\int_D \rho_0\psi dx + \intT \phi \int_D \rho u\cdot \nabla\psi dxdt,
		\end{equation}
		holds for all $\phi\in C_c^\infty(\ico{0}{T})$ and all $\psi\in C_c^\infty\left(\RR^3\right)$ $\PP$-a.s.;
		\item[(8)] the momentum equation 
		\begin{align}
			&-\intT \pp_t \phi\int_D \rho u\cdot \bphi dxdt-\phi(0)\int_D \rho_0 u_0 \cdot \bphi dx  \notag\\
			&\quad = \intT \phi\int_D[\rho u\otimes u:\nabla\bphi + p(\rho)\rdiv\bphi]dxdt - \intT \phi\int_D \SSS(\nabla u):\nabla\bphi dxdt -\intT\phi\int_{\pD}g u\cdot \bphi d\Gamma dt \notag \\ 
			&\quad \quad +\intT \phi\int_D \GG(\rho,\rho u)\cdot \bphi dx dW 
		\end{align}
		holds for all $\phi\in C_c^\infty \left(\ico{0}{T}\right)$ and all $\bphi\in C^\infty(\overline{D})$ with $\bphi\cdot \rn|_{\pD}=0$ $\PP$-a.s.;
		\item[(9)] the energy inequality 
		\begin{align}
			&-\intT \pp_t \phi \int_D \left[\frac12 \rho \left|u \right|^2 + P(\rho)\right]dx dt + \intT \phi \int_D \SSS(\nabla u):\nabla u dx dt + \intT \phi \int_{\pD} g|u|^2 d\Gamma dt\notag \\ 
			&\quad \leq \phi(0)\int_D \left[\frac12 \rho_0 \left|u_0 \right|^2 + P(\rho_0)\right]dx + \frac12 \sum_{k=1}^{\infty}\intT \phi\int_D \rho^{-1} \left|G_k(\rho,\rho u) \right|^2 dx dt \notag \\ 
			&\quad \quad +\intT \phi \int_D \GG(\rho,\rho u)\cdot u dx dW  \label{eq:1.23}
		\end{align}
		holds for all $\phi\in C_c^\infty(\ico{0}{T})$ and for all $\bphi\in C^\infty(\overline{D})$ with $\bphi\cdot \rn|_{\pD}=0$ $\PP$-a.s.; 
		\item[(10)] if $b\in C^1(\RR)$ such that $b'(z)=0$ for all $z\geq M_b$, then, for all $\phi\in C_c^\infty(\ico{0}{T})$ and all $\psi\in C_c^\infty(\RR^3)$, we have $\PP$-a.s. 
		\begin{align*}
			-\intT \pp_t \phi \int_D b(\rho)\psi dxdt &= \phi(0) \int_D b(\rho_0)\psi dx + \intT \phi\int_D b(\rho)u\cdot \nabla \psi dxdt \\ 
			&\quad - \intT \phi \int_D \left(b'(\rho) \rho - b(\rho) \right) \rdiv u \psi dx dt.
		\end{align*}
	\end{itemize}
\end{definition}
\begin{theorem}
	\label{thm:1.8}
	Under the same conditions as Theorem \ref{thm:1.3}, we further assume that  $g\in L^\infty((0,T)\times \pD)$. Then there exists a dissipative martingale solution to \eqref{eq:1.1}--\eqref{eq:1.4} and \eqref{eq:1.20} in the sense of Definition \ref{def:1.7}.
	Furthermore, this solution satisfies the following:
	\begin{align*}
		&\EE\left[\sup_{t\in [0,T]}\norm{\rho(t)}_{L^\gamma(D)}^{\gamma r}\right] + \EE\left[\norm{u}_{L^2(0,T;W^{1,2}(D))}^r \right]< \infty,\\ 
		&\EE\left[\sup_{t\in [0,T]}\norm{\rho |u|^2(t)}_{L^1(D)}^r\right] + \EE\left[\sup_{t\in [0,T]}\norm{\rho u (t)}_{L^{\frac{2\gamma}{\gamma +1}}(D)}^{\frac{2\gamma}{\gamma +1}r}\right]<\infty.
	\end{align*}
\end{theorem}
The proof of Theorem \ref{thm:1.8} is given in Appendix \ref{sec:proof of Theorem 1.8}.

\section{Preliminaries}
\subsection{Elements of functional analysis}
Similarly to the proof of \cite[Theorem 1.8.5]{BFHbook18}, the following statement holds.
\begin{theorem}
    \label{thm:weak_conti_embedding}
    Let $\alpha\geq 0,1<p<\infty,$ and $l\in \RR$. Then 
    $$
        L^\infty(0,T;L^p(D)) \cap C^{\alpha}([0,T];W^{l,2}(D)) \hookrightarrow C_w([0,T];L^p(D)).
    $$
    If $\alpha>0$, then the embedding is sequentially compact, meaning any sequence 
    $$
        (v_n)_{v\in \NN} \quad \text{bounded in } L^\infty(0,T;L^p(D)) \cap C^\alpha([0,T];W^{l,2}(D))
    $$
    contains a subsequence $(v_{n_k})_{k\in \NN}$ such that 
    $$
        v_{n_k} \to v\quad \text{in }  C_w([0,T];L^p(D)).
    $$
\end{theorem}
This paper implicitly employs the standard trace theorem and Green's formula (see \cite{Necas11}).
\begin{theorem}
    \label{thm:trace}
    Let $D\subset \RR^N$ be a bounded Lipschitz domain. Then there exists a linear operator $\gamma_0$ with the following properties:
    \begin{align*}
        &\gamma_0 (v) (x) = v(x)\quad \text{for }x\in \pD\quad \text{provided }v\in C^\infty(\oD), \\ 
        &\norm{\gamma_0(v)}_{W^{1-\frac1p,p}(\pD)} \leq c\norm{v}_{W^{1,p}(D)}\quad \text{for all }v\in W^{1,p}(D), \\ 
        &\ker[\gamma_0] = W_0^{1,p}(D)
    \end{align*}
    provided $1<p<\infty$. Moreover, there exists a continuous linear operator 
    \begin{align*}
        &l:W^{1-\frac1p,p}(\pD) \to W^{1,p}(D)  
    \end{align*}
    such that 
    \begin{align*}
        &\gamma_0(l(v)) = v\quad \text{for all }v\in W^{1-\frac1p,p}(\pD)   
    \end{align*}
    provided $1<p<\infty$. 
    
    In addition, the following formula holds: 
    \begin{align*}
        &\int_D \pp_{x_i}u v dx + \int_D u\pp_{x_i}  v dx = \int_{\pD} \gamma_0(u)\gamma_0(v)\rn_i d\Gamma,\quad i=1,\ldots,N,
    \end{align*}
    for any $u\in W^{1,p}(D)$, $v\in W^{1,p'}(D)$, where $\rn_i$ is the $i$-th component of the outer normal vector $\rn$ on the boundary $\pD$.
\end{theorem}
\subsection{Elements of stochastic analysis}
\begin{definition}[Cylindrical Wiener process]
    \begin{itemize}
        \item[1.] Let $\fU$ be a separable Hilbert space with a complete orthonormal system $(e_k)_{k\in \NN}$ and let $(W_k)_{k\in \NN}$ be a sequence of mutually independent real-valued Wiener processes on some probability space $(\Omega,\FFF,\PP)$. The stochastic process $W$ given by the formal expansion $W = \sum_{k=1}^{\infty}e_kW_k$ is called a \textit{cylindrical Wiener process on} $\fU$.
                  Let $\fU_0$ be a separable Hilbert space such that the embedding $\fU \hookrightarrow\fU_0$ is Hilbert--Schmidt. Then the formal expansion converges in $\fU_0$ and thus $W$ makes sense as a $\fU_0$-valued stochastic process. Moreover, trajectories of $W$ are in $C(\ico{0}{\infty};\fU_0)$ $\PP$-a.s., and the law of $W$ on $C(\ico{0}{\infty};\fU_0)$ is denoted by $\LL[W]$.
        \item[2.] If the sequence $(W_k)_{k\in \NN}$ is a sequence of mutually independent real-valued $(\FFF_t)$-Wiener processes on some stochastic basis $\stochbasis$, then $W$ is called a \textit{cylindrical $(\FFF_t)$-Wiener process on} $\fU$.
    \end{itemize}
\end{definition}
As a consequence of the fact that a real-valued Wiener process is determined by its law, we deduce the same for a cylindrical Wiener process.
\begin{lemma}
    \label{lem:sufficient_cond_of_Wiener_by_law}
    Let $W$ be a cylindrical Wiener process on $\fU$ and let $B$ be a stochastic process on $\fU_0$ defined on some probability space $(\Omega,\FFF,\PP)$ and having the law $\LL[W]$. Then $B$ is a cylindrical Wiener process on $\fU$, namely, there exists a collection of mutually independent real-valued Wiener processes $(B_k)_{k\in \NN}$ on $(\Omega,\FFF,\PP)$ such that $B=\sum_{k=1}^{\infty}e_kB_k$.  
\end{lemma}
\begin{proof}
    Since 
    $$
        \LL[((B,e_{k_1}),\ldots, (B,e_{k_n}))] = \LL[((W,e_{k_1}),\ldots, (W,e_{k_n}))] = \LL[(W_{k_1},\ldots, W_{k_n})],
    $$
    for all $k_1,\ldots, k_n\in \NN$, $n\in\NN$,
    $B_k:= (B,e_k)$ are real-valued, mutually independent Wiener processes defined on $(\Omega,\FFF,\PP)$ and the expansion $B=\sum_{k=1}^{\infty}e_k B_k$ follows.
\end{proof}
\begin{lemma}
    \label{lem:sufficient_cond_of_(G_t)-Wiener }
    Let $W$ be a cylindrical $(\FFF_t)$-Wiener process on $\fU$ defined on some probability space $\probsp$ with respect to its canonical filtration $\FFF_t := \sigma(W(r),0\leq r\leq t),t \geq 0$. Assume that $(\GGG_t)_{t\geq 0}$ is a filtration on $\probsp$ such that $\FFF_t\subset \GGG_t$ for any $t\geq 0$ and $(\GGG_t)_{t\geq 0}$ is non-anticipative with respect to $W$, that is, for any $t\geq 0$, $\GGG_t$ is independent of $\sigma(W(t+h)-W(t))$ for any $h>0$. Then $W$ is a cylindrical $(\GGG_t)$-Wiener process on $\fU$.
\end{lemma}
\begin{proof}
    In view of the L\'{e}vy martingale characterization of the Wiener process, \cite[Theorem 4.6]{DPZ14}, it suffices to prove that $W$ is a $(\GGG_t)$-martingale.
    This immediately follows from the assumptions on $(\GGG_t)$.
\end{proof}
\begin{definition}[Random distribution]
    \label{def:r.d.}
    Let $D$ be an open subset of $\RR^N$, and let $\probsp$ be a complete probability space. A mapping 
    $$
        U:\Omega \mapsto \DD'(\RR\times D)
    $$
    is called a \textit{random distribution} if $\ev{U,\phi}:\Omega\mapsto \RR$ is a measurable function for any $\phi\in \DD(\RR\times D)$.
\end{definition}
Similarly to the proof of \cite[Theorem 2.2.3]{BFHbook18}, we have the following measurability theorem. This will be implicitly used in subsequent sections.
\begin{lemma}
    \label{lem:m'ble_of_r.d._in_X}
    Let $X$ be a topological vector space continuously embedded into $\DD'(\RR\times D)$.  Let $U\in \DD'(\RR\times D)$ be a random distribution such that, for any $\varepsilon >0 $, there is a compact set $K_\varepsilon\subset X$ such that 
    $$
        \PP(U\in K_\varepsilon) >1-\varepsilon.
    $$
    Then $U\in X$ $\PP$-a.s., $U$ is an $X$-valued Borel random variable and the law of $U$ is a Radon measure on $X$.
\end{lemma}
Similarly to \cite[Definition 2.2.13]{BFHbook18}, we define the adaptedness of a random distribution.
\begin{definition}[Adaptedness]
    Let $U$ be a random distribution. Then:
    \begin{itemize}
        \item[(1)] We say that $U$ is \textit{adapted} to $(\FFF_t)_{t\geq 0}$ if $\ev{U,\phi}$ is $(\FFF_t)$-measurable for any $\phi\in \DD((-\infty,t)\times D)$. 
        \item[(2)] The family of $\sigma$-fields $(\sigma_t[U])_{t\geq 0}$, given as 
        $$
            \sigma_t[U] := \bigcap_{s>t} \sigma\left(\bigcup_{\phi\in \DD((-\infty,s)\times D)}\left\{ \ev{U,\phi}<1 \right\}\cup \left\{ N\in \FFF,\PP(N)=0 \right\}\right),
        $$
        is called the \textit{history} of $U$.
    \end{itemize}
\end{definition}
\begin{remark}
    Note that we use the same notation for the history of a random distribution and for the canonical filtration of a stochastic process. This will be justified by Lemma \ref{lem:prog_m'ble_r.d.}.
    The proof is exactly the same as in \cite[Lemma 2.2.18]{BFHbook18}.
\end{remark}
\begin{lemma}
    \label{lem:prog_m'ble_r.d.}
    Let $U$ be an $(\FFF_t)$-adapted random distribution taking values in $L^2_{\text{loc}}(\RR;W^{l,2}(D))\subset \DD'(\RR\times D)$ in the sense of Definition \ref{def:r.d.}. Then, for any $T>0$, there exists a $W^{l,2}(D)$-valued $(\FFF_t)$-progressively measurable process $\overline{U}$ such that $\overline{U}\in L^2(0,T;W^{l,2}(D))$ a.s. and 
    $$
        U= \overline{U},\quad \text{a.e. in } \Omega\times [0,T].
    $$
    In particular, we have 
    \begin{align*}
        &U = \overline{U} \quad \text{in } L^2(0,T;W^{l,2}(D))\quad \text{a.s.}  
    \end{align*}
\end{lemma}
\begin{lemma}
    \label{lem:stability_of_nonanti}
    Let $(U_n)_{n\in\NN}$, and $U$ be random distributions on $\probsp$ and $(W_n)_{n\in\NN}$, and $W$ be cylindrical Wiener processes on $\fU$ defined on the same probability space. Assume that the filtration 
    $$
        \sigma(\sigma_t[U_n]\cup \sigma_t[W_n]),\quad t\geq 0
    $$
    is non-anticipative with respect to $W_n$ for every $n\in \NN$. If 
    \begin{align*}
        &\ev{U_n,\phi}\to \ev{U,\phi}\quad \text{in probability for any} \ \phi\in \DD(\RR\times \RR^N), \\ 
        &W_n\to W\quad \text{in probability},
    \end{align*}
    then the filtration 
    $$
        \sigma\left(\sigma_t[U]\cup \sigma_t[W]\right),\quad t\geq 0
    $$
    is non-anticipative with respect to $W$.
\end{lemma}
\begin{proof}
    Since $\sigma(\sigma_t[U_n]\cup \sigma_t[W_n])$ is non-anticipative with respect to $W_n$, we have 
    \begin{align*}
        &\EE\left[F(\ev{U_n,\phi_1},\ldots, \ev{U_n,\phi_m},W_{n,1}(t_1),\ldots, W_{n,l}(t_l))H(W_{n,1}(t+s)-W_{n,1}(t),\ldots, W_{n,k}(t+s)-W_{n,k}(t))\right]  \\ 
        &\quad = \EE\left[F(\ev{U_n,\phi_1},\ldots, \ev{U_n,\phi_m},W_{n,1}(t_1),\ldots, W_{n,l}(t_l))\right] \\ 
        &\quad \quad \times \EE\left[H(W_{n,1}(t+s)-W_{n,1}(t),\ldots, W_{n,k}(t+s)-W_{n,k}(t))\right]  
    \end{align*}
    for any $0\leq t_i\leq t, i=1,\ldots,l,s>0,l,k\in \NN, \phi_j$ supported in $(-\infty,t)$, $j=1,\ldots,m$, and bounded continuous $F$ and $H$.
    Therefore, we may pass to the limit to obtain the corresponding statement for $[U,W]$.
\end{proof}
We use the following lemma on the convergence of stochastic integrals, proven in \cite[Lemma 2.1]{DGHT11}.
For the proof in the setting of this paper, see also \cite[Lemma 2.6.6]{BFHbook18}.
\begin{lemma}
    \label{lem:conv_of_stoch_int}
    Let $\probsp$ be a complete probability space and $H$ be a separable Hilbert space. For $n\in\NN$, let $W_n$ be a cylindrical $(\FFF_t^n)$-Wiener process on $\fU$, and let $G_n$ be an $(\FFF_t^n)$-progressively measurable stochastic process ranging in $L_2(\fU,H)$. Suppose that 
    \begin{align*}
        &W_n\to W\quad \text{in } C([0,T];\fU_0) \quad \text{in probability,} \\ 
        &G_n\to G \quad \text{in }L^2(0,T;L_2(\fU,H))\quad \text{in probability},  
    \end{align*}
    where $W$ is a cylindrical $(\FFF_t)$-Wiener process, and $G$ is $(\FFF_t)$-progressively measurable. Then 
    $$
        \int_0^\cdot G_n dW_n \to \int_0^\cdot G dW\quad \text{in } L^2(0,T;H)\quad \text{in probability.}
    $$
\end{lemma}
Finally, we state Jakubowski's extension of the Skorokhod representation theorem (see \cite[Theorem 2]{jak17}), which is crucial in the stochastic compactness method used in this paper.
\begin{definition}
    A topological space $X$ is called a \textit{sub-Polish space} if there exists a countable family 
    $$
        \left\{ g_n:X\to (-1,1):n\in\NN \right\}
    $$ 
    of continuous functions that separates points of $X$.
\end{definition}
\begin{theorem}[Jakubowski]
    \label{thm:Jakubowski--Skorokhod}
    Let $(X,\tau)$ be a sub-Polish space. If $(\mu_n)_{n\in\NN}$ is a tight sequence of Borel probability measures on $(X,\tau)$, then there exists a subsequence $(n_k)_{k\in\NN}$ and $X$-valued Borel measurable random variables $(U_k)_{k\in\NN}$ and $U$ defined on the standard probability space $([0,1],\overline{\fB([0,1])},\fL)$, where $\fL$ is the Lebesgue measure, 
    such that $\mu_{n_k}$ is the law of $U_k$ and $U_k(\omega)$ converges to $U(\omega)$ in $X$ for every $\omega\in[0,1]$. Moreover, the law of $U$ is a Radon measure.
\end{theorem}

\section{Approximation system}\label{sec:Approximation system}
\subsection{Outline of the proof of the main result}\label{sec:Outline of the proof of the main result}
Our approach in proving the existence of dissipative martingale solutions to \eqref{eq:1.1}--\eqref{eq:1.5} combines the four-layer approximation based on \cite{BFHbook18} with the convex approximation layer used in \cite{NOS23}.
Specifically, the approximation scheme for the problem \eqref{eq:1.8}--\eqref{eq:1.10} is decomposed into the following five-layers, approximating from the bottom layer upward:
\begin{itemize}
    \item[1.] (Artificial pressure layer) To achieve integrability of the density, modify the pressure $p$ to $p_\delta (\rho) = p(\rho) + \delta(\rho+\rho^\Gamma)$ with $\Gamma\geq \max \left\{6,\gamma \right\}$.
    \item[2.] (Artificial viscosity layer) By adding viscosity terms $\varepsilon\Delta \rho$ and $\varepsilon\Delta (\rho u)$ to the continuity equation and the momentum equation, respectively, we partially reduce the problem to the theory of parabolic equations.
    \item[3.] (Faedo--Galerkin approximation layer) We apply standard finite-dimensional approximation to the velocity and the momentum equation. Specifically, by $V_m\subset C^{2}(\oD)\subset L^2(D)$ we denote an $m$-dimensional vector space equipped with the $L^2(D)$-inner product, such that 
    \begin{equation}
        \bigcup_{m\in \NN} V_m \quad \text{is dense in} \quad W_\rn^{1,p}(D),\quad \forall 1\leq p<\infty.
    \end{equation}    
    Furthermore, the $L^2$-orthogonal projection onto $V_m$ is denoted by $\Pi_m$.
    For technical reasons, we assume without loss of generality that the sequence of spaces $(V_m)_{m\in \NN}$ contains a subsequence $(V_{0,m})_{m\in \NN}$ of spaces $V_{0,m}\subset C_c^2(D)$ such that
    \begin{equation}
        \label{eq:3.2}
        \bigcup_{m\in \NN} V_{0,m} \quad \text{is dense in} \quad W_0^{1,p}(D),\quad \forall 1\leq p<\infty.
    \end{equation} 
    Such a choice of $V_m$ is possible provided $D$ belongs to the regularity class $C^{2+\nu}$ (see \cite[Theorem 11.19]{FN17}). 
    \item[4.] (Convex approximation layer) We introduce a convex approximation $(j_\alpha)_{\alpha>0}\subset C^1(\RR^3)\cap C_{\text{loc}}^{1,1}(\RR^3)$ for the boundary integral appearing in \eqref{eq:1.10} as follows:
    $$
        j_\alpha(v):= \begin{cases}
        |v| &\text{for}\ |v|>\alpha, \\
        \frac{|v|^2}{2\alpha}+\frac \alpha 2 &\text{for}\  |v|\leq \alpha. 
        \end{cases}
    $$
    The approximation $j_\alpha$ has the properties
    \begin{align}
        &j_\alpha(0) =\tfrac2\alpha, \\ 
        &\nabla j_\alpha(v)\cdot v\geq 0,\quad \forall v\in \RR^3, \\ 
        &\left|\nabla j_\alpha (v)\right|\leq v,\quad \forall v\in \RR^3, \\ 
        &\left|j_\alpha(v)-\left|v \right| \right|\leq \alpha,\quad \forall v\in \RR^3. \label{eq:3.6}
    \end{align}
    \item[5.] (Cut-off layer for the norm of the velocity) To derive the uniform boundedness of $\rho$ and the energy balance, we introduce the following cut-off function:
    $$
        \chi\in C^\infty(\RR),\quad 
        \begin{cases}
            \chi(z) = 1 &\text{for} \ z\leq 0, \\
            \chi(z) = 0 &\text{for}\ z\geq 1,\\ 
            \chi'(z)\leq 0 &\text{for}\ 0<z\leq 1,       
        \end{cases}  
    $$
    together with the operators 
    $$
        [v]_R = \chi(\norm{v}_{V_m}-R)v,\quad \text{defined for}\ v\in V_m.
    $$
    Similarly, we consider a suitable approximation of the diffusion coefficient. It is convenient to introduce $\FF=(F_k)_{k\in \NN}$ by 
    $$
        F_k(\rho,v) = \frac{G_k(\rho,\rho v)}{\rho}.
    $$
    Note that, in accordance with hypotheses \eqref{eq:1.6}--\eqref{eq:1.7}, the function $F_k$ satisfy 
    $$
        F_k:D\times \ico{0}{\infty}\times \RR^3\to \RR^3,\quad F_k\in C^1(D\times (0,\infty)\times \RR^3)
    $$
    and there exist constants $(f_k)_{k\in \NN}\subset \ico{0}{\infty}$ such that 
    \begin{equation}
        \norm{F_k(\cdot,\cdot,0)}_{L_{x,\rho}^\infty}+\norm{\nabla_v F_k}_{L_{x,\rho,v}^\infty}\leq f_k,\quad \sum_{k=1}^{\infty}f_k^2<\infty.
    \end{equation}
    Finally, we define the noise coefficient $\GG_\varepsilon=(G_{k,\varepsilon})_{k\in \NN}$ by 
    \begin{equation}
        \label{eq:3.8}
        G_{k,\varepsilon}(\rho,q) = \rho F_{k,\varepsilon} \left(\rho,\frac q \rho\right),
    \end{equation}
    where 
    \begin{equation}
        F_{k,\varepsilon}(\rho,v) = \chi\left(\frac \varepsilon \rho - 1\right)\chi\left(|v|-\frac 1 \varepsilon \right)F_k(\rho,v).
    \end{equation}
    Consequently, there exist constants $(f_{k,\varepsilon})\subset \ico{0}{\infty}$ such that 
    \begin{equation}
        \label{eq:3.10}
        \norm{F_{k,\varepsilon}}_{L_{x,\rho,v}^\infty}+\norm{\nabla_{\rho,v}F_{k,\varepsilon}}_{L_{x,\rho,v}^\infty}\leq f_{k,\varepsilon},\quad \sum_{k=1}^{\infty}f_{k,\varepsilon}^2<\infty,
    \end{equation}
    with a bound depending on $\varepsilon$. 
\end{itemize} 

Therefore, the basic approximate problem reads 
\begin{align}
   &d\rho + \rdiv (\rho[u]_R)dt = \varepsilon \Delta \rho dt,\quad \nabla \rho \cdot \rn |_{\pD} =0,\quad \text{in}\ (0,T)\times D, \label{eq:3.11} \\
   \int_D \rho u\cdot \bphi(\tau) dx &= \int_D \rho_0 u_0 \cdot \bphi dx  + \inttau \int_D \left[(\rho[u]_R \otimes u):\nabla\bphi + \chi\left(\norm{u}_{V_m}-R\right)p_\delta (\rho)\rdiv \bphi\right]dxdt  \notag \\ 
   &\quad - \inttau \int_D \left[\SSS(\nabla u):\nabla\bphi -\varepsilon \rho u\cdot \Delta \bphi \right]dxdt - \inttau \int_{\pD}g\nabla j_\alpha(u)\cdot \bphi d\Gamma dt \notag \\ 
   &\quad + \inttau\int_D \GG_\varepsilon(\rho,\rho u)\cdot \bphi dxdW, \quad \ \tau\in [0,T],\  \bphi\in V_m. \label{eq:3.12}
\end{align}
Note that, compared with the conventional basic approximate problem in probabilistic settings, the diffusion coefficient has been replaced with a form independent of the properties of the Galerkin projection.
This modification arises from the fact that, in the compressible flow with slip boundary conditions, the properties of the Galerkin projection $\Pi_m$ independent of spatial dimensions may degenerate. 
More precisely, the following properties 
\begin{align*}
    &\norm{\Pi_m \bphi}_{L^p(D)} \lesssim \norm{\bphi}_{L^p(D)}\quad \text{uniformly in }m\in\NN,\quad \text{for } \bphi\in L^p(D),\\ 
    &\quad \norm{\Pi_m \bphi - \bphi}_{W^{s,p}(D)} \to 0\quad \text{as }m\to \infty,\quad \text{for }\bphi\in W^{s,p}(D),
\end{align*}
are non-trivial unless $s=0$ and $p=2$.
Hence, we avoid using the orthogonal projection $\Pi_m$, and, before passing to the limit in the Galerkin approximation, we must ensure that such a modification of the approximation does not cause any difficulties.
Moreover, using the symbols in \eqref{eq:3.27}, the approximate momentum equation \eqref{eq:3.12} can be rewritten as follows:
\begin{align*}
    &d [\rho u]_m^\ast = \NNN[\rho](u)dt + [\GG_\varepsilon(\rho,\rho u)]_m^\ast dW \quad \text{on }V_m^\ast.
\end{align*}
Next, the reason why the convex approximation layer is placed after the cut-off layer is that we need to ensure the unique existence of approximation solutions before passing to the limit $R\to \infty$. In fact, at the limit $\alpha\to 0$, the momentum equation degenerates, making it difficult to establish uniqueness (see Section \ref{sec:convex_approx}). 
Therefore, note that in the convex approximation, the convergence of the velocity is only guaranteed in the weak sense, and arguments similar to \cite{NOS23} cannot be applied.
Our strategy is as follows:
\begin{itemize}
    \item[(1)] In Section \ref{sec:Approximation system}, we discuss the well-posedness of the basic approximate problem \eqref{eq:3.11}--\eqref{eq:3.12}. We also show that the constructed approximate solutions satisfy the essential energy balance in the uniform estimates, and we pass to the limit $R\to \infty$ based on the uniqueness of solutions.
    \item[(2)] In Section \ref{sec:convex_approx}, we study the limit $\alpha\to 0$ in the approximation of the barrier term appearing in the boundary condition.
    Starting from uniform estimates derived from parabolic maximal regularity and energy bounds, we employ the compactness method based on Jakubowski's extension of the Skorokhod representation theorem.
    \item[(3)] In Section \ref{sec:The limit in the Galerkin approximation scheme}, fixing $\varepsilon$ and $\delta>0$, we pass to the limit $m\to \infty$ in the Galerkin approximation. Since the Galerkin projection $\Pi_m$ degenerates in this setting, in addition to uniform estimates similar to those in Section \ref{sec:convex_approx}, we require refined estimates for the momentum $\rho u$.
    Specifically, we adopt the approach developed in \cite{Smi17} to recover time continuity in the limit.
    \item[(4)] In Section \ref{sec:Vanishing viscosity limit}, we consider the artificial viscosity limit $\varepsilon \to 0$. Here we adopt the method developed in \cite{FN17} for the case of complete slip boundary conditions to the probabilistic setting.
    \item[(5)] Finally, in Section \ref{sec:Vanishing artificial pressure limit}, we deal with the artificial pressure limit $\delta\to 0$. As in the previous section, we proceed with the similar argument as in the case of complete slip boundary conditions, using Jakubowski's extension of the Skorokhod representation theorem.
\end{itemize}

\subsection{Solvability of the basic approximate problem}
First, we give a precise definition of a martingale solution to the basic approximate problem \eqref{eq:3.11}--\eqref{eq:3.12}.
\begin{definition}
    \label{def:3.1}
	Let $\Lambda$ be a Borel probability measure on $C^{2+\nu}(\oD)\times V_m$. Then $((\Omega,\FFF,(\FFF_t)_{t\geq 0},\PP),\rho,u,W)$ is called a \textit{martingale solution} to \eqref{eq:3.11}--\eqref{eq:3.12} with the initial law $\Lambda$ if:
	\begin{itemize}
		\item[(1)] $(\Omega,\FFF,(\FFF_t)_{t\geq 0},\PP)$ is a stochastic basis with a complete right-continuous filtration;
		\item[(2)] $W$ is a cylindrical $(\FFF_t)$-Wiener process on $\fU$;
		\item[(3)] $\rho$ and $u$ are $(\FFF_t)$-adapted stochastic processes such that $\PP$-a.s.
		\begin{align*}
		   &\rho> 0,\quad \rho \in C([0,T];C^{2+\nu}(\oD))\cap C^1([0,T];C^{\nu}(\oD)),\quad u\in C([0,T];V_m);   
		\end{align*}
		\item[(4)] there exists a $C^{2+\nu}(\oD)\times V_m$-valued $\FFF_0$-measurable random variable $(\rho_0,u_0)$ such that $\Lambda = \LL[\rho_0,u_0]$;
		\item[(5)] the approximate equation of continuity 
		\begin{equation}
            \label{eq:3.13}
            \pp_t \rho + \rdiv (\rho [u]_R) = \varepsilon \Delta \rho,\quad \nabla \rho \cdot \rn |_{\pD} =0
        \end{equation}
		holds in $(0,T)\times D$ $\PP$-a.s. and we have $\rho(0) = \rho_0$ $\PP$-a.s.;
		\item[(6)] the approximate momentum equation 
		\begin{align}
            \int_D \rho u\cdot \bphi (\tau)dx &= \int_D \rho_0 u_0 \cdot \bphi dx  + \inttau \int_D \left[(\rho[u]_R \otimes u ):\nabla \bphi+ \chi\left(\norm{u}_{V_m}-R\right)p_\delta (\rho)\rdiv \bphi\right]dxdt  \notag \\ 
            &\quad - \inttau \int_D \left[\SSS(\nabla u):\nabla\bphi -\varepsilon \rho u\cdot \Delta \bphi \right]dxdt - \inttau \int_{\pD} g\nabla j_\alpha(u)\cdot \bphi d\Gamma dt \notag \\ 
            &\quad + \inttau\int_D \GG_\varepsilon(\rho,\rho u)\cdot \bphi dxdW \label{eq:3.14}
        \end{align}
		holds for all $\tau\in [0,T]$ and all $\bphi\in V_m$ $\PP$-a.s., and we have $u(0) = u_0$ $\PP$-a.s.
    \end{itemize}
\end{definition}
\begin{remark}
    \label{rem:3.2}
    As the processes $\rho,u$ are $(\FFF_t)$-adapted and continuous, the composition $\GG_\varepsilon(\rho,\rho u)$ is progressively measurable as a mapping into $L_2(\fU,(W^{b,2}(D))^\ast)$ with $b>3/2$ and the stochastic integral is well-defined. 
    Indeed, by \eqref{eq:3.10}, and the embedding $W^{b,2}(D)\hookrightarrow L^\infty(D)$, $b>3/2$, we have
    \begin{align*}
        \left|\int_D G_{k,\varepsilon}(\rho,\rho u)\cdot \bphi dx \right|  &\leq \norm{\rho F_{k,\varepsilon}(\rho,u)}_{L_x^1}\norm{\bphi}_{L_x^\infty} \\ 
        &\lesssim \norm{\rho}_{L_x^2}\norm{F_{k,\varepsilon}(\rho,u)}_{L_x^2}\norm{\bphi}_{W_x^{b,2}} \\ 
        &\leq \norm{\rho}_{L_x^2}f_{k,\varepsilon}\norm{\bphi}_{W_x^{b,2}}
    \end{align*}
    for all $\bphi\in W^{b,2}(D)$, and, consequently
    \begin{align*}
        &\norm{\GG_\varepsilon(\rho,\rho u)}_{L_2(\fU,(W^{b,2}(D))^\ast)}^2 \lesssim \norm{\rho}_{L_x^2}^2.
    \end{align*}
\end{remark}
Our main goal in this section is to prove the following result.
\begin{theorem}
    \label{thm:sol1}
    Let $\Lambda$ be a Borel probability measure on $C^{2+\nu}(\oD)\times V_m$ such that 
    \begin{equation}
        \label{eq:3.15}
        \Lambda \left\{ \urho\leq \rho,\ \norm{\rho}_{C_x^{2+\nu}}\leq \orho,\ \nabla \rho\cdot \rn |_{\pD}=0 \right\}=1,\quad \int_{C_x^{2+\nu}\times V_m}\norm{v}_{V_m}^r d\Lambda(\rho,v)\leq \overline{u}
    \end{equation}
    for some deterministic constants $\urho,\orho,\overline{u}>0$ and some $r>2$. Then the approximate problem \eqref{eq:3.11}--\eqref{eq:3.12} admits a martingale solution in the sense of Definition \ref{def:3.1}. The solution satisfies in addition 
    \begin{align}
       &\sup_{t\in [0,T]} \left(\norm{\rho(t)}_{C_x^{2+\nu}}+\norm{\pp_t \rho(t)}_{C_x^{\nu}} +\norm{\rho^{-1}(t)}_{C_x^0}\right)\leq c\quad \PP\text{-a.s.,}  \label{eq:3.16}\\
       &\EE\left[\sup_{\tau\in [0,T]}\norm{u(\tau)}_{V_m}^r\right] \leq c \left(1+\EE\left[\norm{u_0}_{V_m}^r\right]\right), \label{eq:3.17}
    \end{align}
    with a constant $c=c(m,R,T,\urho,\orho,\nu,\varepsilon,r)$.
\end{theorem}
\subsubsection{Iteration scheme}\label{sec:Iteration scheme}
Let $\stochbasis$ be a stochastic basis with a complete right-continuous filtration and $W$ be a cylindrical $(\FFF_t)$-Wiener process. Consider an $\FFF_0$-measurable initial datum $(\rho_0,u_0)$ with law $\Lambda$. By the assumptions \eqref{eq:3.15}, we have
\begin{equation}
    \label{eq:3.18}
    0<\urho\leq \rho_0,\quad \norm{\rho_0}_{C_x^{2+\nu}}\leq \orho,\quad \nabla \rho_0\cdot \rn |_{\pD}=0,\quad \PP\text{-a.s.,}\quad  \EE\left[\norm{u_0}_{V_m}^r\right]\leq \overline{u}.
\end{equation}
Fixing a time step $h>0$, we set 
\begin{equation}
    \label{eq:3.19}
    \rho(t) = \rho_0,\quad u(t)=u_0\quad \text{for} \  t\leq 0,
\end{equation}
and recursively construct the solutions to the following equations

\begin{align}
    \label{eq:3.20}
    &\begin{aligned}
        \pp_t \rho + \rdiv (\rho[u(nh)]_R) &= \varepsilon \Delta \rho ,\quad \nabla \rho \cdot \rn |_{\pD} =0,\quad \text{in}\ [nh,(n+1)h)\times D,  \\ 
        \rho(nh) &= \rho(nh-) := \lim_{s\uparrow nh}\rho(s)
    \end{aligned}
\end{align}
and 
\begin{align}
    &\int_D \rho u \cdot \bphi (\tau) dx \notag\\ 
    &\quad = \int_D \rho u \cdot \bphi (nh) dx  + \int_{nh }^{\tau} \int_D \left[(\rho(t)[u(nh)]_R \otimes u(nh)):\nabla \bphi + \chi\left(\norm{u(nh)}_{V_m}-R\right)p_\delta (\rho(t))\rdiv \bphi\right]dxdt  \notag \\ 
    &\quad \quad - \int_{nh }^{\tau} \int_D \left[\SSS(\nabla u(nh)):\nabla\bphi -\varepsilon \rho(t) u(nh)\cdot \Delta \bphi \right]dxdt - \int_{nh }^{\tau}\int_{\pD} g\nabla j_\alpha(u(nh))\cdot \bphi d\Gamma dt \notag \\ 
    &\quad \quad + \int_{nh }^{\tau}\int_D \GG_\varepsilon (\rho(nh),\rho u(nh)) \cdot \bphi dxdW, \quad \ \tau\in \ico{nh}{(n+1)h},\  \bphi\in V_m, \label{eq:3.21}
\end{align}
where $u(nh) = u(nh-):= \lim_{s\uparrow nh}u(s)$ and $n\in \{0,\ldots,\lfloor h^{-1}T\rfloor\}$. 
Note that the velocity on the left-hand side of \eqref{eq:3.20} and on the right-hand side of \eqref{eq:3.21} is always frozen at the time $nh$, whereas the density is evaluated at time $t$ everywhere except for the stochastic integral. 

To construct these solutions, we first solve the following parabolic equation 
\begin{align}
    \label{eq:3.22}
    &\begin{aligned}
        \pp_t \rho + \rdiv (\rho [u_0]_R) &= \varepsilon \Delta \rho ,\quad \nabla \rho \cdot \rn |_{\pD} =0,\quad \text{in}\ [0,h)\times D,  \\ 
        \rho(0) &= \rho_0.  
    \end{aligned}
\end{align}
Then by the classical theory for the parabolic Neumann problem (see \cite[Lemma 3.1, Theorems 11.29, and 11.30]{FN17}), there exists a unique solution $\rho\in C([0,h];C^{2+\nu}(\oD))\cap C^1([0,h];C^{\nu}(\oD))$ to the problem \eqref{eq:3.22} which satisfies the estimate
\begin{equation}
    \label{eq:3.23}
    0<\urho \exp \left(-\int_{0}^{\tau}\norm{\rdiv [u_0]_R}_{L^\infty(D)}dt\right)\leq \rho(\tau,\cdot)\leq \orho \exp \left(\int_{0}^{\tau}\norm{\rdiv [u_0]_R}_{L^\infty(D)}dt\right)<\infty,\quad \text{in}\ \oD
\end{equation}
for all $\tau\in [0,h]$. Furthermore, this solution satisfies the estimate 
\begin{equation}
    \label{eq:3.24}
    \norm{\rho}_{C([0,h];C^{2+\nu}(\oD))} +\norm{\pp_t\rho}_{C([0,h];C^{\nu}(\oD))}\leq c, 
\end{equation}
for some constant $c=c(m,R,T,\nu,\orho,\varepsilon)$. Indeed, by classical parabolic maximum regularity, we have
\begin{align*}
    &\norm{\rho}_{C([0,h];C^{2+\nu}(\oD))} +\norm{\pp_t\rho}_{C([0,h];C^{\nu}(\oD))}\leq c\left(\norm{\rho_0}_{C^{2+\nu}(\oD)}+ \norm{\rdiv (\rho [u_0]_R)}_{C([0,h];C^{\nu}(\oD))}\right),
\end{align*}
and by the equivalence of norms on $V_m$, the definition of $[\cdot]_R$, and \eqref{eq:3.23}, we have 
\begin{align*}
    &\norm{\rdiv (\rho [u_0]_R)}_{C([0,h];C^{\nu}(\oD))} \\
    &\quad \leq \norm{\nabla\rho \cdot [u_0]_R}_{C([0,h];C^{\nu}(\oD))} + \norm{\rho \rdiv[u_0]_R}_{C([0,h];C^{\nu}(\oD))} \\ 
    &\quad \leq \norm{\nabla \rho}_{C([0,h];C^{\nu}(\oD))}\norm{[u_0]_R}_{C([0,h];C^{0}(\oD))} + \norm{\nabla\rho}_{C([0,h];C^{0}(\oD))}\norm{[u_0]_R}_{C([0,h];C^{\nu}(\oD))}\\ 
    &\quad \quad + \norm{\rho}_{C([0,h];C^{\nu}(\oD))}\norm{\rdiv[u_0]_R}_{C([0,h];C^{0}(\oD))} + \norm{\rho}_{C([0,h];C^{0}(\oD))}\norm{\rdiv[u_0]_R}_{C([0,h];C^{\nu}(\oD))} \\ 
    &\quad \leq c(m,R)\norm{\rho}_{C([0,h];C^{1+\nu}(\oD))} \\ 
    &\quad \leq \frac 1 2 \norm{\rho}_{C([0,h];C^{2+\nu}(\oD))} + c(m,R,\nu)\norm{\rho}_{C([0,h];C^{0}(\oD))} \\ 
    &\quad \leq \frac 1 2 \norm{\rho}_{C([0,h];C^{2+\nu}(\oD))} + c(m,R,\nu)\orho \exp\left(\inttau \norm{\rdiv [u_0]_R}_{L^\infty(D)}dt \right), \\
    &\quad \leq \frac 1 2 \norm{\rho}_{C([0,h];C^{2+\nu}(\oD))} + c(m,R,\nu,\orho,T), 
\end{align*}
where we used an interpolation argument in the inequality in the third to last line. Combining the above two inequalities yields \eqref{eq:3.24}.

Next, to solve \eqref{eq:3.21} on $[0,h]$, it is convenient to rewrite \eqref{eq:3.21} in terms of $du$. To this end, we introduce (for a given function $\rho\in L^1(D)$) the linear mapping
$$
    \MMM[\rho]:V_m\to V_m^\ast,\quad \MMM[\rho](v)=[\rho v]_m^\ast,
$$
where the symbol $[\cdot]_m^\ast:L^2(D)\to V_m^\ast$ is given by 
$$
    \ev{[v]_m^\ast,\bphi}_{V_m^\ast,V_m} = \int_D v\cdot \bphi dx,\quad \bphi\in V_m,\ v\in L^2(D).
$$
It is easy to check that the operator $\MMM$ has the following properties. $\MMM[\rho]$ is invertible, and we have 
\begin{equation}
    \label{eq:3.25}
    \norm{\MMM^{-1}[\rho]}_{\LL(V_m^\ast,V_m)}\leq \left(\inf_{x\in \oD} \rho\right)^{-1}
\end{equation}
as long as $\rho\in L^1(D)$ is bounded below from zero, and making use of the identity 
$$
    \MMM^{-1}[\rho_1]-\MMM^{-1}[\rho_2]=\MMM^{-1}[\rho_2]\left(\MMM[\rho_2]-\MMM[\rho_1]\right)\MMM^{-1}[\rho_1],
$$
we see that 
\begin{equation}
    \label{eq:3.26}
    \norm{\MMM^{-1}[\rho_1]-\MMM^{-1}[\rho_2]}_{\LL(V_m^\ast,V_m)}\leq c(m,\urho) \norm{\rho_1-\rho_2}_{L_x^1},
\end{equation}
provided both $\rho_1$ and $\rho_2$ are bounded below by some positive constant $\urho$.
Furthermore, we introduce (for a given function $\rho\in C^2(\oD)$) the mapping $\NNN[\rho]:V_m\to V_m^\ast$ as follows: 
\begin{align}
    \NNN[\rho](v) &= \left[-\rdiv (\rho[v]_R\otimes v) - \chi(\norm{v}_{V_m}-R)\nabla p_\delta(\rho)+\varepsilon \Delta (\rho v) +\rdiv \SSS(\nabla v)\right]_m^\ast  \notag \\
    &\quad - [\varepsilon \rho(\nabla v)\rn]_m^{\ast \pp} + [\varepsilon \rho v \cdot \nabla_\rn]_m^{\ast \pp} - [\SSS(\nabla v)\rn]_m^{\ast \pp} - [g\nabla j_\alpha(v)]_m^{\ast \pp} \label{eq:3.27} 
\end{align}
where for $v\in L^{2}(\pD)$, $[v]_m^{\ast \pp},\ [v \nabla_\rn]_m^{\ast \pp}\in V_m^\ast$ are given by 
\begin{align*}
    \ev{[v]_m^{\ast \pp},\bphi}_{V_m^\ast,V_m} &= \int_{\pD} v\cdot \bphi d\Gamma,\  \\
    \ev{[v\nabla_\rn]_m^{\ast \pp},\bphi}_{V_m^\ast,V_m} &= \int_{\pD} v\cdot (\nabla\bphi)\rn d\Gamma,\quad \bphi\in V_m.
\end{align*}
Note that by the inclusion $V_m\subset C^2(\oD)$ and the equivalence of norms on $V_m$, all boundary integrals appearing above are bounded linear functionals on $V_m$. Therefore, the mapping $\NNN[\rho]:V_m\to V_m^\ast$ is well-defined.
Accordingly, the relation \eqref{eq:3.21} can be written in the form 
\begin{align}
    \label{eq:3.28}
    u(\tau) &= \MMM^{-1}[\rho(\tau)] \left([\rho_0 u_0]_m^\ast + \int_{0}^{\tau}\NNN[\rho(t)](u_0)dt + \int_{0}^{\tau}\left[\GG_\varepsilon(\rho_0,\rho_0 u_0)\right]_m^\ast dW\right),\quad \tau\in \ico{0}{h},
\end{align}
where the Bochner integral and the stochastic integral on the right-hand side are well-defined in view of \eqref{eq:3.24}. Therefore, by recursively repeating a similar argument for $n\in \{1,\ldots,\lfloor h^{-1}T\rfloor\}$, we can construct a unique solution to \eqref{eq:3.20}--\eqref{eq:3.21}, denoted as $(\rho_h,u_h)$.

\subsubsection{The limit for vanishing time step}\label{sec:The limit for vanishing time step}
Our next goal is to let $h\to 0$ in \eqref{eq:3.20}--\eqref{eq:3.21} in order to obtain a solution to the approximate problem \eqref{eq:3.11}--\eqref{eq:3.12}. 
In this section, the approximation parameters $R>0, \alpha\in (0,1), m\in \NN, \varepsilon\in (0,1)$ and $\delta\in (0,1)$ are kept fixed. We call uniform estimate that are independent of $h$ but may depend on $R,\alpha,m,\varepsilon,\delta$ and $T>0$.
Our strategy is as follows: 
\begin{itemize}
    \item[1.] We establish a uniform estimate for the approximation sequence $(\rho_h,u_h)$ with respect to $h\in (0,1)$ to derive the tightness of laws of $[\rho_h,u_h,W],h\in (0,1)$.
    \item[2.] We apply Prokhorov's theorem and Skorokhod's theorem to show the existence of the limit as $h\to 0$.
    \item[3.] We prove that the limit is a solution in the sense of Definition \ref{def:3.1}.
\end{itemize}
For simplicity of notation, the subscript $h$ in $(\rho_h,u_h)$ is temporarily omitted, and the following notation is used
$$
    [v]_h = v(nh,\cdot),\quad [v]_{h,R} = [v(nh,\cdot)]_R\quad \text{for}\ t\in \ico{nh}{(n+1)h},\ n\in \NN.
$$
Then, $\rho$ given in Section \ref{sec:Iteration scheme} is the unique solution to the following parabolic Neumann problem 
\begin{align}
    \label{eq:3.29}
    &\begin{aligned}
        \pp_t \rho + \rdiv (\rho [u]_{h,R}) &= \varepsilon \Delta \rho ,\quad \nabla \rho \cdot \rn |_{\pD} =0,\quad \text{in}\ [0,T)\times D,  \\ 
        \rho(0) &= \rho_0
    \end{aligned}
\end{align}
in the class 
\begin{equation*}
    \rho\in C([0,T];C^{2+\nu}(\oD)),\quad \pp_t \rho \in L^\infty(0,T;C^\nu(\oD)).
\end{equation*}
Note that as the regularized velocity $[u]_{h,R}$ is only piecewise continuous, the same is true for $\pp_t\rho$. In general, we do not expect $\pp_t\rho\in C([0,T];C^{\nu}(\oD))$. 
In view of \eqref{eq:3.29}, following the same argument used to derive \eqref{eq:3.23} and \eqref{eq:3.24}, we obtain the estimate
\begin{align}
    &\esssup_{t\in [0,T]} \left(\norm{\rho(t)}_{C_x^{2+\nu}}+\norm{\pp_t \rho(t)}_{C_x^{\nu}} +\norm{\rho^{-1}(t)}_{C_x^0}\right)\leq c\quad \PP\text{-a.s.,} \label{eq:3.30}
\end{align}
with a deterministic constant $c = c(m,R,T,\urho,\orho,\nu,\varepsilon)$ independent of $h$. 
\begin{remark}
    \label{rem:3.4}
    By the Sobolev embedding and interpolation (see \cite[Theorem 5.2]{Ama00} and \cite{Lun12}), the following embedding holds:
    \begin{align}
        L^\infty(0,T;C^{2+\nu}(\oD))\cap W^{1,\infty}(0,T;C^{\nu}(\oD))\hookrightarrow C^{\iota}([0,T];C^{2+\iota}(\oD)),\quad \iota\in \left(0,\tfrac \nu 3\right], 
    \end{align} 
    where, for $\iota \in (0,\nu/3)$, this embedding is actually compact. Therefore, $\rho_h$ is bounded in this space.
\end{remark}
Next, we consider uniform estimates for the velocity $u$. What we require is compactness of the approximate velocities in the space $C([0,T];V_m)$, and we have to control the difference $u-[u]_h$ uniformly in time. To this end, we need to derive an estimate for the H\"{o}lder norm of $u$.
\begin{proposition}
    The approximate sequence $(u_h)_{h\in (0,1)}$ satisfies the estimate
    \begin{align}
        \label{eq:3.32}
        \EE\left[\norm{u}_{C_t^\beta V_m}^r\right]\leq c \left(1+ \EE\left[\norm{u_0}_{V_m}^r\right]\right)
    \end{align}
    uniformly in $h$, whenever $r>2$ and $\beta\in (0,\nu/3\wedge (1/2-1/r))$ with a constant $c=c(m,R,T,\urho,\orho,\nu,r,\beta)$.
\end{proposition}
\begin{proof}
    By the definition of $\MMM$ and \eqref{eq:3.28}, $u$ has the representation 
    \begin{align}
        u(\tau) &= \MMM^{-1}[\rho(\tau)]([\rho u]_m^\ast(\tau)), \\ 
        [\rho u]_m^\ast(\tau) &= [\rho_0 u_0]_m^\ast + \int_{0}^{\tau}\NNN[\rho(t)]([u]_h)dt + \int_{0}^{\tau}\left[\GG_\varepsilon([\rho]_h,[\rho u]_h )\right]_m^\ast dW,\quad \tau\in [0,T].
    \end{align}
    Hence, we first show that 
    \begin{align*}
        \EE\left[\norm{[\rho u]_m^\ast(\tau)}_{C_t^{\beta'}V_m^\ast}^r\right]\leq c\left(1+\EE\left[\norm{u_0}_{V_m}^r\right]\right)  
    \end{align*}
    for some $\beta'\in (0,1)$. It follows from \eqref{eq:3.21} and the equivalence of norms on $V_m$, that, uniformly in $h$, 
    \begin{align}
        \label{eq:3.35}
        \int_D \rho u(\tau)\cdot \bphi dx \lesssim \norm{u_0}_{V_m} + \inttau \sup_{0\leq s\leq t}\norm{u}_{V_m}dt + T + \norm{\inttau \left[\GG_\varepsilon([\rho]_h,[\rho u]_h )\right]_m^\ast dW}_{V_m^\ast}
    \end{align}
    for any $\bphi\in V_m,\ \norm{\bphi}_{V_m}\leq 1$ whenever $0\leq \tau\leq T$. Here we take into account \eqref{eq:3.30} and the definitions of $\chi$ as well as the cut-off $[\cdot]_R$. Taking the supremum over $\bphi$, we obtain 
    \begin{align}
        \label{eq:3.36}
        \norm{[\rho u]_m^\ast(\tau)}_{V_m^\ast}\lesssim \norm{u_0}_{V_m} + \inttau \sup_{0\leq s\leq t}\norm{u}_{V_m}dt + T + \norm{\inttau\left[\GG_\varepsilon([\rho]_h,[\rho u]_h )\right]_m^\ast dW}_{V_m^\ast}
    \end{align}
    and
    \begin{align}
        \label{eq:3.37}
        \norm{[\rho u]_m^\ast(\tau)}_{V_m^\ast}^r\lesssim \norm{u_0}_{V_m}^r + \inttau \sup_{0\leq s\leq t}\norm{u}_{V_m}^r dt + 1 + \norm{\inttau\left[\GG_\varepsilon([\rho]_h,[\rho u]_h )\right]_m^\ast dW}_{V_m^\ast}^r
    \end{align}
    uniformly in $h$ for all $0\leq \tau\leq T$ and for any $r\geq 1$. Finally, we pass to expectation and apply the Burkholder--Davis--Gundy inequality to the last integral with consideration of Remark \ref{rem:3.2}, obtaining  
    \begin{align}
        &\EE\left[\sup_{0\leq t\leq \tau}\norm{[\rho u]_m^\ast(t)}_{V_m^\ast}^r\right] \notag \\ 
        &\quad \lesssim \EE\left[\norm{u_0}_{V_m}^r \right]+\inttau \EE\left[\sup_{0\leq s\leq t}\norm{u}_{V_m}^r\right]dt + 1 + \EE\left[\left(\inttau \sum_{k=1}^{\infty}\norm{\left[G_{k,\varepsilon}([\rho]_h,[\rho u]_h )\right]_m^\ast}_{V_m^\ast}^2 dt\right)^{r/2}\right]\notag \\ 
        &\quad \lesssim \EE\left[\norm{u_0}_{V_m}^r \right]+\inttau \EE\left[\sup_{0\leq s\leq t}\norm{u}_{V_m}^r\right]dt + 1 + \EE\left[\left(\inttau \norm{[\rho]_h}_{L_x^2}^2 \sum_{k=1}^{\infty}f_{k,\varepsilon}^2 dt\right)^{r/2}\right] \notag\\ 
        &\quad \lesssim \EE\left[\norm{u_0}_{V_m}^r \right]+\inttau \EE\left[\sup_{0\leq s\leq t}\norm{u}_{V_m}^r\right]dt + 1 \label{eq:3.38}
    \end{align}
    uniformly in $h$. Due to
    $$
        \norm{[\rho u]_m^\ast}_{V_m^\ast}\lesssim \norm{u}_{V_m} \lesssim \norm{[\rho u]_m^\ast}_{V_m^\ast},
    $$
    Gronwall's lemma gives rise to the following estimate:
    \begin{align}
        \label{eq:3.39}
        &\EE\left[\sup_{\tau\in [0,T]}\norm{[\rho u]_m^\ast(\tau)}_{V_m^\ast}^r\right] +\EE\left[\sup_{\tau\in [0,T]}\norm{u(\tau)}_{V_m}^r\right]\leq c(m,R,T,\urho,\orho,r)\left(1+\EE\left[\norm{u_0}_{V_m}^r\right]\right).
    \end{align}
    On the other hand, following the same argument used to derive \eqref{eq:3.35}, we obtain 
    \begin{align*}
        &\int_D (\rho u(\tau_1)-\rho u (\tau_2))\cdot \bphi dx  \\ 
        &\quad \lesssim \norm{u_0}_{V_m} + \sup_{0\leq \tau\leq T}\norm{u}_{V_m}|\tau_1-\tau_2| + |\tau_1-\tau_2| + \norm{\int_{\tau_2}^{\tau_1} \left[\GG_\varepsilon([\rho]_h,[\rho u]_h )\right]_m^\ast dW}_{V_m}
    \end{align*}
    for any $\bphi\in V_m,\ \norm{\bphi}_{V_m}\leq 1$ whenever $0\leq \tau_1<\tau_2\leq T$.
    Therefore, with the bound \eqref{eq:3.39} at hand, we repeat the arguments leading to \eqref{eq:3.38} to obtain 
    \begin{align*}
        &\EE\left[\norm{[\rho u(\tau_1)-\rho u(\tau_2)]_m^\ast}_{V_m^\ast}^r\right]  \lesssim |\tau_1-\tau_2|^{r/2}\left(1+\EE\left[\norm{u_0}_{V_m}^r\right]\right),\quad r\geq 1
    \end{align*}
    uniformly in $h$ whenever $0\leq \tau_1<\tau_2\leq T$. Thus, for $r>2$, we apply Kolmogorov's continuity theorem to conclude that $[\rho u]_m^\ast$ has $\PP$-a.s. $\beta'$-H\"{o}lder continuous trajectories for all $\beta'\in (0,1/2-1/r)$. In addition, we have 
    \begin{align}
        \label{eq:3.40}
        &\EE\left[\norm{[\rho u]_m^\ast}_{C_t^{\beta'}V_m^\ast}^r\right]  \lesssim 1+\EE\left[\norm{u_0}_{V_m}^r\right],\quad r>2
    \end{align}
    uniformly in $h$, where the proportional constant depends on $\beta'$. Next, we proceed to the proof of the estimate \eqref{eq:3.32}. We decompose $u(\tau_1)-u(\tau_2)$ into two parts
    \begin{align*}
        &u(\tau_1)-u(\tau_2) = \left(\MMM^{-1}[\rho(\tau_1)]-\MMM^{-1}[\rho(\tau_2)]\right) [\rho u(\tau_1)]_m^\ast + \MMM^{-1}[\rho(\tau_2)][\rho u(\tau_1)-\rho u(\tau_2)]_m^\ast.
    \end{align*}
    Considering \eqref{eq:3.30}, Remark \ref{rem:3.4} and \eqref{eq:3.40}, we have $\PP$-a.s.
    \begin{align*}
        \norm{\left(\MMM^{-1}[\rho(\tau_1)]-\MMM^{-1}[\rho(\tau_2)]\right) [\rho u(\tau_1)]_m^\ast}_{V_m} &\leq \norm{\MMM^{-1}[\rho(\tau_1)]-\MMM^{-1}[\rho(\tau_2)]}_{\LL(V_m^\ast,V_m)}\norm{[\rho u(\tau_1)]_m^\ast}_{V_m^\ast}  \\ 
        &\leq c(m,\urho)\norm{\rho(\tau_1)-\rho(\tau_2)}_{L_x^1} \norm{[\rho u(\tau_1)]_m^\ast}_{V_m^\ast}  \\ 
        &\lesssim \left|\tau_1-\tau_2 \right|^{\nu/3}\norm{\rho}_{C_t^{\nu/3}C_x^{2+\nu/3}}\norm{[\rho u]_m^\ast}_{C_t^0 V_m^\ast}  \\
        &\lesssim \left|\tau_1-\tau_2 \right|^{\nu/3}\norm{[\rho u]_m^\ast}_{C_t^0 V_m^\ast},
    \end{align*}
    and 
    \begin{align*}
        \norm{\MMM^{-1}[\rho(\tau_2)][\rho u(\tau_1)-\rho u(\tau_2)]_m^\ast}_{V_m} &\lesssim \norm{[\rho u(\tau_1)-\rho u(\tau_2)]_m^\ast }_{V_m^\ast}  \\ 
        &\lesssim \left|\tau_1-\tau_2 \right|^{\beta'}\norm{[\rho u]_m^\ast}_{C_t^{\beta'}V_m^\ast}.
    \end{align*}
    Therefore, combining these two estimates yields 
    \begin{align*}
        \EE\left[\norm{u}_{C_t^\beta V_m}^r\right]\leq c \left(1+ \EE\left[\norm{u_0}_{V_m}^r\right]\right)
    \end{align*}
    uniformly in $h$, whenever $r>2$ and $\beta\in (0,\nu/3\wedge (1/2-1/r))$ with a constant $c=c(m,R,T,\urho,\orho,\nu,r,\beta)$.
\end{proof}

With \eqref{eq:3.30} and \eqref{eq:3.32} at hand, we are ready to perform the limit $h\to 0$. Let $[\rho_h,u_h,W]$ be the unique approximate solution issuing from the iteration scheme \eqref{eq:3.19}--\eqref{eq:3.21}, with the initial data satisfying \eqref{eq:3.18}. The corresponding path space is defined as 
$$
    \XX = \XX_\rho\times \XX_u\times \XX_W = \overline{C^{\nu/3}([0,T];C^{2+\nu/3}(\oD))}^{\norm{\cdot}_{C_t^{\iota}C_x^{2+\iota}}}\times \overline{C^\beta([0,T];V_m)}^{\norm{\cdot}_{C_t^\kappa V_m}}\times C([0,T];\fU_0),
$$
where $\iota\in (0,\nu/3),\kappa\in (0,\beta)$, and $\nu$ and $\beta$ are the H\"{o}lder exponents in \eqref{eq:3.30} and \eqref{eq:3.32}, respectively. 
Let $\LL[\rho_h,u_h,W]$ denote the joint law of $[\rho_h,u_h,W]$ on $\XX$, whereas $\LL[\rho_h]$, $\LL[u_h]$, and $\LL[W]$ denote the corresponding marginals on $\XX_\rho$, $\XX_u$, and $\XX_W$, respectively.
In view of the bounds \eqref{eq:3.30} and \eqref{eq:3.32}, we obtain tightness of joint laws $\LL[\rho_h,u_h,W],\ h\in (0,1)$.
\begin{proposition}
    The set $\left\{ \LL[\rho_h,u_h,W],\ h\in (0,1) \right\}$ is tight on $\XX$.
\end{proposition}
\begin{proof}
    First, in view of Remark \ref{rem:3.4}, the set 
    $$
        B_L = \left\{ \rho\in C^{\nu/3}([0,T];C^{2+\nu/3}(\oD)) : \ \norm{\rho}_{C_t^{\nu/3}C_x^{2+\nu/3}}\leq L\right\}
    $$
    is relatively compact in $\XX_\rho$, and we have
    $$
        \LL[\rho_h](B_L^c)= \PP(\norm{\rho}_{C_t^{\nu/3}C_x^{2+\nu/3}}>L) =0,
    $$
    provided $L$ is sufficiently large. Therefore, $\left\{ \LL[\rho_h],h\in (0,1) \right\}$ is tight on $\XX_\rho$. Similarly, the set 
    $$
        B_L = \left\{ u\in C^{\beta}([0,T];V_m) : \ \norm{u}_{C_t^{\beta}V_m}\leq L\right\}
    $$
    is relatively compact in $\XX_u$ and by Chebyshev's inequality and \eqref{eq:3.32}, we have 
    $$
        \LL[u_h](B_L^c) = \PP(\norm{u_h}_{C_t^\beta V_m}>L)\leq \frac{1}{L^r}\EE\left[\norm{u_h}_{C_t^\beta V_m}^r\right]\leq \frac{C}{L^r}.
    $$
    Hence, $\left\{ \LL[u_h],h\in (0,1) \right\}$ is tight on $\XX_u$. Finally, $\LL[W]$ is tight on $\XX_W$ since it is a Radon measure on a Polish space. 
    Therefore, tightness of $\left\{ \LL[\rho_h,u_h,W],h\in (0,1) \right\}$ follows from Tychonoff's theorem.
\end{proof}
Accordingly, since $\XX$ is a Polish space, the following result follows from Prokhorov's theorem and Skorokhod's theorem.
\begin{proposition}
    \label{prop:repr1}
    There exists a complete probability space $(\tOmega,\tFFF,\tPP)$ with $\XX$-valued Borel measurable random variables $(\trho_h,\tu_h,\tW_h),h\in (0,1)$, and $(\trho,\tu,\tW)$ such that (up to a subsequence): 
    \begin{itemize}
        \item[(1)] $\LL[\trho_h,\tu_h,\tW_h] = \LL[\rho_h,u_h,W],h\in (0,1)$;
        \item[(2)] $(\trho_h,\tu_h,\tW_h)$ converges $\tPP$-almost surely to $(\trho,\tu,\tW)$ in the topology of $\XX$, i.e., 
        \begin{align}
            \label{eq:3.41}
            \begin{aligned}
                &\trho_h\to \trho\quad \text{in} \ C^{\iota}([0,T];C^{2+\iota}(\oD))\quad \tPP\text{-a.s.,} \\ 
                &\tu_h\to \tu\quad \text{in} \ C^{\kappa}([0,T];V_m)\quad \tPP\text{-a.s.,} \\
                &\tW_h\to \tW\quad \text{in} \ C([0,T];\fU_0)\quad \tPP\text{-a.s,}
            \end{aligned}
        \end{align}
    \end{itemize}
    where, for simplicity, we omit the notation for subsequences.
\end{proposition}
\begin{remark}
    Since trajectories of $\trho$, $\tu$ and $\tW$ are $\tPP$-a.s. continuous, $[\trho,\tu,\tW]$ is progressively measurable with respect to the canonical filtration  
    $$
        \tFFF_t:= \sigma \left(\sigma_t[\trho]\cup \sigma_t[\tu]\cup \sigma_t[\tW]\right),\quad t\in [0,T].
    $$
    In view of Lemma \ref{lem:sufficient_cond_of_Wiener_by_law}, the process $\tW$ is a cylindrical Wiener process with respect to its canonical filtration. In order to show that $\tW$ is a cylindrical $(\tFFF_t)$-Wiener process, we intend to apply Lemma \ref{lem:sufficient_cond_of_(G_t)-Wiener }. 
    Hence, we need to show that filtration is non-anticipative with respect to $\tW$. 
    To this end, noting that the iterative construction of $[\rho_h,u_h]$, we see that the filtration 
    $$
        \sigma \left(\sigma_t[\trho_h]\cup \sigma_t[\tu_h]\cup  \sigma_t[\tW_{h}]\right),\quad t\in [0,T]
    $$
    is non-anticipative with respect to $\tW_h$. Therefore, Lemma \ref{lem:stability_of_nonanti} yields the claim.
\end{remark}
We prove that $[\trho,\tu,\tW]$ is a solution of \eqref{eq:3.11}--\eqref{eq:3.12} in the sense of Definition \ref{def:3.1}. First, we show that $[\trho,\tu]$ solves the approximate continuity equation.
\begin{proposition}
    \label{prop:eq_of_conti1}
    The process $[\trho,\tu]$ satisfies \eqref{eq:3.13} in $(0,T)\times D$, $\tPP$-a.s.
\end{proposition}
\begin{proof}
    As a consequence of the equality of laws from Proposition \ref{prop:repr1}, \eqref{eq:3.20} is satisfied by $[\trho_h,\tu_h]$. Indeed, it is easy to check that $\trho_h$ satisfies the Neumann boundary condition, and that
    \begin{align*}
        &\intT \pp_t\phi\int_D \trho_h\psi dxdt +\phi(0)\int_D\trho_h(0)\psi dx +\intT \phi\int_D\trho_h [\tu_h]_{h,R}\cdot \nabla\psi dxdt - \intT \phi\int_D \varepsilon \nabla \trho_h \cdot \nabla \psi dxdt \\ 
        &\quad \overset{d}{=} \intT \pp_t\phi\int_D \rho_h\psi dxdt +\phi(0)\int_D\rho_h(0)\psi dx +\intT \phi\int_D\rho_h [u_h]_{h,R}\cdot \nabla\psi dxdt - \intT \phi\int_D \varepsilon \nabla\rho_h \cdot \nabla \psi dxdt  
    \end{align*}
    for all $\phi\in C_c^\infty(\ico{0}{T})$ and all $\psi\in C_c^\infty(\RR^3)$. Since the right-hand side is $\PP$-a.s. equal to zero and $C_c^\infty(\ico{0}{T}\times\RR^3)$ is separable, $\trho_h$ is a solution to the problem \eqref{eq:3.20} in the weak sense. 
    But, the regularity of $[\trho_h,\tu_h]$ and the approximation of $1_{[0,\tau]}$ by test functions via mollification imply that the equation 
    \begin{align*}
        &\trho_h(\tau,x) = \trho_h(0,x) - \inttau \rdiv\left(\trho_h [\tu_h]_{h,R}\right)(t,x) dt + \inttau \varepsilon \Delta \trho_h (t,x) dt
    \end{align*}
    holds for a.a. $\tau\in [0,T]$ and all $x\in D$ $\tPP$-a.s., and thus $\trho_h$ is a solution to \eqref{eq:3.20} in the class 
    \begin{align*}
        &\rho\in C^\iota([0,T];C^{2+\iota}(\oD)),\quad \pp_t \rho \in L^\infty(0,T;C^\iota(\oD)).  
    \end{align*}
    On the other hand, since $\norm{\cdot}_{C_t^\nu C_x^{2+\nu}}$ is measurable in $C^\iota([0,T];C^{2+\iota}(\oD))$, the uniform bound of $\norm{\trho_h}_{C_t^\nu C_x^{2+\nu}}$ follows from \eqref{eq:3.30} and the equality of laws. In particular $\trho_h(0)$ belongs to $C^{2+\nu}(\oD)$ and satisfies the initial conditions \eqref{eq:3.18}. 
    Therefore, following the same argument used to derive \eqref{eq:3.23} and \eqref{eq:3.24}, the estimate \eqref{eq:3.30} is satisfied by $\trho_h$ uniformly in $h$.
    
    Next, by Proposition \ref{prop:repr1} and the dominated convergence theorem, passing to the limit in the weak form for $[\trho_h,\tu_h]$ yields that $[\trho,\tu]$ satisfies the identity 
    \begin{align*}
        &-\intT \pp_t\phi\int_D \trho\psi dxdt  = \phi(0)\int_D\trho(0)\psi dx +\intT \phi\int_D\trho [\tu]_{R}\cdot \nabla\psi dxdt +\intT \phi\int_D \varepsilon \nabla\trho\cdot \nabla \psi dxdt \\   
    \end{align*}
    for all $\phi\in C_c^\infty(\ico{0}{T})$ and all $\psi\in C_c^\infty(\RR^3)$ $\tPP$-a.s., and satisfies the Neumann boundary condition. Hence, due to the regularity of $[\tu]_R$ and the fact that the laws of $\trho_h(0)$ and $\trho(0)$ coincide with each other, by the same argument as above, $[\trho,\tu]$ satisfies \eqref{eq:3.13} in $(0,T)\times D$, $\tPP$-a.s., and has the regularity as in Definition \ref{def:3.1} (3) with the estimate \eqref{eq:3.16}.
\end{proof} 
Next, we show that $[\trho,\tu,\tW]$ satisfies the approximate momentum equation \eqref{eq:3.14}.
\begin{proposition}
    \label{prop:ME1}
    The process $[\trho,\tu,\tW]$ satisfies \eqref{eq:3.14} for all $\tau\in [0,T]$ and all $\bphi\in V_m$ $\tPP$-a.s.
\end{proposition}
\begin{proof}
    Similarly to the proof of Proposition \ref{prop:eq_of_conti1}, we see that $[\trho_h,\tu_h,\tW_h]$ satisfies \eqref{eq:3.21} and the uniform bound \eqref{eq:3.32}. Indeed, by considering the finite-dimensional approximation and the time-discrete approximation of the stochastic integral in \eqref{eq:3.21}, we can show the coincidence of the laws of the weak forms, and the claim follows by passing to the limit. 

    We observe that 
    \begin{align}
        \norm{[\tu_h]_h(t)-\tu_h(t)}_{V_m}&\lesssim h^\beta \norm{\tu_h}_{C_t^\beta V_m},\notag \\ 
        \norm{[\trho_h]_h(t)-\trho_h(t)}_{C_x^{2+\nu}}&\lesssim h^\nu \norm{\trho_h}_{C_t^\nu C_x^{2+\nu}}. \label{eq:3.42}
    \end{align}
    Now, with the convergence \eqref{eq:3.41}, the bounds \eqref{eq:3.30}, \eqref{eq:3.32} and \eqref{eq:3.42}, and the assumption \eqref{eq:3.18} at hand, we may pass to the limit in the approximate momentum equation \eqref{eq:3.21}. 
    Since the convergence of the deterministic part in \eqref{eq:3.21} can be easily verified, it is only necessary to explain the convergence of the stochastic integral.
    It is easy to see that $\tPP$-a.s. 
    \begin{align*}
        \left[[\trho_h]_h F_{k,\varepsilon}([\trho_h]_h,[\tu_h]_h )\right]_m^\ast \to [\trho F_{k,\varepsilon}(\trho,\tu)]_m^\ast\quad \text{in}\ L^2(0,T;V_m^\ast),
    \end{align*}
    and by Vitali's convergence theorem, we have 
    \begin{align}
        \label{eq:3.43}
        \left[[\trho_h]_h F_{k,\varepsilon}([\trho_h]_h,[\tu_h]_h )\right]_m^\ast \to [\trho F_{k,\varepsilon}(\trho,\tu)]_m^\ast \quad \text{in}\ L^2(\Omega\times (0,T);V_m^\ast)
    \end{align}
    for all $k\in \NN$. On the other hand, we have 
    \begin{align}
        \tEE\left[\intT \norm{[[\trho_h]_h \FF_{\varepsilon}([\trho_h]_h,[\tu_h]_h)]_m^\ast}_{L_2(\fU,V_m^\ast)}^2 dt \right] &= \sum_{k=1}^{\infty}\tEE\left[\intT \norm{[[\trho_h]_hF_{k,\varepsilon}([\trho_h]_h,[\tu_h]_h)]_m^\ast}_{V_m^\ast}^2\right]\notag \\ 
        &\leq \sum_{k=1}^{\infty}\tEE\left[\intT \norm{[[\trho_h]_hF_{k,\varepsilon}([\trho_h]_h,[\tu_h]_h)]_m^\ast}_{L_x^2}^2\right] \notag\\ 
        &\leq \norm{\trho_h}_{L_{\omega,t,x}^\infty}^2\sum_{k=1}^{\infty}\tEE\left[\intT \norm{F_{k,\varepsilon}([\trho_h]_h,[\tu_h]_h)}_{L_x^2}^2\right] \notag \\ 
        &\lesssim \norm{\trho_h}_{L_{\omega,t,x}^\infty}^2 \sum_{k=1}^{\infty}f_{k,\varepsilon}^2 \lesssim 1 \label{eq:3.44}
    \end{align}
    uniformly in $h$, using \eqref{eq:3.10} and \eqref{eq:3.30}. Therefore, combine \eqref{eq:3.43} and \eqref{eq:3.44} to obtain 
    \begin{align}
        [[\trho_h]_h \FF_{\varepsilon}([\trho_h]_h,[\tu_h]_h)]_m^\ast\to [\trho \FF_{\varepsilon}(\trho,\tu)]_m^\ast\quad \text{in}\ L^2(\Omega\times (0,T);L_2(\fU,V_m^\ast)), 
    \end{align}
    and we may apply Lemma \ref{lem:conv_of_stoch_int} to pass to the limit in the stochastic integral and hence complete the proof.
\end{proof} 
The proof of Theorem \ref{thm:sol1} is hereby complete.

\subsection{Well-posedness of the basic approximate problem}\label{sec:Well-posedness of the basic approximate problem}
In this section, we show that for the basic approximate problem, a pathwise solution (a strong solution in the probabilistic sense) exists uniquely, and that the assumption on the moment of the initial velocity can be dropped.
This is essential for employing a stopping time argument taking the limit in the cut-off approximation layer.
Furthermore, we derive the energy balance essential for approximations in subsequent sections.
We adopt the following strategy.
\begin{itemize}
    \item[1.] With a classical Yamada--Watanabe type argument in mind, we show pathwise uniqueness.
    \item[2.] Using the following characterization of convergence in probability observed in Gy\"{o}ngy--Krylov \cite[Lemma 1.1]{GK96}, we provide a direct proof of the classical result and derive the existence of a pathwise solution.
    \item[3.] We show the existence of a unique pathwise solution for general initial data.
    \item[4.] Using It\^{o}'s formula, we prove that any solution to \eqref{eq:3.11}--\eqref{eq:3.12} satisfies a variant of the energy balance.
\end{itemize}

\begin{lemma}
    \label{lem:GK96}
    Let $X$ be a Polish space equipped with the Borel $\sigma$-algebra. A sequence of $X$-valued random variables $(U_n)_{n\in \NN}$ converges in probability if and only if for every sequence of joint laws of $(U_{n_k},U_{m_k})_{k\in \NN}$ there exists a further subsequence which converges weakly to a probability measure $\mu$ such that 
    $$
        \mu\left(\left\{ (x,y)\in X\times X: x=y \right\}\right)=1.
    $$
\end{lemma}
We first show the pathwise uniqueness of the approximate problem \eqref{eq:3.11}--\eqref{eq:3.12}.
\begin{proposition}
    \label{prop:pathwise_uniq}
    Let $(\stochbasis,\rho_1,u_1,W),\ (\stochbasis,\rho_2,u_2,W)$ be two martingale solutions of problem \eqref{eq:3.11}--\eqref{eq:3.12} in the sense of Definition \ref{def:3.1} satisfying \eqref{eq:3.16}--\eqref{eq:3.17} and 
    \begin{align*}
        &[\rho_1,u_1](0) = [\rho_2,u_2](0)\quad \text{in} \ C^{2+\nu}(\oD)\times V_m\ \PP\text{-a.s.}  
    \end{align*}
    Then 
    \begin{align*}
        &[\rho_1,u_1] = [\rho_2,u_2]\quad \text{in} \ C([0,T];C^{2+\nu}(\oD)\times V_m)\ \PP\text{-a.s.}    
    \end{align*}
\end{proposition} 
\begin{proof}
    Introduce the $(\FFF_t)$-stopping times 
    \begin{align*}
        &\tau_M^i = \inf \left\{ t\in [0,T] : \norm{u_i(t)}_{V_m}>M \right\},\quad i=1,2
    \end{align*}
    with the convention $\inf \emptyset =T$, as well as $\tau_M= \tau_M^1\wedge \tau_M^2$.  
    Note that the stopping times are well-defined, due to the continuity of the involved quantity. Moreover, as a consequence of \eqref{eq:3.16} and \eqref{eq:3.17}, we have 
    \begin{align*}
        &\tau_M\leq \tau_L,\quad \text{whenever}\ M\leq L, \\ 
        &\tau_M\uparrow T\quad \PP\text{-a.s.}  
    \end{align*}
    Noting that 
    \begin{align*}
        d[\rho_i u_i]_m^\ast &= [u_i d\rho_i]_m^\ast + [\rho_i du_i]_m^\ast \\ 
        &=[u_i \pp_t \rho_i]_m^\ast dt + [\rho_i du_i]_m^\ast, \quad i=1,2,
    \end{align*}
    and recalling $u= \MMM^{-1}[\rho]([\rho u]_m^\ast)$, we can rewrite \eqref{eq:3.14} in the form 
    \begin{align}
        du + \MMM^{-1}[\rho]\left([ud\rho]_m^\ast\right) &= \MMM^{-1}[\rho]\left(d[\rho u]_m^\ast\right) \notag\\ 
        &= \MMM^{-1}[\rho]\NNN[\rho](u) dt + \MMM^{-1}[\rho]([\GG_\varepsilon(\rho,\rho u)]_m^\ast) dW. \label{eq:3.46}
    \end{align}
    Taking the difference of equations \eqref{eq:3.13} and \eqref{eq:3.46}, we have 
    \begin{align}
        &\pp_t(\rho_1-\rho_2)  - \varepsilon \Delta (\rho_1-\rho_2) = - \rdiv \left(\rho_1[u_1]_R - \rho_2 [u_2]_R\right),\label{eq:3.47}
    \end{align}
    and
    \begin{align}
        d(u_1 - u_2) &= \left(\MMM^{-1}[\rho_2]-\MMM^{-1}[\rho_1]\right)\left([\pp_t \rho_1 u_1]_m^\ast\right)dt + \MMM^{-1}[\rho_2]\left([\pp_t \rho_2 u_2 - \pp_t \rho_1 u_1]_m^\ast\right)dt \notag \\ 
        &\quad + \left(\MMM^{-1}[\rho_1]-\MMM^{-1}[\rho_2]\right)\NNN[\rho_1](u_1)dt + \MMM^{-1}[\rho_2]\left(\NNN[\rho_1](u_1)-\NNN[\rho_2](u_2)\right)dt \notag \\ 
        &\quad + \left(\MMM^{-1}[\rho_1]-\MMM^{-1}[\rho_2]\right)([\GG_\varepsilon(\rho_1,\rho_1 u_1)]_m^\ast)dW \notag \\ 
        &\quad + \MMM^{-1}[\rho_2]([\GG_\varepsilon(\rho_1,\rho_1 u_1)]_m^\ast-[\GG_\varepsilon(\rho_2,\rho_2 u_2)]_m^\ast)dW. \label{eq:3.48}
    \end{align}
    It\^{o}'s formula applied to \eqref{eq:3.48} yields 
    \begin{align*}
        \frac 1 2 d\norm{u_1-u_2}_{V_m}^2 &= \int_D\left(\MMM^{-1}[\rho_2]-\MMM^{-1}[\rho_1]\right)\left([\pp_t \rho_1 u_1]_m^\ast\right)\cdot (u_1-u_2)dx dt \\ 
        &\quad + \int_D \MMM^{-1}[\rho_2]\left([\pp_t \rho_2 u_2 - \pp_t \rho_1 u_1]_m^\ast\right) \cdot (u_1-u_2)dxdt \\ 
        &\quad + \int_D \left(\MMM^{-1}[\rho_1]-\MMM^{-1}[\rho_2]\right)\NNN[\rho_1](u_1) \cdot (u_1-u_2)dx dt \\ 
        &\quad + \int_D \MMM^{-1}[\rho_2]\left(\NNN[\rho_1](u_1)-\NNN[\rho_2](u_2)\right) \cdot (u_1-u_2)dx dt \\ 
        &\quad +\frac 1 2 \sum_{k=1 }^{\infty} \int_D \left|\MMM^{-1}[\rho_1]([G_{k,\varepsilon}(\rho_1,\rho_1 u_1)]_m^\ast) - \MMM^{-1}[\rho_2]([G_{k,\varepsilon}(\rho_2,\rho_2 u_2)]_m^\ast)\right|^2 dx dt \\ 
        &\quad + \int_D \left(\MMM^{-1}[\rho_1]-\MMM^{-1}[\rho_2]\right)([\GG_\varepsilon(\rho_1,\rho_1 u_1)]_m^\ast) \cdot (u_1-u_2)dx dW\\ 
        &\quad + \int_D \MMM^{-1}[\rho_2]([\GG_\varepsilon(\rho_1,\rho_1 u_1)]_m^\ast-[\GG_\varepsilon(\rho_2,\rho_2 u_2)]_m^\ast) \cdot (u_1-u_2)dx dW.
    \end{align*}
    Therefore, with \eqref{eq:3.10}, \eqref{eq:3.16}, \eqref{eq:3.26}, \eqref{eq:3.27}, the equivalence of norms on $V_m$ and the definition of $\tau_M$ at hand, we take the supremum in time and pass to expectation, to obtain 
    \begin{align}
        &\EE\left[\sup_{t\in [0,T]}\norm{(u_1-u_2)(t\wedge \tau_M)}_{V_m}^2\right] \notag \\ 
        &\quad \leq \EE\left[\norm{(u_1-u_2)(0)}_{V_m}^2\right]+ c(M)\EE\left[\int_{0}^{T\wedge \tau_M}\norm{\rho_1-\rho_2}_{C_x^{2+\nu}}\norm{u_1-u_2}_{V_m}dt\right] \notag \\ 
        &\quad \quad +c \EE\left[\int_{0}^{T\wedge \tau_M}\left(\norm{u_1-u_2}_{V_m}^2 + \norm{\rho_1-\rho_2}_{L_x^2}^2\right)dt \right] + \EE\left[\sup_{t\in [0,T]}|\fM_{t\wedge \tau_M}|\right] \notag \\ 
        &\quad \leq \kappa \EE\left[\sup_{t\in [0,T]}\norm{(\rho_1-\rho_2)(t\wedge \tau_M)}_{C_x^{2+\nu}}^2\right] + \EE\left[\norm{(u_1-u_2)(0)}_{V_m}^2\right] \notag \\ 
        &\quad \quad +c(M,\kappa) \EE\left[\int_{0}^{T\wedge \tau_M}\left(\norm{u_1-u_2}_{V_m}^2 + \norm{\rho_1-\rho_2}_{L_x^2}^2\right)dt \right] + \EE\left[\sup_{t\in [0,T]}|\fM_{t\wedge \tau_M}|\right], \label{eq:3.49}
    \end{align}
    for any $\kappa>0$, where 
    \begin{align*}
        d\fM &= \int_D \left(\MMM^{-1}[\rho_1]-\MMM^{-1}[\rho_2]\right)([\GG_\varepsilon(\rho_1,\rho_1 u_1)]_m^\ast) \cdot (u_1-u_2)dx dW\\ 
        &\quad + \int_D \MMM^{-1}[\rho_2]([\GG_\varepsilon(\rho_1,\rho_1 u_1)]_m^\ast-[\GG_\varepsilon(\rho_2,\rho_2 u_2)]_m^\ast) \cdot (u_1-u_2)dx dW.
    \end{align*}
    By Burkholder--Davis--Gundy's inequality and \eqref{eq:3.10}, we estimate $\fM$ similarly by 
    \begin{align*}
        &\EE\left[\sup_{t\in [0,T]}\left|\fM_{t \wedge \tau_M} \right|\right] \\ 
        &\quad\lesssim \EE\left[\left(\int_{0}^{T\wedge \tau_M}\sum_{k=1}^{\infty}\left( \int_D \left(\MMM^{-1}[\rho_1]-\MMM^{-1}[\rho_2]\right)([G_{k,\varepsilon}(\rho_1,\rho_1 u_1)]_m^\ast) \cdot (u_1-u_2)dx\right)^2dt \right)^{1/2}\right] \\ 
        &\quad \quad  + \EE\left[\left(\int_{0}^{T\wedge \tau_M}\sum_{k=1}^{\infty} \left(\int_D  \MMM^{-1}[\rho_2]([G_{k,\varepsilon}(\rho_1,\rho_1 u_1)]_m^\ast-[G_{k,\varepsilon}(\rho_2,\rho_2 u_2)]_m^\ast) \cdot (u_1-u_2)dx\right)^2 dt \right)^{1/2}\right] \\
        &\quad \lesssim \EE\left[\sup_{t\in [0,T]}\norm{\rho_1(t\wedge \tau_M)-\rho_2(t\wedge \tau_M)}_{L_x^1}\norm{u_1-u_2}_{L^2(0,T\wedge \tau_M;L^2(D))}\right] \\ 
        &\quad \quad + \EE\left[\sup_{t\in [0,T]}\left(\norm{\rho_1(t\wedge \tau_M)-\rho_2(t\wedge \tau_M)}_{L_x^2} + \norm{u_1(t\wedge \tau_M) - u_2(t\wedge \tau_M)}_{L_x^2}\right)\norm{u_1-u_2}_{L^2(0,T\wedge \tau_M;L^2(D))}\right] \\ 
        &\quad \leq \kappa \EE\left[\sup_{t\in [0,T]}\int_D \left(|(\rho_1-\rho_2)(t\wedge \tau_M)|^2 + |(u_1-u_2)(t\wedge \tau_M)|^2\right)dx\right] + c(\kappa) \EE\left[\int_{0}^{T\wedge \tau_M}\int_D |u_1-u_2|^2dx dt\right],
    \end{align*}
    where $\kappa>0$ is arbitrary. Applying this to \eqref{eq:3.49} implies 
    \begin{align}
        \EE\left[\sup_{t\in [0,T]}\norm{(u_1-u_2)(t\wedge \tau_M)}_{V_m}^2\right] &\leq  \kappa \EE\left[\sup_{t\in [0,T]}\norm{(\rho_1-\rho_2)(t\wedge \tau_M)}_{C_x^{2+\nu}}^2\right] + \EE\left[\norm{(u_1-u_2)(0)}_{V_m}^2\right] \notag \\ 
        &\quad +c(M,\kappa) \EE\left[\int_{0}^{T\wedge \tau_M}\left(\norm{u_1-u_2}_{V_m}^2 + \norm{\rho_1-\rho_2}_{L_x^2}^2\right)dt \right], \label{eq:3.50}
    \end{align}
    for any $\kappa>0$. On the other hand, the standard parabolic regularity theory (see \cite[Theorem 11.30]{FN17}) applied to \eqref{eq:3.47} provides the estimate
    \begin{align*}
        &\sup_{t\in [0,\tau]} \left(\norm{(\rho_1-\rho_2)(t)}_{C_x^{2+\nu}} + \norm{\pp_t (\rho_1-\rho_2)(t)}_{C_x^\nu}\right)  \\ 
        &\quad \lesssim \sup_{t\in [0,\tau]} \norm{\rdiv\left(\rho_1[u_1]_R -\rho_2[u_2]_R\right)(t)}_{C_x^\nu} +\norm{(\rho_1-\rho_2)(0)}_{C_x^{2+\nu}},
    \end{align*}
    such that, using the definition of $[\cdot]_R$, 
    \begin{align}
        &\sup_{t\in [0,\tau]} \left(\norm{(\rho_1-\rho_2)(t)}_{C_x^{2+\nu}} + \norm{\pp_t (\rho_1-\rho_2)(t)}_{C_x^\nu}\right)  \notag \\ 
        &\quad \lesssim \sup_{t\in [0,\tau]} \norm{(\rho_1-\rho_2)(t)}_{C_x^{1+\nu}} + \sup_{t\in [0,\tau]}\norm{(u_1-u_2)(t)}_{V_m} +\norm{(\rho_1-\rho_2)(0)}_{C_x^{2+\nu}}. \label{eq:3.51}
    \end{align}
    It is easy to show (for instance by contradiction) that, for every $\kappa>0$, there is some $c(\kappa)$ such that 
    \begin{align*}
        \sup_{t\in [0,\tau]}\norm{v}_{C_x^{1+\nu}} \leq \kappa \left(\sup_{t\in [0,\tau]}\norm{v}_{C_x^{2+\nu}} + \sup_{t\in [0,\tau]} \norm{\pp_t v}_{C_x^\nu}\right) + c(\kappa) \left(\inttau \norm{v}_{C_x^\nu}^2 dt\right)^{1/2}, 
    \end{align*}
    for all $v\in W^{1,\infty}(0,\tau;C^{\nu}(\oD))\cap L^\infty(0,\tau;C^{2+\nu}(\oD))$. 
    Indeed, if we choose $\kappa>0$, and
    \begin{align*}
        &(v_n)_{n\in\NN}\subset W^{1,\infty}(0,\tau;C^{\nu}(\oD))\cap L^\infty(0,\tau;C^{2+\nu}(\oD))  
    \end{align*}
    such that 
    \begin{align*}
        &\norm{v_n}_{L_t^\infty C_x^{1+\nu}} = 1, \\ 
        &1 > \kappa \left(\norm{v_n}_{L_t^\infty C_x^{2+\nu}} + \norm{\pp_t v_n}_{L_t^\infty C_x^\nu}\right) + n \norm{v_n}_{L_t^2 C_x^\nu}, 
    \end{align*}
    for all $n\in\NN$, then, by the Aubin--Lions--Simon theorem (see \cite[Section 8]{Sim87}), $(v_n)$ contains a convergence subsequence in $C_t C^{1+\nu}_x$, and write the limit as $v$.
    In particular, 
    \begin{align*}
        &\norm{v}_{L_t^\infty C_x^{1+\nu}} =1.   
    \end{align*}
    On the other hand, since 
    \begin{align*}
        &\norm{v_n}_{L_t^2 C_x^\nu} \to 0,   
    \end{align*}
    we have $v\equiv 0$, which contradicts $\norm{v}_{L_t^\infty C_x^{1+\nu}}=1$.
    
    Applying this inequality to \eqref{eq:3.51} and choosing $\kappa$ small enough yields 
    \begin{align}
        &\sup_{t\in [0,\tau]} \left(\norm{(\rho_1-\rho_2)(t)}_{C_x^{2+\nu}} + \norm{\pp_t (\rho_1-\rho_2)(t)}_{C_x^\nu}\right)  \notag \\ 
        &\quad \lesssim \left(\inttau \norm{\rho_1-\rho_2}_{C_x^{2+\nu}}^2 dt\right)^{1/2} + \sup_{t\in [0,\tau]}\norm{(u_1-u_2)(t)}_{V_m} +\norm{(\rho_1-\rho_2)(0)}_{C_x^{2+\nu}}. \label{eq:3.52}
    \end{align}
    Next, we combine \eqref{eq:3.50} and \eqref{eq:3.52} at $\tau= T\wedge \tau_M$, and choose $\kappa$ small enough to get 
    \begin{align*}
        &\EE\left[\sup_{t\in [0,T]}\left(\norm{(u_1-u_2)(t\wedge \tau_M)}_{V_m}^2 + \norm{(\rho_1-\rho_2)(t\wedge \tau_M)}_{C_x^{2+\nu}}^2\right)\right] \\ 
        &\quad \lesssim \EE\left[\norm{(u_1-u_2)(0)}_{V_m}^2 + \norm{(\rho_1-\rho_2)(0)}_{C_x^{2+\nu}}^2\right] + \EE\left[\int_{0}^{T\wedge \tau_M}\left(\norm{u_1-u_2}_{V_m}^2 + \norm{\rho_1-\rho_2}_{C_x^{2+\nu}}^2\right)dt \right],
    \end{align*}
    with a constant depending on $M$. Finally, a direct application of Gronwall's lemma yields 
    \begin{align*}
        &\EE\left[\sup_{t\in [0,T]}\left(\norm{(u_1-u_2)(t\wedge \tau_M)}_{V_m}^2 + \norm{(\rho_1-\rho_2)(t\wedge \tau_M)}_{C_x^{2+\nu}}^2\right)\right] \\ 
        &\quad \lesssim \EE\left[\norm{(u_1-u_2)(0)}_{V_m}^2 + \norm{(\rho_1-\rho_2)(0)}_{C_x^{2+\nu}}^2\right],
    \end{align*}
    with a constant depending on $M$. Now, as the initial data coincide, the desired conclusion follows by sending $M\to \infty$.
\end{proof} 

Thanks to the pathwise uniqueness established in Proposition \ref{prop:pathwise_uniq}, we are able to show existence of a unique pathwise solution to \eqref{eq:3.11}--\eqref{eq:3.12}. To be more precise, we solve \eqref{eq:3.11}--\eqref{eq:3.12} in the context of the following definition.
\begin{definition}
    \label{def:pathwise_sol}
    Let $\stochbasis$ be a given stochastic basis with a complete right-continuous filtration, let $W$ be a cylindrical $(\FFF_t)$-Wiener process on $\fU$, and let $(\rho_0,u_0)$ be an $\FFF_0$-measurable random variable taking values in $C^{2+\nu}(\oD)\times V_m$. Then $(\rho,u)$ is called a \textit{pathwise solution} to \eqref{eq:3.11}--\eqref{eq:3.12} with the initial datum $(\rho_0,u_0)$, if the following hold: 
	\begin{itemize}
		\item[(1)] $\rho$ and $u$ are $(\FFF_t)$-adapted stochastic processes such that $\PP$-a.s.
		\begin{align*}
		   &\rho> 0,\quad \rho \in C([0,T];C^{2+\nu}(\oD))\cap C^1([0,T];C^{\nu}(\oD)),\quad u\in C([0,T];V_m);   
		\end{align*}
		\item[(2)] $(\rho(0),u(0)) = (\rho_0,u_0)$ $\PP$-a.s.;
		\item[(3)] the approximate equation of continuity 
		\begin{equation}
            \pp_t \rho + \rdiv (\rho [u]_R) = \varepsilon \Delta \rho,\quad \nabla \rho \cdot \rn |_{\pD} =0 \label{eq:3.53}
        \end{equation}
		holds in $(0,T)\times D$ $\PP$-a.s.;
		\item[(4)] the approximate momentum equation 
		\begin{align}
            \int_D \rho u\cdot \bphi (\tau)dx &= \int_D \rho_0 u_0 \cdot \bphi dx  + \inttau \int_D \left[(\rho[u]_R \otimes u ):\nabla \bphi+ \chi\left(\norm{u}_{V_m}-R\right)p_\delta (\rho)\rdiv \bphi\right]dxdt  \notag \\ 
            &\quad - \inttau \int_D \left[\SSS(\nabla u):\nabla\bphi -\varepsilon \rho u\cdot \Delta \bphi \right]dxdt - \inttau \int_{\pD}g\nabla j_\alpha(u)\cdot \bphi d\Gamma dt \notag \\ 
            &\quad + \inttau\int_D \GG_\varepsilon(\rho,\rho u) \cdot \bphi dxdW \label{eq:3.54}
        \end{align}
		holds for all $\tau\in [0,T]$ and all $\bphi\in V_m$ $\PP$-a.s.
    \end{itemize}
\end{definition}
The main result of this section reads as follows. 
\begin{theorem}
    \label{thm:pathwise_sol}
    Let $\rho_0\in C^{2+\nu}(\oD),u_0\in V_m$ be $\FFF_0$-measurable random variables satisfying 
    \begin{align}
        \PP\left(\left\{ \urho\leq \rho_0,\norm{\rho_0}_{C_x^{2+\nu}}\leq \orho, \nabla\rho_0\cdot \rn |_{\pD}=0 \right\}\right) = 1,\quad \EE\left[\norm{u_0}_{V_m}^r\right]\leq \overline{u},
    \end{align}
    for some deterministic constants $\urho,\orho,\overline{u}>0,r>2$. Then the approximate problem \eqref{eq:3.11}--\eqref{eq:3.12} admits a pathwise solution $[\rho,u]$ in the sense of Definition \ref{def:pathwise_sol}, which is unique among solutions satisfying the following
    \begin{align}
       &\sup_{t\in [0,T]} \left(\norm{\rho(t)}_{C_x^{2+\nu}}+\norm{\pp_t \rho(t)}_{C_x^{\nu}} +\norm{\rho^{-1}(t)}_{C_x^0}\right)\leq c\quad \PP\text{-a.s.,} \label{eq:3.56} \\ 
       &\EE\left[\sup_{\tau\in [0,T]}\norm{u(\tau)}_{V_m}^r\right] \leq c \left(1+\EE\left[\norm{u_0}_{V_m}^r\right]\right), 
    \end{align}
    with a constant $c=c(m,R,T,\urho,\orho,\nu,\varepsilon)$.
\end{theorem}
\begin{proof}
    Consider a family of approximate solutions 
    \begin{align*}
        \left\{ [\rho_l,u_l]:l\in \NN \right\} := \left\{ [\rho_{h_l},u_{h_l}]:l\in \NN \right\}
    \end{align*}
    constructed on the original probability space $\stochbasis$, by means of the iteration procedure \eqref{eq:3.20}--\eqref{eq:3.21}. Similarly to Section \ref{sec:The limit for vanishing time step}, we apply the Skorokhod representation theorem to the joint laws generated by the random variables 
    $$
        \left\{ [\rho_{n_k},u_{n_k},\rho_{m_k},u_{m_k},W]:k\in \NN \right\}
    $$
    in the space 
    $$
        \left[\XX_\rho\times \XX_u\right]^2\times \XX_W.
    $$
    Consequently, we obtain a subsequence 
    $$
        \left\{ [\trho_{n_{k_l}},\tu_{n_{k_l}},\trho_{m_{k_l}},\tu_{m_{k_l}},\tW_{l}]:l\in \NN \right\}
    $$
    defined on a new probability space $(\tOmega,\tFFF,\tPP)$ converging $\tPP$-a.s. to some random variable $[\trho_1,\tu_1,\trho_2,\tu_2,\tW]$. 
    Note that the $\sigma$-field $(\tFFF_t)_{t\geq 0}$ generated by $[\trho_1,\tu_1,\trho_2,\tu_2,\tW]$ is non-anticipative with respect to $\tW$ due to the construction of $[\rho_h,u_h]$, Lemma \ref{lem:sufficient_cond_of_Wiener_by_law} and Lemma \ref{lem:stability_of_nonanti}.
    Therefore, applying the arguments used in the proof of Theorem \ref{thm:sol1}, $(\tstochbasis,\trho_i,\tu_i,\tW), i=1,2$ are solutions of \eqref{eq:3.11}--\eqref{eq:3.12} with the initial data 
    $$
        [\trho_1(0),\tu_1(0)] =[\trho_2(0),\tu_2(0)]\quad \tPP\text{-a.s.}
    $$
    As a consequence of Proposition \ref{prop:pathwise_uniq}, the two solutions $[\trho_i,\tu_i],i=1,2$ coincide $\tPP$-a.s. Therefore, the family $\left\{ [\rho_l,u_l] : l\in \NN\right\}$ satisfies the hypothesis of Gy\"{o}ngy--Krylov's result, Lemma \ref{lem:GK96}. Passing to a subsequence if necessary, we conclude that $[\rho_l,u_l]$ converges $\PP$-a.s. and therefore gives rise to a solution $[\rho,u]$ defined on the original probability space $\probsp$. 
    Note that the fact that the limit is a solution \eqref{eq:3.11}--\eqref{eq:3.12} follows by repeating the argument used in Section \ref{sec:The limit for vanishing time step}.
\end{proof}

As a corollary of Theorem \ref{thm:pathwise_sol}, the existence of a unique pathwise solution for general initial data follows.
\begin{corollary}
    \label{cor:general_data}
    Let $\stochbasis$ be a stochastic basis with a complete right-continuous filtration and let $W$ be a cylindrical $(\FFF_t)$-Wiener process on $\fU$. Let $[\rho_0,u_0]$ be a given $\FFF_0$-measurable initial datum satisfying
    \begin{align}
        \rho_0\in C^{2+\nu}(\oD),\quad \urho\leq \rho_0\leq \orho,\quad \nabla\rho_0\cdot \rn |_{\pD}=0,\quad u_0\in V_m,\quad \PP\text{-a.s.,}
    \end{align}
    for some deterministic constants $\urho,\orho>0$. Then the approximate problem \eqref{eq:3.11}--\eqref{eq:3.12} admits a unique pathwise solution $[\rho,u]$ in the sense of Definition \ref{def:pathwise_sol}. The solution satisfies in addition the estimate \eqref{eq:3.56}.
\end{corollary}

\begin{proof}
    For a given $M>0$, we consider the set 
    $$
        I_M = \left\{ (\rho,v)\in C^{2+\nu}(\oD)\times V_m: \norm{\rho}_{C_x^{2+\nu}}\leq M,\norm{v}_{V_m}\leq M \right\}.
    $$
    We modify the initial data introducing 
    $$
        [\rho_{0,M},u_{0,M}] =\begin{cases}
        [\rho_0,u_0] &\text{if} \  [\rho_0,u_0]\in I_M \\
        [M+1,0] &\text{otherwise}. 
        \end{cases}
    $$
    By Theorem \ref{thm:pathwise_sol}, the problem \eqref{eq:3.11}--\eqref{eq:3.12} admits a unique pathwise solution $[\rho_M,u_M]$ with the initial data $[\rho_{0,M},u_{0,M}]$, for any $M>0$. As a consequence of the uniqueness result stated in Theorem \ref{thm:pathwise_sol},
    $$
        [\rho_N,u_N] = [\rho_M,u_M]\quad \PP\text{-a.s. on the set } \left\{ [\rho_0,u_0]\in I_N \right\},
    $$
    whenever $M \geq N$.
    Seeing that 
    $$
        \PP\left([\rho_0,u_0]\in \bigcup_{N=1}^{\infty}\bigcap_{M\geq N}I_M\right)=1,
    $$
    we define $[\rho,u]$ -- the unique solution of the approximate problem \eqref{eq:3.11}--\eqref{eq:3.12} -- as 
    $$
        [\rho,u] =[\rho_M,u_M]\quad \text{whenever}\ [\rho_0,u_0]\in I_M.
    $$
    This completes the proof of Corollary~\ref{cor:general_data}.
\end{proof}
Finally, we show that any solution of the approximate problem \eqref{eq:3.11}--\eqref{eq:3.12} in the sense of Definition \ref{def:pathwise_sol} satisfies a variant of the energy balance.
\begin{proposition}
    \label{prop:energy_balance}
    Let $[\rho,u]$ be a pathwise solution to \eqref{eq:3.11}--\eqref{eq:3.12} in the sense of Definition \ref{def:pathwise_sol}. Then the energy balance 
    \begin{align}
        &\int_D \left[\frac 1 2 \rho |u|^2 + P_\delta(\rho)\right](\tau)dx + \inttau \int_D \left[\SSS(\nabla u):\nabla u + \varepsilon \rho |\nabla u|^2 + \varepsilon P_\delta''(\rho)|\nabla \rho|^2\right]dxdt +\inttau \int_{\pD} g \nabla j_\alpha(u)\cdot u d\Gamma dt \notag\\ 
        &\quad = \int_D \left[\frac 1 2 \rho_0 |u_0|^2 + P_\delta(\rho_0)\right]dx +\frac 1 2 \sum_{k=1}^{\infty} \inttau \int_D G_{k,\varepsilon}(\rho,\rho u) \cdot \MMM^{-1}[\rho] ([G_{k,\varepsilon}(\rho,\rho u)]_m^\ast) dxdt \notag \\ 
        &\quad \quad + \inttau \int_D \GG_\varepsilon(\rho,\rho u) \cdot u dx dW \label{eq:3.59}
    \end{align}
    holds for all $\tau\in [0,T]$ $\PP$-a.s. with the approximate pressure potential 
    $$
        P_\delta(\rho) = \rho \int_{1}^{\rho}\frac{p_\delta(z)}{z^2}dz,
    $$
    where $p_\delta$ is given at the beginning of Section \ref{sec:Outline of the proof of the main result}.
\end{proposition}
\begin{proof}
    Recall that \eqref{eq:3.54} can be rewritten in the form 
    \begin{align}
        &d[\rho u]_m^\ast = \NNN[\rho](u) dt + [\GG_\varepsilon (\rho,\rho u)]_m^\ast dW, \label{eq:3.60}
    \end{align}
    where the operator $\NNN$ is given in \eqref{eq:3.27}.
    We formally test this by $u$, that is we apply It\^{o}'s formula to the system \eqref{eq:3.53} and \eqref{eq:3.60}, and the functional 
    $$
        f:L^2(D)\times V_m^\ast \to \RR,\quad (\rho,q)\mapsto \frac 1 2 \int_D q\cdot \MMM^{-1}[\rho]q dx.
    $$
    More precisely, we introduce a cut-off function $\chi_c\in C^\infty(\RR)$ such that 
    \begin{align*}
        &\chi_c(x) = x\quad \text{for } c\leq x\leq \tfrac 1 c, \\ 
        &\chi_c \subset \left[\tfrac c 2, \tfrac{1}{2c}\right], \\
        &\chi_c \ \text{is monotonically increasing},
    \end{align*}
    and the functional $f_c:L^2(D)\times V_m^\ast\to \RR$, 
    $$
        f_c[\rho,q] := \frac 1 2  \int_D q \MMM^{-1}[\chi_c(\rho)]q dx = \frac 1 2 \evm{q}{\MMM^{-1}[\chi_c(\rho)]q}
    $$
    for any $c>0$. Then $\pp_\rho f, \pp_q f_c$, and $\pp_{qq}f_c$ exist, and we have 
    \begin{align*}
        \pp_\rho f_c[\rho,q][h] &= -\frac 1 2 \evm{q}{\MMM^{-1}[\chi_c(\rho)]\MMM[\chi_c'(\rho)h]\MMM^{-1}[\chi_c(\rho)]q},\quad h\in L^2(D), \\ 
        \pp_q f_c[\rho,q][h] &= \evm{h}{\MMM^{-1}[\chi_c(\rho)]q},\quad h\in V_m^\ast, \\ 
        \pp_{qq}f_c[\rho,q][h,k] &= \evm{h}{\MMM^{-1}[\chi_c(\rho)]k},\quad h,k\in V_m^\ast,   
    \end{align*}
    for all $(\rho,q)\in L^2(D)\times V_m^\ast$. 
    Note that all derivatives of $f$ appearing here can be verified directly. Indeed, for $\pp_\rho f_c$, we write $\cA_c[\rho][h] = \MMM^{-1}[\chi_c(\rho)]\MMM[\chi_c'(\rho)h]\MMM^{-1}[\chi_c(\rho)]$, and observe that  
    \begin{align*}
        &\evm{q}{\left(\MMM^{-1}[\chi_c(\rho + h)] - \MMM^{-1}[\chi_c(\rho)] + \cA_c[\rho][h]\right)q} \\ 
        &\quad =  \evm{q}{\left(- \MMM^{-1}[\chi_c(\rho + h)]\left(\MMM[\chi_c(\rho + h)] - \MMM[\chi_c(\rho)]\right)\MMM^{-1}[\chi_c(\rho)] + \cA_c[\rho][h]\right)q} \\ 
        &\quad = \evm{q}{\left(- \MMM^{-1}[\chi_c(\rho + h)]\left(\MMM[\chi_c(\rho + h)] - \MMM[\chi_c(\rho)]\right) + \MMM^{-1}[\chi_c(\rho)]\MMM[\chi_c'(\rho)h] \right) \MMM^{-1}[\chi_c(\rho)]q} \\
        &\quad = \evm{q}{- \left(\MMM^{-1}[\chi_c(\rho + h)] - \MMM^{-1}[\chi_c(\rho)]\right)\left(\MMM[\chi_c(\rho + h)] - \MMM[\chi_c(\rho)]\right) \MMM^{-1}[\chi_c(\rho)]q} \\ 
        &\quad \quad + \evm{q}{\MMM^{-1}[\chi_c(\rho)]\left(\MMM[\chi_c'(\rho)h] - \MMM[\chi_c(\rho + h)] + \MMM[\chi_c(\rho)]\right)\MMM^{-1}[\chi_c(\rho)]q},
    \end{align*}
    for all $(\rho,q)\in L^2(D)\times V_m^\ast$ and all $h\in L^2(D)$. With the definition of $\MMM$, \eqref{eq:3.25}, \eqref{eq:3.26}, and the definition of $\chi_c$ at hand, we have 
    \begin{align*}
        &\norm{\left(\MMM^{-1}[\chi_c(\rho + h)] - \MMM^{-1}[\chi_c(\rho)]\right)\left(\MMM[\chi_c(\rho + h)] - \MMM[\chi_c(\rho)]\right) \MMM^{-1}[\chi_c(\rho)]}_{\LL(V_m^\ast,V_m)} \\ 
        &\quad \lesssim \norm{\MMM^{-1}[\chi_c(\rho + h)] - \MMM^{-1}[\chi_c(\rho)]}_{\LL(V_m^\ast,V_m)} \norm{\MMM[\chi_c(\rho + h)] - \MMM[\chi_c(\rho)]}_{\LL(V_m,V_m^\ast)} \\ 
        &\quad \lesssim \norm{\chi_c(\rho + h) - \chi_c(\rho)}_{L^1(D)}\norm{\chi_c(\rho + h) - \chi_c(\rho)}_{L^2(D)} \\ 
        &\quad \lesssim \norm{h}_{L^2(D)}^2,
    \end{align*}
    and 
    \begin{align*}
        &\norm{\MMM^{-1}[\chi_c(\rho)]\left(\MMM[\chi_c'(\rho)h] - \MMM[\chi_c(\rho + h)] + \MMM[\chi_c(\rho)]\right)\MMM^{-1}[\chi_c(\rho)]}_{\LL(V_m^\ast,V_m)} \\ 
        &\quad \lesssim \norm{\MMM[\chi_c'(\rho)h] - \MMM[\chi_c(\rho + h)] + \MMM[\chi_c(\rho)]}_{\LL(V_m,V_m^\ast)} \\ 
        &\quad \lesssim \norm{\chi_c'(\rho)h - \chi_c(\rho + h) + \chi_c(\rho)}_{L^1(D)} \\ 
        &\quad \lesssim \norm{h^2}_{L^1(D)} = \norm{h}_{L^2(D)}^2,
    \end{align*}
    where the equivalence of norms on $V_m$ is used in the second inequality in the latter. These observations imply $\evm{q}{\cA_c[\rho][\cdot]q}$ is the desired partial derivative $\pp_\rho f_c$.
    
    Therefore, with this, \eqref{eq:3.53} and \eqref{eq:3.60} at hand, It\^{o}'s formula applied to $f_c$ yields
    \begin{align*}
        df_c[\rho,[\rho u]_m^\ast] &= \pp_\rho f_c[\rho,[\rho u]_m^\ast]d\rho + \pp_q f_c[\rho,[\rho u]_m^\ast]d [\rho u]_m^\ast \\ 
        &\quad + \frac 1 2 \Tr\left[\pp_{qq}f_c[\rho,[\rho u]_m^\ast] [\GG_\varepsilon(\rho,\rho u)]_m^\ast\left([\GG_\varepsilon(\rho,\rho u)]_m^\ast\right)^\ast \right] dt \\ 
        &= \pp_\rho f_c[\rho,[\rho u]_m^\ast]d\rho + \evm{\NNN[\rho]([\rho u]_m^\ast)}{\MMM^{-1}[\chi_c(\rho)]([\rho u]_m^\ast)}dt \\ 
        &\quad + \evm{[\GG_\varepsilon(\rho,\rho u)]_m^\ast}{\MMM^{-1}[\chi_c(\rho)]([\rho u]_m^\ast)}dW \\ 
        &\quad + \frac 1 2 \Tr\left[\pp_{qq}f_c[\rho,[\rho u]_m^\ast] [\GG_\varepsilon(\rho,\rho u)]_m^\ast\left([\GG_\varepsilon(\rho,\rho u)]_m^\ast\right)^\ast \right] dt, 
    \end{align*}
    where the symbol $(\cdot)^\ast$ denotes the adjoint of an operator. Since $\rho$ satisfies the estimate \eqref{eq:3.56}, choosing $c$ small enough yields 
    \begin{align*}
        d\left(\frac 1 2 \int_D \rho |u|^2 dx\right) &=-\frac 1 2 \int_D |u|^2 d\rho dx + \evm{\NNN[\rho]([\rho u]_m^\ast)}{u}dt + \int_D \GG_\varepsilon(\rho,\rho u) \cdot u dx dW \\ 
        &\quad + \frac 1 2 \sum_{k=1}^{\infty}\int_D G_{k,\varepsilon}(\rho,\rho u) \cdot \MMM^{-1}[\rho] ([G_{k,\varepsilon}(\rho,\rho u)]_m^\ast) dxdt.
    \end{align*}
    With the boundary conditions for $\rho$ and $u$, respectively, the definition of $\NNN$, and \eqref{eq:3.53} at hand, integrating by parts yields 
    \begin{align}
        d\left(\frac 1 2 \int_D \rho |u|^2 dx\right) &= -\int_D \chi(\norm{u}_{V_m}-R)\nabla p_\delta(\rho) \cdot u dx dt -\int_D \left(\varepsilon \rho |\nabla u|^2 + \SSS(\nabla u):\nabla u\right)dx dt \notag \\ 
        &\quad - \int_{\pD} g\nabla j_\alpha(u)\cdot u d\Gamma dt + \int_D \GG_\varepsilon(\rho,\rho u) \cdot u dx dW \notag \\ 
        &\quad +  \frac 1 2 \sum_{k=1}^{\infty}\int_D G_{k,\varepsilon}(\rho,\rho u) \cdot \MMM^{-1}[\rho] ([G_{k,\varepsilon}(\rho,\rho u)]_m^\ast) dxdt. \label{eq:3.61}
    \end{align}
    Similarly, applying It\^{o}'s formula to the functional 
    $$
        L^2(D)\ni \rho\mapsto \int_D P_\delta(\chi_c(\rho))dx\in \RR,
    $$
    we have
    \begin{align}
        d\int_D P_\delta(\rho)dx &= \int_D P_\delta'(\rho)dx d\rho \notag \\ 
        &= -\int_D P_\delta'(\rho)\rdiv(\rho[u]_R)dx dt + \varepsilon \int_D P_\delta'(\rho)\Delta \rho dx dt \notag \\ 
        &= \int_D \chi(\norm{u}_{V_m}-R)\nabla p_\delta(\rho) \cdot u dx dt - \varepsilon\int_D P_\delta''(\rho)|\nabla \rho|^2 dx dt. \label{eq:3.62}
    \end{align}
    Therefore, taking the sum of \eqref{eq:3.61} and \eqref{eq:3.62} yields the assertion.
\end{proof}
\begin{remark}
    \label{rem:quad_var_est}
    The quadratic variation term in \eqref{eq:3.59} has the following property: 
    \begin{align}
        0&\leq \frac 1 2 \sum_{k=1}^{\infty}\int_D G_{k,\varepsilon}(\rho,\rho u) \cdot \MMM^{-1}[\rho] ([G_{k,\varepsilon}(\rho,\rho u)]_m^\ast) dx \notag \\ 
        &\leq \frac 1 2 \sum_{k=1}^{\infty} \int_D \frac{\left|G_{k,\varepsilon}(\rho,\rho u) \right|^2}{\rho}dx = \frac 1 2 \sum_{k=1}^{\infty} \int_D \rho \left|F_{k,\varepsilon}(\rho,u) \right|^2 dx. \label{eq:3.63}
    \end{align}
    Indeed, since $F=\MMM^{-1}[\rho] ([G_{k,\varepsilon}(\rho,\rho u)]_m^\ast)$ satisfies that 
    $$
        \int_D \rho F\cdot \bphi dx = \int_D  G_{k,\varepsilon}(\rho,\rho u)\cdot \bphi dx,\quad \bphi\in V_m,
    $$
    taking $\bphi= F$ yields
    $$
        0\leq \int_D \rho |F|^2 dx = \int_ D G_{k,\varepsilon}(\rho,\rho u)\cdot F dx,
    $$
    and by Cauchy--Schwarz inequality, we have 
    $$
        \int_D G_{k,\varepsilon}(\rho,\rho u)\cdot F dx \leq \norm{\frac{G_{k,\varepsilon}(\rho,\rho u)}{\sqrt{\rho}}}_{L_x^2}\norm{\sqrt{\rho}F}_{L_x^2}.
    $$
    These observations imply 
    $$
        \norm{\sqrt{\rho}F}_{L_x^2} \leq \norm{\frac{G_{k,\varepsilon}(\rho,\rho u)}{\sqrt{\rho}}}_{L_x^2},
    $$
    and thus we obtain that 
    $$
        \int_D G_{k,\varepsilon}(\rho,\rho u)\cdot F dx \leq \int_D \frac{|G_{k,\varepsilon}(\rho,\rho u)|^2}{\rho}dx. 
    $$
\end{remark}

\subsection{Limit passage with respect to \texorpdfstring{$R\to \infty$}{R to infty}}
Our target problem reads 
\begin{align}
   &d\rho + \rdiv (\rho u)dt = \varepsilon \Delta \rho dt,\quad \nabla \rho \cdot \rn |_{\pD} =0,\quad \text{in}\ (0,T)\times D, \label{eq:approx_CE_R} \\
   \int_D \rho u\cdot \bphi(\tau) dx &= \int_D \rho_0 u_0 \cdot \bphi dx  + \inttau \int_D \left[(\rho u \otimes u):\nabla\bphi + p_\delta (\rho)\rdiv \bphi\right]dxdt  \notag \\ 
   &\quad - \inttau \int_D \left[\SSS(\nabla u):\nabla\bphi -\varepsilon \rho u\cdot \Delta \bphi \right]dxdt - \inttau \int_{\pD}g\nabla j_\alpha(u)\cdot \bphi d\Gamma dt \notag \\ 
   &\quad + \inttau\int_D \GG_\varepsilon(\rho,\rho u) \cdot \bphi dxdW, \quad \ \tau\in [0,T],\  \bphi\in V_m. \label{eq:approx_ME_R}
\end{align}
It will be solved in the context of strong pathwise solutions given by the following definition.
\begin{definition}
    \label{def:sol_to_cutofflayer}
    Let $\stochbasis$ be a given stochastic basis with a complete right-continuous filtration, let $W$ be a cylindrical $(\FFF_t)$-Wiener process on $\fU$, and let $(\rho_0,u_0)$ be an $\FFF_0$-measurable random variable taking values in $C^{2+\nu}(\oD)\times V_m$. Then $(\rho,u)$ is called a \textit{pathwise solution} to \eqref{eq:3.11}--\eqref{eq:3.12} with the initial datum $(\rho_0,u_0)$, if the following hold: 
	\begin{itemize}
		\item[(1)] $\rho$ and $u$ are $(\FFF_t)$-adapted stochastic processes such that $\PP$-a.s.
		\begin{align}
		   &\rho> 0,\quad \rho \in C([0,T];C^{2+\nu}(\oD))\cap C^1([0,T];C^{\nu}(\oD)),\quad u\in C([0,T];V_m);   
		\end{align}
		\item[(2)] We have $\PP$-a.s.
		\begin{align*}
            (\rho(0),u(0)) = (\rho_0,u_0);
        \end{align*}
		\item[(3)] the approximate equation of continuity 
		\begin{equation}
            \pp_t \rho + \rdiv (\rho u) = \varepsilon \Delta \rho,\quad \nabla \rho \cdot \rn |_{\pD} =0
        \end{equation}
		holds in $(0,T)\times D$ $\PP$-a.s.;
		\item[(4)] the approximate momentum equation 
		\begin{align}
            \int_D \rho u\cdot \bphi(\tau) dx &= \int_D \rho_0 u_0 \cdot \bphi dx  + \inttau \int_D \left[(\rho u \otimes u):\nabla\bphi + p_\delta (\rho)\rdiv \bphi\right]dxdt  \notag \\ 
            &\quad - \inttau \int_D \left[\SSS(\nabla u):\nabla\bphi -\varepsilon \rho u\cdot \Delta \bphi \right]dxdt - \inttau \int_{\pD}g\nabla j_\alpha(u)\cdot \bphi d\Gamma dt \notag \\ 
            &\quad + \inttau\int_D \GG_\varepsilon(\rho,\rho u) \cdot \bphi dxdW, \quad \ \tau\in [0,T],\  \bphi\in V_m 
        \end{align}
		holds for all $\tau\in [0,T]$ and all $\bphi\in V_m$ $\PP$-a.s.;
        \item[(5)] the energy equality
        \begin{align}
            &\int_D \left[\frac 1 2 \rho |u|^2 + P_\delta(\rho)\right](\tau)dx \notag \\ 
            &\quad + \inttau \int_D \left[\SSS(\nabla u):\nabla u + \varepsilon \rho |\nabla u|^2 + \varepsilon P_\delta''(\rho)|\nabla \rho|^2\right]dxdt +\inttau \int_{\pD} g \nabla j_\alpha(u)\cdot u d\Gamma dt \notag\\ 
            &\quad = \int_D \left[\frac 1 2 \rho_0 |u_0|^2 + P_\delta(\rho_0)\right]dx + \frac 1 2 \sum_{k=1}^{\infty}\int_D G_{k,\varepsilon}(\rho,\rho u) \cdot \MMM^{-1}[\rho] ([G_{k,\varepsilon}(\rho,\rho u)]_m^\ast) dxdt \notag \\ 
            &\quad \quad + \inttau \int_D \GG_\varepsilon(\rho,\rho u) \cdot u dx dW \label{eq:EB_R}
        \end{align}
        holds for all $\tau\in [0,T]$ $\PP$-a.s. 
    \end{itemize}
\end{definition}
\begin{theorem}
    \label{thm:sol_R}
    Let $\stochbasis$ be a stochastic basis with a complete right-continuous filtration and let $W$ be a cylindrical $(\FFF_t)$-Wiener process on $\fU$. Let $[\rho_0,u_0]$ be a given $\FFF_0$-measurable initial datum such that 
    \begin{align}
        \rho_0\in C^{2+\nu}(\oD),\quad \urho\leq \rho_0\leq \orho,\quad \nabla\rho_0\cdot \rn |_{\pD}=0, \quad u_0\in V_m,\quad \PP\text{-a.s.,}
    \end{align}
    for some deterministic constants $\urho,\orho>0$, and 
    \begin{align}
        \label{eq:3.71}
        \EE\left[\left(\int_D\left[\frac 1 2 \rho_0 |u_0|^2 + P_\delta(\rho_0)\right]dx\right)^r\right]<\infty,
    \end{align}
    for some $r>2$. Then the approximate problem \eqref{eq:approx_CE_R}--\eqref{eq:approx_ME_R} admits a pathwise unique solution $[\rho,u]$ in the sense of Definition \ref{def:sol_to_cutofflayer}. 
\end{theorem}

In order to prove Theorem \ref{thm:sol_R} we adopt the following strategy:
\begin{itemize}
    \item[1.] Using the energy balance from Proposition \ref{prop:energy_balance} we derive uniform bounds independent of the parameters $R,\alpha,m,\varepsilon$ and $\delta$.
    \item[2.] We perform the limit $R\to \infty$. The proof is based on a suitable stopping time argument. 
\end{itemize}
\begin{proposition}\label{prop:unif.bound_by_energy}
    Let $(\rho,u)$ be a solution to \eqref{eq:3.11}--\eqref{eq:3.12} in the sense of Definition \ref{def:pathwise_sol}. Then, we have, uniformly in $R,\alpha,m,\varepsilon$ and $\delta$, 
    \begin{align}
        &\EE\left[\left|\sup_{\tau\in [0,T]} \int_D \left[\frac 1 2 \rho |u|^2 + P_\delta(\rho)\right](\tau)dx \right|^{r'}\right] \notag\\ 
        &\quad + \EE\left[\left|\intT \int_D \left[\SSS(\nabla u):\nabla u + \varepsilon \rho |\nabla u|^2 + \varepsilon P_\delta''(\rho)|\nabla \rho|^2\right]dxdt \right|^{r'}\right] + \EE\left[\left|\intT \int_{\pD} g \nabla j_\alpha(u)\cdot u d\Gamma dt \right|^{r'}\right] \notag\\ 
        &\quad \leq c \left(\EE\left[\left|\int_D \left[\frac 1 2 \rho_0 |u_0|^2 + P_\delta(\rho_0)\right]dx \right|^{r'}\right] + 1\right)\quad \text{for any} \  r'\in [1,r], \label{eq:3.72}
    \end{align}
    with a constant $c=c(r',T,\gamma)$.
\end{proposition}
\begin{proof}
    According to Remark \ref{rem:quad_var_est}, we have 
    \begin{align*}
        \frac 1 2 \sum_{k=1}^{\infty}\int_D G_{k,\varepsilon}(\rho,\rho u) \cdot \MMM^{-1}[\rho] ([G_{k,\varepsilon}(\rho,\rho u)]_m^\ast) dx &\leq \frac 1 2 \sum_{k=1}^{\infty} \int_D \rho \left|F_{k,\varepsilon}(\rho,u) \right|^2 dx \\ 
        &\leq \frac 1 2 \sum_{k=1}^{\infty} \int_D \rho \left|F_{k}(\rho,u) \right|^2 dx \\ 
        &\leq \sum_{k=1}^{\infty} \int_D \rho f_k^2 (1+|u|^2) dx \\ 
        &\lesssim \int_D (\rho +\rho |u|^2)dx,  
    \end{align*}
    and by the Burkholder--Davis--Gundy inequality and the definitions of $p_\delta$ and $P_\delta$,
    \begin{align*}
        &\EE\left[\sup_{t\in [0,\tau]}\left|\int_{0}^{t}\int_D \GG_\varepsilon(\rho,\rho u)\cdot udx dW \right|^{r'}\right] \\ 
        &\quad \lesssim \EE\left[\left(\inttau \sum_{k=1}^{\infty}\left|\int_D G_{k,\varepsilon}(\rho,\rho u)\cdot u dx \right|^2 dt \right)^{r'/2}\right]  \\ 
        &\quad \lesssim \EE\left[\left(\inttau \sum_{k=1}^{\infty}\left(\int_D \rho|F_k(\rho,u)\cdot u| dx \right)^2 dt \right)^{r'/2}\right]  \\ 
        &\quad \lesssim \EE\left[\left(\inttau\left(\int_D (\rho+\rho |u|^2) dx \right)^2 dt \right)^{r'/2}\right] \\ 
        &\quad \lesssim \EE\left[\left(\inttau\left(\int_D (P_\delta(\rho)+\rho |u|^2) dx \right)^2 dt \right)^{r'/2}\right] + 1
    \end{align*}
    uniformly in $R,\alpha,m,\varepsilon$, and $\delta$. Therefore, passing to expectations in \eqref{eq:3.59}, we apply Gronwall's inequality, to deduce \eqref{eq:3.72}.
\end{proof}
As a direct consequence of the approximate equation of continuity \eqref{eq:3.11}, we have 
$$
    \int_D \rho(\tau)dx = \int_D \rho_0 dx,\quad \tau\in [0,T].
$$

Hence, $\norm{\rho}_{L^1(D)}$ is bounded below by a positive deterministic constant. On the other hand, $\norm{\rho}_{L^\gamma(D)},\gamma>1$, is only bounded in expectation. 
Thus, the generalized Korn--Poincar\'{e} inequality \cite[Theorem 11.23]{FN17}, which is standard in deterministic settings, cannot be applied directly. 
Instead, we adopt the generalized Korn--Poincar\'{e} inequality involving boundary traces, which was first established by \cite[Theorem2.3]{Pom03}. 
In our setting, we further generalize this inequality as follows.

\begin{lemma}
    \label{lem:korn_poincare}        
    Let $D$ be a bounded Lipschitz domain in $\RR^3$, that is, a connected and bounded open subset of $\RR^3$ with Lipschitz boundary, and let $M,K>0$. Then 
    there exists a positive constant $c=c(D,M,K)$ such that the inequality
    \begin{align*}
        &\norm{v}_{W^{1,2}(D;\RR^3)} \leq c \left(\norm{\nabla v + (\nabla v)^T - \frac23\rdiv v\II}_{L^2(D;\RR^{3\times 3})} + \int_{\pD} r |v|d\Gamma \right)
    \end{align*}
    holds for any $v\in W^{1,2}(D;\RR^3)$ and any non-negative function $r$ such that 
    $$
        0<M\leq \int_{\pD} rd\Gamma,\quad \int_{\pD} r^2 d\Gamma \leq K.
    $$
\end{lemma}
\begin{proof}
    Fixing the parameters $M,K$, we argue by contradiction. Specifically, we construct a sequence $w_n\in W^{1,p}(D;\RR^3)$ such that 
    \begin{align*}
        &\norm{w_n}_{W^{1,2}(D;\RR^3)} = 1,\quad w_n \weakarrow w \quad \text{in }W^{1,2}(D;\RR^3)  
    \end{align*}
    and 
    \begin{align*}
        &\norm{\nabla w_n + (\nabla w_n)^T - \frac23\rdiv w_n\II}_{L^2(D;\RR^{3\times 3})} + \int_{\pD} r_n |w_n|d\Gamma <\frac1n,
    \end{align*}
    for certain 
    \begin{align*}
        &r_n \weakarrow r\quad \text{in }L^2(\pD),\quad \int_{\pD} r d\Gamma  \geq M > 0.
    \end{align*}
    Due to the compact embedding $W^{1,2}(D;\RR^3)$ into $L^2(D;\RR^3)$, 
    \begin{align*}
        &\norm{w_n-w}_{L^2(D;\RR^3)} \to 0,  
    \end{align*}
    and by virtue of a standard generalization of Korn--Poincar\'{e} inequality \cite[Theorem 11.22]{FN17}, we have 
    \begin{align*}
        &\norm{w_m-w_n}_{W^{1,2}(D;\RR^3)} \\ 
        &\quad \lesssim \norm{\nabla (w_m-w_n) + (\nabla (w_m-w_n))^T - \frac23\rdiv (w_m-w_n)\II}_{L^2(D;\RR^{3\times 3})} + \norm{w_m-w_n}_{L^2(D;\RR^3)}.\\ 
    \end{align*}
    These observations imply  
    \begin{align*}
        &w_n\to w\quad \text{strongly in }W^{1,2}(D;\RR^3). 
    \end{align*}
    Moreover, due to the construction of $w_n$, the limit $w$ satisfies the following:
    \begin{align*}
        &\norm{w}_{W^{1,2}(D;\RR^3)} = 1,\\ 
        &\nabla w + (\nabla w)^T - \frac23\rdiv w\II =0,\\ 
        &\int_{\pD} r|w|d\Gamma = 0.
    \end{align*}
    The second identity implies that, $w$ is a weak solution to the problem 
    \begin{align*}
        &\begin{cases}
        \Delta w + \frac13 \nabla \rdiv w =0 &\text{in } D, \\
        w =\gamma_0 (w)  &\text{on }\pD, 
        \end{cases}  
    \end{align*}
    in the class $W^{1,2}(D;\RR^3)$, where $\gamma_0$ is the trace operator given by Theorem \ref{thm:trace}.
    Since the differential operator $\Delta + \frac13 \nabla \rdiv$ is a strictly elliptic operator with constant coefficients, $w$ is smooth and also analytic in $D$ (see \cite[Theorem 4.16]{McL00}, and \cite[Chapter VII]{Joh55}).
    Moreover, $w$ is the restriction of a polynomial $v$ to $D$ of the following form:
    \begin{align*}
        &v(x) = 2(a\cdot x)x -|x|^2 a + A x + \alpha x + b,\quad x\in \RR^3,\ A\in \RR_{\text{skew}}^{3\times 3},\ a,b\in \RR^3,\ \alpha\in \RR,  
    \end{align*}
    which is called a \textit{conformal Killing vector field} (see \cite[Lemma 2.1, Example 2.2]{GST25}, \cite{LMN21}). In general, the set of zeros of such a polynomial has finitely many connected components, each of which is a totally umbilical submanifold of $\RR^n$, of even codimension when not reduced to a point (see, for example, \cite[Theorem 3.1]{BMO11}). 
    In particular, if $v\not\equiv 0$, then the measure of the set of zeros of $v$ on the boundary $\pD$ is zero.
    Note that, in the present setting, we can also verify this fact directly. For instance, when $a$ is the first standard basis vector $e_1$, decompose 
    \begin{align*}
        &x= ue_1 + y,\quad u\in \RR,\ y\in \operatorname{span} \left\{ e_2,e_3 \right\},
    \end{align*}
    and transform the condition $v = 0$ equivalently. On the other hand, the third identity implies that $v$ vanishes on the set $\left\{ x\in \pD:r(x)>0 \right\}$ of a nonzero measure. 
    Thus, we have $v\equiv 0$, which implies $w\equiv 0$ in $D$. This contradicts the first identity.
\end{proof}
\begin{corollary}
    \label{cor:est_u_R}
    Under the assumptions of Theorem \ref{thm:sol_R}, the velocity $u$ of a solution to \eqref{eq:3.11}--\eqref{eq:3.12} in the sense of Definition \ref{def:pathwise_sol} satisfies the estimate 
    \begin{align}
        \EE\left[\left| \intT \norm{u}_{W_x^{1,2}}^2 dt \right|^{r/2} \right] \lesssim c(r,T,\gamma) \left(\EE\left[\left|\int_D \left[\frac 1 2 \rho_0 |u_0|^2 + P_\delta(\rho_0)\right]dx \right|^r\right] + 1\right),  \label{eq:3.73}
    \end{align}  
    where the proportional constant is independent of $R,\alpha,m,\varepsilon$, and $\delta$.
\end{corollary}
\begin{proof}
    From the assumptions on $g$, there exists a $t\in (0,T)$ such that 
    \begin{align*}
        &0<\int_{\pD}g(t,x) d\Gamma (x),\quad \int_{\pD}g(t,x)^2 d\Gamma (x)<\infty,  
    \end{align*}
    and from the definition of $j_\alpha$, we have 
    \begin{align*}
        &\int_{\pD} g|u| d\Gamma \lesssim \int_{\pD} g\nabla j_\alpha(u)\cdot u d\Gamma + 1
    \end{align*}
    uniformly in $\alpha$. Combining these with Lemma \ref{lem:korn_poincare} and applying Proposition \ref{prop:unif.bound_by_energy}, the assertion follows.
\end{proof}

\begin{proof}[Proof of Theorem \ref{thm:sol_R}]
    Let $(\rho_0,u_0)$ be the initial data in Theorem \ref{thm:sol_R}, and let us denote $(\rho_R,u_R)$ the unique pathwise solution to \eqref{eq:3.11}--\eqref{eq:3.12} starting from $(\rho_0,u_0)$, which was constructed in Corollary \ref{cor:general_data}.
    Let us consider 
    $$
        \tau_R = \inf \left\{ t\in [0,T] : \norm{u_R(t)}_{V_m}>R \right\}
    $$
    with the convention $\inf \emptyset = T$. Due to uniqueness, if $R'>R$, then 
    $$
        \tau_{R'}\geq \tau_R,\quad \text{and}\quad  (\rho_{R'},u_{R'}) = (\rho_R,u_R)\quad \text{on}\  \ico{0}{\tau_R}.
    $$
    In addition, $(\rho_R,u_R)$ is a solution to \eqref{eq:approx_CE_R}--\eqref{eq:approx_ME_R} on $\ico{0}{\tau_R}$. 
    Therefore, we define $(\rho,u)$ by $(\rho,u):= (\rho_R,u_R)$ on $\ico{0}{\tau_R}$. In order to make sure that $(\rho,u)$ solves \eqref{eq:approx_CE_R}--\eqref{eq:approx_ME_R} on the whole time interval $[0,T]$, i.e., the blow-up cannot occur in a finite time, we are going to show that 
    \begin{align}
        \PP\left(\sup_{R\in\NN}\tau_R = T\right) = 1. \label{eq:3.74}
    \end{align}
    By classical maximum principle for the approximate equation of continuity \eqref{eq:3.11} (see \cite[Lemma 3.1]{FN17}), we have 
    $$
        \urho \exp \left(-\int_{0}^{\tau}\norm{\rdiv [u_R]_R}_{L^\infty_x}dt\right)\leq \rho_R(\tau,x)\leq \orho \exp \left(\int_{0}^{\tau}\norm{\rdiv [u_R]_R}_{L^\infty_x}dt\right);
    $$
    whence, 
    \begin{align}
        \urho \exp \left(-\int_{0}^{\tau}\norm{u_R}_{W_x^{1,\infty}}dt\right)\leq \rho_R(\tau,x) \leq \orho \exp \left(\int_{0}^{\tau}\norm{u_R}_{W_x^{1,\infty}}dt\right), \label{eq:3.75}
    \end{align}
    for all $\tau\in [0,T],x\in \oD$. Using the equivalence of norms on $V_m$ and H\"{o}lder's inequality at hand, we have 
    $$
        \urho \exp \left(-T -c \inttau \norm{u_R}_{W_x^{1,2}}^2 dt \right) \leq \rho_R(\tau,x).
    $$ 
    Plugging this into \eqref{eq:3.72}, we obtain that 
    \begin{align}
        \EE\left[\exp\left(-c\inttau \norm{u_R}_{W_x^{1,2}}^2dt\right)\sup_{t\in[0,\tau]}\norm{u_R}_{L_x^2}^2\right]\leq \tilde{c}. \label{eq:3.76}
    \end{align}
    Next, let us fix two increasing sequence $(a_R)$ and $(b_R)$ such that $a_R,b_R\to \infty$ and $a_R\exp(b_R)=R$ for each $R\in \NN$. As in \cite{FM12}, we introduce the following events 
    \begin{align*}
        &A = \left\{ \exp\left(-c\inttau \norm{u_R}_{W_x^{1,2}}^2dt\right)\sup_{t\in[0,\tau]}\norm{u_R}_{L_x^2}^2\leq a_R \right\}, \\ 
        &B = \left\{ c\inttau \norm{u_R}_{W_x^{1,2}}^2dt \leq b_R \right\}, \\ 
        &C = \left\{ \sup_{t\in [0,\tau]}\norm{u_R}_{L_x^2}^2\leq a_R e^{b_R}\right\}.
    \end{align*}
    Then $A\cap B\subset C$, because on $A\cap B$ we know 
    $$
        \sup_{t\in [0,\tau]}\norm{u_R}_{L_x^2}^2 = e^{b_R}e^{-b_R}\sup_{t\in [0,\tau]}\norm{u_R}_{L_x^2}^2 \leq e^{b_R} \exp\left(-c\inttau \norm{u_R}_{W_x^{1,2}}^2dt \right)\sup_{t\in [0,\tau]} \norm{u_R}_{L_x^2}^2\leq e^{b_R}a_R.
    $$
    Furthermore, according to \eqref{eq:3.73}, \eqref{eq:3.76}, and Chebyshev's inequality, we have 
    $$
        \PP(A)\geq 1- \frac{\tilde{c}}{a_R},\quad \PP(B)\geq 1- \frac{\overline{c}}{b_R}.
    $$
    Due to the general inequality $\PP(C)\geq \PP(A)+\PP(B)-1$ we deduce 
    $$
        \PP(C)\geq 1- \frac{\tilde{c}}{a_R} - \frac{c}{b_R}\to 1,\quad R\to \infty.
    $$
    This in turn implies \eqref{eq:3.74}. Finally, uniqueness follows by Theorem \ref{thm:pathwise_sol} and \eqref{eq:3.74}.
\end{proof}

\section{Limit passage with respect to \texorpdfstring{$\alpha\to 0$}{alpha to 0}}\label{sec:convex_approx}
Our next goal is to let $\alpha\to 0$ in the approximate system \eqref{eq:approx_CE_R}--\eqref{eq:approx_ME_R}. 
From the approximate layer in this section, the approximate momentum equation essentially become an inequality combined with the energy inequality. Therefore, a rigorous formulation reads as follows.
\begin{definition}
    \label{def:sol_alpha}
	Let $\Lambda$ be a Borel probability measure on $C^{2+\nu}(\oD)\times V_m$. Then $((\Omega,\FFF,(\FFF_t)_{t\geq 0},\PP),\rho,u,W)$ is called a \textit{dissipative martingale solution in $m$ layer} with the initial law $\Lambda$ if:
	\begin{itemize}
		\item[(1)] $(\Omega,\FFF,(\FFF_t)_{t\geq 0},\PP)$ is a stochastic basis with a complete right-continuous filtration;
		\item[(2)] $W$ is a cylindrical $(\FFF_t)$-Wiener process on $\fU$;
		\item[(3)] the density $\rho$ is an $(\FFF_t)$-adapted stochastic process, and the velocity $u$ is an $(\FFF_t)$-adapted random distribution such that $\PP$-a.s.
		\begin{align*}
		   &\rho \geq 0,\quad \rho \in C_w([0,T];L^\Gamma(D)), \quad u\in L^2(0,T;V_m),\quad [\rho u]_m^\ast \in C([0,T];V_m^\ast);   
		\end{align*}
		\item[(4)] there exists a $C^{2+\nu}(\oD)\times V_m$-valued $\FFF_0$-measurable random variable $(\rho_0,u_0)$ such that $\Lambda = \LL[\rho_0,u_0]$;
		\item[(5)] the approximate equation of continuity 
		\begin{equation}
            \label{eq:approx_CE_alpha}
            \pp_t \rho + \rdiv (\rho u) = \varepsilon \Delta \rho,\quad \nabla \rho \cdot \rn |_{\pD} =0
        \end{equation}
		holds a.e. in $(0,T)\times D$ $\PP$-a.s.;
		\item[(6)] the approximate interior momentum equation 
		\begin{align}
            -\intT \pp_t\phi \int_D \rho u\cdot \bphi dx dt &= \phi(0)\int_D \rho_0 u_0 \cdot \bphi dx  + \intT\phi \int_D \left[(\rho u \otimes u):\nabla\bphi + p_\delta (\rho)\rdiv \bphi\right]dxdt  \notag \\ 
            &\quad - \intT \phi\int_D \left[\SSS(\nabla u):\nabla\bphi -\varepsilon \rho u\cdot \Delta \bphi \right]dxdt \notag \\ 
            &\quad + \intT \phi\int_D \GG_\varepsilon(\rho,\rho u) \cdot \bphi dxdW  \label{eq:approx_ME_alpha}
        \end{align}
		holds for all $\phi\in C_c^\infty(\ico{0}{T})$ and all $\bphi\in V_m$ such that $\bphi|_{\pD} =0$ $\PP$-a.s.;
        \item[(7)] the approximate momentum and energy inequality 
        \begin{align}
            &\int_D \left[\frac 1 2 \rho_0 |u_0|^2 + P_\delta(\rho_0)\right]dx - \int_D \left[\frac 1 2 \rho |u|^2 + P_\delta(\rho)\right](\tau)dx - \inttau \pp_t\phi \int_D \rho u\cdot \bphi dx - \phi(0)\int_D \rho_0 u_0 \cdot \bphi dx \notag\\ 
            &\quad - \inttau \phi \int_D \left[(\rho u \otimes u):\nabla\bphi + p_\delta (\rho)\rdiv \bphi\right]dxdt + \inttau \int_D \left[\SSS(\nabla u):\nabla(\phi\bphi- u) -\varepsilon \rho u\cdot \Delta (\phi\bphi) \right]dxdt \notag \\ 
            &\quad - \inttau \int_D \left[\varepsilon \rho |\nabla u|^2 + \varepsilon P_\delta''(\rho)|\nabla \rho|^2\right]dxdt +\frac 1 2 \sum_{k=1}^{\infty} \inttau \int_D \rho \left|F_{k,\varepsilon}(\rho,u) \right|^2 dx dt \notag \\ 
            &\quad - \inttau \int_D \GG_\varepsilon(\rho,\rho u) \cdot (\phi\bphi - u) dx dW + \inttau \int_{\pD} g|\phi\bphi| - g |u| d\Gamma dt\geq 0 \label{eq:approx_MEI_alpha}
        \end{align}
        holds for all $\tau\in [0,T]$, all $\phi\in C_c^\infty(\ico{0}{\tau})$, and all $\bphi\in V_m$ $\PP$-a.s.
    \end{itemize}
\end{definition}
\begin{theorem}
    \label{thm:sol_alpha}
    Let $\Lambda$ be a Borel probability measure on $C^{2+\nu}(\oD)\times V_m$ such that 
    \begin{equation}
        \Lambda \left\{ \urho\leq \rho \leq \orho,\ \nabla \rho\cdot \rn |_{\pD}=0 \right\}=1, \label{eq:4.4}
    \end{equation}
    for some deterministic constants $\urho,\orho>0$ and 
    \begin{align}
        \int_{C_x^{2+\nu}\times V_m} \left(\left|\int_D \left[\frac 1 2 \rho |v|^2 + P_\delta(\rho)\right]dx \right|^{2r} + \norm{\rho}_{C_x^{2+\nu}}^r\right)d \Lambda (\rho,v) <\infty, \label{eq:4.5}
    \end{align}
    for some $r \geq 2$. Then there exists a dissipative martingale solution in the sense of Definition \ref{def:sol_alpha}.
\end{theorem}
\begin{remark}
    \label{rem:alpha(integrability_of_sol)}
    As a matter of fact, the solution constructed in this section belongs to the following class: 
    \begin{align*}
        &\EE\left[\sup_{t\in [0,T]}\left|\int_D\left[\frac 1 2 \rho |u|^2 + P_\delta(\rho)\right]dx \right|^{2r}\right] + \EE\left[\left|\intT \int_{\pD} g |u| d\Gamma dt \right|^{2r}\right] <\infty, \\ 
        &\EE\left[\left|\intT \int_D \left[\SSS(\nabla u):\nabla u + \varepsilon \rho |\nabla u|^2 + \varepsilon P_\delta''(\rho)|\nabla \rho|^2\right]dxdt\right|^{2r}\right]<\infty,
    \end{align*}
    with the same $r \geq 2$ as in \eqref{eq:4.5}. In particular, similarly to the proof of Corollary \ref{cor:est_u_R}, we have 
    \begin{align*}
        &\EE \left[\sup_{t\in [0,T]} \norm{\rho(t)}_{L_x^\Gamma}^{2\Gamma r}\right] + \EE \left[\sup_{t\in [0,T]} \norm{\rho |u|^2(t)}_{L_x^1}^{2 r}\right] + \EE \left[\sup_{t\in [0,T]} \norm{\rho u(t)}_{L_x^{\frac{2\Gamma}{\Gamma +1}}}^{\frac{4\Gamma}{\Gamma + 1}r}\right] + \EE\left[\norm{u}_{L_t^2 V_m}^{2r}\right]<\infty.
    \end{align*}
\end{remark}
In order to prove Theorem \ref{thm:sol_alpha}, we adopt the following strategy:
\begin{itemize}
    \item[1.] Using the energy balance we derive uniform bounds independent of the parameter $\alpha$.
    \item[2.] We let $\alpha\to 0$ with the help of the stochastic compactness method based on the Jakubowski--Skorokhod representation theorem, Theorem \ref{thm:Jakubowski--Skorokhod}. 
\end{itemize}
\subsection{Uniform estimates}\label{sec:alpha:Uniform estimates}
In this section, the approximation parameters $m\in \NN, \varepsilon,\delta\in (0,1)$ are kept fixed. We call uniform estimate that are independent of $\alpha$ but may depend on $m,\varepsilon,\delta$ and $T>0$.
Let $(\rho_\alpha,u_\alpha)$ be the solutions constructed in Theorem \ref{thm:sol_R} satisfying \eqref{eq:3.71} for some $2r\geq 4$, and for simplicity, the subscript $\alpha$ is temporarily omitted.
Note that the H\"{o}lder continuity of $\rho$ and $u$ cannot be expected since the $R$-cut-off of $\norm{u}_{V_m}$ is lost. On the other hand, as described later, the H\"{o}lder continuity of $[\rho u]_m^\ast$ can be obtained similarly to Section \ref{sec:The limit for vanishing time step}.  
Recall that the following uniform estimate follows from Proposition \ref{prop:unif.bound_by_energy}:
\begin{align}
    &\EE\left[\left|\sup_{\tau\in [0,T]} \int_D \left[\frac 1 2 \rho |u|^2 + P_\delta(\rho)\right](\tau)dx \right|^{r}\right] \notag\\ 
    &\quad  +\EE\left[\left|\intT \int_D \left[\SSS(\nabla u):\nabla u + \varepsilon \rho |\nabla u|^2 + \varepsilon P_\delta''(\rho)|\nabla \rho|^2\right]dxdt \right|^{r}\right] \notag \\ 
    &\quad + \EE\left[\left|\intT \int_{\pD} g \nabla j_\alpha(u)\cdot u d\Gamma dt \right|^{r}\right] \lesssim  c(r), \label{eq:4.6} 
\end{align}
and by Corollary \ref{cor:est_u_R}, we have the estimate 
\begin{align}
    \EE\left[\left| \intT \norm{u}_{W_x^{1,2}}^2 dt \right|^{r} \right] \lesssim  c(2r) \label{eq:4.7}
\end{align}
uniformly in $\alpha$, where 
$$
    c(r) = \EE\left[\left|\int_D \left[\frac 1 2 \rho_0 |u_0|^2 + P_\delta(\rho_0)\right]dx \right|^{r}\right] + 1,\quad r\geq 2.
$$
From \eqref{eq:4.6} and the definitions of $p_\delta$ and $P_\delta$, we obtain the following estimates which are independent of $\alpha$:
\begin{align}
    &\EE\left[\norm{\rho}_{L_t^\infty L_x^\Gamma}^{\Gamma r}\right] + \EE \left[\norm{\rho^{-1/2}\nabla \rho}_{L_t^2 L_x^2}^{2r}\right] + \EE \left[\norm{\rho^{\Gamma / 2-1}\nabla \rho}_{L_t^2 L_x^2}^{2r}\right] \lesssim c(r), \label{eq:4.8}\\ 
    & \EE\left[\norm{\rho |u|^2}_{L_t^\infty L_x^1}^{r}\right] + \EE\left[\norm{\rho u}_{L_t^\infty L_x^{\frac{2\Gamma}{\Gamma +1}}}^{\frac{2\Gamma}{\Gamma +1}r}\right] \lesssim c(r), \label{eq:4.9}
\end{align}
where the estimate of the second term in \eqref{eq:4.9} follows from applying H\"older's inequality to $\sqrt{\rho}$ and $\sqrt{\rho}|u|$.
Next, we recall the standard parabolic maximal regularity estimates (see \cite[Lemmas 7.37, 7.38]{NS04}), applied to \eqref{eq:approx_CE_R}: 
\begin{align}
    \begin{aligned}
        \norm{\pp_t \rho}_{L_t^p L_x^q} + \norm{\rho}_{L_t^p W_x^{2,q}} &\lesssim \norm{\rdiv (\rho u)}_{L_t^p L_x^q} + \norm{\rho_0}_{C_x^{2+\nu}}, \\ 
        \norm{\rho}_{L_t^p W_x^{1,q}} &\lesssim \norm{\rho u}_{L_t^p L_x^q} + \norm{\rho_0}_{C_x^{2+\nu}} 
    \end{aligned}\label{eq:4.10}
\end{align}  
for $1<p,q<\infty$. In \eqref{eq:4.10}, the regularity of the initial data can be considerably weakened. However, such generality is not needed here. 
Now observe that \eqref{eq:4.8}, together with $\Gamma \geq 6$, \eqref{eq:4.7}, and the standard Sobolev embedding $W^{1,2}(D)\hookrightarrow L^6(D)$, gives rise to 
\begin{align*}
    &\EE\left[\norm{\rho u}_{L_t^1 L_x^3}^r \right]  \leq \EE \left[\left|\intT \norm{\rho}_{L_x^6}\norm{u}_{L_x^6} dt \right|^r\right] \leq \EE\left[\left|\norm{\rho}_{L_t^\infty L_x^6}\norm{u}_{L_t^1 L_x^6} \right|^r\right] \\ 
    &\quad \lesssim \EE\left[\norm{\rho}_{L_t^\infty L_x^\Gamma}^{2r}\right] + \EE\left[\norm{u}_{L_t^2 W_x^{1,2}}^{2r}\right] \lesssim c(2r).
\end{align*}
Interpolating this with \eqref{eq:4.9} yields 
\begin{align*}
    \EE\left[\norm{\rho u}_{L_t^p L_x^p}^r\right] \lesssim c(2r) \quad \text{for some} \  p>2, 
\end{align*}
where, for the existence and specific range of $p>2$, refer to the general results on interpolation between Bochner spaces (see, for example, \cite[Theorem 5.1.2]{BL76}).
This estimate, plugged in the right-hand side of \eqref{eq:4.10}, implies 
\begin{align}
    \EE\left[\norm{\rho}_{L_t^p W_x^{1,p}}^r\right]\lesssim \tilde{c}(2r) \quad \text{for some} \  p>2, \label{eq:4.11}
\end{align}
where 
$$
    \tilde{c}(r) = \EE\left[\left|\int_D \left[\frac 1 2 \rho_0 |u_0|^2 + P_\delta(\rho_0)\right]dx \right|^r + \norm{\rho_0}_{C_x^{2+\nu}}^{r / 2}\right] + 1,\quad r\geq 2.
$$
Using this and \eqref{eq:4.7} yields 
\begin{align}
    \EE\left[\left|\norm{\pp_t \rho}_{L_t^pL_x^p} + \norm{\rho}_{L_t^p W_x^{2,p}} \right|^{r / 2}\right] \lesssim \tilde{c}(2r),\quad \text{for some } p>1. \label{eq:4.12}
\end{align}
Finally, we derive the H\"{o}lder continuity of $[\rho u]_m^\ast$.  Similarly to Section \ref{sec:The limit for vanishing time step}, using equivalence of norms on $V_m$, we have 
\begin{align*}
    &\norm{[\rho u]_m^\ast (\tau_1) - [\rho u]_m^\ast (\tau_2)}_{V_m^\ast}  \\ 
    &\quad \lesssim \int_{\tau_2}^{\tau_1} \left( \norm{\rho |u|^2}_{L_x^1} + \norm{\rho}_{L_x^\Gamma} + \norm{u}_{V_m}\right)dt + \norm{\int_{\tau_2}^{\tau_1}\int_D [\GG_\varepsilon (\rho,\rho u)]_m^\ast dx dW}_{V_m^\ast} \\ 
    &\quad \lesssim \left|\tau_1 -\tau_2 \right| \left(\norm{\rho |u|^2}_{L_x^\infty L_x^1} + \norm{\rho}_{L_t^\infty L_x^\Gamma}\right) + \left|\tau_1 -\tau_2 \right|^{1 / 2} \norm{u}_{L_t^2 V_m} + \norm{\int_{\tau_2}^{\tau_1}\int_D [\GG_\varepsilon (\rho,\rho u)]_m^\ast dx dW}_{V_m^\ast}.
\end{align*}
Therefore, with the estimates \eqref{eq:4.7}--\eqref{eq:4.9}, the properties of $\GG_\varepsilon$, and the Burkholder--Davis--Gundy inequality at hand, we obtain that 
$$
    \EE\left[\norm{[\rho u]_m^\ast (\tau_1) - [\rho u]_m^\ast (\tau_2)}_{V_m^\ast}^{2r} \right] \lesssim \left|\tau_1 -\tau_2 \right|^{r} c(2r),\quad r\geq 2
$$
uniformly in $\alpha$ whenever $0\leq \tau_1<\tau_2\leq T$. Thus, applying Kolmogorov's continuity theorem yields that $[\rho u]_m^\ast$ has $\PP$-a.s. $\beta$-H\"{o}lder continuous trajectories for all $\beta\in (0,1/2 - 1/(2r))$. In addition, we have the estimate 
\begin{align}
    \EE\left[\norm{[\rho u]_m^\ast}_{C_t^\beta V_m^\ast}^{2r}\right] \lesssim c(2r) \label{eq:4.13}
\end{align}
uniformly in $\alpha$. 

\subsection{Asymptotic limit} \label{sec:alpha(Asymptotic limit)}
With the uniform bounds established in the preceding subsection, we are ready to perform the limit $\alpha\to 0$. Let $\Lambda$ be a probability measure on $C^{2+\nu}(\oD)\times V_m$ satisfying \eqref{eq:4.4}--\eqref{eq:4.5}. Let $\stochbasis$ be a stochastic basis with a complete right-continuous filtration, let $W$ be a cylindrical $(\FFF_t)$-Wiener process on $\fU$, and let $(\rho_0,u_0)$ be $\FFF_0$-measurable random variables with values in $C^{2+\nu}(\oD)\times V_m$ and law $\Lambda$. Finally, let $(\rho_\alpha,u_\alpha)$ be the solution of problem \eqref{eq:approx_CE_R}--\eqref{eq:approx_ME_R} obtained in Theorem \ref{thm:sol_R} and starting from $(\rho_0,u_0)$. 
Then all the estimates in Section \ref{sec:alpha:Uniform estimates} hold true for $(\rho_\alpha,u_\alpha)$ uniformly in $\alpha$.
Therefore, we choose the path space 
$$
    \XX = \XX_{\rho_0}\times \XX_{u_0}\times \XX_\rho\times \XX_{\rho u}\times \XX_u \times \XX_W,
$$
where 
\begin{align*}
    &\XX_{\rho_0} = C(\oD), \\
    &\XX_{u_0} = V_m, \\
    &\XX_\rho = L^p(0,T;W^{1,p}(D)) \cap \left[W^{1,q}(0,T;L^q(D))\cap L^q(0,T;W^{2,q}(D)),w\right]  \cap C_w([0,T];L^\Gamma(D)), \\ 
    &\XX_{\rho u} = \overline{C^\beta([0,T];V_m^\ast)}^{C_t^\kappa V_m^\ast}, \\ 
    &\XX_u = \left[L^2(0,T;V_m),w\right], \\ 
    &\XX_W = C([0,T];\fU_0),
\end{align*}
for some $p>2,q>1$, and $\kappa\in (0,\beta)$. Now we claim that the parameter $p,q$ can be adjusted in such a way that the family of joint laws 
$$
    \left\{ \LL\left[\rho_0,u_0,\rho_\alpha,[\rho_\alpha u_\alpha]_m^\ast,u_\alpha,W\right]: \alpha\in (0,1) \right\} \quad \text{is tight on}\  \XX.
$$
To this end, it is enough to prove tightness of $\LL[\rho_\alpha]$ and $\LL[u_\alpha]$, $\alpha\in (0,1)$, as in Section \ref{sec:The limit for vanishing time step}.
\begin{proposition}
    \label{prop:tight_alpha}
    The sets $\left\{ \LL[\rho_\alpha]: \alpha\in (0,1)\right\}$ and $\left\{ \LL[u_\alpha]: \alpha\in (0,1)\right\}$ are tight on $\XX_\rho$ and $\XX_u$, respectively.
\end{proposition}
\begin{proof}
    Tightness of $\left\{ \LL[u_\alpha]: \alpha\in (0,1) \right\}$ immediately follows from the fact that the set 
    $$
        \left\{ u\in L^2(0,T;V_m): \norm{u}_{L_t^2V_m}\leq L \right\}
    $$
    is relatively compact in $\XX_u$. Next, using \eqref{eq:4.8} and \eqref{eq:4.9}, we have 
    \begin{align*}
        &\EE\left[\norm{\Delta \rho_\alpha}_{L_t^\infty (W_x^{2,\Gamma})^\ast}^{\Gamma r}\right] + \EE\left[\norm{\rdiv (\rho_\alpha u_\alpha)}_{L_t^\infty (W_x^{1, \frac{2\Gamma}{\Gamma +1}})^\ast}^{\frac{2\Gamma}{\Gamma +1}r}\right]  \lesssim c(r).
    \end{align*}
    As a consequence of the equation of continuity \eqref{eq:approx_CE_R}, 
    $$
        \EE\left[\norm{\rho_\alpha}_{C_t^{0,1}(W_x^{2,\frac{2\Gamma}{\Gamma +1}})^\ast }\right] \leq c.
    $$
    Therefore, the required tightness in $C_w([0,T];L^\Gamma(D))$ follows from the compact embedding (see Theorem \ref{thm:weak_conti_embedding}):
    $$
        L^\infty(0,T;L^\Gamma(D)) \cap C^{0,1}([0,T];W^{-2,\frac{2\Gamma}{\Gamma + 1}}(D)) \cptarrow C_w([0,T];L^\Gamma(D)).
    $$
    Tightness in $\left[W^{1,q}(0,T;L^q(D))\cap L^q(0,T;W^{2,q}(D)),w\right]$ is a direct consequence of \eqref{eq:4.12}. In order to show tightness in $L^p(0,T;W^{1,p}(D))$ for some $p>2$, we observe that, in view of compactness of the embedding 
    $$
        W^{1,q}(0,T;L^q(D))\cap L^q(0,T;W^{2,q}(D)) \cptarrow L^r(0,T;W^{1,r}(D)),
    $$
    which holds for a certain $r(q)>1$ by the Aubin--Lions theorem, we obtain tightness in $L^r(0,T;W^{1,r})$ for $r>1$. Using an interpolation argument, we observe that, if $q>1$ and $r'>2$, the set 
    \begin{align*}
        &\left\{ \rho \in  L^{r'}(0,T;W^{1,r'}(D)) \cap W^{1,q}(0,T;L^q(D))\cap L^q(0,T;W^{2,q}(D)):\right. \\ 
        &\quad \left. \norm{\rho}_{L_t^{r'}W_x^{1,r'}} + \norm{\rho}_{W_t^{1,q}L_x^q} + \norm{\rho}_{L_t^q W_x^{2,q}} \leq L   \right\}  
    \end{align*}
    is relatively compact in $L^p(0,T;W^{1,p}(D))$ for some $p>2$. Hence, the desired tightness can be deduced from \eqref{eq:4.11} and \eqref{eq:4.12}.
\end{proof}
The path space $\XX$ is not a Polish space, but it is a sub-Polish space. Therefore, our compactness argument is based on the Jakubowski--Skorokhod representation theorem, Theorem \ref{thm:Jakubowski--Skorokhod}, and from tightness we obtain the following result.
\begin{proposition}
    \label{prop:repr_alpha}
    There exists a complete probability space $\tprobsp$ with $\XX$-valued Borel measurable random variables $(\trho_{0,\alpha},\tu_{0,\alpha},\trho_\alpha,\tq_\alpha,\tu_\alpha,\tW_\alpha),\alpha\in (0,1)$ and $(\trho_0,\tu_0,\trho,\tq,\tu,\tW)$ such that (up to a subsequence):
    \begin{itemize}
        \item[(1)] the laws of $(\trho_{0,\alpha},\tu_{0,\alpha},\trho_\alpha,\tq_\alpha,\tu_\alpha,\tW_\alpha)$ and $(\rho_0,u_0,\rho_\alpha,[\rho_\alpha u_\alpha]_m^\ast,u_\alpha,W)$ coincide on $\XX$;
        \item[(2)] $(\trho_{0,\alpha},\tu_{0,\alpha},\trho_\alpha,\tq_\alpha,\tu_\alpha,\tW_\alpha)$ converges in the topology of $\XX$ $\tPP$-a.s. to $(\trho_0,\tu_0,\trho,\tq,\tu,\tW)$, i.e., 
        \begin{align}
            \begin{aligned}
                \trho_{0,\alpha} &\to \trho_0 \quad \text{in } C(\oD), \\ 
                \tu_{0,\alpha} &\to \tu_0\quad \text{in } V_m, \\
                \trho_\alpha &\to \trho\quad \text{in } L^p(0,T;W^{1,p}(D)), \\ 
                \trho_\alpha &\weakarrow \trho\quad \text{in } W^{1,q}(0,T;L^q(D))\cap L^q(0,T;W^{2,q}(D)), \\ 
                \trho_\alpha &\to \trho\quad \text{in } C_w([0,T];L^\Gamma(D)), \\ 
                \tq_\alpha &\to \tq \quad \text{in } C^\kappa([0,T];V_m^\ast), \\
                \tu_\alpha &\weakarrow \tu \quad \text{in } L^2(0,T;V_m), \\ 
                \tW_\alpha &\to \tW\quad \text{in } C([0,T];\fU_0),
            \end{aligned}  \label{eq:4.14}
        \end{align}
    \end{itemize}
    $\tPP$-a.s. for some $p>2,q>1$, and $\kappa>0$.
\end{proposition}  
\begin{remark}
    \label{rem:r.d.}
    At this stage of the proof it becomes convenient to work with random distributions as introduced in \cite[Section 2.2]{BFHbook18} due to the limit velocity $\tu$ is not a stochastic process in the classical sense.
    From Lemma \ref{lem:prog_m'ble_r.d.}, there is a stochastic process belonging to the same equivalence class as $\tu$, which is progressively measurable with respect to the canonical filtration. 
    Therefore, noting that $\GG_\varepsilon(\rho,\rho u) = \rho \FF_\varepsilon(\rho,u)$, and following a similar argument in \cite[Remark 2.7]{BFHbook18}, for the stochastic integral $\GG_\varepsilon(\trho,\trho \tu)d\tW$ to be well-defined, it suffices to show that $\tW$ is non-anticipative with respect to the filtration 
    $$
        \tFFF_t := \sigma\left(\sigma_t\left[\trho\right]\cup \sigma_t\left[\tu\right]\cup \sigma_t [\tW]\right),\quad t\in [0,T].
    $$
    But, this follows from Lemmas \ref{lem:sufficient_cond_of_Wiener_by_law}, \ref{lem:sufficient_cond_of_(G_t)-Wiener }, and \ref{lem:stability_of_nonanti}, as in Section \ref{sec:The limit for vanishing time step}.
\end{remark}
\begin{remark}
    \label{rem:alpha:initial repr}
    In view of the equality of laws and the convergence of the initial data in Proposition \ref{prop:repr_alpha}, we have 
    $$
        \Lambda = \LL[\trho_0,\tu_0]\quad \text{in } C^{2+\nu}(\oD)\times V_m.
    $$
\end{remark}
\begin{remark}
    \label{rem:NeumannBC}
    By the weak convergence of $\trho_\alpha$ in $L^q(0,T;W^{2,q}(D))$, the limit $\trho$ automatically satisfies the Neumann boundary condition in \eqref{eq:approx_CE_alpha}.
\end{remark}
\begin{lemma}
    \label{lem:alpha(q = rho u)}
    We have 
    $$
        \tq_\alpha =[\trho_\alpha \tu_\alpha]_m^\ast,\quad \tq = [\trho \tu]_m^\ast,\quad \tPP\text{-a.s.}
    $$
\end{lemma}
\begin{proof}
    The first statement follows from the equality of joint laws of 
    $$
        (\rho_\alpha,u_\alpha,[\rho_\alpha u_\alpha]_m^\ast) \quad \text{and} \quad (\trho_\alpha,\tu_\alpha,[\trho_\alpha \tu_\alpha]_m^\ast).
    $$
    As a consequence of the convergence of $\trho_\alpha$ and $\tu_\alpha$ in $\XX_\rho$ and $\XX_u$, respectively, we have 
    $$
        \trho_\alpha \tu_\alpha \weakarrow \trho \tu \quad \text{in}\ L^1(0,T;L^1(D))\quad \tPP\text{-a.s.}
    $$
    and thus
    $$
        [\trho_\alpha \tu_\alpha]_m^\ast \weakarrow [\trho \tu]_m^\ast \quad \text{in}\ L^1(0,T;V_m^\ast)\quad \tPP\text{-a.s.}
    $$
    Therefore, the second statement follows by the uniqueness of the weak limit.
\end{proof}
\begin{lemma}
    \label{lem:alpha(conv rho u otimes u)}
    The following convergence holds true $\tPP$-a.s.: 
    \begin{align}
        \label{eq:4.15}
        \trho_\alpha \tu_\alpha\otimes \tu_\alpha \to \trho \tu\otimes \tu\quad \text{in}\ L^1(0,T;L^1(D)).
    \end{align}
\end{lemma}
\begin{proof}
    First, we have 
    \begin{align*}
        &\norm{\sqrt{\trho}\tu \otimes \sqrt{\trho}\tu - \sqrt{\trho_\alpha}\tu_\alpha\otimes \sqrt{\trho_\alpha}\tu_\alpha}_{L_t^1 L_x^1} \\ 
        &\quad \leq \norm{\left|\sqrt{\trho}\tu - \sqrt{\trho_\alpha}\tu_\alpha \right||\sqrt{\trho}\tu|}_{L_t^1 L_x^1} + \norm{\left|\sqrt{\trho}\tu - \sqrt{\trho_\alpha}\tu_\alpha \right||\sqrt{\trho_\alpha}\tu_\alpha|}_{L_t^1 L_x^1} \\ 
        &\quad \leq \norm{\sqrt{\trho}\tu - \sqrt{\trho_\alpha}\tu_\alpha}_{L_t^2 L_x^2}\left(\norm{\sqrt{\trho}\tu}_{L_t^2 L_x^2} + \sup_{\alpha}\norm{\sqrt{\trho_\alpha}\tu_\alpha}_{L_t^2 L_x^2}\right).
    \end{align*}
    Hence, it suffices to show that 
    \begin{align}
        \label{eq:4.16}
            \sqrt{\trho_\alpha}\tu_\alpha \to  \sqrt{\trho}\tu \quad \text{in}\ L^2(0,T;L^2(D)).
    \end{align}
    To this end, it suffices to prove the following assertions $\tPP$-a.s.:
    \begin{align*}
        &\norm{\sqrt{\trho_\alpha}\tu_\alpha}_{L_t^2 L_x^2} \to \norm{\sqrt{\trho}\tu}_{L_t^2 L_x^2}\quad \text{in}\ \RR, \\ 
        &\sqrt{\trho_\alpha}\tu_\alpha \weakarrow  \sqrt{\trho}\tu \quad \text{in}\ L^2(0,T;L^2(D)).
    \end{align*}
    The first assertion follows from 
    \begin{align*}
        &\norm{\sqrt{\trho_\alpha}\tu_\alpha}_{L_t^2 L_x^2}^2 = \intT \int_D \trho_\alpha \tu_\alpha \cdot \tu_\alpha dxdt \\ 
        &\quad \to \intT \int_D \trho \tu \cdot \tu dx dt = \norm{\sqrt{\trho}\tu}_{L_t^2 L_x^2}^2\quad \tPP\text{-a.s.,}
    \end{align*}
    using Lemma \ref{lem:alpha(q = rho u)} and the convergence of $[\trho_\alpha \tu_\alpha]_m^\ast$ and $\tu_\alpha$ in $\XX_{\rho u}$ and $\XX_u$, respectively. 
    In particular, $\norm{\sqrt{\trho_\alpha}\tu_\alpha}_{L_t^2 L_x^2}$ is bounded, and thus it suffices to show that 
    \begin{align*}
        &\left(\sqrt{\trho_\alpha}\tu_\alpha,\phi\right)_{L_{t,x}^2} \to \left(\sqrt{\trho}\tu,\phi\right)_{L_{t,x}^2},
    \end{align*}
    for any $\phi\in C_c^\infty((0,T)\times D)$ $\tPP$-a.s. 
    As a consequence of the strong convergence of $\trho_\alpha$ and the weak convergence $\tu_\alpha$, we have $\tPP$-a.s.,
    \begin{align*}
        \left|\left(\sqrt{\trho_\alpha}\tu_\alpha - \sqrt{\trho}\tu,\phi\right)_{L_{t,x}^2} \right| &\leq \left|\left(\left(\sqrt{\trho_\alpha}- \sqrt{\trho}\right)\tu_\alpha,\phi\right)_{L_{t,x}^2} \right| + \left|\left(\sqrt{\trho}\left(\tu_\alpha - \tu\right),\phi\right)_{L_{t,x}^2} \right| \\ 
        &\leq \norm{\phi}_{L_{t,x}^\infty}\norm{\trho_\alpha -\trho}_{L_{t,x}^1}^{1/2}\sup_{\alpha}\norm{\tu_\alpha}_{L_{t,x}^2} + \left|\left(\left(\tu_\alpha - \tu\right),\sqrt{\trho}\phi\right)_{L_{t,x}^2} \right| \to 0.
    \end{align*}
    This completes the proof of Lemma~\ref{lem:alpha(conv rho u otimes u)}.
\end{proof}
Similarly to the proof of Proposition \ref{prop:eq_of_conti1}, by virtue of \eqref{eq:4.14}, the equation of continuity \eqref{eq:approx_CE_alpha} is satisfied by $[\trho,\tu]$ a.e. in $(0,T)\times D$, $\tPP$-a.s.
\begin{proposition}
    \label{prop:CE_alpha}
    $(\trho,\tu)$ satisfies \eqref{eq:approx_CE_alpha} a.e. in $(0,T)\times D$ $\tPP$-a.s.
\end{proposition}
Next, we perform the limit $\alpha\to 0$ in the momentum equation \eqref{eq:approx_CE_R}.
\begin{proposition}
    \label{prop:IME_alpha}
    $(\trho,\tu,\tW)$ satisfies \eqref{eq:approx_ME_alpha} for all $\tau\in [0,T]$ and all $\bphi\in V_m$ such that $\bphi|_{\pD} = 0$ $\tPP$-a.s.
\end{proposition}
\begin{proof}
    Similarly to Proposition \ref{prop:ME1}, in view of the equality of laws from Proposition \ref{prop:repr_alpha}, \eqref{eq:approx_ME_R} is satisfied on the new probability space by $(\trho_\alpha,\tu_\alpha,\tW_\alpha)$. 
    By multiplying $\phi\in C_c^\infty(\ico{0}{T})$ and by It\^{o}'s formula, $(\trho_\alpha,\tu_\alpha,\tW_\alpha)$ satisfies that 
    \begin{align}
        -\intT \pp_t\phi \int_D \trho_\alpha \tu_\alpha \cdot \bphi dx dt &= \phi(0)\int_D \trho_{0,\alpha} \tu_{0,\alpha} \cdot \bphi dx  + \intT\phi \int_D \left[(\trho_\alpha \tu_\alpha \otimes \tu_\alpha):\nabla\bphi + p_\delta (\trho_\alpha)\rdiv \bphi\right]dxdt  \notag \\ 
        &\quad - \intT \phi\int_D \left[\SSS(\nabla \tu_\alpha):\nabla\bphi -\varepsilon \trho_\alpha \tu_\alpha\cdot \Delta \bphi \right]dxdt - \inttau \phi \int_{\pD}g\nabla j_\alpha(\tu_\alpha)\cdot \bphi d\Gamma dt \notag \\ 
        &\quad + \intT \phi\int_D \GG_\varepsilon(\trho_\alpha,\trho_\alpha \tu_\alpha) \cdot \bphi dxd\tW_\alpha, \label{eq:4.17}
    \end{align}
    for all $\phi\in C_c^\infty(\ico{0}{T})$ and all $\bphi\in V_m$ $\tPP$-a.s.
    Using \eqref{eq:4.14}, the limit passage in the deterministic terms excluding the boundary integral follows immediately for all $\phi\in C_c^\infty(\ico{0}{T})$ and all $\bphi\in V_m$ $\tPP$-a.s. 
    Let us now discuss in detail the convergence of the stochastic integral. If
    \begin{align}
        \label{eq:4.18}
        F_{k,\varepsilon}(\trho_\alpha,\tu_\alpha) \to F_{k,\varepsilon}(\trho,\tu)\quad \text{in}\  L^2(\tOmega\times (0,T)\times D),
    \end{align}
    as $\alpha\to 0$ for any $k\in\NN$, then, from the first one in \eqref{eq:4.8} (which continue to hold on the new probability space due to Proposition \ref{prop:repr_alpha}),
    \begin{align}
        \label{eq:4.19}
        \trho_\alpha F_{k,\varepsilon}(\trho_\alpha,\tu_\alpha) \to \trho F_{k,\varepsilon}(\trho,\tu)\quad \text{in}\  L^1(\tOmega\times (0,T)\times D).    
    \end{align}
    Thus, we first show \eqref{eq:4.18}.
    We consider the set 
    \begin{align*}
        &\OO_\kappa = \left\{ (\omega,t,x)\in \tOmega\times (0,T)\times D:\trho\geq \kappa \right\},\quad \kappa>0.
    \end{align*}
    Using the strong convergence of $\trho_\alpha$ and $\sqrt{\trho_\alpha}\tu_\alpha$ from \eqref{eq:4.14} and \eqref{eq:4.16}, respectively, the estimates \eqref{eq:4.8} and \eqref{eq:4.9}, and Vitali's convergence theorem, we have (up to a subsequence) 
    $$
        \trho_\alpha\to \trho,\quad \sqrt{\trho_\alpha}\tu_\alpha\to \sqrt{\trho}\tu\quad \text{a.e. in } \tOmega\times (0,T) \times D,
    $$
    and thus, we have 
    $$
        F_{k,\varepsilon}(\trho_\alpha,\tu_\alpha) = F_{k,\varepsilon}\left(\trho_\alpha,\frac{\sqrt{\trho_\alpha}\tu_\alpha}{\sqrt{\trho_\alpha}}\right) \to F_{k,\varepsilon}\left(\trho,\frac{\sqrt{\trho}\tu}{\sqrt{\trho}}\right) = F_{k,\varepsilon}(\trho,\tu)\quad  \text{a.e. in } \OO_\kappa.
    $$
    On the other hand, by the definition of $F_{k,\varepsilon}$, for $\kappa<\varepsilon$, $F_{k,\varepsilon}(\trho,\tu) = 0$ a.e. in $\OO_\kappa^c$, and 
    $$
        F_{k,\varepsilon}(\trho_\alpha,\tu_\alpha) \to 0 \quad \text{a.e. in } \OO_\kappa^c.
    $$
    Therefore, applying the dominated convergence theorem yields the desired convergence \eqref{eq:4.18}.
    
    Next, noting that 
    \begin{align*}
        &\tEE\left[\intT \int_D |\trho_\alpha F_{k,\varepsilon}(\trho_\alpha,\tu_\alpha)|^p dx dt\right] \leq f_{k,\varepsilon}^p \tEE\left[\intT \int_D |\trho_\alpha |^p dx dt\right] ,\quad p\geq 1,
    \end{align*}
    and \eqref{eq:4.8}, combining \eqref{eq:4.19} with this inequality via interpolation yields
    \begin{align}
        \trho_\alpha F_{k,\varepsilon}(\trho_\alpha,\tu_\alpha) \to \trho F_{k,\varepsilon}(\trho,\tu)\quad \text{in}\  L^2(\tOmega\times (0,T)\times D),    \label{eq:4.20}
    \end{align}
    and thus, we obtain that  
    $$
        [\trho_\alpha F_{k,\varepsilon}(\trho_\alpha,\tu_\alpha)]_m^\ast \to [\trho F_{k,\varepsilon}(\trho,\tu)]_m^\ast\quad \text{in}\  L^2(\tOmega\times (0,T);V_m^\ast),    
    $$
    for any $k\in \NN$. On the other hand, by \eqref{eq:3.10} and \eqref{eq:4.8}, we have uniformly in $\alpha$
    \begin{align*}
        \tEE\left[\intT \norm{\left[\trho_\alpha \FF_\varepsilon(\trho_\alpha,\tu_\alpha)\right]_m^\ast}_{L_2(\fU,V_m^\ast)}^2 dt\right] &= \sum_{k=1}^\infty\tEE\left[\intT \norm{\left[\trho_\alpha F_{k,\varepsilon}(\trho_\alpha,\tu_\alpha)\right]_m^\ast}_{V_m^\ast}^2 dt\right] \\ 
        &\lesssim \tEE\left[\intT \norm{\trho_\alpha}_{L_x^2}^2 dt\right] \\ 
        &\lesssim c(r).
    \end{align*}
    Combining this with \eqref{eq:4.20} yields
    $$
        \left[\trho_\alpha \FF_\varepsilon(\trho_\alpha,\tu_\alpha)\right]_m^\ast \to \left[\trho \FF_\varepsilon(\trho,\tu)\right]_m^\ast \quad \text{in } L^2(\tOmega\times (0,T);L_2(\fU,V_m^\ast)).
    $$
    Therefore, combining this with the convergence of $\tW_\alpha$ from Proposition \ref{prop:repr_alpha}, we may apply Lemma \ref{lem:conv_of_stoch_int} to pass to the limit in the stochastic integral and hence complete the proof.
\end{proof} 
\begin{remark}
    \label{rem:alpha(conv of each term of ME)}
    The convergence to each term in the momentum equation \eqref{eq:approx_ME_alpha} actually holds for any $\bphi\in V_m$, but the equality holds only for $\bphi\in V_m$ such that $\bphi|_{\pD} =0$.
\end{remark}
As the next step, we take the limit $\alpha\to 0$ in the stochastic integral appearing in the momentum and energy inequality.

\begin{proposition}
    \label{prop:alpha(conv of stoch int in MEI)}
    We have 
    \begin{align*}
        &\inttau \int_D \GG_\varepsilon(\trho_\alpha,\trho_\alpha \tu_\alpha)\cdot \tu_\alpha dx d\tW_\alpha \to \inttau \int_D \GG_\varepsilon(\trho,\trho \tu) \cdot \tu dx d\tW\quad \text{in } L^2(0,T)\quad \text{in probability,} \\ 
        &\frac 1 2 \sum_{k=1}^{\infty} \inttau \int_D \trho_\alpha \left|F_{k,\varepsilon}(\trho_\alpha,\tu_\alpha) \right|^2 dx dt \to \frac 1 2 \sum_{k=1}^{\infty} \inttau \int_D \trho \left|F_{k,\varepsilon}(\trho,\tu) \right|^2 dx dt \quad \text{for all }\tau\in [0,T] \ \tPP\text{-a.s.}
    \end{align*}
\end{proposition}
\begin{proof}
    First, note that the relationship between $\GG_\varepsilon= (G_{k,\varepsilon})_{k\in \NN}$ and $\FF_\varepsilon=(F_{k,\varepsilon})_{k\in\NN}$ is given in \eqref{eq:3.8}.
    By interpolation, we observe that \eqref{eq:4.18} can be strengthened to 
    $$
        F_{k,\varepsilon}(\trho_\alpha,\tu_\alpha) \to F_{k,\varepsilon}(\trho,\tu)\quad \text{in}\  L^q((0,T)\times D),  
    $$
    $\tPP$-a.s., for all $q\in \ico{0}{\infty}$ and all $k\in \NN$. With this, \eqref{eq:3.10}, the estimates \eqref{eq:4.8}--\eqref{eq:4.9}, and the strong convergence of $\trho_\alpha$ and $\sqrt{\trho_\alpha}\tu_\alpha$ at hand, we obtain that 
    \begin{align*}
        &\int_D \trho_\alpha F_{k,\varepsilon}(\trho_\alpha,\tu_\alpha)\cdot \tu_\alpha dx \to \int_D \trho F_{k,\varepsilon}(\trho,\tu)\cdot \tu dx \quad \text{in } L^2(0,T)  
    \end{align*}
    $\tPP$-a.s., for all $k\in \NN$. On the other hand, from \eqref{eq:3.10} and \eqref{eq:4.9}, 
    \begin{align*}
        &\tEE\left[\intT \norm{\int_D \trho_\alpha \FF_\varepsilon(\trho_\alpha,\tu_\alpha)\cdot \tu_\alpha dx}_{L_2(\fU,\RR)}^2 dt \right] \\ 
        &\quad = \tEE\left[\intT \sum_{k=1}^{\infty}\left(\int_D \trho_\alpha F_{k,\varepsilon}(\trho_\alpha,\tu_\alpha)\cdot \tu_\alpha dx \right)^2 dt\right] \\ 
        &\quad \lesssim \sum_{k=1}^{\infty} f_{k,\varepsilon}^2 \tEE\left[\intT \left(\int_D \left|\trho_\alpha \tu_\alpha \right|^q dx\right)^2\right] \lesssim \sum_{k=1}^{\infty}f_{k,\varepsilon}^2,
    \end{align*}
    for all $q\in \ioc{1}{\tfrac{2\Gamma}{\Gamma +1}}$. This implies 
    \begin{align*}
        &\int_D \trho_\alpha \FF_\varepsilon(\trho_\alpha,\tu_\alpha)\cdot \tu_\alpha dx \to \int_D \trho \FF_\varepsilon(\trho,\tu)\cdot \tu dx  \quad \text{in } L^2(0,T;L_2(\fU,\RR))
    \end{align*}
    $\tPP$-a.s. Combining this with the convergence of $\tW_\alpha$ from Proposition \ref{prop:repr_alpha}, we apply Lemma \ref{lem:conv_of_stoch_int} and the first assertion follows. 
    Similarly, using \eqref{eq:4.18} and \eqref{eq:4.20}, it can be verified that the second assertion holds.
\end{proof}
\begin{proposition}
    \label{prop:alpha:MEI}
    The momentum and energy inequality \eqref{eq:approx_MEI_alpha} is satisfied by $(\trho,\tu,\tW)$ for all $\tau\in [0,T]$, all $\phi\in C_c^\infty(\ico{0}{\tau})$, and all $\bphi\in V_m$ $\tPP$-a.s.
\end{proposition}
\begin{proof}
    Recall that \eqref{eq:4.17} is satisfied by $(\trho_\alpha,\tu_\alpha,\tW_\alpha)$ for all $\phi\in C_c^\infty(\ico{0}{T})$ and all $\bphi\in V_m$ $\tPP$-a.s. In particular, this identity is satisfied for all $\tau\in [0,T]$, $\phi\in C_c^\infty(\ico{0}{\tau})$ and all $\bphi\in V_m$ $\tPP$-a.s.
    On the other hand, from the equality of laws from Proposition \ref{prop:repr_alpha}, $(\trho_\alpha,\tu_\alpha,\tW_\alpha)$ satisfies the energy balance \eqref{eq:EB_R} by replacing the initial value term 
    \begin{align*}
        &\int_D \left[\frac 1 2 \rho_0 |u_0|^2 + P_\delta(\rho_0)\right]dx
    \end{align*}
    with
    $$
        \int_D \left[\frac12 \trho_{0,\alpha} \left|\tu_{0,\alpha} \right|^2 + P_\delta(\trho_{0,\alpha})\right]dx.
    $$
    Further, as a consequence of the convexity and $C^1$-regularity of $j_\alpha$, we have 
    $$
        \nabla j_\alpha(\tu_\alpha) \cdot (\phi\bphi-\tu_\alpha) \leq j_\alpha(\phi \bphi)- j_\alpha(\tu_\alpha).
    $$
    Therefore, by subtracting the energy equality \eqref{eq:EB_R} from the approximate momentum equation \eqref{eq:4.17}, and using this estimate of $j_\alpha$ and the estimate \eqref{eq:3.63}, we obtain that 
    \begin{align*}
        &\int_D \left[\frac 1 2 \trho_{0,\alpha} |\tu_{0,\alpha}|^2 + P_\delta(\trho_{0,\alpha})\right]dx - \int_D \left[\frac 1 2 \trho_\alpha |\tu_\alpha|^2 + P_\delta(\trho_\alpha)\right](\tau)dx \\ 
        &\quad - \inttau \pp_t\phi \int_D \trho_\alpha \tu_\alpha\cdot \bphi dx - \phi(0)\int_D \trho_{0,\alpha} \tu_{0,\alpha} \cdot \bphi dx \notag\\ 
        &\quad - \inttau \phi \int_D \left[(\trho_\alpha \tu_\alpha \otimes \tu_\alpha):\nabla\bphi + p_\delta (\trho_\alpha)\rdiv \bphi\right]dxdt + \inttau \int_D \left[\SSS(\nabla \tu_\alpha):\nabla(\phi\bphi- \tu_\alpha) -\varepsilon \trho_\alpha \tu_\alpha\cdot \Delta (\phi\bphi) \right]dxdt \notag \\ 
        &\quad - \inttau \int_D \left[\varepsilon \trho_\alpha |\nabla \tu_\alpha|^2 + \varepsilon P_\delta''(\trho_\alpha)|\nabla \trho_\alpha|^2\right]dxdt +\frac 1 2 \sum_{k=1}^{\infty} \inttau \int_D \trho_\alpha \left|F_{k,\varepsilon}(\trho_\alpha,\tu_\alpha) \right|^2 dx dt \notag \\ 
        &\quad - \inttau \int_D \GG_\varepsilon(\trho_\alpha,\trho_\alpha \tu_\alpha) \cdot (\phi\bphi - \tu_\alpha) dx dW + \inttau \int_{\pD} g \left[j_\alpha(\phi\bphi) - j_\alpha(\tu_\alpha)\right] d\Gamma dt\geq 0 
    \end{align*}
    holds for all $\tau\in [0,T]$, all $\phi\in C_c^\infty(\ico{0}{\tau})$, and all $\bphi\in V_m$ $\tPP$-a.s.
    Due to the uniform convergence \eqref{eq:3.6} of $j_\alpha$, the weak convergence \eqref{eq:4.14} of $\tu_\alpha$, and the non-negativity and the integrability of $g$, we obtain that
    \begin{align*}
        &\limsup_{\alpha\to 0} \inttau \int_{\pD}  g \left[j_\alpha(\phi\bphi) - j_\alpha(\tu_\alpha)\right] d\Gamma dt \leq \inttau \int_{\pD} g\left|\phi\bphi \right| - g\left|\tu \right|d\Gamma dt.
    \end{align*}
    Note that due to \eqref{eq:4.14}, the mapping
    $$
        (\rho,u) \mapsto \int_D \left[\frac 1 2 \trho |\tu|^2 + P_\delta(\trho)\right](\tau)dx + \inttau \int_D \left[\SSS(\nabla \tu):\nabla \tu + \varepsilon \trho \left|\nabla \tu \right|^2 + \varepsilon P_\delta''(\trho)|\nabla \trho|^2\right]dx dt
    $$
    is at least lower semi-continuous with respect to the topologies in \eqref{eq:4.14}, for almost all $\tau\in [0,T]$.
    Therefore, recalling Remark \ref{rem:alpha(conv of each term of ME)}, and combining Proposition \ref{prop:alpha(conv of stoch int in MEI)} with the above discussion, we see that the desired inequality \eqref{eq:approx_MEI_alpha} is satisfied by $(\trho,\tu,\tW)$ for almost all $\tau\in [0,T]$, all $\phi\in C_c^\infty(\ico{0}{\tau})$, and all $\bphi\in V_m$ $\tPP$-a.s.
    However, since $\trho$ and $\trho \tu$ are at least weakly continuous in time, the weak lower semi-continuity of the energy yields \eqref{eq:approx_MEI_alpha} for  all $\tau\in [0,T]$, all $\phi\in C_c^\infty(\ico{0}{\tau})$, and all $\bphi\in V_m$ $\tPP$-a.s.
\end{proof}
The proof of Theorem \ref{thm:sol_alpha} is hereby complete.

\section{The limit in the Galerkin approximation scheme}\label{sec:The limit in the Galerkin approximation scheme}
Our next goal is to let $m\to \infty$ in the approximate system \eqref{eq:approx_CE_alpha}--\eqref{eq:approx_MEI_alpha}. 
A rigorous formulation reads as follows.
\begin{definition}
    \label{def:sol_m}
	{\hspace{-5pt} Let $\Lambda$ be a Borel probability measure on $C^{2+\nu}(\oD)\times L^1(D)$. \hspace{-3pt}Then $((\Omega,\FFF,(\FFF_t)_{t\geq 0},\PP),\rho,u,W)$ is called a \textit{dissipative martingale solution in $\varepsilon$ layer} with the initial law $\Lambda$, if:}
	\begin{itemize}
		\item[(1)] $(\Omega,\FFF,(\FFF_t)_{t\geq 0},\PP)$ is a stochastic basis with a complete right-continuous filtration;
		\item[(2)] $W$ is a cylindrical $(\FFF_t)$-Wiener process on $\fU$;
		\item[(3)] the density $\rho$ is an $(\FFF_t)$-adapted stochastic process, and the velocity $u$ is an $(\FFF_t)$-adapted random distribution such that $\PP$-a.s.
		\begin{align*}
		   &\rho \geq 0,\quad \rho \in C_w([0,T];L^\Gamma(D)),\quad u\in L^2(0,T;W_{\rn }^{1,2}(D)),\quad \rho u \in C_w([0,T];L^{\frac{2\Gamma}{\Gamma + 1}}(D));   
		\end{align*}
		\item[(4)] there exists a $C^{2+\nu}(\oD)\times L^2(D)$-valued $\FFF_0$-measurable random variable $(\rho_0,u_0)$ such that \!$\Lambda = \LL[\rho_0,\rho_0 u_0]$;
		\item[(5)] the approximate equation of continuity 
		\begin{equation}
            \label{eq:approx_CE_m}
            \pp_t \rho + \rdiv (\rho u) = \varepsilon \Delta \rho,\quad \nabla \rho \cdot \rn |_{\pD} =0
        \end{equation}
		holds a.e. in $(0,T)\times D$ $\PP$-a.s.;
		\item[(6)] the approximate interior momentum equation 
		\begin{align}
            -\intT \pp_t\phi \int_D \rho u\cdot \bphi dx dt &= \phi(0)\int_D \rho_0 u_0 \cdot \bphi dx  + \intT\phi \int_D \left[(\rho u \otimes u):\nabla\bphi + p_\delta (\rho)\rdiv \bphi\right]dxdt  \notag \\ 
            &\quad - \intT \phi\int_D \left[\SSS(\nabla u):\nabla\bphi -\varepsilon \rho u\cdot \Delta \bphi \right]dxdt \notag \\ 
            &\quad + \intT \phi\int_D \GG_\varepsilon(\rho,\rho u) \cdot \bphi dxdW  \label{eq:approx_ME_m}
        \end{align}
		holds for all $\phi\in C_c^\infty(\ico{0}{T})$ and all $\bphi\in C_c^\infty(D)$ $\PP$-a.s.;
        \item[(7)] the approximate momentum and energy inequality 
        \begin{align}
            &\int_D \left[\frac 1 2 \rho_0 |u_0|^2 + P_\delta(\rho_0)\right]dx - \int_D \left[\frac 1 2 \rho |u|^2 + P_\delta(\rho)\right](\tau)dx - \inttau \pp_t\phi \int_D \rho u\cdot \bphi dx - \phi(0)\int_D \rho_0 u_0 \cdot \bphi dx \notag\\ 
            &\quad - \inttau \phi \int_D \left[(\rho u \otimes u):\nabla\bphi + p_\delta (\rho)\rdiv \bphi\right]dxdt + \inttau \int_D \left[\SSS(\nabla u):\nabla(\phi\bphi- u) -\varepsilon \rho u\cdot \Delta (\phi\bphi) \right]dxdt \notag \\ 
            &\quad - \inttau \int_D \left[\varepsilon \rho |\nabla u|^2 + \varepsilon P_\delta''(\rho)|\nabla \rho|^2\right]dxdt +\frac 1 2 \sum_{k=1}^{\infty} \inttau \int_D \rho \left|F_{k,\varepsilon}(\rho,u) \right|^2 dx dt \notag \\ 
            &\quad - \inttau \int_D \GG_\varepsilon(\rho,\rho u) \cdot (\phi\bphi - u) dx dW + \inttau \int_{\pD} g|\phi\bphi| - g |u| d\Gamma dt\geq 0 \label{eq:approx_MEI_m}
        \end{align}
        holds for all $\tau\in [0,T]$, all $\phi\in C_c^\infty(\ico{0}{\tau})$, and all $\bphi\in C^\infty(\oD)$ with $\bphi\cdot \rn|_{\pD} = 0$ $\PP$-a.s.
    \end{itemize}
\end{definition}
\begin{theorem}
    \label{thm:sol_m}
    Let $\Lambda$ be a Borel probability measure on $C^{2+\nu}(\oD)\times L^1(D)$ such that 
    \begin{equation}
        \Lambda \left\{ \urho\leq \rho \leq \orho,\ \nabla \rho\cdot \rn |_{\pD}=0 \right\}=1, \label{eq:5.4}
    \end{equation}
    for some deterministic constants $\urho,\orho>0$ and 
    \begin{align}
        \int_{C_x^{2+\nu}\times L_x^1} \left(\left|\int_D \left[\frac 1 2 \frac{|q|^2}{\rho} + P_\delta(\rho)\right]dx \right|^{2r} + \norm{\rho}_{C_x^{2+\nu}}^r\right)d \Lambda (\rho,q) <\infty, \label{eq:5.5}
    \end{align}
    for some $r \geq 2$. Then there exists a dissipative martingale solution in the sense of Definition \ref{def:sol_m}.
\end{theorem}
\begin{remark}
    \label{rem:m(integrability_of_sol)}
    As in Remark \ref{rem:alpha(integrability_of_sol)}, the solution constructed in this section belongs to the following class 
    \begin{align*}
        &\EE\left[\sup_{t\in [0,T]}\left|\int_D\left[\frac 1 2 \rho |u|^2 + P_\delta(\rho)\right]dx \right|^{2r}\right] + \EE\left[\left|\intT \int_{\pD} g |u| d\Gamma dt \right|^{2r}\right] <\infty, \\ 
        &\EE\left[\left|\intT \int_D \left[\SSS(\nabla u):\nabla u + \varepsilon \rho |\nabla u|^2 + \varepsilon P_\delta''(\rho)|\nabla \rho|^2\right]dxdt\right|^{2r}\right]<\infty,
    \end{align*}
    with the same $r\geq 2$ as in \eqref{eq:4.5}, and thus we have 
    \begin{align*}
        &\EE \left[\sup_{t\in [0,T]} \norm{\rho(t)}_{L_x^\Gamma}^{2\Gamma r}\right] + \EE \left[\sup_{t\in [0,T]} \norm{\rho |u|^2(t)}_{L_x^1}^{2 r}\right] + \EE \left[\sup_{t\in [0,T]} \norm{\rho u(t)}_{L_x^{\frac{2\Gamma}{\Gamma +1}}}^{\frac{4\Gamma}{\Gamma + 1}r}\right] + \EE\left[\norm{u}_{L_t^2 W^{1,2}_x}^{2r}\right]<\infty.
    \end{align*}
\end{remark}
\subsection{Uniform estimates}\label{sec:m:Uniform estimates}
In this section, the approximation parameters $\varepsilon,\delta\in (0,1)$ are kept fixed. We call uniform estimate that are independent of $m$ but may depend on $\varepsilon,\delta$ and $T>0$.
The proof of Theorem \ref{thm:sol_m} can be carried out similarly by modifying some of the arguments in Section \ref{sec:convex_approx}. 
First, let $(\rho,u)$ denote a solution constructed in Theorem \ref{thm:sol_alpha}. Then choosing $\phi=0$ in the momentum and energy inequality \eqref{eq:approx_MEI_alpha} yields the following standard energy inequality:
\begin{align*}
    &\int_D \left[\frac 1 2 \rho |u|^2 + P_\delta(\rho)\right](\tau)dx \\ 
    &\quad + \inttau \int_D \left[\SSS(\nabla u):\nabla u + \varepsilon \rho |\nabla u|^2 + \varepsilon P_\delta''(\rho)|\nabla \rho|^2\right]dxdt +\inttau \int_{\pD} g |u| d\Gamma dt \notag\\ 
    &\quad \leq \int_D \left[\frac 1 2 \rho_0 |u_0|^2 + P_\delta(\rho_0)\right]dx + \frac 1 2 \sum_{k=1}^{\infty} \inttau \int_D \rho \left|F_{k,\varepsilon}(\rho,u) \right|^2 dx dt + \inttau \int_D \GG_\varepsilon(\rho,\rho u) \cdot u dx dW   
\end{align*}
for all $\tau\in [0,T]$ $\PP$-a.s. 
Therefore,  as in Section \ref{sec:alpha:Uniform estimates}, using Gronwall's lemma together with Remark \ref{rem:alpha(integrability_of_sol)} yields the uniform estimates:
\begin{align}
    &\EE\left[\left|\sup_{\tau\in [0,T]} \int_D \left[\frac 1 2 \rho |u|^2 + P_\delta(\rho)\right](\tau)dx \right|^r\right] \notag\\ 
    &\quad  +\EE\left[\left|\intT \int_D \left[\SSS(\nabla u):\nabla u + \varepsilon \rho |\nabla u|^2 + \varepsilon P_\delta''(\rho)|\nabla \rho|^2\right]dxdt \right|^r\right]\notag \\ 
    &\quad  + \EE\left[\left|\intT \int_{\pD} g |u| d\Gamma dt \right|^{r}\right] \lesssim  c(r), \label{eq:5.6} \\ 
    &\EE\left[\norm{\rho}_{L_t^\infty L_x^\Gamma}^{\Gamma r}\right] + \EE \left[\norm{\rho^{-1/2}\nabla \rho}_{L_t^2 L_x^2}^{2r}\right] + \EE \left[\norm{\rho^{\Gamma / 2-1}\nabla \rho}_{L_t^2 L_x^2}^{2r}\right] \lesssim c(r), \label{eq:5.7}\\ 
    & \EE\left[\norm{\rho |u|^2}_{L_t^\infty L_x^1}^{r}\right] + \EE\left[\norm{\rho u}_{L_t^\infty L_x^{\frac{2\Gamma}{\Gamma +1}}}^{\frac{2\Gamma}{\Gamma +1}r}\right] \lesssim c(r), \label{eq:5.8} \\
    &\EE\left[\left| \intT \norm{u}_{W_x^{1,2}}^2 dt \right|^{r} \right] \lesssim  c(2r) \label{eq:5.9}
\end{align}
uniformly in $m$, where $\Gamma=\max \left\{ 6,\gamma \right\}$ is given at the beginning of Section \ref{sec:Outline of the proof of the main result}, and 
$$
    c(r) = \EE\left[\left|\int_D \left[\frac 1 2 \rho_0 |u_0|^2 + P_\delta(\rho_0)\right]dx \right|^r\right] + 1,\quad r\geq 2.
$$
Furthermore, using exactly the same argument as in Section \ref{sec:alpha:Uniform estimates}, we obtain the uniform estimates: 
\begin{align}
    &\EE\left[\norm{\rho}_{L_t^p W_x^{1,p}}^r\right]\lesssim \tilde{c}(2r) \quad \text{for some} \  p>2, \label{eq:5.10} \\ 
    &\EE\left[\left|\norm{\pp_t \rho}_{L_t^pL_x^p} + \norm{\rho}_{L_t^p W_x^{2,p}} \right|^{r / 2}\right] \lesssim \tilde{c}(2r),\quad \text{for some } p>1, \label{eq:5.11}
\end{align}
uniformly in $m$, where
$$
    \tilde{c}(r) = \EE\left[\left|\int_D \left[\frac 1 2 \rho_0|u_0|^2 + P_\delta(\rho_0)\right]dx \right|^r + \norm{\rho_0}_{C_x^{2+\nu}}^{r / 2}\right] + 1,\quad r\geq 2.
$$
Finally, unlike the case in \cite[Section 4.3]{BFHbook18}, uniform estimates for the H\"{o}lder norm of $[\rho u]_m^\ast$ cannot be expected due to the degeneracy of the momentum equation \eqref{eq:approx_ME_alpha} and the Galerkin approximation. 
Nevertheless, due to the properties of the subsequence $V_{0,m}$ assumed in the Galerkin approximation (see Section \ref{sec:Approximation system}), we can obtain the time regularity of $\rho u$ in a new probability space.
To this end, let $(\rho_m,u_m),m\in\NN$ be a sequence of solutions constructed in Theorem \ref{thm:sol_alpha} starting from $(\rho_{0,m},u_{0,m})$. Without loss of generality, we assume the energies of $(\rho_{0,m},u_{0,m})$ coincide and use the notation $c(r)$ as above. Then, we choose a subsequence $(n_k)_{k\in \NN}\subset \NN$ and a sequence $(\phi_{n_k})_{k\in\NN}\subset \bigcup_{m\in\NN}V_{0,m}$ such that
$$
    (\phi_{n_k})_{k\in\NN} \quad \text{is a dense subset of }L^{\frac{2\Gamma }{\Gamma - 1}}(D),\quad \phi_{n_k}\in V_{n_k},\quad \phi_{n_k}|_{\pD} = 0.
$$ 
By the definition of $V_{0,m}$, such a choice is possible. We introduce 
$$
    Y_m = (Y_{m,k})_{k=1}^\infty := \left(\inner{\rho_m u_m}{\phi_{n_k}}1_{m\geq n_k}\right)_{k\in \NN}\in \left[C([0,T])\right]^\NN.
$$
Note that $\left[C([0,T])\right]^\NN$ is a sub-Polish space. To derive tightness of laws of $Y_m$ later, we prepare the following.
\begin{lemma}
    \label{lem:m:unif_est_others}
    We have 
    \begin{align*}
        &\EE\left[\norm{\rho u\otimes u }_{L_t^2 L_x^{\frac{6\Gamma}{4\Gamma + 3}}}^{r}\right] + \EE\left[\norm{\SSS(\nabla u)}_{L_t^2 L_x^2}^{2r}\right] \lesssim c(2r),\\ 
        &\EE\left[\norm{p_\delta(\rho)}_{L_t^2 L_x^1}^r\right] + \EE\left[\norm{\nabla (\rho u)}_{L_t^2 W_x^{-1,\frac{2\Gamma}{\Gamma + 1}}}^r\right] \lesssim c(r)
    \end{align*}
    uniformly in $m$.
\end{lemma}
\begin{proof}
    Recall that the definition of $p_\delta$ is given at the beginning of Section \ref{sec:Outline of the proof of the main result}.
    The estimates of the first line follow from \eqref{eq:5.8}--\eqref{eq:5.9}, the embedding $W^{1,2}(D)\hookrightarrow L^6(D)$, and H\"older's inequality.
    Note that 
    \begin{align*}
        &\frac16 + \frac{\Gamma + 1}{2\Gamma} = \frac{4\Gamma + 3}{6\Gamma}.  
    \end{align*}
    The estimates of the second line follow from \eqref{eq:5.7}--\eqref{eq:5.8}.
\end{proof}
\begin{proposition}
    \label{prop:m:Holder_conti_of_rho u}
    We have 
    \begin{align}
        \label{eq:5.12}
        \sup_{m\geq n_k}\EE\left[\norm{\inner{\rho_m u_m}{\phi_{n_k}}}_{C_t^\kappa}^{2r}\right] \lesssim c(2r)\norm{\phi_{n_k}}_{C_x^1}^{2r},\quad \text{for all }k\in \NN, 
    \end{align}
    for some $\kappa>0$, where the proportional constant does not depend on $k$.
\end{proposition}
\begin{proof}
    We first show that 
    \begin{align}
        \EE\left[\left|\inner{\rho_m u_m(\tau_1)}{\phi_{n_k}} - \inner{\rho_m u_m(\tau_2)}{\phi_{n_k}} \right|^{2r}\right] \lesssim c(2r)\norm{\phi_{n_k}}_{W_x^{1,\beta}}^{2r} |\tau_1-\tau_2|^{r}, \label{eq:5.13}
    \end{align}
    whenever $m\geq n_k$, and $0\leq \tau_2<\tau_1\leq T$, for $\beta>3$, where the proportional constant does not depend on $m$ and $k$. 
    With the interior momentum equation \eqref{eq:approx_ME_alpha} at hand, it is easy to verify that 
    \begin{align}
        &\EE\left[\left|\inner{\rho_m u_m(\tau_1)}{\phi_{n_k}} - \inner{\rho_m u_m(\tau_2)}{\phi_{n_k}} \right|^{2r}\right] \notag  \\ 
        &\lesssim \EE \left[\left(\int_{\tau_2}^{\tau_1} \left(\norm{\rho_m u_m \otimes u_m}_{W_x^{-1,\beta'}} +\norm{p_\delta(\rho_m)}_{W_x^{-1,\beta'}}\right)\norm{\phi_{n_k}}_{W_x^{1,\beta}}dt \right)^{2r}\right] \notag \\ 
        &\quad + \EE \left[\left(\int_{\tau_2}^{\tau_1} \left(\norm{\SSS(\nabla u_m)}_{W_x^{-1,\beta'}} + \varepsilon\norm{\nabla(\rho_m u_m)}_{W_x^{-1,\beta'}}\right)\norm{\phi_{n_k}}_{W_x^{1,\beta}}dt \right)^{2r}\right] \notag \\ 
        &\quad + \left|\tau_1 -\tau_2 \right|^r \EE\left[\sup_{t\in [0,T]} \left(\sum_{k=1 }^{\infty} \norm{\rho_m F_{k,\varepsilon}(\rho_m,u_m)}_{L_x^2}^2\right)^r\right] \norm{\phi_{n_k}}_{L_x^2}^{2r}, \quad \text{whenever } n_k\leq m, \label{eq:5.14}
    \end{align}
    where $\beta'$ is the H\"older conjugate exponent of $\beta$, the Burkholder--Davis--Gundy inequality was used in the last term, and the proportional constant does not depend on $m$ or $k$. 
    By H\"{o}lder's inequality, \eqref{eq:3.10}, and $\Gamma \geq 6$, we have 
    \begin{align*}
        &\norm{\rho F_{k,\varepsilon}(\rho,u)}_{L_x^2} \leq \norm{\rho}_{L_x^4}\norm{F_{k,\varepsilon}(\rho,u)}_{L^4_x}\lesssim \norm{\rho}_{L_x^\Gamma}\norm{F_{k,\varepsilon}(\rho,u)}_{L_x^\infty} \leq f_{k,\varepsilon}\norm{\rho}_{L_x^\Gamma}.  
    \end{align*}
    Therefore, choosing $\beta>3$ so that $L^1(D)\hookrightarrow W^{-1,\beta'}(D)$ and applying Lemma \ref{lem:m:unif_est_others} to \eqref{eq:5.14} yield \eqref{eq:5.13}.
    By Kolmogorov's continuity theorem, the desired assertion \eqref{eq:5.12} holds for $\kappa\in (0,1/2-1/(2r))$. 
\end{proof}

\subsection{Asymptotic limit}\label{sec:m:Asymptotic limit}
With the uniform bounds \eqref{eq:5.6}--\eqref{eq:5.12} established in the previous section, we are ready to perform the limit $m\to \infty$ similarly as in Section \ref{sec:alpha(Asymptotic limit)}.
Let $\Lambda$ be a probability measure on $C^{2+\nu}(\oD)\times L^1(D)$ satisfying \eqref{eq:5.4} and \eqref{eq:5.5}, and consider a random variable $(\rho_0,q_0)$ with law $\Lambda$ on some probability space $\probsp$. 
Since $\rho_0\geq \urho>0$, we can set $u_0 = q_0/\rho_0$. Note that, due to the assumption \eqref{eq:5.5}, we have $u_0\in L^2(D)$ $\PP$-a.s. We define the initial velocities $u_{0,m}=\Pi_m u_0$, where $\Pi_m:L^2(D)\to V_m$ is the orthogonal projection onto $V_m$. Then, we observe that the initial laws $\overline{\Lambda}_m = \PP\circ (\rho_0,u_{0,m})^{-1}$ satisfy the assumptions on the initial condition in Theorem \ref{thm:sol_alpha} uniformly in $m$.
Therefore, Theorem \ref{thm:sol_alpha} yields existence of 
$$
    ((\Omega^m,\FFF^m,(\FFF_t^m),\PP^m),\hat{\rho}_{0,m},\hat{u}_{0,m},\rho_m,u_m,W_m),
$$ 
which is a dissipative martingale solution in the sense of Definition \ref{def:sol_alpha} with the initial law $\overline{\Lambda}_m$, and 
\begin{align}
    \label{eq:m:unif_of_initial}
    \int_{C_x^{2+\nu}\times L_x^2} \left(\left|\int_D \left[\frac12 \rho |v|^2 + P_\delta(\rho)\right]dx \right|^{2r} + \norm{\rho}_{C_x^{2+\nu}}^r \right)d\overline{\Lambda}_m(\rho,v) \leq c, 
\end{align}
uniformly in $m$. As a consequence of \eqref{eq:m:unif_of_initial}, the estimates \eqref{eq:5.6}--\eqref{eq:5.12} hold true for $(\rho_m,u_m)$ uniformly in $m$.
Moreover, since $(\rho_0,\rho_0 u_{0,m}) \to (\rho_0,\rho_0 u_0)$ in $C^{2+\nu}(\oD)\times L^1(D)$ $\PP$-a.s., and $C^{2+\nu}(\oD)$ is a Borel measurable subset of $C(\oD)$, we have 
\begin{align}
    &\Lambda_m := \PP^m\circ (\hat{\rho}_{0,m},\hat{\rho}_{0,m}\hat{u}_{0,m})^{-1} = \PP\circ (\rho_0,\rho_0 u_{0,m})^{-1} \to \Lambda \quad \text{weakly in } C(\oD) \times L^1(D). \label{eq:5.16}
\end{align}
In view of \cite[Remark 4.0.4]{BFHbook18}, we may assume without loss of generality that 
$$
    (\Omega^m,\FFF^m,\PP^m) = ([0,1],\overline{\fB([0,1])},\fL),\quad m\in \NN
$$
and that 
$$
    \FFF^m_t = \sigma \left(\sigma_t[\rho_m]\cup \sigma_t [u_m]\cup \sigma_t[W_m]\right),\quad t\in [0,T],
$$
where $([0,1],\overline{\fB([0,1])},\fL)$ is the Lebesgue measure space.
In accordance with the uniform estimates \eqref{eq:5.6}--\eqref{eq:5.12} derived in the previous section, we choose the path space 
$$
    \XX = \XX_{\rho_0}\times \XX_{u_0}\times \XX_\rho \times \XX_{\rho u}\times \XX_u \times \XX_{Y} \times \XX_W,
$$
where 
\begin{align*}
    &\XX_{\rho_0} = C(\oD), \\ 
    &\XX_{u_0} = L^2(D), \\ 
    &\XX_\rho = L^p(0,T;W^{1,p}(D)) \cap \left[W^{1,q}(0,T;L^q(D))\cap L^q(0,T;W^{2,q}(D)),w\right]  \cap C_w([0,T];L^\Gamma(D)), \\ 
    &\XX_{\rho u} = \left[L^\infty(0,T;L^{\frac{2\Gamma}{\Gamma + 1}}(D)),w^\ast\right], \\ 
    &\XX_u = \left[L^2(0,T;W^{1,2}_{\rn }(D)),w\right], \\ 
    &\XX_Y = \left[C([0,T])\right]^{\NN},\\ 
    &\XX_W = C([0,T];\fU_0),
\end{align*}
for some $p>2$, and $ q>1$. 
\begin{proposition}
    \label{prop:m:tightness}
    The family of the probability measures $\left\{ \LL[\hat{\rho}_{0,m},\hat{u}_{0,m},\rho_m,\rho_m u_m, u_m,Y_m, W_m]:m\in \NN \right\}$ is tight on $\XX$,
    where $Y_m$ is defined as in Section \ref{sec:m:Uniform estimates}. 
\end{proposition} 
\begin{proof}
    Similarly to Section \ref{sec:alpha(Asymptotic limit)}, $p>2$ and $q>1$ can be adjusted so that 
    $$
        \left\{ \LL\left[\rho_m\right]: m\in \NN \right\}\quad \text{is tight on }\XX_\rho. 
    $$ 
    Tightness of $\left\{ \LL[\hat{u}_{0,m}]:m\in \NN \right\}$ on $\XX_{u_0}$ follows immediately from $u_{0,m}\to u_0$ in $L^2(D)$ $\PP$-a.s.
    Tightness of the laws of $\rho_m$ and $u_m$ can be verified in the same way as in Section \ref{sec:alpha(Asymptotic limit)}, and tightness of the laws of $\rho_m u_m$ follows from weak$^\ast$ compactness of bounded sets in $L^\infty(0,T;L^{\frac{2\Gamma}{\Gamma + 1}}(D))$.
    Thus, only tightness of $\left\{ \LL[Y_m]:m\in \NN \right\}$ on $\XX_{Y}$ is non-trivial. By Tychonoff's theorem, the set 
    $$
        B_L= \prod_{k=1}^\infty \left\{ f\in C([0,T]): \norm{f}_{C_t^\kappa} \leq L^{\frac{1}{2r}}2^{\frac{k }{2r}}\norm{\phi_{n_k}}_{C_x^1} \right\}
    $$
    is relatively compact in $\left[C([0,T])\right]^\NN$. Thus, using \eqref{eq:5.12} and Chebyshev's inequality, we have 
    \begin{align*}
        &\PP(Y_m\notin B_L) \leq \sum_{k:m\geq n_k}\PP\left(\norm{\inner{\rho_m u_m}{\phi_{n_k}}}_{C_t^\kappa}\geq L^{\frac{1}{2r}}2^{\frac{k }{2r}}\norm{\phi_{n_k}}_{C_x^1} \right) \\ 
        &\quad \leq L^{-1}\sum_{k:m\geq n_k} 2^{-k}\norm{\phi_{n_k}}_{C_x^1}^{-2r}\sup_{m\geq n_k} \EE \left[\norm{\inner{\rho_m u_m}{\phi_{n_k}}}_{C_t^\kappa}^{2r}\right] \\ 
        &\quad \lesssim L^{-1},
    \end{align*} 
    which implies tightness of $\left\{ \LL[Y_m]:m\in \NN \right\}$ on $\XX_{Y}$.
\end{proof}
Consequently, we may apply Jakubowski's theorem, Theorem \ref{thm:Jakubowski--Skorokhod} as well as Proposition \ref{prop:m:tightness} to obtain the following.
\begin{proposition}
    \label{prop:m:repr}
    There exists a complete probability space $\tprobsp$ with $\XX$-valued Borel measurable random variables $(\trho_{0,m},\tu_{0,m},\trho_m,\tq_m,\tu_m,\tY_m,\tW_m),m\in\NN$ and $(\trho_0,\tu_0,\trho,\tq,\tu,\tY,\tW)$ such that (up to a subsequence):
    \begin{itemize}
        \item[(1)] the laws of $(\trho_{0,m},\tu_{0,m},\trho_m,\tq_m,\tu_m,\tY_m,\tW_m)$ and $(\hat{\rho}_{0,m},\hat{u}_{0,m},\rho_m,\rho_m u_m, u_m,Y_m,W_m)$ coincide on $\XX$;
        \item[(2)] $(\trho_{0,m},\tu_{0,m},\trho_m,\tq_m,\tu_m,\tY_m,\tW_m)$ converges in the topology of $\XX$ $\tPP$-a.s. to $(\trho_0,\tu_0,\trho,\tq,\tu,\tY,\tW)$, i.e., 
        \begin{align}
            \begin{aligned}
                &\trho_{0,m} \to \trho_0\quad \text{in } C(D), \\
                &\tu_{0,m}\to \tu_0 \quad \text{in } L^2(D), \\
                &\trho_m \to \trho\quad \text{in } L^p(0,T;W^{1,p}(D)), \\ 
                &\trho_m \weakarrow \trho\quad \text{in } W^{1,q}(0,T;L^q(D))\cap L^q(0,T;W^{2,q}(D)), \\ 
                &\trho_m \to \trho\quad \text{in } C_w([0,T];L^\Gamma(D)), \\ 
                &\tq_m \weakstararrow \tq \quad \text{in } L^\infty(0,T;L^{\frac{2\Gamma}{\Gamma + 1}}(D)), \\
                &\tu_m \weakarrow \tu \quad \text{in } L^2(0,T;W^{1,2}_{\rn }(D)), \\ 
                &\tY_m \to \tY \quad \text{in } \left[C([0,T])\right]^\NN, \\
                &\tW_m \to \tW\quad \text{in } C([0,T];\fU_0),
            \end{aligned}  \label{eq:5.17}
        \end{align}
    \end{itemize}
    $\tPP$-a.s. for some $p>2,q>1$, and $k\in\NN$.
\end{proposition}  
As stated in Remark \ref{rem:r.d.}, we may deduce that the filtration 
$$
    \tFFF_t:= \sigma \left(\sigma_t[\trho] \cup \sigma_t[\tu]\cup \sigma_t[\tW]\right),\quad t\in [0,T],
$$
is non-anticipating with respect to $\tW= \sum_{k=1}^{\infty} e_k \tW_k$, which is a cylindrical $(\tFFF_t)$-Wiener process on $\fU$. 
\begin{remark}
    \label{rem:m:initial_repr}
    In view of the convergence of initial conditions \eqref{eq:5.16} and \eqref{eq:5.17}, we obtain that 
    $$
        \Lambda = \LL[\trho_0,\trho_0 \tu_0]\quad \text{on } C^{2+\nu}(\oD) \times L^1(D).
    $$
\end{remark}
As remarked in Section \ref{sec:m:Uniform estimates}, the convergence of \eqref{eq:5.17} allows us to derive the convergence of $\trho_m \tu_m$ in $C_w([0,T];L^{\frac{2\Gamma }{\Gamma + 1}}(D))$.
\begin{lemma}
    \label{lem:m:q = rho u}
    We have $\tPP$-a.s.,
    $$
        \tq_m = \trho_m \tu_m,\quad \tY_m = \left(\inner{\trho_m \tu_m}{\phi_{n_k}}\right)_{k:m\geq n_k},\quad \tq = \trho \tu
    $$
    and 
    $$
        \trho_m \tu_m \to \trho \tu \quad \text{in } C_w([0,T];L^{\frac{2\Gamma }{\Gamma + 1}}(D)).
    $$
\end{lemma}
\begin{proof}
    The first assertion follows in the same way as Lemma \ref{lem:alpha(q = rho u)}. In \eqref{eq:5.17}, the weak convergence of $\trho_m \tu_m$ implies the uniform boundedness of $\trho_m \tu_m$, and the convergence of $\tY_m$ implies
    $$
        \int_D \trho_m \tu_m \cdot \phi_{n_k} dx \to \int_D \trho \tu\cdot \phi_{n_k} dx \quad \text{in }C([0,T]),\quad \text{for any } k\in\NN\  \tPP\text{-a.s.}
    $$
    Combining these observations with the density of $(\phi_{n_k})_{k\in\NN}$ in $L^{\frac{2\Gamma }{\Gamma - 1}}(D)$ yields 
    $$
        \int_D \trho_m \tu_m \cdot \phi dx \to \int_D \trho \tu\cdot \phi dx \quad \text{in }C([0,T]),\quad \text{for all }\phi\in L^{\frac{2\Gamma }{\Gamma - 1}}(D)\  \tPP\text{-a.s.,}
    $$
    which implies the second assertion.
\end{proof}
By Lemma \ref{lem:m:q = rho u}, it is now possible to carry out a discussion as in Section \ref{sec:alpha(Asymptotic limit)}.
As in Lemma \ref{lem:alpha(conv rho u otimes u)}, we obtain the following convergence.
\begin{lemma}
    \label{lem:m:conv rho u otimes u}
    We have $\tPP$-a.s.,
    \begin{align}
        \trho_m \tu_m\otimes \tu_m \to \trho \tu\otimes \tu\quad \text{in}\ L^1(0,T;L^1(D)). \label{eq:5.18}
    \end{align}
\end{lemma}
\begin{proof}
    Similarly to the proof of Lemma \ref{lem:alpha(conv rho u otimes u)}, it suffices to prove that $\tPP$-a.s.,
    \begin{align}
        \sqrt{\trho_m}\tu_m \to  \sqrt{\trho}\tu \quad \text{in}\ L^2(0,T;L^2(D)).    \label{eq:5.19}
    \end{align}
    To this end, we will prove that $\tPP$-a.s.,
    \begin{align*}
        &\norm{\sqrt{\trho_m}\tu_m}_{L_t^2 L_x^2} \to \norm{\sqrt{\trho}\tu}_{L_t^2 L_x^2}\quad \text{in}\ \RR.
    \end{align*}
    By Lemma \ref{lem:m:q = rho u}, and the compact embedding 
    $$
        L^{\frac{2\Gamma}{\Gamma +1}}(D) \cptarrow W^{-1,2}(D),
    $$
    we have $\tPP$-a.s. 
    $$
        \trho_m\tu_m\to \trho \tu \quad \text{in } L^2(0,T;W^{-1,2}(D)).
    $$
    Thus, combining this with the convergence of $\tu_m$ in \eqref{eq:5.17} yields 
    \begin{align*}
        &\norm{\sqrt{\trho_m}\tu_m}_{L_t^2 L_x^2}^2 = \intT \int_D \trho_m \tu_m \cdot \tu_m dxdt \\ 
        &\quad \to \intT \int_D \trho \tu \cdot \tu dx dt = \norm{\sqrt{\trho}\tu}_{L_t^2 L_x^2}^2\quad \tPP\text{-a.s.}
    \end{align*}
\end{proof}
Similarly to Proposition \ref{prop:eq_of_conti1}, by virtue of \eqref{eq:5.17}, $(\trho,\tu)$ satisfies the equation of continuity \eqref{eq:approx_CE_m} a.e. in $(0,T)\times D$, $\tPP$-a.s.
\begin{proposition}
    \label{prop:m:CE}
    $(\trho,\tu)$ satisfies the approximate equation of continuity \eqref{eq:approx_CE_m} a.e. in $(0,T)\times D$ $\tPP$-a.s.
\end{proposition}
Next, we pass to the limit $m\to \infty$ in the interior momentum equation \eqref{eq:approx_ME_alpha}. In view of the density of $\bigcup_{m=1}^\infty V_{m}$ in $W_{\rn}^{1,p}(D)$, the proof can be carried out by the same way as in the proof of Proposition \ref{prop:IME_alpha} with the replacement of $V_m$ by $C_c^\infty(D)$. Therefore, the following result follows. 
\begin{proposition}
    \label{prop:m:IME}
    $(\trho,\tu,\tW)$ satisfies the approximate momentum equation \eqref{eq:approx_ME_m} for all $\tau\in [0,T]$ and all $\bphi\in C_c^\infty(D)$ $\tPP$-a.s.
\end{proposition}
\begin{remark}
    \label{rem:m:conv of each term of ME}
    From the same reason as Remark \ref{rem:alpha(conv of each term of ME)}, the convergence to each term in the momentum equation \eqref{eq:approx_ME_m} actually holds for any $\bphi\in C^\infty(\oD)$, but the equality holds only for $\bphi\in C_c^\infty(D)$.
\end{remark}
Finally, we pass to the limit as $m\to \infty$ in the momentum and energy inequality \eqref{eq:approx_MEI_alpha}. 
For the momentum part, recalling the density of $\bigcup_{m=1}^\infty V_m$ in $W^{1,p}_{\rn}(D)$, we can pass to the limit in the same way as in Proposition \ref{prop:m:IME}.
For the energy part, we can pass to the limit in the same way as in Proposition \ref{prop:alpha:MEI}. 
Thus, we obtain the following result.
\begin{proposition}
    \label{prop:m:MEI}
    $(\trho,\tu,\tW)$ satisfies the approximate momentum and energy inequality \eqref{eq:approx_MEI_m} for all $\tau\in [0,T]$, all $\phi\in C_c^\infty(\ico{0}{\tau})$, and all $\bphi\in C^\infty(\oD)$ with $\bphi\cdot \rn |_{\pD} = 0$ $\tPP$-a.s.
\end{proposition}
The proof of Theorem \ref{thm:sol_m} is hereby complete.

\section{Vanishing viscosity limit}\label{sec:Vanishing viscosity limit}
Our next goal is to let $\varepsilon\to 0$ in the approximate system \eqref{eq:approx_CE_m}--\eqref{eq:approx_MEI_m}. A rigorous formulation reads as follows.
\begin{definition}
    \label{def:sol_ep}
	Let $\Lambda$ be a Borel probability measure on $L^1(D)\times L^1(D)$. Then $((\Omega,\FFF,(\FFF_t)_{t\geq 0},\PP),\rho,u,W)$ is called a \textit{dissipative martingale solution in $\delta$ layer} with the initial law $\Lambda$, if:
	\begin{itemize}
		\item[(1)] $(\Omega,\FFF,(\FFF_t)_{t\geq 0},\PP)$ is a stochastic basis with a complete right-continuous filtration;
		\item[(2)] $W$ is a cylindrical $(\FFF_t)$-Wiener process on $\fU$; 
		\item[(3)] the density $\rho$ is an $(\FFF_t)$-adapted stochastic process, and the velocity $u$ is an $(\FFF_t)$-adapted random distribution such that $\PP$-a.s.
		\begin{align*}
		   &\rho \geq 0,\quad \rho \in C_w([0,T];L^\Gamma(D)),\quad u\in L^2(0,T;W_{\rn }^{1,2}(D)),\quad \rho u \in C_w([0,T];L^{\frac{2\Gamma}{\Gamma + 1}}(D));   
		\end{align*}
		\item[(4)] there exists an $L^1(D)\times L^1(D)$-valued $\FFF_0$-measurable random variable $(\rho_0,u_0)$ such that $\rho_0 u_0\in L^1(D)$ $\PP$-a.s., and $\Lambda = \LL[\rho_0,\rho_0 u_0]$;
		\item[(5)] the equation of continuity 
		\begin{equation}
            \label{eq:approx_CE_ep}
			-\intT \pp_t\phi \int_D \rho \psi dxdt = \phi(0)\int_D \rho_0\psi dx + \intT \phi \int_D \rho u\cdot \nabla\psi dxdt
        \end{equation}
		holds for all $\phi\in C_c^\infty(\ico{0}{T})$ and all $\psi\in C_c^\infty\left(\RR^3\right)$ $\PP$-a.s.;
		\item[(6)] the approximate interior momentum equation 
		\begin{align}
            -\intT \pp_t\phi \int_D \rho u\cdot \bphi dx dt &= \phi(0)\int_D \rho_0 u_0 \cdot \bphi dx  + \intT\phi \int_D \left[(\rho u \otimes u):\nabla\bphi + p_\delta (\rho)\rdiv \bphi\right]dxdt  \notag \\ 
            &\quad - \intT \phi\int_D \SSS(\nabla u):\nabla\bphi dxdt + \intT \phi\int_D \GG(\rho,\rho u) \cdot \bphi dxdW  \label{eq:approx_ME_ep}
        \end{align}
		holds for all $\phi\in C_c^\infty(\ico{0}{T})$ and all $\bphi\in C_c^\infty(D)$ $\PP$-a.s.;
        \item[(7)] the approximate momentum and energy inequality 
        \begin{align}
            &\int_D \left[\frac 1 2 \rho_0 |u_0|^2 + P_\delta(\rho_0)\right]dx - \int_D \left[\frac 1 2 \rho |u|^2 + P_\delta(\rho)\right](\tau)dx - \inttau \pp_t\phi \int_D \rho u\cdot \bphi dx - \phi(0)\int_D \rho_0 u_0 \cdot \bphi dx \notag\\ 
            &\quad - \inttau \phi \int_D \left[(\rho u \otimes u):\nabla\bphi + p_\delta (\rho)\rdiv \bphi\right]dxdt + \inttau \int_D \SSS(\nabla u):\nabla(\phi\bphi- u) dxdt \notag \\ 
            &\quad +\frac 1 2 \sum_{k=1}^{\infty} \inttau \int_D \rho^{-1} \left|G_{k}(\rho,\rho u) \right|^2 dx dt - \inttau \int_D \GG(\rho,\rho u) \cdot (\phi\bphi - u) dx dW  \notag \\ 
            &\quad + \inttau \int_{\pD} g|\phi\bphi| - g |u| d\Gamma dt\geq 0 \label{eq:approx_MEI_ep}
        \end{align}
        holds for all $\tau\in [0,T]$, all $\phi\in C_c^\infty(\ico{0}{\tau})$, and all $\bphi\in C^\infty(\oD)$ with $\bphi\cdot \rn|_{\pD} = 0$ $\PP$-a.s.
    \end{itemize}
\end{definition}
\begin{theorem}
    \label{thm:sol_ep}
    Let $\Lambda$ be a Borel probability measure on $L^1(D)\times L^1(D)$ such that 
    \begin{equation}
        \Lambda\left\{ \rho>0 \right\} = 1,\quad \Lambda \left\{ \urho\leq \int_D \rho dx \leq \orho \right\}=1, \label{eq:6.4}
    \end{equation}
    for some deterministic constants $\urho,\orho>0$ and 
    \begin{align}
        \int_{L^1_x\times L_x^1} \left|\int_D \left[\frac 1 2 \frac{|q|^2}{\rho} + P_\delta(\rho)\right]dx \right|^{r}d \Lambda (\rho,q) <\infty, \label{eq:6.5}
    \end{align}
    for some $r \geq 4$. Then there exists a dissipative martingale solution in the sense of Definition \ref{def:sol_ep}.
\end{theorem}
\begin{remark}
    \label{rem:ep:integrability_of_sol}
    As in Remark \ref{rem:m(integrability_of_sol)}, the solution constructed in this section belongs to the following class 
    \begin{align*}
        &\EE\left[\sup_{t\in [0,T]}\left|\int_D\left[\frac 1 2 \rho |u|^2 + P_\delta(\rho)\right]dx \right|^{r}\right] <\infty, \\ 
        &\EE\left[\left|\intT \int_D \SSS(\nabla u):\nabla u dxdt\right|^{r}\right] + \EE\left[\left|\intT \int_{\pD} g |u| d\Gamma dt \right|^{r}\right]<\infty,
    \end{align*}
    with the same $r\geq 4$ as in \eqref{eq:4.5}, and thus we have 
    \begin{align*}
        &\EE \left[\sup_{t\in [0,T]} \norm{\rho(t)}_{L_x^\Gamma}^{\Gamma r}\right] + \EE \left[\sup_{t\in [0,T]} \norm{\rho |u|^2(t)}_{L_x^1}^{r}\right] + \EE \left[\sup_{t\in [0,T]} \norm{\rho u(t)}_{L_x^{\frac{2\Gamma}{\Gamma +1}}}^{\frac{2\Gamma}{\Gamma + 1}r}\right] + \EE\left[\norm{u}_{L_t^2 W^{1,2}_x}^{r}\right]<\infty.
    \end{align*}
    Note that the exponent $r\geq 4$ in Theorem \ref{thm:sol_ep} corresponds to $2r\geq 4$ in Theorem \ref{thm:sol_m}.
\end{remark}
\begin{remark}
    \label{rem:ep:renormalized CE}
    It follows from the DiPerna--Lions theory (see \cite[Section 11.19]{FN17} and Section \ref{sec:ep:Strong convergence of the densities}) that the martingale solution constructed in Theorem \ref{thm:sol_ep} satisfies $\PP$-a.s. the renormalized equation of continuity
    \begin{align*}
        -\intT \pp_t \phi \int_D b(\rho)\psi dxdt &= \phi(0) \int_D b(\rho_0)\psi dx + \intT \phi\int_D b(\rho)u\cdot \nabla \psi dxdt \\ 
        &\quad - \intT \phi \int_D \left(b'(\rho) \rho - b(\rho) \right) \rdiv u \psi dx dt,
    \end{align*}
    for all $\phi\in C_c^\infty(\ico{0}{T})$, all $\psi\in C_c^\infty(\RR^3)$, and any $b\in C(\ico{0}{\infty}) \cap C^1((0,\infty))$ such that 
    \begin{align*}
        &\lim_{s\to 0 +} (sb'(s)-b(s))\in \RR, \\ 
        &|b'(s)| \leq c s^\lambda \quad \text{if } s\in (1,\infty)\quad \text{for a certain }\lambda \leq \frac{\Gamma}{2}-1. 
    \end{align*}
    Moreover, in view of \cite[Lemma 11.14]{FN17} and $\Gamma \geq 6$, $\rho$ satisfies 
    \begin{align*}
        &\rho \in C([0,T];L^p(D))\quad \PP\text{-a.s.,}  
    \end{align*}
    for any $1\leq p<\Gamma$.
\end{remark}
\subsection{Uniform estimates}\label{sec:ep:Uniform estimates}
In this section, the approximation parameter $\delta\in (0,1)$ is kept fixed. We call uniform estimate that are independent of $\varepsilon$ but may depend on $\delta$ and $T>0$.
Let $(\rho,u)$ denote the solution constructed in Theorem \ref{thm:sol_m}. As in Section \ref{sec:m:Uniform estimates}, choosing $\phi=0$ in the momentum and energy inequality \eqref{eq:approx_MEI_m} yields the standard energy inequality
\begin{align*}
    &\int_D \left[\frac 1 2 \rho |u|^2 + P_\delta(\rho)\right](\tau)dx \\ 
    &\quad + \inttau \int_D \left[\SSS(\nabla u):\nabla u + \varepsilon \rho |\nabla u|^2 + \varepsilon P_\delta''(\rho)|\nabla \rho|^2\right]dxdt +\inttau \int_{\pD} g |u| d\Gamma dt \notag\\ 
    &\quad \leq \int_D \left[\frac 1 2 \rho_0 |u_0|^2 + P_\delta(\rho_0)\right]dx + \frac 1 2 \sum_{k=1}^{\infty} \inttau \int_D \rho \left|F_{k,\varepsilon}(\rho,u) \right|^2 dx dt + \inttau \int_D \GG_\varepsilon(\rho,\rho u) \cdot u dx dW   
\end{align*}
for all $\tau\in [0,T]$ $\PP$-a.s., and using Gronwall's lemma together with Remark \ref{rem:m(integrability_of_sol)} yields the uniform estimates:
\begin{align}
    &\EE\left[\left|\sup_{\tau\in [0,T]} \int_D \left[\frac 1 2 \rho |u|^2 + P_\delta(\rho)\right](\tau)dx \right|^r\right] \notag\\ 
    &\quad  +\EE\left[\left|\intT \int_D \left[\SSS(\nabla u):\nabla u + \varepsilon \rho |\nabla u|^2 + \varepsilon P_\delta''(\rho)|\nabla \rho|^2\right]dxdt \right|^r\right]\notag \\ 
    &\quad  + \EE\left[\left|\intT \int_{\pD} g |u| d\Gamma dt \right|^{r}\right] \lesssim  c(r), \label{eq:6.6} \\ 
    &\EE\left[\norm{\rho}_{L_t^\infty L_x^\Gamma}^{\Gamma r}\right] + \EE \left[\norm{\sqrt{\varepsilon}\nabla \rho}_{L_t^2 L_x^2}^{2r}\right] \lesssim c(r), \label{eq:6.7}\\ 
    & \EE\left[\norm{\rho |u|^2}_{L_t^\infty L_x^1}^{r}\right] + \EE\left[\norm{\rho u}_{L_t^\infty L_x^{\frac{2\Gamma}{\Gamma +1}}}^{\frac{2\Gamma}{\Gamma +1}r}\right] \lesssim c(r), \label{eq:6.8} \\
    &\EE\left[\left| \intT \norm{u}_{W_x^{1,2}}^2 dt \right|^{r} \right] \lesssim  c(2r) \label{eq:6.9}
\end{align}
uniformly in $\varepsilon$, where 
$$
    c(r) = \EE\left[\left|\int_D \left[\frac 1 2 \rho_0 |u_0|^2 + P_\delta(\rho_0)\right]dx \right|^r\right] + 1,\quad r\geq 2.
$$
Finally, it follows from \eqref{eq:approx_CE_m} that 
\begin{align*}
    \norm{\rho(\tau)}_{L_x^1} = \norm{\rho_0}_{L_x^1} \leq \orho,\quad \tau\in [0,T].
\end{align*}
The above estimates are not sufficient to control the pressure term proportional to $\rho^\Gamma$, and we need to find a uniform bound for the density in $L^q(D)$ for some $q>\Gamma$. 
To this end, we consider the problem 
\begin{align*}
    \begin{cases}
    \rdiv u = f &\text{in }D, \\
    u=0 &\text{on }\pD
    \end{cases}
\end{align*}
for a given function $f$. This problem admits many solutions that may be constructed in different manners, but here we use a solution map $\Bog$ constructed by Bogovskii \cite{Bog80}. The properties of $\Bog$ are recalled in Appendix \ref{sec:Bogovskii}. In accordance with \eqref{eq:approx_CE_m}, $\PP$-a.s.,
$$
    \pp_t \Phi = - \Bog \left(\rdiv \left(\rho u -\varepsilon \nabla \rho\right)\right),
$$
where $\Phi = \Bog (\rho -(\rho)_D)$ and $(\rho)_D =  \frac{1}{|D|} \int_D \rho dx$. Formally test the interior momentum equation \eqref{eq:approx_ME_m} by $\Phi$ to derive the pressure estimate. To this end, we introduce the standard mollifier $\theta_c$ localized at scale $c$, and for $v\in \DD'(\RR^3)$, we define its regularization 
$$
    [v]_c (x) := v\ast \theta_c(x) = \ev{v,\theta_c(x-\cdot)}.
$$ 
Extending $(\rho,u)$ outside $D$, we may define 
\begin{align*}
    &g_c(\tau,x) = \int_{\RR^d} \left[\rho u \otimes u + p_\delta (\rho) \II - \SSS(\nabla u) - \varepsilon \nabla (\rho u) \right] (\tau,y)\nabla\theta_c(x-y) dy, \quad \text{for all } x\in \RR^3,\tau\in [0,T].\\ 
\end{align*}
Since $\theta_c(x-\cdot)\in C_c^\infty(D)$ for all $c>0$ and all $x\in D_c = \left\{ x\in D: d(x,\pD)>c \right\}$, we have $\PP$-a.s., 
$$
    [\rho u]_c (\tau,x) = [\rho_0 u_0]_c(x) + \inttau g_c(t,x)dt + \inttau [\GG_\varepsilon(\rho,\rho u)]_c(t,x)dW,\quad \text{for all } x\in D_c,\tau\in [0,T].
$$
Here, note that by approximating $1_{[0,\tau]}$ in the momentum equation \eqref{eq:approx_ME_m}, the equation \eqref{eq:approx_ME_m} can be transformed into an integral equation with respect to time.
Similarly, we have 
$$
    d[\Phi]_c = - [\Bog \left(\rdiv \left(\rho u -\varepsilon \nabla \rho\right)\right)]_c dt,
$$ 
pointwise in $D_c$. Applying It\^{o}'s product rule, we obtain for each $x\in D_c$, $\PP$-a.s., 
\begin{align*}
    &d([\rho u]_c \cdot [\Phi]_c)(x) = [\rho u]_c(x) \cdot d [\Phi]_c (x) + [\Phi]_c (x)\cdot  d[\rho u]_c(x).
\end{align*}
Noting the spatial regularity of the appearing terms, integrating over $D_c$ and applying the stochastic Fubini theorem yields $\PP$-a.s.
\begin{align}
    \label{eq:6.10}
    &d\int_{D_c} [\rho u]_c \cdot [\Phi]_c dx = \int_{D_c}[\rho u]_c \cdot [\pp_t \Phi]_c dx dt + \int_{D_c} [\Phi]_c \cdot g_c dx dt + \int_{D_c} [\Phi]_c \cdot [\GG_\varepsilon(\rho,\rho u)]_c dx dW. 
\end{align}
Rewrite the second term on the right-hand side as follows:
\begin{align*}
    &\inttau \int_{D_c} [\Phi]_c \cdot g_c dx dt = \inttau \int_{D_c} \left([\Phi]_c -\Phi\right) \cdot g_c dx dt + \inttau \int_{D_c} \Phi \cdot g_c dx dt.
\end{align*}
The first term on the right-hand side vanishes as $c\to 0$, and the second term can be transformed as follows:
\begin{align*}
    \inttau \int_{D_c} \Phi \cdot g_c dx dt &= - \inttau \int_{D_c}\Phi(t,x) \cdot \rdiv \left(\int_{\RR^d} \left[\rho u \otimes u + p_\delta (\rho) \II - \SSS(\nabla u) - \varepsilon \nabla (\rho u) \right] (t,y)\theta_c(x-y) dy\right) dx dt \\ 
    &= \inttau \int_{D_c}\nabla \Phi:\left[\rho u \otimes u + p_\delta (\rho) \II - \SSS(\nabla u) - \varepsilon \nabla (\rho u)\right]_c dx  dt\\ 
    &\quad - \inttau \int_{\pD_c} \Phi\cdot \left( \left[\rho u \otimes u + p_\delta (\rho) \II - \SSS(\nabla u) - \varepsilon \nabla (\rho u)\right]_c \rn \right) d\Gamma dt,
\end{align*}
where the outward unit normal vector on $\pD_c$ is also denoted by the same symbol $\rn$. Since $\Phi$ satisfies Dirichlet boundary conditions on $\pD$, applying the Green's formula to the domain $D\backslash \overline{D_c}$ yields
\begin{align*}
    - &\inttau \int_{\pD_c} \Phi\cdot \left( \left[\rho u \otimes u + p_\delta (\rho) \II - \SSS(\nabla u) - \varepsilon \nabla (\rho u)\right]_c \rn \right) d\Gamma dt\\ 
    &\quad = - \inttau \int_{\pD_c} \Phi\cdot \left( \left[\rho u \otimes u + p_\delta (\rho) \II - \SSS(\nabla u) - \varepsilon \nabla (\rho u)\right]_c \rn \right) d\Gamma dt \\ 
    &\quad \quad + \inttau \int_{\pD} \Phi\cdot \left( \left[\rho u \otimes u + p_\delta (\rho) \II - \SSS(\nabla u) - \varepsilon \nabla (\rho u)\right]_c \rn \right) d\Gamma dt \\ 
    &\quad = \inttau \int_{D\backslash \overline{D_c}}\nabla \Phi:\left[\rho u \otimes u + p_\delta (\rho) \II - \SSS(\nabla u) - \varepsilon \nabla (\rho u)\right]_c dx  dt,
\end{align*}
which vanishes as $c\to 0$. Therefore, noting that $\nabla \Phi:p_\delta(\rho)\II = (\rho- (\rho)_D)p_\delta(\rho)$, and letting $c\to 0$ in \eqref{eq:6.10}, we obtain that $\PP$-a.s.,
\begin{align}
    \inttau \int_D p_\delta(\rho)(\rho - (\rho)_D) dx dt &= \left[\int_D \rho u \cdot \Phi dx \right]_{t=0}^{t=\tau} - \inttau \int_D \left[\rho u \otimes u - \SSS(\nabla u) - \varepsilon \nabla(\rho u)\right]:\nabla \Phi dx dt \notag \\ 
    &\quad + \inttau \int_D \rho u\cdot \pp_t \Phi dx dt - \inttau \int_D \GG_\varepsilon(\rho,\rho u)\cdot \Phi dx dW\notag  \\ 
    &=: \sum_{j=1}^{4} I_j.  \label{eq:6.11}
\end{align}
By estimating each term on the right-hand side in \eqref{eq:6.11}, we will find the desired uniform bound:
\begin{align}
    \EE\left[\left|\intT\int_D \left(p(\rho) + \delta \rho^\Gamma \right)\rho dx dt \right|^r\right] \lesssim c(2r). \label{eq:6.12}
\end{align}
To this end, first, using $W^{1,\Gamma}(D)\hookrightarrow L^\infty(D)$ and properties of $\Bog$ (see Theorem \ref{thm:Bog}) yields 
\begin{align}
    \label{eq:6.13}
    &\norm{\Phi}_{L^\infty(D)} \lesssim \norm{\Phi}_{W^{1,\Gamma}(D)} \lesssim \norm{\rho}_{L^\Gamma (D)}.
\end{align}
Thus, we have 
\begin{align*}
    \left|I_1 \right|   &\lesssim \sup_{\tau\in [0,T]}\norm{\sqrt{\rho}}_{L_x^2} \norm{\sqrt{\rho} u}_{L_x^2}\norm{\Phi}_{L_x^\infty} \\ 
    &\lesssim \sup_{\tau\in [0,T]} \norm{\rho}_{L_x^1}^{1/2} \norm{\rho |u|^2}_{L_x^1}^{1/2} \norm{\rho}_{L_x^\Gamma}.
\end{align*}
This implies 
$$
    \EE\left[\left|I_1 \right|^r\right]\lesssim c(r).
$$
By \eqref{eq:6.13}, we have 
\begin{align*}
    \norm{\rho u\otimes u:\nabla \Phi}_{L_t^1 L_x^1}  &\lesssim \sup_{\tau\in [0,T]}\norm{\rho}_{L_x^\Gamma}\intT \norm{\rho |u|^2}_{L_x^{\Gamma'}}dt \\ 
    &\lesssim \sup_{\tau\in [0,T]} \norm{\rho}_{L_x^\Gamma}^2 \intT \norm{u}_{W_x^{1,2}}^2 dt,\\ 
    \norm{\SSS(\nabla u):\nabla \Phi}_{L_t^1 L_x^1} &\lesssim \sup_{\tau\in [0,T]}\norm{\rho}_{L_x^\Gamma}^2 + \norm{u}_{W_x^{1,2}}^2, \\  
    \norm{\varepsilon \nabla(\rho u):\nabla \Phi}_{L_t^1 L_x^1} &\lesssim \sup_{\tau\in [0,T]}\norm{\rho}_{L_x^\Gamma}\left(\norm{\varepsilon \nabla \rho \otimes u}_{L_t^1 L_x^{\Gamma'}} + \norm{\varepsilon \rho \nabla u}_{L_x^1 L_x^{\Gamma'}}\right) \\ 
    &\lesssim \sup_{\tau\in [0,T]}\norm{\rho}_{L_x^\Gamma}\left(\norm{\varepsilon \nabla \rho}_{L^2_{t,x}}\norm{u}_{L_t^2 W_x^{1,2}} + \norm{\varepsilon \rho}_{L_t^\infty L_x^\Gamma} \norm{u}_{L_x^2 W_x^{1,2}}\right),
\end{align*}
where $\Gamma'$ is the H\"older conjugate exponent of $\Gamma$. Thus, 
$$
    \EE\left[\left|I_2 \right|^r\right]\lesssim c(2r).
$$
Next, noting that 
$$
    \rho u\cdot \Bog\left(\rdiv \left(\rho u - \varepsilon \nabla \rho\right)\right) = \rho u \cdot \Bog\left(\rdiv \left(\rho u\right)\right) - \varepsilon \rho u \cdot \Bog \left(\rdiv\nabla\rho\right),
$$
Theorem \ref{thm:Bog} (iii), and $W^{1,2}(D)\hookrightarrow L^6(D)$, each term on the right-hand side can be estimated as follows:
\begin{align*}
    &\norm{\rho u \cdot \Bog\left(\rdiv \left(\rho u\right)\right)}_{L^1_{t,x}} \lesssim \norm{\rho u}_{L^2_{t,x}}^2 \lesssim \sup_{\tau\in [0,T]}\norm{\rho}_{L_x^\Gamma}^2 \norm{u}_{L_t^2 W_x^{1,2}}^2 \lesssim \sup_{\tau\in [0,T]}\norm{\rho}_{L_x^\Gamma}^4 + \left(\intT \norm{u}_{W_x^{1,2}}^2dt\right)^2,
\end{align*}
and 
\begin{align*}
    &\norm{\varepsilon \rho u \cdot \Bog \left(\rdiv\nabla\rho\right)}_{L^{1}_{t,x}} \leq \norm{\rho u}_{L^2_{t,x}}\norm{\varepsilon \nabla \rho}_{L^2_{t,x}}\lesssim  \sup_{\tau\in [0,T]}\norm{\rho}_{L_x^\Gamma}^4 + \left(\intT \norm{u}_{W_x^{1,2}}^2dt\right)^2 + \varepsilon \norm{\sqrt{\varepsilon}\nabla \rho}_{L^2_{t,x}}^2.
\end{align*}
Therefore, we have 
$$
    \EE\left[\left|I_3 \right|^r\right]\lesssim c(2r).    
$$
Finally, applying the Burkholder--Davis--Gundy inequality and \eqref{eq:6.13} yield 
\begin{align*}
    \EE\left[\left|I_4 \right|^r\right] &\lesssim \EE\left[\left(\inttau \sum_{k=1}^{\infty} \left|\int_D \rho F_{k,\varepsilon}(\rho,u)\cdot \Phi dx \right|^2 dt\right)^{r/2}\right]  \\ 
    &\leq \EE\left[\left(\inttau \sum_{k=1}^{\infty} \left[\int_D \rho\left| F_{k}(\rho,u)\right| dx\right]^2 \norm{\Phi}_{L^\infty_x}^2 dt\right)^{r/2}\right]  \\
    &\lesssim \EE\left[\left(\sup_{\tau\in [0,T]}\left[\int_D \left[\rho + \rho |u|^2\right] dx\right]^2 \sup_{\tau\in [0,T]}\norm{\rho}_{L^\Gamma_x}^2\right)^{r/2}\right] \\ 
    &\lesssim c(2r).
\end{align*}

\subsection{Asymptotic limit}\label{sec:ep:Asymptotic limit}
Assume that $\Lambda$ is the initial law given by Theorem \ref{thm:sol_ep}, that is, $\Lambda$ satisfies \eqref{eq:6.4} and \eqref{eq:6.5} with $r\geq 4$. We consider a random variable $(\rho_0,q_0)$ with law $\Lambda$ on some probability space $\probsp$. 
Then one can find random variables $\rho_{0,\varepsilon}$ with values in $C^{2+\kappa}(\oD)$, for some $\kappa>0$, such that $\PP$-a.s. 
$$
    0<\varepsilon\leq \rho_{0,\varepsilon}\leq \frac 1 \varepsilon, \quad \frac{\urho}{2}\leq \int_D \rho_{0,\varepsilon}dx \leq 2\orho, \quad \nabla \rho_{0,\varepsilon}\cdot \rn|_{\pD} =0,\quad \norm{\rho_{0,\varepsilon}}_{C^{2+\kappa}_x} \leq c(\orho,\varepsilon)
$$
as well as 
\begin{align}
    \rho_{0,\varepsilon}\to \rho_0\quad \text{in } L^p(\Omega;L^\Gamma(D)),\quad \forall p\in [1,\Gamma r] \label{eq:6.14}
\end{align}
(see \cite[Section 4.4.3]{BFHbook18} and \cite[Section 4]{FN17}). Next setting 
$$
    \hat{q}_{0,\varepsilon} = \begin{cases}
    q_0\sqrt{\frac{\rho_{0,\varepsilon}}{\rho_0}} &\text{if }\rho_0>0, \\
    0 &\text{if } \rho_0=0,
    \end{cases}
$$
it follows from the assumptions on $\Lambda$ that 
$$
    \left\{ \frac{|\hat{q}_{0,\varepsilon}|^2}{\rho_{0,\varepsilon}}:\varepsilon\in (0,1) \right\} \quad \text{is bounded in}\  L^p(\Omega;L^1(D)),\quad \forall p\in [1,r].
$$
Moreover, by mollification we choose random variables $h_\varepsilon$ with values in $C^2(\oD)$ such that 
$$
    \frac{\hat{q}_{0,\varepsilon}}{\sqrt{\rho_{0,\varepsilon}}} - h_\varepsilon\to 0 \quad \text{in } L^p(\Omega;L^2(D)),\quad \forall p\in [1,2r]. 
$$
Let $q_{0,\varepsilon} = h_\varepsilon \sqrt{\rho_{0,\varepsilon}}$. Then 
$$
    \left\{ \frac{|q_{0,\varepsilon}|^2}{\rho_{0,\varepsilon}}:\varepsilon\in (0,1) \right\} \quad \text{is bounded in}\ L^p(\Omega;L^1(D)),\quad \forall p\in [1,r],
$$
and 
\begin{align}
    q_{0,\varepsilon} &\to q_0 \quad \text{in } L^p(\Omega;L^1(D)),\quad \forall p\in [1,r], \label{eq:6.15} \\ 
    \frac{q_{0,\varepsilon}}{\sqrt{\rho_{0,\varepsilon}}} &\to \frac{q_{0}}{\sqrt{\rho_{0}}} \quad \text{in } L^p(\Omega;L^2(D)),\quad \forall p\in [1,2r]. \label{eq:6.16}
\end{align}
Therefore, applying Theorem \ref{thm:sol_m} to the laws $\PP\circ (\rho_{0,\varepsilon},q_{0,\varepsilon})^{-1}$ on $C^{2+\kappa}(\oD) \times L^1(D)$ yields for every $\varepsilon\in (0,1)$ a multiplet 
$$
    ((\Omega^\varepsilon,\FFF^\varepsilon,(\FFF_t)^\varepsilon,\PP^\varepsilon),\overline{\rho}_{0,\varepsilon},\overline{u}_{0,\varepsilon},\rho_\varepsilon,u_\varepsilon,W_\varepsilon),
$$
which is a weak martingale solution in the sense of Definition \ref{def:sol_m}. Moreover, in view of \eqref{eq:6.14}--\eqref{eq:6.16}, the laws $\Lambda_\varepsilon := \PP^\varepsilon \circ (\overline{\rho}_{0,\varepsilon},\overline{\rho}_{0,\varepsilon} \overline{u}_{0,\varepsilon})^{-1}$ on $L^1(D)\times L^1(D)$ satisfy \eqref{eq:6.4} and \eqref{eq:6.5} uniformly in $\varepsilon$, and 
\begin{align}
    \Lambda_\varepsilon \to \Lambda\quad \text{weakly on } L^1(D)\times L^1(D). \label{eq:6.17}
\end{align}
As in Section \ref{sec:m:Asymptotic limit}, we may assume without loss of generality that 
$$
    (\Omega^\varepsilon,\FFF^\varepsilon,\PP^\varepsilon) = ([0,1],\overline{\fB([0,1])},\fL),\quad \varepsilon\in (0,1)
$$
and that 
$$
    \FFF^\varepsilon_t = \sigma \left(\sigma_t[\rho_\varepsilon]\cup \sigma_t [u_\varepsilon]\cup \sigma_t[W_\varepsilon]\right),\quad t\in [0,T].
$$
Next, in the stochastic compactness method based on Jakubowski's theorem stated in Theorem \ref{thm:Jakubowski--Skorokhod}, it is convenient to include the energy and the Young measure associated to $(\rho_\varepsilon,u_\varepsilon,\nabla u_\varepsilon)$. More precisely, we define 
$$
    E_\varepsilon = E(\rho_\varepsilon,u_\varepsilon) := \frac12 \rho_\varepsilon \left|u_\varepsilon \right|^2 + P_\delta(\rho_\varepsilon).
$$
As we have shown in Section \ref{sec:ep:Uniform estimates}, the expectation of the energy $E_\varepsilon$ is bounded in the space 
$$
    L^\infty(0,T;L^1(D))\hookrightarrow L^\infty(0,T;\MMM_b(D)) \simeq \left(L^1(0,T;C_0(D))\right)^\ast.
$$
Since $L^1(0,T;C_0(D))$ is a separable Banach space, $[L^\infty(0,T;\MMM_b(D)), w^\ast]$ is a sub-Polish space. Next, the canonical Young measure $\nu_\varepsilon$ associated to $(\rho_\varepsilon,u_\varepsilon,\nabla u_\varepsilon)$ is defined as a (random) weakly-$\ast$ measurable mapping 
$$
    \nu_\varepsilon:(0,T)\times D\to \PPP(\RR\times \RR^3\times \RR^{3\times 3}) \simeq \PPP(\RR^{13})
$$
given by 
$$
    \nu_{\varepsilon,t,x} (\cdot) = \delta_{(\rho_\varepsilon,u_\varepsilon,\nabla u_\varepsilon)(t,x)} (\cdot),
$$
where a mapping $\nu:(0,T)\times D\to \PPP(\RR^N)$ is said to be weakly-$\ast$ measurable if for any $\phi\in C_b(\RR^N)$, the mapping 
$$
    (0,T)\times D\ni (t,x)\mapsto \nu(\phi)\in \RR
$$
is measurable, and in this case, $\nu$ is called a Young measure.
In view of the discussion in \cite[Section 2.8]{BFHbook18}, $\nu_\varepsilon$ can be regarded as a random variable taking values in the space of Young measures which we denoted by 
$$
    (L_{w^\ast}^\infty((0,T)\times D;\PPP(\RR^{13})),w^\ast),
$$
where the topology on this space is determined by functionals 
$$
    (L_{w^\ast}^\infty((0,T)\times D;\PPP(\RR^{13})),w^\ast)\ni \nu\mapsto \intT\int_D \psi(t,x)\int_{\RR^{13}} \phi(\xi)d\nu_{t,x}(\xi)dxdt\in \RR,
$$
where $\psi\in L^1((0,T)\times D),\phi\in C_b(\RR^{13})$. Since this topology is finer than the weak-$\ast$ topology on $L^\infty((0,T)\times D;\MMM_b(\RR^{13}))$, this space belongs to the class of sub-Polish spaces. 

The corresponding path space is as follows:
$$
    \XX = \XX_{\rho_0}\times \XX_{q_0}\times \XX_{\frac{q_0}{\sqrt{\rho_0}}}\times \XX_\rho \times \XX_{\rho u}\times \XX_{u}\times \XX_W \times \XX_E \times \XX_\nu, 
$$
where 
\begin{align*}
    &\XX_{\rho_0} = L^\Gamma(D),\quad \XX_{q_0} = L^1(D),\quad \XX_{\frac{q_0}{\sqrt{\rho_0}}} = L^2(D), \\ 
    &\XX_\rho = \left[L^{\Gamma+1}((0,T)\times D),w\right]\cap C_w([0,T];L^\Gamma(D)), \\ 
    &\XX_{\rho u} = C_w([0,T];L^{\frac{2\Gamma }{\Gamma +1}}(D)) \cap C([0,T];W^{-k,2}(D)), \\ 
    &\XX_u = \left[L^2(0,T;W^{1,2}_{\rn }(D)),w\right], \\ 
    &\XX_W = C([0,T];\fU_0), \\ 
    &\XX_E = \left[L^\infty(0,T;\MMM_b(D)),w^\ast\right], \\ 
    &\XX_\nu = (L_{w^\ast}^\infty((0,T)\times D;\PPP(\RR^{13})),w^\ast),
\end{align*}
for certain $k\in \NN$.
\begin{proposition}
    \label{prop:ep:tightness}
    The family of probability measures
    $$
        \left\{ \LL\left[\overline{\rho}_{0,\varepsilon}, \overline{\rho}_{0,\varepsilon} \overline{u}_{0,\varepsilon},\frac{\overline{\rho}_{0,\varepsilon} \overline{u}_{0,\varepsilon}}{\sqrt{\overline{\rho}_{0,\varepsilon}}},\rho_\varepsilon,,\rho_\varepsilon u_\varepsilon,u_\varepsilon,W_\varepsilon,E(\rho_\varepsilon,u_\varepsilon),\nu_\varepsilon \right]:\varepsilon\in (0,1) \right\}
    $$
    is tight on $\XX$.
\end{proposition}
\begin{proof}
    Tightness of the initial laws 
    $$
        \left\{ \LL\left[\overline{\rho}_{0,\varepsilon}, \overline{\rho}_{0,\varepsilon} \overline{u}_{0,\varepsilon},\frac{\overline{\rho}_{0,\varepsilon} \overline{u}_{0,\varepsilon}}{\sqrt{\overline{\rho}_{0,\varepsilon}}}\right]:\varepsilon\in(0,1) \right\}
    $$
    immediately follows from \eqref{eq:6.14}--\eqref{eq:6.16} and Prokhorov's theorem. Tightness of the set 
    $$
        \left\{ \LL\left[\rho_\varepsilon,u_\varepsilon,W_\varepsilon\right] :\varepsilon\in (0,1) \right\}
    $$
    follows by the same arguments as in Section \ref{sec:alpha(Asymptotic limit)}. To prove tightness of the set 
    $$
        \left\{ \LL\left[\rho_\varepsilon u_\varepsilon\right]:\varepsilon\in (0,1) \right\},
    $$
    we use the following compact embedding (see Theorem \ref{thm:weak_conti_embedding}):
    $$
        L^\infty(0,T;L^{\frac{2\Gamma }{\Gamma + 1}}(D)) \cap C^\alpha([0,T];W^{-k,2}(D))\cptarrow C_w([0,T];L^{\frac{2\Gamma }{\Gamma + 1 }}(D)),\quad \alpha>0.
    $$
    Here, to prove tightness in $C([0,T];W^{-k,2}(D))$, i.e., boundedness in $C^\alpha([0,T];W^{-k,2}(D))$ for some $\alpha>0$, we proceed as in Section \ref{sec:alpha:Uniform estimates}. Note that the estimates as in Lemma \ref{lem:m:unif_est_others} hold, that is, by the uniform estimates \eqref{eq:6.6}--\eqref{eq:6.9}, we have 
    \begin{align*}
        &\EE\left[\norm{\rho u\otimes u }_{L_t^2 L_x^{\frac{6\Gamma}{4\Gamma + 3}}}^{r/2}\right] + \EE\left[\norm{\SSS(\nabla u)}_{L_t^2 L_x^2}^{r}\right] \lesssim c(r),\\ 
        &\EE\left[\norm{p_\delta(\rho)}_{L_t^2 L_x^1}^r\right] + \EE\left[\norm{\varepsilon \nabla (\rho u)}_{L_t^2 W_x^{-1,\frac{2\Gamma}{\Gamma + 1}}}^r\right] \lesssim c(r)
    \end{align*}
    uniformly in $\varepsilon$.
    Next, since bounded sets in $L^\infty(0,T;\MMM_b(D))$ are relatively compact with respect to the weak-$\ast$ topology, $\left\{ \LL[E_\varepsilon]:\varepsilon\in(0,1) \right\}$ is tight on $\XX_E$. 
    Finally, tightness of $\left\{ \LL[\nu_\varepsilon]:\varepsilon\in (0,1) \right\}$ follows from the fact that the set
    $$
        B_R = \left\{ \nu\in (L_{w^\ast}^\infty((0,T)\times D;\PPP(\RR^{13})),w^\ast) : \intT\int_D\int_{\RR^{13}} \left(\left|\xi_1 \right|^{\Gamma + 1} + \sum_{i=2}^{13}\left|\xi_i \right|^2\right) d\nu_{t,x}(\xi)dx dt \leq R \right\}
    $$
    is relatively compact in $ (L_{w^\ast}^\infty((0,T)\times D;\PPP(\RR^{13})),w^\ast)$ (see \cite[Corollary 2.8.6]{BFHbook18}). The proof is complete.
\end{proof}
Consequently, we may apply Jakubowski's theorem, Theorem \ref{thm:Jakubowski--Skorokhod} as well as Corollary \ref{prop:ep:tightness} to obtain the following.
\begin{proposition}
    \label{prop:ep:repr}
        There exists a complete probability space $\tprobsp$ with $\XX$-valued Borel measurable random variables $(\trho_{0,\varepsilon},\tq_{0,\varepsilon},\tk_{0,\varepsilon},\trho_\varepsilon,\tq_\varepsilon,\tu_\varepsilon,\tW_\varepsilon,\tE_\varepsilon,\tnu_\varepsilon),\varepsilon\in (0,1)$, as well as 
        $(\trho_{0},\tq_{0},\tk_{0},\trho,\tq,\tu,\tW,\tE,\tnu)$ such that (up to a subsequence):
        \begin{itemize}
            \item[(1)] the laws $\LL[(\trho_{0,\varepsilon},\tq_{0,\varepsilon},\tk_{0,\varepsilon},\trho_\varepsilon,\tq_\varepsilon,\tu_\varepsilon,\tW_\varepsilon,\tE_\varepsilon,\tnu_\varepsilon)]$ and $\LL[\overline{\rho}_{0,\varepsilon}, \overline{\rho}_{0,\varepsilon} \overline{u}_{0,\varepsilon},(\overline{\rho}_{0,\varepsilon} \overline{u}_{0,\varepsilon}) / \sqrt{\overline{\rho}_{0,\varepsilon}}, \rho_\varepsilon, \rho_\varepsilon u_\varepsilon, u_\varepsilon, W_\varepsilon, \linebreak E(\rho_\varepsilon, u_\varepsilon), \nu_\varepsilon ]$ 
            coincide on $\XX$. In particular, 
            \begin{align*}
                &\trho_{0,\varepsilon} = \trho_\varepsilon(0),\quad \tq_{0,\varepsilon} = \trho_\varepsilon \tu_\varepsilon(0),\quad \tk_{0,\varepsilon} = \frac{\tq_{0,\varepsilon}}{\sqrt{\trho_{0,\varepsilon}}} = \frac{\trho_\varepsilon \tu_\varepsilon(0)}{\sqrt{\trho_\varepsilon(0)}}, \\ 
                &\tq_\varepsilon = \trho_\varepsilon \tu_\varepsilon,\quad \tE_\varepsilon = E(\trho_\varepsilon,\tu_\varepsilon),\quad \tnu_\varepsilon = \delta_{(\trho_\varepsilon,\tu_\varepsilon,\nabla \tu_\varepsilon)},   
            \end{align*}
            $\tPP$-a.s., as well as 
            \begin{align}
                \EE\left[\left|\sup_{t\in [0,T]}\int_D \left[\frac12 \trho_\varepsilon |\tu_\varepsilon|^2 + P_\delta(\trho_\varepsilon)\right]dx \right|^{r} \right] \lesssim c(r);
            \end{align}
            \item[(2)] the law of $(\trho_{0},\tq_{0},\tk_{0},\trho,\tq,\tu,\tW,\tE,\tnu)$ on $\XX$ is a Radon measure;
            \item[(3)] $(\trho_{0,\varepsilon},\tq_{0,\varepsilon},\tk_{0,\varepsilon},\trho_\varepsilon,\tq_\varepsilon,\tu_\varepsilon,\tW_\varepsilon,\tE_\varepsilon,\tnu_\varepsilon)$ converges in the topology of $\XX$ $\tPP$-a.s. to $(\trho_{0},\tq_{0},\tk_{0},\trho,\tq,\tu,\tW,\tE,\tnu)$, i.e., 
            \begin{align}
                \begin{aligned}
                    \trho_{0,\varepsilon} &\to \trho_0 \quad \text{in } L^\Gamma (D), \\ 
                    \tq_{0,\varepsilon}&\to \tq_0 \quad \text{in } L^1(D), \\ 
                    \tk_{0,\varepsilon} &\to \tk_0 \quad \text{in } L^2(D), \\ 
                    \trho_\varepsilon&\to \trho \quad \text{in } C_w([0,T];L^\Gamma(D)), \\ 
                    \trho_\varepsilon&\weakarrow \trho \quad \text{in } L^{\Gamma+1}((0,T)\times D), \\ 
                    \trho_\varepsilon \tu_\varepsilon &\to \tq \quad \text{in } C_w([0,T];L^{\frac{2\Gamma }{\Gamma +1}}(D)), \\ 
                    \tu_\varepsilon &\weakarrow \tu\quad \text{in } L^2(0,T;W^{1,2}_{\rn }(D)), \\ 
                    \tW_\varepsilon &\to \tW\quad \text{in } C([0,T];\fU_0), \\
                    E(\trho_\varepsilon,\tu_\varepsilon) &\weakstararrow \tE\quad \text{in } L^\infty(0,T;\MMM_b(D)), \\ 
                    \delta_{(\trho_\varepsilon,\tu_\varepsilon,\nabla \tu_\varepsilon)} &\to \tnu \quad \text{in } (L_{w^\ast}^\infty((0,T)\times D;\PPP(\RR^{13})),w^\ast),
                \end{aligned}\label{eq:6.19}
            \end{align}
            as $\varepsilon\to 0$ $\tPP$-a.s.;
            \item[(4)] for any Carath\'{e}odory function $H=H(t,x,\rho,v,V)$ where $(t,x)\in (0,T)\times D$, $(\rho,v,V)\in \RR^{13}$, satisfying for some $q_1,q_2>0$ the growth condition 
            $$
                \left|H(t,x,\rho,v,V) \right| \lesssim 1 + |\rho|^{q_1} + |v|^{q_2} +|V|^{q_2}
            $$
            uniformly in $(t,x)$, denote $\overline{H(\trho,\tu,\nabla \tu)}(t,x) = \ev{\tnu_{t,x},H}$. Then we have 
            \begin{align}
                H(\trho_\varepsilon,\tu_\varepsilon,\nabla \tu_\varepsilon)\weakarrow \overline{H(\trho,\tu,\nabla \tu)}\quad \text{in } L^r((0,T)\times D)\quad \text{for all } 1<r\leq \frac{\Gamma + 1}{q_1}\wedge \frac{2}{q_2}, \label{eq:6.20}
            \end{align}
            as $\varepsilon\to 0$ $\tPP$-a.s.
        \end{itemize}
\end{proposition}
\begin{remark}
    The assertion (4) in Proposition \ref{prop:ep:repr} is a direct consequence of \cite[Corollary 2.8.3]{BFHbook18}.
\end{remark}
\begin{remark}
    As stated in Remark \ref{rem:r.d.}, we may deduce that the filtration 
    $$
        \tFFF_t:= \sigma \left(\sigma_t[\trho] \cup \sigma_t[\tu]\cup \sigma_t[\tW]\right),\quad t\in [0,T],
    $$
    is non-anticipating with respect to $\tW= \sum_{k=1}^{\infty} e_k \tW_k$, which is a cylindrical $(\tFFF_t)$-Wiener process on $\fU$.    
\end{remark}
\begin{remark}
    In view of the convergence of the initial conditions \eqref{eq:6.17} and \eqref{eq:6.19}, we obtain that 
    $$
        \Lambda = \LL[\trho_0,\trho_0 \tu_0] \quad \text{on } L^1(D)\times L^1(D),
    $$
    where $\tu_0$ is defined as $\tu_0 = \tq_0/ \trho_0 1_{\left\{ \trho_0>0 \right\}}$.
\end{remark}
As a direct consequence of Proposition \ref{prop:ep:repr}, similarly to the proofs of Lemma \ref{lem:alpha(q = rho u)} and Lemma \ref{lem:alpha(conv rho u otimes u)}, we obtain the following result. Note that since strong convergence of $\trho_\varepsilon$ does not hold at present, only weak convergence holds in the following unlike Lemma \ref{lem:alpha(conv rho u otimes u)}.
\begin{lemma}
    \label{lem:ep:conv rho u otimes u}
    We have $\tPP$-a.s., $\trho \tu = \tq$, and
    \begin{align}
        \label{eq:6.21}
        \trho_\varepsilon \tu_\varepsilon\otimes \tu_\varepsilon\weakarrow \trho \tu\otimes \tu\quad \text{in }L^1(0,T;L^1(D)).
    \end{align}
\end{lemma}
Similarly to Proposition \ref{prop:eq_of_conti1}, the equation of continuity \eqref{eq:approx_CE_ep} is satisfied by $(\trho,\tu)$ on the new probability space. Note that we only have a weak solution.
\begin{proposition}
    \label{prop:ep:CE}
    The random distribution $(\trho,\tu)$ satisfies \eqref{eq:approx_CE_ep} for all $\phi\in C_c^\infty(\ico{0}{T})$ and $\psi\in C_c^\infty(\RR^3)$ $\tPP$-a.s.
\end{proposition}
\begin{remark}
    \label{rem:ep:conv_of_Delta rho}
    Since $(\trho_\varepsilon,\tu_\varepsilon)$ satisfies the equation of continuity \eqref{eq:approx_CE_m} a.e. in $(0,T)\times D$ $\tPP$-a.s., and $\trho_\varepsilon \tu_\varepsilon$ and $\nabla \trho_\varepsilon$ have zero normal trace, we may multiply \eqref{eq:approx_CE_m} by $\trho_\varepsilon$ and integrate by parts, obtaining 
    $$
        \dv{t} \int_D \frac12 \left|\trho_\varepsilon \right|^2 dx + \varepsilon \int_D \left|\nabla \trho_\varepsilon \right|^2dx = - \frac12 \int_D \left|\trho_\varepsilon \right|^2 \rdiv \tu_\varepsilon dx.
    $$
    In accordance with the first one in \eqref{eq:6.7}, and \eqref{eq:6.9} (both of which continue to hold on the new probability space), the expectation of the integral on the right-hand side is bounded, and we obtain 
    \begin{align*}
        \varepsilon \tEE \left[\intT \norm{\nabla \trho_\varepsilon }_{L_x^2}^2 dt \right] \lesssim c(r),
    \end{align*}
    where $c(r)$ is given in Section \ref{sec:ep:Uniform estimates}. In particular, for a suitable subsequence, $\tPP$-a.s., 
    $$
        \varepsilon \nabla \trho_\varepsilon \to 0 \quad \text{in } L^2(0,T;L^2(D)).
    $$
\end{remark}
Next, we perform the limit $\varepsilon\to 0$ in the interior momentum equation \eqref{eq:approx_ME_ep}. 
To this end, we apply Proposition \ref{prop:ep:repr}, part (4), to the compositions $p_\delta(\trho_\varepsilon)$ and $\trho_\varepsilon F_k(\trho_\varepsilon,\tu_\varepsilon)$, $k\in\NN$.
Specifically, we have 
\begin{align}
    &p_\delta(\trho_\varepsilon) \weakarrow \overline{p_\delta(\trho)}\quad \text{in } L^q((0,T)\times D), \label{eq:6.22}\\ 
    &\trho_\varepsilon F_k(\trho_\varepsilon,\tu_\varepsilon) \weakarrow \overline{\trho F_k(\trho,\tu)} \quad \text{in } L^q((0,T)\times D), \label{eq:6.23}
\end{align}
for some $q>1$ $\tPP$-a.s. Let us now define the Hilbert--Schmidt operator $\overline{\GG(\trho,\trho \tu)}= \overline{\trho \FF(\trho,\tu)}$ by 
$$
    \overline{\trho \FF(\trho,\tu)}e_k := \overline{\trho F_k(\trho,\tu)},\quad k\in \NN.
$$
We obtain the following result.
\begin{proposition}
    \label{prop:ep:IME}
    The random distribution $(\trho,\tu,\tW)$ satisfies 
    \begin{align}
        &-\intT \pp_t \phi \int_D \trho \tu \cdot \bphi dxdt - \phi(0) \int_D \trho_0 \tu_0 \cdot \bphi dx \notag\\ 
        &\quad = \intT \phi \int_D [\trho \tu\otimes \tu:\nabla \bphi + \overline{p_\delta(\trho)}\rdiv \bphi]dxdt \notag \\ 
        &\quad \quad - \intT \phi\int_D \SSS(\nabla \tu):\nabla \bphi dx dt + \intT \phi \int_D \overline{\trho \FF(\trho,\tu)} \cdot \bphi dx d\tW,\label{eq:6.24}
    \end{align}
    for all $\phi\in C_c^\infty(\ico{0}{T})$ and all $\bphi\in C_c^\infty(D)$ $\tPP$-a.s.
\end{proposition}
\begin{remark}
    \label{rem:ep:conv of each term of ME}
    From the same reason as Remark \ref{rem:alpha(conv of each term of ME)}, the convergence to each term in the momentum equation \eqref{eq:6.24} actually holds for any $\bphi\in C^\infty(\oD)$, but the equality holds only for $\bphi\in C_c^\infty(D)$.
\end{remark}
\begin{proof}
    As in the proof of Proposition \ref{prop:ME1}, by the equality of laws from Proposition \ref{prop:ep:repr}, $(\trho_\varepsilon,\tu_\varepsilon,\tW_\varepsilon)$ satisfies the interior momentum equation \eqref{eq:approx_ME_m}, and by \eqref{eq:6.22}, the deterministic part can pass to the limit.
    Note that at the current stage we are only able to show that a certain limit exists, but we are unable to identify it. For the identification, strong convergence of the approximate densities $\trho_\varepsilon$ is necessary, which is the main goal of Section \ref{sec:ep:Strong convergence of the densities}.
    
    Finally, the passage to the limit for the stochastic integral part using \eqref{eq:6.23} can be justified in exactly the same way as \cite[Proposition 4.4.14]{BFHbook18}. 
    To this end, we first show that, for $l>\frac32$,
    \begin{align}
        \label{eq:6.25}
        \trho_\varepsilon F_{k,\varepsilon}(\trho_\varepsilon,\tu_\varepsilon)\to \overline{\trho F_k(\trho,\tu)}\quad \text{in }L^2(0,T;(W^{l,2}(D))^\ast) 
    \end{align}
    $\tPP$-a.s., for any $k\in\NN$. We observe that 
    \begin{align*}
        &\int_D \trho_\varepsilon \left|F_{k,\varepsilon}(\trho_\varepsilon,\tu_\varepsilon) - F_{k}(\trho_\varepsilon,\tu_\varepsilon) \right|dx \\ 
        &\quad \leq \int_{\trho_\varepsilon<\varepsilon}\trho_\varepsilon \left|F_{k}(\trho_\varepsilon,\tu_\varepsilon) \right|dx + \int_{\left|\tu_\varepsilon \right|>\frac1\varepsilon}\trho_\varepsilon \left|F_{k}(\trho_\varepsilon,\tu_\varepsilon) \right|dx \\ 
        &\quad \leq f_k \left[\int_{\trho_\varepsilon<\varepsilon}\left(\trho_\varepsilon +\trho_\varepsilon |\tu_\varepsilon|\right)dx + \int_{\left|\tu_\varepsilon\right|>\frac1\varepsilon }\left(\trho_\varepsilon +\trho_\varepsilon |\tu_\varepsilon|\right)dx\right] \\ 
        &\quad \lesssim f_k \left[\varepsilon \int_D \left(1 + |\tu_\varepsilon|\right)dx + \int_{\left|\tu_\varepsilon \right|>\frac1\varepsilon}\trho_\varepsilon |\tu_\varepsilon|dx\right] \\ 
        &\quad \lesssim \varepsilon f_k \int_D \left(1 + |\tu_\varepsilon| + \trho_\varepsilon |\tu_\varepsilon|^2 \right)dx.
    \end{align*}
    Hence, due to the uniform estimates \eqref{eq:6.8}--\eqref{eq:6.9} (which continue to hold the new probability space by Proposition \ref{prop:ep:repr}), we obtain 
    \begin{align*}
        &\trho_\varepsilon F_{k,\varepsilon}(\trho_\varepsilon,\tu_\varepsilon)- \trho_\varepsilon F_{k}(\trho_\varepsilon,\tu_\varepsilon)  \to 0 \quad \text{in }L^2(0,T;(W^{l,2}(D))^\ast),
    \end{align*}
    $\tPP$-a.s., for any $k\in\NN$. To see \eqref{eq:6.25}, we first write 
    \begin{align*}
        &\int_D \trho_\varepsilon F_{k}(\trho_\varepsilon,\tu_\varepsilon)\cdot \bphi dx = \int_D \trho_\varepsilon \left(F_{k}(\trho_\varepsilon,\tu_\varepsilon) - F_{k}(\trho_\varepsilon,\tu)\right)\cdot \bphi dx + \int_D \trho_\varepsilon F_{k}(\trho_\varepsilon,\tu)\cdot \bphi dx,
    \end{align*}
    for any $\bphi\in C^\infty(\oD)$. By H\"older's inequality, we have 
    \begin{align*}
          \int_D \trho_\varepsilon \left(F_{k}(\trho_\varepsilon,\tu_\varepsilon) - F_{k}(\trho_\varepsilon,\tu)\right)\cdot \bphi dx &\leq f_k \norm{\bphi}_{L_x^\infty}\int_D \trho_\varepsilon |\tu_\varepsilon -\tu|dx \\ 
          &\leq f_k \norm{\bphi}_{L_x^\infty}\norm{\sqrt{\trho_\varepsilon}}_{L_x^2}\norm{\trho_\varepsilon |\tu_\varepsilon -\tu|^2}_{L_x^1}^{1/2}.
    \end{align*}
    It follows from Proposition \ref{prop:ep:repr} and Lemma \ref{lem:ep:conv rho u otimes u} that $\tPP$-a.s.,
    \begin{align*}
        &\norm{\trho_\varepsilon |\tu_\varepsilon -\tu|^2}_{L_x^1} = \int_D \trho_\varepsilon |\tu_\varepsilon|^2dx - 2\int_D \trho_\varepsilon \tu_\varepsilon \cdot \tu dx + \int_D \trho_\varepsilon |\tu|^2 dx \to 0 \quad \text{in }L^1(0,T). 
    \end{align*}
    This implies 
    \begin{align*}
        &\trho_\varepsilon \left(F_k(\trho_\varepsilon,\tu_\varepsilon)- F_k(\trho_\varepsilon, \tu)\right) \to 0 \quad \text{in }L^2(0,T;(W^{l,2}(D))^\ast)\quad \tPP\text{-a.s.}
    \end{align*}
    Thus, \eqref{eq:6.25} follows if we show that 
    \begin{align}
        \label{eq:6.26}
        &\trho_\varepsilon F_k(\trho_\varepsilon,\tu)  \to \overline{\trho F_k(\trho,\tu)}\quad \text{in }L^2(0,T;(W^{l,2}(D))^\ast)\quad \tPP\text{-a.s.}
    \end{align}
    To see \eqref{eq:6.26}, we need the following renormalized form of the equation of continuity:
    \begin{align}
        \label{eq:6.27}
        \pp_t b(\trho_\varepsilon) + \rdiv (b(\trho_\varepsilon)\tu_\varepsilon) + (b'(\trho_\varepsilon)\trho_\varepsilon - b(\trho_\varepsilon))\rdiv \tu_\varepsilon = \varepsilon \rdiv (b'(\trho_\varepsilon)\nabla \trho_\varepsilon) - \varepsilon b''(\trho_\varepsilon)|\nabla \trho_\varepsilon|^2
    \end{align}
    for any $b$ having at most quadratic growth, 
    $$
        \left|b''(\rho) \right| \lesssim 1\quad \text{for all } \rho\geq 0.
    $$
    \eqref{eq:6.27} follows easily by multiplying \eqref{eq:approx_CE_m} by $b'(\trho_\varepsilon)$. Moreover, if $b$ is globally Lipschitz, by an Arzel\`a--Ascoli type result in $C_w([0,T];L^2(D))$ (see \cite[Lemma 6.2]{NS04}), \eqref{eq:6.20}, and Remark \ref{rem:ep:conv_of_Delta rho}, we deduce from \eqref{eq:6.27} 
    \begin{align}
        \label{eq:6.28}
        &b(\trho_\varepsilon) \to \overline{b(\trho)}\quad \text{in }C_w([0,T];L^2(D))\quad \tPP\text{-a.s.}\quad \text{as }\varepsilon \to 0.
    \end{align} 
    Indeed, first, note that we do not have to choose a subsequence in \eqref{eq:6.28} as the limit is uniquely determined by \eqref{eq:6.20}. 
    Next, in view of \cite[Lemma 6.2]{NS04}, it suffices to show that 
    \begin{align*}
        &b(\trho_\varepsilon) \quad \text{is bounded in }L^\infty(0,T;L^2(D))\cap C([0,T];W^{-1,r}(D))  \quad \tPP\text{-a.s.,}
    \end{align*}
    for some $1<r<\infty$. The boundedness of $b(\trho_\varepsilon)$ in $L^\infty_tL^2_x$ follows immediately from the Lipschitz continuity of $b$ and the boundedness of $\trho_\varepsilon$ in $L_t^\infty L_x^\Gamma$ $\tPP$-a.s.
    Thus, by \eqref{eq:6.27}, it suffices to show that 
    \begin{align*}
        &\pp_t b(\trho_\varepsilon) \quad \text{is bounded in }L^1(0,T;W^{-1,r}(D))  \quad \tPP\text{-a.s.,}
    \end{align*}
    for some $1<r<\infty$. Clearly, 
    \begin{align*}
        &\norm{\rdiv (b(\trho_\varepsilon)\tu_\varepsilon)}_{L_t^1W_x^{-1,r}} \lesssim \norm{b(\trho_\varepsilon)\tu_\varepsilon}_{L_t^1L_x^r} \lesssim \norm{\trho_\varepsilon \tu_\varepsilon}_{L_t^1L_x^r},\quad 1<r<\infty.  
    \end{align*}
    Since $b$ is globally Lipschitz and $b'$ is bounded, we have 
    \begin{align*}
        \norm{ (b'(\trho_\varepsilon)\trho_\varepsilon - b(\trho_\varepsilon))\rdiv \tu_\varepsilon}_{L_t^1L_x^1}  &\lesssim \norm{\trho_\varepsilon}_{L_t^2L_x^2}\norm{\rdiv \tu_\varepsilon}_{L_t^2L_x^2} \\ 
        &\lesssim \norm{\trho_\varepsilon}_{L_t^2L_x^2}\norm{\tu_\varepsilon}_{L_t^2W_x^{1,2}}.
    \end{align*}
    If $1<r\leq 2$, we have 
    \begin{align*}
        \norm{ \varepsilon \rdiv (b'(\trho_\varepsilon)\nabla \trho_\varepsilon)}_{L_t^1W_x^{-1,r}} &\lesssim \norm{b'(\trho_\varepsilon)\varepsilon \nabla \trho_\varepsilon}_{L_t^1L_x^r} \\ 
        &\lesssim \norm{\varepsilon\nabla \trho_\varepsilon}_{L_t^1L_x^2},
    \end{align*}    
    and by $|b''|\lesssim 1$, 
    \begin{align*}
        &\norm{\varepsilon b''(\trho_\varepsilon)|\nabla \trho_\varepsilon|^2}_{L_t^1L_x^1}  \lesssim \norm{\sqrt{\varepsilon}\nabla \trho_\varepsilon}_{L_t^2L_x^2}^2.
    \end{align*}
    Therefore, by choosing $1<r<3/2$ so that $L^1(D) \hookrightarrow W^{-1,r}(D)$, the boundedness of $\pp_t b(\trho_\varepsilon)$ in $L^1_tW^{-1,r}_x$ follows from the convergence \eqref{eq:6.19} and Remark \ref{rem:ep:conv_of_Delta rho}. 
    Thus, we obtain \eqref{eq:6.28}.

    To obtain a similar statement for the Carath\'eodory composition 
    \begin{align*}
        &\trho_\varepsilon F_k(\trho_\varepsilon,\tu), 
    \end{align*}
    we use an approximation argument. First, we recall the compact embedding $L^2(D)\cptarrow (W^{1,2}(D))^\ast$ and 
    \begin{align*}
        &C_w([0,T];L^2(D)) \hookrightarrow L^2(0,T;(W^{1,2}(D))^\ast).   
    \end{align*}
    Consequently, we deduce from \eqref{eq:6.28} and the fact that $\tu\in L^2(0,T;W^{1,2}(D))$ $\tPP$-a.s. that 
    \begin{align}
        \label{eq:6.29}
        &b(\trho_\varepsilon) B(\tu) \to \overline{b(\trho)B(\tu)}\quad \text{in }L^1(0,T;(W^{l,2}(D))^\ast),  \quad l>\tfrac32,
    \end{align}
    for any $b$ having at most quadratic growth and $|b''(\rho)|\lesssim 1,\rho\geq 0$, and for any $B\in C^1(\RR^3)$ with $\nabla B$ bounded. Indeed, we have 
    \begin{align*}
        &\left|\intT\int_D \left(b(\trho_\varepsilon)-b(\trho_{\varepsilon'})\right) B(\tu)\cdot \bphi dx dt \right| \\ 
        &\quad \leq \norm{b(\trho_\varepsilon)-b(\trho_{\varepsilon'})}_{L^2_t(W_x^{1,2})^\ast} \norm{B(\tu)\cdot\bphi}_{L^2_t W^{1,2}_x} \\ 
        &\quad \lesssim \norm{b(\trho_\varepsilon)-b(\trho_{\varepsilon'})}_{L^2_t(W_x^{1,2})^\ast} \norm{B(\tu)}_{L^2_t W^{1,2}_x}\norm{\bphi}_{L^\infty_t W^{l,2}_x},\quad l>\tfrac32,
    \end{align*}
    for all $\bphi\in L^\infty(0,T;W^{l,2}(D))$, $\varepsilon,\varepsilon'>0$ $\tPP$-a.s., and taking the supremum over $\bphi\in L^\infty(0,T;W^{l,2}(D))$ with $\norm{\bphi}_{L_t^\infty W_x^{l,2}}\leq 1$ yields the desired convergence \eqref{eq:6.29}.

    By interpolation between $L^1(0,T;(W^{l,2}(D))^\ast)$ and $L^p(0,T;(W^{l,2}(D))^\ast),\ p>2$, and a density argument, we also have $\tPP$-a.s.,
    \begin{align*}
        &b(\trho_\varepsilon) B(\tu) \to \overline{b(\trho)B(\tu)}\quad \text{in }L^2(0,T;(W^{l,2}(D))^\ast),  
    \end{align*}
    for any globally Lipschitz $b$ and $B$. Thus, \eqref{eq:6.26} follows by a density argument via locally uniform approximation of $\rho F_k(\rho,u)$ by finite sums $\sum_j b_j(\rho)B_j(u)$. This completes the proof of \eqref{eq:6.25}.

    For $l>\frac32$, we have, by Sobolev's embedding $L^1(D)\cptarrow (W^{l,2}(D))^\ast$, \eqref{eq:6.7} and \eqref{eq:6.8}, 
    \begin{align*}
        &\tEE\left[\intT \norm{\trho_\varepsilon \FF_\varepsilon(\trho_\varepsilon,\tu_\varepsilon)}_{L_2(\fU;(W^{l,2}_x)^\ast)}^2dt \right]   \\ 
        &\quad \lesssim \tEE\left[\intT \sum_{k=1}^{\infty}\norm{\sqrt{\trho_\varepsilon}\sqrt{\trho_\varepsilon}^{-1}G_{k,\varepsilon}(\trho_\varepsilon,\trho_\varepsilon \tu_\varepsilon)}_{L^1_x}^2dt \right] \\ 
        &\quad \lesssim \tEE\left[\intT (\trho_\varepsilon)_D\sum_{k=1}^{\infty}\int_D \trho_\varepsilon^{-1}|G_{k,\varepsilon}(\trho_\varepsilon,\trho_\varepsilon \tu_\varepsilon)|^2dx dt \right] \\ 
        &\quad \lesssim \tEE\left[\intT (\trho_\varepsilon)_D\int_D \left(\trho_\varepsilon + \trho_\varepsilon |\tu_\varepsilon|^2 \right)dx dt \right] \\ 
        &\quad \lesssim c(r).
    \end{align*} 
    Consequently, \eqref{eq:6.25} implies 
    \begin{align*}
        &\trho_\varepsilon \FF_\varepsilon(\trho_\varepsilon,\tu_\varepsilon)\to \overline{\trho \FF(\trho,\tu)}\quad \text{in }L^2(0,T;L_2(\fU;(W^{l,2}(D)))^\ast) 
    \end{align*}
    $\tPP$-a.s. Together with the convergence of $\tW_\varepsilon$ from Proposition \ref{prop:ep:repr}, this allows us to apply Lemma \ref{lem:conv_of_stoch_int} to pass to the limit in the stochastic integral and hence complete the proof.
\end{proof}
\subsection{Strong convergence of the densities} \label{sec:ep:Strong convergence of the densities}
We adapt the method used in the deterministic case based on weak continuity of the \textit{effective viscous flux} (c.f. \cite[Secion 3.6.5]{FN17})
$$
    (\eta + \mu) \rdiv u - p(\rho),
$$
where $\eta = \frac\mu3 +\lambda$.
First, we introduce \textit{pseudo-differential operators} $\RRR= \Delta^{-1}\nabla \otimes \nabla = (\RRR_{ij})_{i,j=1}^3$, and $\cA = \nabla \Delta^{-1} = (\pp_j \Delta^{-1})_{j=1}^3$ on $\RR^3$, which are identified by the Fourier symbols:
\begin{itemize}
    \item the ``double'' Riesz transform:
    \begin{align*}
        &\RRR_{ij} \approx \frac{\xi_i \xi_j}{|\xi|^2},\quad i,j= 1,2,3,
    \end{align*}
    meaning that 
    \begin{align*}
        &\RRR_{ij}(v) := \cF_{\xi\to x}^{-1}\left[\frac{\xi_i \xi_j}{|\xi|^2} \cF_{x\to \xi}(v)\right],\quad v\in \cS(\RR^3);
    \end{align*}
    \item the inverse divergence:
    \begin{align*}
        &\pp_j \Delta^{-1}\approx -\frac{i \xi_j}{|\xi|^2},\quad j=1,2,3,
    \end{align*}
    meaning that 
    \begin{align*}
        &\pp_j \Delta^{-1}(v) := - \cF_{\xi\to x}^{-1}\left[\frac{i \xi_j}{|\xi|^2} \cF_{x\to \xi}(v)\right],\quad v\in \cS(\RR^3),
    \end{align*}
\end{itemize}
where $\cS(\RR^3)$ denotes the space of \textit{smooth rapidly decreasing functions}, $\cF$ is the \textit{Fourier transform} on the space of \textit{tempered distributions} $\cS'(\RR^3)$, and the values of the above mappings make sense as elements of $\cS'(\RR^3)$.
Moreover, by virtue of the H\"olmander--Mikhlin theorem, $\RRR_{ij}$ is a continuous linear operator mapping $L^p(\RR^3)$ into $L^p(\RR^3)$ for any $i,j=1,2,3$ and any $1<p<\infty$, and $\nabla \Delta^{-1}$ also possesses appropriate $L^p$-regularity.
For details, see Appendix \ref{sec:Properties of pseudo-differential operators}.

Extending $\trho_\varepsilon$ and $\tu_\varepsilon$ to be zero outside $D$, we use the quantity 
$$
    \phi \cA(1_D \trho_\varepsilon )= \phi \nabla \Delta^{-1} (1_D \trho_\varepsilon ),\quad \phi\in C_c^\infty(D)
$$
as a test function in \eqref{eq:approx_ME_m}. Since $(\trho_\varepsilon,\tu_\varepsilon)$ satisfies the equation of continuity \eqref{eq:approx_CE_m} a.e. in $(0,T)\times D$ $\tPP$-a.s., and $\trho_\varepsilon \tu_\varepsilon$ and $\nabla \trho_\varepsilon$ have zero normal trace, the equation \eqref{eq:approx_CE_m} can be extended to the whole space $\RR^3$, that is,
\begin{align*}
    \pp_t (1_D \trho_\varepsilon ) + \rdiv \left(1_D \trho_\varepsilon \tu_\varepsilon\right) = \varepsilon \rdiv \left(1_D \nabla \trho_\varepsilon\right)
\end{align*}
a.e. in $(0,T)\times \RR^3$ $\tPP$-a.s. Indeed, applying the boundary conditions on $\nabla \trho_\varepsilon$ and $\tu_\varepsilon$, we have 
\begin{align*}
    &-\intT \pp_t \phi \int_D \trho_\varepsilon \psi dx dt  - \intT \phi \int_D \trho_\varepsilon \tu_\varepsilon \cdot \nabla \psi dxdt = -\varepsilon \intT\phi \int_D \nabla \trho_\varepsilon \cdot \nabla \psi dx dt,
\end{align*}
for all $\phi\in C_c^\infty((0,T))$, and all $\psi\in C_c^\infty(\RR^3)$ $\tPP$-a.s., which is equivalent to 
\begin{align*}
    &-\intT \pp_t \phi \int_{\RR^3} 1_D \trho_\varepsilon \psi dx dt  - \intT \phi \int_{\RR^3} 1_D \trho_\varepsilon \tu_\varepsilon \cdot \nabla \psi dxdt = -\varepsilon \intT\phi \int_{\RR^3} 1_D \nabla \trho_\varepsilon \cdot \nabla \psi dx dt,
\end{align*}
for all $\phi\in C_c^\infty((0,T))$, and all $\psi\in C_c^\infty(\RR^3)$ $\tPP$-a.s. Similarly, weak differentiability of appearing terms can be verified.

Therefore, we have 
\begin{align*}
    \pp_t \Psi_\varepsilon = -\phi\nabla \Delta^{-1} \rdiv (1_D\trho_\varepsilon \tu_\varepsilon -\varepsilon 1_D \nabla \trho_\varepsilon),
\end{align*}
where $\Psi_\varepsilon = \phi \nabla \Delta^{-1}(1_D \trho_\varepsilon)$. 
Using It\^{o}'s product rule with the same argument as in Section \ref{sec:ep:Uniform estimates} yields $\tPP$-a.s., 
\begin{align}
    \left[\int_D \trho_\varepsilon \tu_\varepsilon \cdot \Psi_\varepsilon dx \right]_{t=0}^{t=\tau} &= \inttau \int_D \trho_\varepsilon \tu_\varepsilon\cdot \pp_t \Psi_\varepsilon dx dt \notag \\ 
    &\quad + \inttau \int_D \left[\trho_\varepsilon \tu_\varepsilon \otimes \tu_\varepsilon + p_\delta(\trho_\varepsilon)\II -\SSS(\nabla \tu_\varepsilon)-\varepsilon \nabla(\trho_\varepsilon \tu_\varepsilon)\right]:\nabla \Psi_\varepsilon dx dt \notag \\ 
    &\quad + \inttau \int_D \GG_\varepsilon (\trho_\varepsilon,\tu_\varepsilon)\cdot \Psi_\varepsilon d\tW_\varepsilon.\label{eq:6.30}
\end{align}
Noting that 
$$
    \nabla \Psi_\varepsilon = \nabla \phi \otimes \nabla \Delta^{-1}(1_D\trho_\varepsilon) + \phi \RRR[1_D \trho_\varepsilon],
$$
rewrite \eqref{eq:6.30} as follows:
\begin{align}
    \label{eq:6.31}
    &\inttau \int_D \phi\left[ p_\delta(\trho_\varepsilon) \trho_\varepsilon - \SSS(\nabla \tu_\varepsilon): \RRR (1_D \trho_\varepsilon)\right] dx dt = \sum_{j=1}^{5} I_{j,\varepsilon},
\end{align}
where 
\begin{align*}
    &I_{1,\varepsilon} = \left[\int_D \phi \trho_\varepsilon \tu_\varepsilon\cdot \nabla \Delta^{-1} \left[1_D \trho_\varepsilon\right]dx\right]_{t=0}^{t=\tau}, \\ 
    &I_{2,\varepsilon} = \inttau \int_D \phi \trho_\varepsilon \tu_\varepsilon \cdot \RRR\left[1_D\trho_\varepsilon \tu_\varepsilon -\varepsilon 1_D \nabla \trho_\varepsilon\right] dxdt, \\ 
    &I_{3,\varepsilon} = - \inttau \int_D \phi \left[\trho_\varepsilon \tu_\varepsilon \otimes \tu_\varepsilon - \varepsilon \nabla(\trho_\varepsilon \tu_\varepsilon) \right]:\RRR(1_D \trho_\varepsilon) dx dt, \\ 
    &I_{4,\varepsilon} = - \inttau \int_D \left[\trho_\varepsilon \tu_\varepsilon \otimes \tu_\varepsilon + p_\delta(\trho_\varepsilon)\II -\SSS(\nabla \tu_\varepsilon)-\varepsilon \nabla(\trho_\varepsilon \tu_\varepsilon)\right]: \nabla \phi \otimes \nabla\Delta^{-1}(1_D \trho_\varepsilon) dx dt, \\ 
    &I_{5,\varepsilon} = - \inttau \int_D \phi \GG_\varepsilon(\trho_\varepsilon,\trho_\varepsilon \tu_\varepsilon)\cdot \invdiv{1_D \trho_\varepsilon} dx d\tW_\varepsilon.
\end{align*}
Similarly, since $(\trho,\tu)$ satisfies the equation of continuity \eqref{eq:approx_CE_ep} in the distribution sense, by the weak continuity of $\trho$ and approximating $1_{[0,\tau]}$ by smooth functions via mollification, we have 
\begin{align*}
    \int_{\RR^3} 1_D \trho(\tau) \psi  dx =\int_{\RR^3} 1_D \trho_0 \psi dx  + \inttau \int_{\RR^3} 1_D \trho \tu \cdot \nabla \psi dx dt,
\end{align*}
for all $\tau\in [0,T]$ and all $\psi\in C_c^\infty(\RR^3)$ $\tPP$-a.s., and 
\begin{align*}
    \left[1_D \trho\right]_c(\tau,x) = \left[1_D \trho\right]_c(0,x) - \inttau \rdiv \left[1_D \trho \tu\right]_c (t,x)dt \quad \text{for all }  \tau\in [0,T], x\in  \RR^3,
\end{align*}
where the symbol $[\cdot]_c$ is the spatial regularization introduced in Section \ref{sec:ep:Uniform estimates}. Thus, as in Section \ref{sec:ep:Uniform estimates}, testing the momentum equation \eqref{eq:6.24} by $\Psi = \phi \nabla \Delta^{-1}([1_D \trho]_c)$ and letting $c\to 0$ yield $\tPP$-a.s.,
\begin{align}
    \label{eq:6.32}
    &\inttau \int_D \phi\left[ \overline{p_\delta(\trho)} \trho - \SSS(\nabla \tu): \RRR (1_D \trho)\right] dx dt = \sum_{j=1}^{5} I_{j},
\end{align}
where 
\begin{align*}
    &I_{1} = \left[\int_D \phi \trho \tu\cdot \nabla \Delta^{-1} \left[1_D \trho\right]dx\right]_{t=0}^{t=\tau}, \\ 
    &I_{2} = \inttau \int_D \phi \trho \tu \cdot\RRR\left[1_D\trho \tu \right] dxdt, \\ 
    &I_{3} = - \inttau \int_D \phi \left[\trho \tu \otimes \tu \right]:\RRR(1_D \trho) dx dt, \\ 
    &I_{4} = - \inttau \int_D \left[\trho \tu \otimes \tu + \overline{p_\delta(\trho)}\II -\SSS(\nabla \tu)\right]: \nabla \phi \otimes \nabla\Delta^{-1}(1_D \trho) dx dt, \\ 
    &I_{5} = - \inttau \int_D \phi \overline{\GG(\trho,\trho \tu)}\cdot \invdiv{1_D \trho}dx d\tW.
\end{align*}
As $\varepsilon \to 0$, by using a variant of the Div-Curl lemma (see \cite[Lemma 3.5]{FN17}), we can verify that the right-hand side of \eqref{eq:6.31} converges to the right-hand side of \eqref{eq:6.32}.
Indeed, applying the Aubin--Lions--Simon lemma (see \cite[Section 8]{Sim87}) together with the embedding $W^{1,\Gamma}(D)\cptarrow C(\oD)$, the convergence in \eqref{eq:6.19}, and the equation for $\nabla\Delta^{-1}(1_D \trho_\varepsilon)$ yields
\begin{align*}
    &\nabla \Delta^{-1} \left(1_D \trho_\varepsilon\right)\to \nabla \Delta^{-1}\left(1_D \trho\right)\quad \text{in }C([0,T]\times \RR^3)\ \tPP\text{-a.s.}
\end{align*}
Combining this with \eqref{eq:6.19}, \eqref{eq:6.22}, and \eqref{eq:6.23} yields 
\begin{align*}
    &I_{1,\varepsilon} \to I_1,\quad I_{j,\varepsilon}\to I_j,\quad j=4,5, \quad \text{for all }\tau\in [0,T] \ \tPP\text{-a.s.}
\end{align*}
Since 
\begin{align*}
    &\varepsilon \nabla \trho_\varepsilon \to 0 \quad \text{in } L^2((0,T)\times D)\ \tPP\text{-a.s.,} \\ 
    &\norm{\Riesz{f}}_{L^p_x} \lesssim \norm{f}_{L^p_x},  
\end{align*}
we have 
\begin{align*}
    &\inttau \int_D \phi \trho_\varepsilon \tu_\varepsilon \cdot \Riesz{\varepsilon 1_D \nabla \trho_\varepsilon}dx dt \to 0, \\ 
    &\inttau \int_D \phi \varepsilon \nabla\left(\trho_\varepsilon \tu_\varepsilon\right):\Riesz{1_D \trho_\varepsilon}dx dt \to 0 
\end{align*}
for all $\tau\in [0,T]$ $\tPP$-a.s. Using a variant of Div-Curl lemma (c.f. \cite[Lemma 3.5]{FN17}) yields 
\begin{align*}
    &\trho_\varepsilon \Riesz{1_D \trho_\varepsilon \tu_\varepsilon}- \trho_\varepsilon \tu_\varepsilon \Riesz{1_D \trho_\varepsilon} \weakarrow \trho \Riesz{1_D \trho \tu}- \trho \tu \Riesz{1_D \trho}\quad \text{in }L^r(\RR^3)  \quad \text{for all } \tau\in [0,T]\ \tPP\text{-a.s.,}
\end{align*}
where 
\begin{align*}
    &\frac1r = \frac1\Gamma + \frac{\Gamma + 1}{2\Gamma}.  
\end{align*}
As 
\begin{align*}
    &\Gamma \geq 5\quad \text{such that }r = \frac{2\Gamma }{\Gamma + 3}>\frac65,  
\end{align*}
we get $L^r\cptarrow W^{-1,2}(D)$ and, consequently 
\begin{align*}
    &\trho_\varepsilon \Riesz{1_D \trho_\varepsilon \tu_\varepsilon}- \trho_\varepsilon \tu_\varepsilon \Riesz{1_D \trho_\varepsilon} \to \trho \Riesz{1_D \trho \tu}- \trho \tu \Riesz{1_D \trho}\quad \text{in }L^2(0,T;(W^{1,2}(\RR^3))^\ast)\  \tPP\text{-a.s.}
\end{align*}
Hence, we have
\begin{align*}
    &I_{2,\varepsilon} + I_{3,\varepsilon} \to I_2 + I_3, \quad \text{for all }\tau\in [0,T] \ \tPP\text{-a.s.}
\end{align*}

Therefore, we obtain the relation 
\begin{align}
    \label{eq:6.33}
    \lim_{\varepsilon\to 0} \inttau \int_D \phi\left[ p_\delta(\trho_\varepsilon) \trho_\varepsilon - \SSS(\nabla \tu_\varepsilon): \RRR (1_D \trho_\varepsilon)\right] dx dt = \inttau \int_D \phi\left[ \overline{p_\delta(\trho)} \trho - \SSS(\nabla \tu): \RRR (1_D \trho)\right] dx dt 
\end{align}
for all $\tau\in [0,T]$ and all $\phi\in C_c^\infty(D)$ $\tPP$-a.s. Next, it can be verified that $\SSS(\nabla \tu_\varepsilon): \RRR (1_D \trho_\varepsilon)$ is replaced by $(\eta + \mu)\rdiv (\tu_\varepsilon) \trho_\varepsilon$, and $\SSS(\nabla \tu): \RRR (1_D \trho)$ is replaced by $(\eta + \mu)\rdiv (\tu) \trho$ exactly as in \cite[Section 3.6.5]{FN17}, where $\eta = \mu/3 + \lambda$. 
To this end, we focus on the symmetric part $\mu(\nabla\tu_\varepsilon + (\nabla \tu_\varepsilon)^T)$ of $\SSS(\nabla \tu_\varepsilon)$. By symmetry of $\RRR$ (see Section \ref{sec:Properties of pseudo-differential operators}), we have 
\begin{align*}
    \inttau \int_D \phi \mu(\nabla\tu_\varepsilon + (\nabla \tu_\varepsilon)^T) : \RRR(1_D \trho_\varepsilon) dx dt &= \inttau \int_D \RRR: \left[\phi \mu(\nabla\tu_\varepsilon + (\nabla \tu_\varepsilon)^T)\right]\trho_\varepsilon dx dt \\ 
    &= \inttau \int_D \left[2\phi \mu \rdiv (\tu_\varepsilon) \trho_\varepsilon + \theta(\tu_\varepsilon)\trho_\varepsilon\right] dx dt,
\end{align*}
and, similarly, 
\begin{align*}
    \inttau \int_D \phi \mu(\nabla\tu + (\nabla \tu)^T) : \RRR(1_D \trho) dx dt &= \inttau \int_D \RRR: \left[\phi \mu(\nabla\tu + (\nabla \tu)^T)\right]\trho dx dt \\ 
    &= \inttau \int_D \left[2\phi \mu \rdiv (\tu) \trho + \theta(\tu)\trho \right]dx dt,
\end{align*}
for all $\tau\in [0,T]$ and all $\phi\in C_c^\infty(D)$ $\tPP$-a.s., where 
\begin{align*}
    &\theta (\tu_\varepsilon) =   \RRR: \left[\phi \mu(\nabla\tu_\varepsilon + (\nabla \tu_\varepsilon)^T)\right] - \phi\mu \RRR:\left[\nabla\tu_\varepsilon + (\nabla \tu_\varepsilon)^T\right], \\ 
    &\theta (\tu) =   \RRR: \left[\phi \mu(\nabla\tu + (\nabla \tu)^T)\right] - \phi\mu \RRR:\left[\nabla\tu + (\nabla \tu)^T\right].
\end{align*}
Next, we claim that 
\begin{align*}
    &\theta(\tu_\varepsilon)\trho_\varepsilon \weakarrow \theta(\tu)\trho\quad \text{in } L^1((0,T)\times D)\quad \tPP\text{-a.s.}  
\end{align*}
To see this, we apply the Div-Curl Lemma \cite[Proposition 3.3]{FN17} to $U_\varepsilon = (\trho_\varepsilon, \trho_\varepsilon \tu_\varepsilon)$ and $V_\varepsilon= (\theta(\tu_\varepsilon),0,0,0)$.
In view of the convergence \eqref{eq:6.19}, the approximate equation of continuity \eqref{eq:approx_CE_m}, Remark \ref{rem:ep:conv_of_Delta rho}, and the continuity of $\RRR$, we have $\tPP$-a.s.,
\begin{align*}
    &\trho_\varepsilon\weakarrow \trho\quad \text{in }  L^{\Gamma}((0,T)\times D), \\ 
    &\theta(\tu_\varepsilon) \weakarrow \theta(\tu)\quad \text{in }  L^{2}((0,T)\times D),
\end{align*}
with $1/\Gamma+ 1/2<1$, and 
\begin{align*}
    &\left.
        \begin{aligned}
            &\rdiv_{t,x} U_\varepsilon = \pp_t \trho_\varepsilon +\rdiv_x (\trho_\varepsilon \tu_\varepsilon),\\ 
            &\operatorname{curl}_{t,x} V_\varepsilon = (\nabla_{t,x} V_\varepsilon - (\nabla_{t,x} V_\varepsilon)^T)
        \end{aligned}  
    \right\}  
    \quad \text{are compact in }
    \left\{
        \begin{aligned}
            &W^{-1,s}((0,T)\times D), \\ 
            &W^{-1,s}((0,T)\times D; \RR^{4\times 4}),
        \end{aligned}
    \right. \quad \text{respectively,}
\end{align*}
for a certain $s>1$ $\tPP$-a.s. Note that the time derivative of $\theta(\tu_\varepsilon)$ does not appear. Therefore, applying the Div-Curl Lemma yields the claimed weak convergence. 
For the remaining part of $\SSS(\nabla \tu_\varepsilon)$, it suffices to use the property of the diagonal components of $\RRR$ (see Section \ref{sec:Properties of pseudo-differential operators}).

Consequently, \eqref{eq:6.33} gives rise to 
\begin{align}
    \label{eq:6.34}
    \lim_{\varepsilon\to 0} \inttau \int_D \phi\left[ p_\delta(\trho_\varepsilon) \trho_\varepsilon - (\eta + \mu)\rdiv (\tu_\varepsilon )\trho_\varepsilon\right] dx dt = \inttau \int_D \phi\left[ \overline{p_\delta(\trho)} \trho - (\eta + \mu)\rdiv (\tu) \trho\right] dx dt 
\end{align}
for all $\tau\in [0,T]$ and all $\phi\in C_c^\infty(D)$ $\tPP$-a.s.

Next, we let $\varepsilon\to 0$ in the renormalized equation \eqref{eq:6.27} with $b(\trho_\varepsilon) = \trho_\varepsilon \log (\trho_\varepsilon)$, to obtain 
\begin{align}
    \label{eq:6.35}
    &\intT \int_D \ov{\trho \rdiv \tu}\psi dx dt \leq \int_D \trho_0 \log (\trho_0)\psi(0)dx + \intT \int_D \left[\ov{\trho \log (\trho)} \pp_t \psi + \ov{\trho \log(\trho)\tu}\cdot \nabla \psi\right]  dx dt 
\end{align}
for any $\psi\in C_c^\infty(\ico{0}{\tau}\times \RR^3)$, $\psi\geq 0$ $\tPP$-a.s. Note that while $b$ does not satisfy $|b''|\lesssim 1$, the inequality \eqref{eq:6.35} holds by using an approximation argument.

Finally, we apply the DiPerna--Lions theory of renormalized solutions (see \cite[Lemma 3.7, Theorem 11.36, and Lemma 11.13]{FN17}) to the limit equation from Proposition \ref{prop:ep:CE}.
Note that $\tu$ must be extended outside $D$ as $\tu\in L^2(0,T;W^{1,2}(\RR^3))$ using an extension theorem for Sobolev spaces, and the regularity $\trho\in L^\infty(0,T;L^\Gamma(D))$ $\tPP$-a.s. where $\Gamma \geq 2$. Extending $\trho$ to be zero outside $D$, $(1_D\trho,\tu)$ satisfies the renormalized equation of continuity: 
\begin{align*}
    &\pp_t b(1_D \trho) + \rdiv \left(b(1_D \trho) \tu\right) + \left(b'(1_D \trho)1_D \trho - b(1_D \trho)\rdiv  \tu\right) =0 \quad \text{in }\DD'((0,T)\times \RR^3)\quad \tPP\text{-a.s.}
\end{align*}
for any $b\in C(\ico{0}{\infty})\cap C^1((0,\infty)) $ such that 
\begin{align*}
    &\lim_{s\to 0+}(sb'(s)-b(s)) \in \RR, \\ 
    &\left|b'(x) \right|  \leq c s^\lambda\quad \text{if }s\in (1,\infty)\quad \text{for a certain }\lambda\leq \tfrac\Gamma2 - 1.
\end{align*}
By directly calculating this equality, we obtain the renormalized equation of continuity in Remark \ref{rem:ep:renormalized CE}, and by choosing $b(\rho)= \rho \log(\rho)$, we have 
\begin{align}
    \label{eq:6.36}
    &\intT\int_D \trho \rdiv \tu \psi dx dt = \int_D \trho_0 \log(\trho_0)\psi(0)dx + \intT\int_D \left[\trho \log(\trho)\pp_t \psi + \trho \log(\trho) \tu \cdot \nabla \psi \right]dxdt 
\end{align} 
for any $\psi\in C_c^\infty(\ico{0}{\tau}\times \RR^3)$, $\psi\geq 0$ $\tPP$-a.s. Subtracting \eqref{eq:6.36} from \eqref{eq:6.35} and taking $\psi$ as a smooth approximation of the indicator function of $[0,\tau]$, we obtain 
\begin{align}
    \label{eq:6.37}
    &\int_D \left[\ov{\trho \log(\trho)} - \rho \log(\trho)\right](\tau) dx + \inttau \int_D \left[\ov{\trho \rdiv \tu} - \trho \rdiv \tu\right]dx dt\leq 0,
\end{align}
for any $\tau\in [0,T]$ $\tPP$-a.s. As the pressure is a non-decreasing function of the density, with the definition of $\ov{p_\delta(\trho)}$ and Chebyshev's sum inequality at hand, we have 
\begin{align*}
    &\lim_{\varepsilon\to 0} \inttau \int_D \phi p_\delta(\trho_\varepsilon)\trho_\varepsilon dx dt \geq \inttau \int_D \phi \ov{p_\delta(\trho)}\trho dx dt,
\end{align*}
for all $\tau\in [0,T]$, and all $\phi\in C_c^\infty(D)$, $\phi\geq 0$ $\tPP$-a.s.
Thus, the relation \eqref{eq:6.34} yields 
\begin{align*}
    &\inttau \int_D \phi \left[\ov{\trho \rdiv \tu} - \trho \rdiv \tu \right] dx dt\geq 0,
\end{align*}
for all $\tau\in [0,T]$, and all $\phi\in C_c^\infty(D)$, $\phi\geq 0$ $\tPP$-a.s., and by approximating $1_D$ with smooth functions, we obtain 
\begin{align*}
    &\inttau \int_D \left[\ov{\trho \rdiv \tu} - \trho \rdiv \tu \right] dx dt\geq 0,
\end{align*}
for any $\tau\in [0,T]$ $\tPP$-a.s.

Therefore, \eqref{eq:6.37} reduces to 
\begin{align}
    \label{eq:6.38}
    \int_D \left[\ov{\trho \log(\trho)} - \rho \log(\trho)\right](\tau) dx \leq 0
\end{align}
for any $\tau\in [0,T]$ $\tPP$-a.s. As the function $\trho\mapsto \trho\log(\trho)$ is strictly convex, we use Jensen's inequality to obtain 
\begin{align*}
    &\ov{\trho \log(\trho)} - \trho \log(\trho) \geq 0.
\end{align*}
Combining this with \eqref{eq:6.38} yields $\ov{\trho \log(\trho)} = \trho \log(\trho)$ and the marginal of the Young measure $\tnu$ with respect to the $\rho$-variable is the Dirac measure $\delta_{\trho}$. 
Thus, we obtain the strong convergence 
\begin{align}
    \label{eq:6.39}
    \trho_\varepsilon \to \trho\quad \text{in } L^q(0,T;L^2(D))\quad \text{for any }1\leq q<\infty \quad \tPP\text{-a.s.}
\end{align}
With the strong convergence \eqref{eq:6.39} at hand, we are able to identify the limit in the stochastic integral. In view of \eqref{eq:6.26}, we only need the $\tPP$-a.s. convergence 
\begin{align}
    \label{eq:6.40}
    \int_D \trho_\varepsilon F_k(\trho_\varepsilon,\tu)\cdot \bphi dx \to \int_D \trho F_k(\trho,\tu)\cdot \bphi dx \quad \text{a.e. in }(0,T)
\end{align}
for any $\bphi\in C^\infty(\oD)$. Obviously \eqref{eq:6.40} follows immediately from \eqref{eq:6.39} and the Lipschitz continuity of $F_k$. Hence, we infer 
\begin{align*}
    &\ov{\trho F_k(\trho,\tu) } =\trho F_k(\trho,\tu)\quad \text{a.e. in }\tOmega\times (0,T) \times D. 
\end{align*}

Next, we can pass to the limit in the stochastic integral appearing in the energy inequality in exactly the same way as in \cite[Proposition 4.4.15]{BFHbook18}. 

\begin{proposition}
    We have 
    $$
        \inttau \int_D \trho_\varepsilon \FF_\varepsilon(\trho_\varepsilon,\tu_\varepsilon)\cdot \tu_\varepsilon dx d\tW_\varepsilon \to \inttau \int_D \trho \FF (\trho,\tu)\cdot \tu dx d\tW\quad \text{in } L^2(0,T)
    $$
    in probability.
\end{proposition}
\begin{proof}
    We proceed similarly as in Proposition \ref{prop:alpha(conv of stoch int in MEI)} and employ Lemma \ref{lem:conv_of_stoch_int}. It suffices to show that 
    \begin{align}
        \label{eq:6.41}
        \int_D \trho_\varepsilon \FF_\varepsilon (\trho_\varepsilon,\tu_\varepsilon)\cdot \tu_\varepsilon dx \to \int_D \trho \FF(\trho,\tu)\cdot \tu dx \quad \text{in }L^2(0,T;L_2(\fU,\RR))\quad \tPP\text{-a.s.}
    \end{align}
    First, let us denote the approximate stochastic integral by $\fM_\varepsilon$. We observe that 
    \begin{align*}
        \tEE\left[\norm{\fM_\varepsilon}_{L^2(0,T;L_2(\fU,\RR))}^2\right]  &= \tEE\left[\sum_{k=1}^{\infty}\intT \left(\int_D \trho_\varepsilon F_{k,\varepsilon}(\trho_\varepsilon,\tu_\varepsilon)\cdot \tu_\varepsilon dx \right)^2 dt \right]\\ 
        &\lesssim \sum_{k=1}^{\infty} f_k^2 \tEE\left[\intT \left(\int_D (\trho_\varepsilon + \trho_\varepsilon |\tu_\varepsilon|^2) dx \right)^2 dt \right] \lesssim \sum_{k=1}^\infty f_k^2,
    \end{align*}
    using \eqref{eq:6.7}--\eqref{eq:6.8}. Due to the summability of $f_k$, the convergence \eqref{eq:6.41} follows, if we can show
    \begin{align}
        \label{eq:6.42}
        \int_D \trho_\varepsilon F_{k,\varepsilon}(\trho_\varepsilon,\tu_\varepsilon)\cdot \tu_\varepsilon dx \to \int_D \trho F_{k}(\trho,\tu)\cdot \tu dx \quad \text{in }L^2(0,T)\quad \tPP\text{-a.s, for all }k\in\NN.
    \end{align}
    In view of \eqref{eq:6.7}--\eqref{eq:6.8}, this can be relaxed to 
    \begin{align}
        \label{eq:6.43}
        \int_D \trho_\varepsilon F_{k,\varepsilon}(\trho_\varepsilon,\tu_\varepsilon)\cdot \tu_\varepsilon dx \to \int_D \trho F_{k}(\trho,\tu)\cdot \tu dx \quad \text{for a.e. }t\in (0,T)\quad \tPP\text{-a.s, for all }k\in\NN.        
    \end{align}
    In order to show \eqref{eq:6.43}, we observe 
    \begin{align*}
        &\int_D \trho_\varepsilon |\tu_\varepsilon| |F_{k,\varepsilon}(\trho_\varepsilon,\tu_\varepsilon) - F_k(\trho_\varepsilon,\tu_\varepsilon)| dx \\ 
        &\quad \lesssim \int_{\trho_\varepsilon<\varepsilon} \trho_\varepsilon |\tu_\varepsilon| |F_k(\trho_\varepsilon,\tu_\varepsilon)| dx + \int_{|\tu_\varepsilon|>\frac1\varepsilon}\trho_\varepsilon |\tu_\varepsilon| |F_k(\trho_\varepsilon,\tu_\varepsilon)| dx\\ 
        &\quad \lesssim f_k \left[\int_{\trho_\varepsilon<\varepsilon} \left(\trho_\varepsilon |\tu_\varepsilon| + \trho_\varepsilon |\tu_\varepsilon|^2\right) dx + \int_{|\tu_\varepsilon|>\frac1\varepsilon} \left(\trho_\varepsilon |\tu_\varepsilon| + \trho_\varepsilon |\tu_\varepsilon|^2 \right)dx\right] \\ 
        &\quad \lesssim f_k \left[\varepsilon \int_D \left(1+ |\tu_\varepsilon|^2\right)dx + \left(\fL\left(|\tu_\varepsilon|>\tfrac1\varepsilon\right)\right)^{\frac{\gamma-1}{2\gamma}} \left(\int_D |\trho_\varepsilon \tu_\varepsilon|^{\frac{2\gamma}{\gamma +1}}\right)^{\frac{\gamma +1}{2\gamma}} \right. \\ 
        &\quad \quad \quad + \left. \left(\fL\left(|\tu_\varepsilon|>\tfrac1\varepsilon\right)\right)^{\frac{2\gamma - 3}{6\gamma}} \left(\int_D (\trho_\varepsilon |\tu_\varepsilon|^2)^{\frac{6\gamma}{4\gamma +3}}\right)^{\frac{4\gamma +3}{6\gamma}}\right],
    \end{align*}
    where $\fL$ denotes the Lebesgue measure. From \eqref{eq:6.7}, \eqref{eq:6.8} and Chebyshev's inequality, the right-hand side vanishes as $\varepsilon\to 0$ $\tPP$-a.s. for a.e. $t\in (0,T)$. Consequently, it suffices to prove 
    \begin{align}
        \label{eq:6.44}
        \int_D \trho_\varepsilon F_{k}(\trho_\varepsilon,\tu_\varepsilon)\cdot \tu_\varepsilon dx \to \int_D \trho F_{k}(\trho,\tu)\cdot \tu dx \quad \text{for a.e. }t\in (0,T)\quad \tPP\text{-a.s, for all }k\in\NN.        
    \end{align}
    Note that, as a consequence of the strong convergence \eqref{eq:6.39}, by the same argument as in the proof of Lemma \ref{lem:alpha(conv rho u otimes u)}, we have (up to a subsequence)
    \begin{align*}
        &\sqrt{\trho_\varepsilon}\tu_\varepsilon \to \sqrt{\trho }\tu \quad \text{a.e. in }(0,T)\times D\quad \tPP\text{-a.s.}  
    \end{align*}
    Let us fix $\kappa>0$. By Egorov's theorem and \eqref{eq:6.39}, there exists a measurable set $\OO_\kappa\subset \tOmega\times (0,T)\times D$ such that $\tPP\otimes \fL([\tOmega\times (0,T)\times D]\backslash \OO_\kappa)<\kappa$ and 
    \begin{align}
        \label{eq:6.45}
        \sqrt{\trho_\varepsilon}\tu_\varepsilon \to \sqrt{\trho }\tu, \quad  \trho_\varepsilon \to \trho \quad \text{uniformly in }\OO_\kappa, 
    \end{align}
    where $\fL$ denotes the Lebesgue measure on $(0,T)\times D$.
    Finally, we consider the sets 
    \begin{align*}
        &\OO_\kappa^1 = \left\{ (\omega,t,x)\in \OO_\kappa: \trho<\kappa \right\}, \\ 
        &\OO_\kappa^2 = \left\{ (\omega,t,x)\in \OO_\kappa : \trho\geq \kappa \right\}.
    \end{align*}
    As a consequence of \eqref{eq:6.45}, we can choose $\varepsilon$ small enough such that 
    \begin{align*}
        &\trho_\varepsilon \leq 2\kappa,\quad \text{in }\OO_\kappa^1, \quad \trho_\varepsilon \geq \frac\kappa2\quad \text{in }\OO_\kappa^2.  
    \end{align*}
    With these preparations at hand, we obtain 
    \begin{align*}
        &\tEE\left[\intT \int_D \left|\trho_\varepsilon F_k(\trho_\varepsilon ,\tu_\varepsilon)\cdot \tu_\varepsilon - \trho F_k(\trho,\tu)\cdot \tu \right|dx dt \right] \\ 
        &\quad = \int_{\OO_\kappa^c} \left|\trho_\varepsilon F_k(\trho_\varepsilon ,\tu_\varepsilon)\cdot \tu_\varepsilon - \trho F_k(\trho,\tu)\cdot \tu \right| dx dt d\tPP + \int_{\OO_\kappa^1} \left|\trho_\varepsilon F_k(\trho_\varepsilon ,\tu_\varepsilon)\cdot \tu_\varepsilon - \trho F_k(\trho,\tu)\cdot \tu \right| dx dt d\tPP \\ 
        &\quad \quad + \int_{\OO_\kappa^2} \left|\trho_\varepsilon F_k\left(\trho_\varepsilon, \tfrac{\sqrt{\trho_\varepsilon} \tu_\varepsilon}{\sqrt{\trho_\varepsilon}}\right)\cdot \tfrac{\sqrt{\trho_\varepsilon} \tu_\varepsilon}{\sqrt{\trho_\varepsilon}} - \trho F_k\left(\trho,\tfrac{\sqrt{\trho} \tu}{\sqrt{\trho}}\right)\cdot \tfrac{\sqrt{\trho} \tu}{\sqrt{\trho}} \right| dx dt d\tPP \\ 
        &\quad =: (I)_\kappa^1 + (I)_\kappa^2 + (I)_\kappa^3.
    \end{align*}
    By H\"older's inequality, we have 
    \begin{align*}
        (I)_\kappa^1 &\lesssim \int_{\OO_\kappa^c} \left(\trho_\varepsilon |\tu_\varepsilon| + \trho_\varepsilon |\tu_\varepsilon|^2 + \trho|\tu| + \trho|\tu|^2\right)dx dt d\tPP \\ 
        &\lesssim \int_{\OO_\kappa^c} \left(\trho_\varepsilon + \trho + \trho_\varepsilon |\tu_\varepsilon|^2 + \trho|\tu|^2\right)dx dt d\tPP \\ 
        &\lesssim \left(\tPP\times \fL(\OO_\kappa^c)\right)^{\frac{2\gamma -3}{6\gamma}} \left(\tEE\left[\intT \int_D \left((\trho_\varepsilon)^{\frac{6\gamma}{4\gamma + 3}} + (\trho)^{\frac{6\gamma}{4\gamma + 3}} + (\trho_\varepsilon |\tu_\varepsilon|^2)^{\frac{6\gamma}{4\gamma + 3}} + (\trho|\tu|^2)^{\frac{6\gamma}{4\gamma + 3}}\right)dx dt\right]\right)^{\frac{4\gamma +3}{6\gamma}} \\ 
        &\lesssim \kappa^{\frac{2\gamma - 3}{6\gamma}},
    \end{align*}
    due to the assumption on the $F_k$, and \eqref{eq:6.7}--\eqref{eq:6.8}. The second integral can be estimated by 
    \begin{align*}
        &(I)_\kappa^2 \lesssim \kappa \tEE\left[\intT\int_D (1 + |\tu_\varepsilon|^2)dx dt\right] + \kappa \tEE\left[\intT\int_D (1 + |\tu|^2)dx dt\right]\lesssim \kappa,
    \end{align*}
    due to \eqref{eq:6.9}. In view of \eqref{eq:6.45}, the continuity of $F_k$, and the lower bounds for $\trho$ and $\trho_\varepsilon$ in $\OO_\kappa^2$, the last integral vanishes as $\varepsilon\to 0$. Since $\kappa$ was arbitrary, we obtain \eqref{eq:6.44}, and, consequently, \eqref{eq:6.41}. Finally, we obtain the claim by Lemma \ref{lem:conv_of_stoch_int}.
\end{proof}
Finally, we pass to the limit in the momentum and energy inequality \eqref{eq:approx_MEI_m}. This proof can be carried out by recalling Remark \ref{rem:ep:conv of each term of ME} and proceeding similarly to Proposition \ref{prop:alpha:MEI}. 
Therefore, we obtain the following result.
\begin{proposition}
    $(\trho,\tu,\tW)$ satisfies the approximate momentum and energy inequality \eqref{eq:approx_MEI_ep} for all $\tau\in [0,T]$, all $\phi\in C_c^\infty(\ico{0}{\tau})$, and all $\bphi\in C^\infty(\oD)$ with $\bphi\cdot \rn |_{\pD} = 0$ $\tPP$-a.s.
\end{proposition} 
The proof of Theorem \ref{thm:sol_ep} is hereby complete.

\section{Vanishing artificial pressure limit}\label{sec:Vanishing artificial pressure limit}
Our ultimate goal is to show that as $\delta\to 0$ in \eqref{eq:approx_CE_ep}--\eqref{eq:approx_MEI_ep}, the limit is a solution in the sense of Definition \ref{Def:1.1}.
We recall the definition of a dissipative martingale solution given in Definition \ref{Def:1.1}.
\begin{definition}
	\label{def:delta:sol}
	Let $\Lambda$ be a Borel probability measure on $L^1(D)\times L^1(D)$ such that 
	\begin{equation}
		\Lambda \left\{ \rho\geq 0 \right\}=1,\quad \int_{L^1_x\times L^1_x}\left|\int_D \left[\frac12\frac{|q|^2}{\rho}+P(\rho)\right]dx \right|^rd\Lambda(\rho,q)<\infty,
	\end{equation}
	where the pressure potential is given by 
	$$
		P(\rho) = \rho \int_1^\rho \frac{p(z)}{z^2}dz
	$$
	and $r\geq 1$, and we assume that $g\in L^2((0,T)\times \pD)$ satisfies $g\geq 0$.
	The quantity $((\Omega,\FFF,(\FFF_t)_{t\geq 0},\PP),\rho,u,W)$ is called a \textit{dissipative martingale solution} to the Navier--Stokes system \eqref{eq:1.1}--\eqref{eq:1.5} with the initial law $\Lambda$ if:
	\begin{itemize}
		\item[(1)] $(\Omega,\FFF,(\FFF_t)_{t\geq 0},\PP)$ is a stochastic basis with a complete right-continuous filtration;
		\item[(2)] $W$ is a cylindrical $(\FFF_t)$-Wiener process on $\fU$;
		\item[(3)] the density $\rho$ is an $(\FFF_t)$-progressively measurable stochastic process such that 
		$$
			\rho \geq 0,\quad \rho\in C_w([0,T];L^\gamma(D))\quad \PP\text{-a.s.;}
		$$
		\item[(4)] the velocity $u$ is an $(\FFF_t)$-adapted random distribution such that  
		$$
			u\in L^2(0,T;W_{\rn }^{1,2}(D))\quad \PP\text{-a.s.;}
		$$ 
		\item[(5)] the momentum $\rho u$ is an $(\FFF_t)$-progressively measurable stochastic process such that 
		$$
			\rho u\in C_w([0,T];L^{\frac{2\gamma}{\gamma + 1}}(D))	\quad \PP\text{-a.s.;}
		$$ 
		\item[(6)] there exists an $L^1(D)\times L^1(D)$-valued $\FFF_0$-measurable random variable $(\rho_0,u_0)$ such that $\rho_0 u_0\in L^1(D)$ $\PP$-a.s., $\Lambda = \LL[\rho_0,\rho_0 u_0]$ and $(\rho_0,\rho_0 u_0) = (\rho(0),\rho u(0))$ $\PP$-a.s.;
		\item[(7)] the equation of continuity 
		\begin{equation}
			\label{eq:7.2}
			-\intT \pp_t\phi \int_D \rho \psi dxdt = \phi(0)\int_D \rho_0\psi dx + \intT \phi \int_D \rho u\cdot \nabla\psi dxdt
		\end{equation}
		holds for all $\phi\in C_c^\infty(\ico{0}{T})$ and all $\psi\in C_c^\infty\left(\RR^3\right)$ $\PP$-a.s.;
		\item[(8)] the interior momentum equation 
		\begin{align}
			&-\intT \pp_t \phi\int_D \rho u\cdot \bphi dxdt-\phi(0)\int_D \rho_0 u_0 \cdot \bphi dx  \notag\\
			&\quad = \intT \phi\int_D[\rho u\otimes u:\nabla\bphi + p(\rho)\rdiv\bphi]dxdt - \intT \phi\int_D \SSS(\nabla u):\nabla\bphi dxdt + \intT \phi\int_D \GG(\rho,\rho u)\cdot \bphi dx dW \label{eq:7.3}
		\end{align}
		holds for all $\phi\in C_c^\infty \left(\ico{0}{T}\right)$ and all $\bphi\in C_c^\infty(D)$ $\PP$-a.s.;
		\item[(9)] the momentum and energy inequality 
		\begin{align}
		   &\int_D \left[\frac12\rho_0 |u_0|^2 + P(\rho_0)\right]dx - \int_D\left[\frac12 \rho \left|u \right|^2 + P(\rho)\right](\tau)dx -\inttau \pp_t\phi \int_D\rho u\cdot \bphi dxdt - \phi(0)\int_D \rho_0 u_0 \cdot \bphi dx  \notag\\
		   &\quad -\inttau \phi \int_D(\rho u\otimes u):\nabla \bphi dx dt- \inttau \phi\int_D p(\rho)\rdiv \bphi dx dt + \inttau \int_D \SSS(\nabla u):\nabla(\phi \bphi-u)dx dt   \notag\\ 
		   &\quad +\frac12 \sum_{k=1}^{\infty}\inttau \int_D \rho^{-1}\left|G_k(\rho,\rho u) \right|^2dx dt - \inttau \int_D \GG(\rho,\rho u)\cdot (\phi\bphi - u)dxdW \notag\\ 
		   &\quad + \inttau \int_{\pD} g\left|\phi\bphi \right|-g\left|u \right|d\Gamma dt\geq 0 \label{eq:7.4}
		\end{align}
		holds for all $\tau\in [0,T],$ for all $ \phi\in C_c^\infty(\ico{0}{\tau})$ and for all $\bphi\in C^\infty(\overline{D})$ with $\bphi\cdot \rn|_{\pD}=0$ $\PP$-a.s.; 
		\item[(10)] if $b\in C^1(\RR)$ such that $b'(z)=0$ for all $z\geq M_b$, then, for all $\phi\in C_c^\infty(\ico{0}{T})$ and all $\psi\in C_c^\infty(\RR^3)$, we have $\PP$-a.s. 
		\begin{align*}
			-\intT \pp_t \phi \int_D b(\rho)\psi dxdt &= \phi(0) \int_D b(\rho_0)\psi dx + \intT \phi\int_D b(\rho)u\cdot \nabla \psi dxdt \\ 
			&\quad - \intT \phi \int_D \left(b'(\rho) \rho - b(\rho) \right) \rdiv u \psi dx dt.
		\end{align*}
	\end{itemize}
\end{definition}
We obtain the following result.
\begin{theorem}
    \label{thm:delta:sol}
	Let $T>0$ and let $D\subset \RR^3$ be a bounded domain of class $C^{2+\nu}$ for some $\nu>0$. Let $\gamma>\frac32$ and let $\Lambda$ be a Borel probability measure defined on $L^1(D)\times L^1(D)$ such that 
	\begin{equation}
		\Lambda\left\{ \rho\geq 0 \right\}=1,\quad \Lambda\left\{ \underline{\rho}\leq \int_D \rho dx \leq \overline{\rho} \right\}=1, \label{eq:7.5}
	\end{equation}
	for some deterministic constants $\urho,\orho>0$ and 
	\begin{equation}
		\int_{L^1_x\times L^1_x}\left|\int_D \left[\frac12\frac{|q|^2}{\rho}+P(\rho)\right]dx \right|^rd\Lambda(\rho,q) <\infty, \label{eq:7.6}
	\end{equation}
	for some $r\geq 4$. Assume that the diffusion coefficients $\GG=\GG(x,\rho,q)$ are continuously differentiable satisfying \eqref{eq:1.6}--\eqref{eq:1.7} and the modulus of friction $g\in L^2((0,T)\times \pD)$ satisfies the conditions 
	\begin{align*}
		&g\geq 0,\quad \intT \int_{\pD}  g d\Gamma dt >0.
	\end{align*} 
    Then there exists a dissipative martingale solution to \eqref{eq:1.1}--\eqref{eq:1.5} in the sense of Definition \ref{def:delta:sol}.
	Furthermore, this solution satisfies the following:
	\begin{align*}
		&\EE\left[\sup_{t\in [0,T]}\norm{\rho(t)}_{L^\gamma(D)}^{\gamma r}\right] + \EE\left[\norm{u}_{L^2(0,T;W^{1,2}(D))}^r \right]< \infty,\\ 
		&\EE\left[\sup_{t\in [0,T]}\norm{\rho |u|^2(t)}_{L^1(D)}^r\right] + \EE\left[\sup_{t\in [0,T]}\norm{\rho u (t)}_{L^{\frac{2\gamma}{\gamma +1}}(D)}^{\frac{2\gamma}{\gamma +1}r}\right]<\infty.
	\end{align*}
\end{theorem}
\subsection{Uniform estimates} \label{sec:delta:Uniform estimates}
Let $(\rho, u)$ be a dissipative martingale solution constructed in Theorem \ref{thm:sol_ep}. As in Section \ref{sec:m:Uniform estimates}, choosing $\phi=0$ in the momentum and energy inequality \eqref{eq:approx_MEI_ep} yields the standard energy inequality:
\begin{align*}
    &\int_D \left[\frac 1 2 \rho |u|^2 + P_\delta(\rho)\right](\tau)dx + \inttau \int_D \SSS(\nabla u):\nabla u dxdt +\inttau \int_{\pD} g |u| d\Gamma dt \notag\\ 
    &\quad \leq \int_D \left[\frac 1 2 \rho_0 |u_0|^2 + P_\delta(\rho_0)\right]dx + \frac 1 2 \sum_{k=1}^{\infty} \inttau \int_D \rho \left|F_{k}(\rho,u) \right|^2 dx dt + \inttau \int_D \GG(\rho,\rho u) \cdot u dx dW   
\end{align*}
for all $\tau\in [0,T]$ $\PP$-a.s., and using Gronwall's lemma together with Remark \ref{rem:ep:integrability_of_sol} yields the uniform estimates:
\begin{align}
    &\EE\left[\left|\sup_{\tau\in [0,T]} \int_D \left[\frac 1 2 \rho |u|^2 + P_\delta(\rho)\right](\tau)dx \right|^r\right] \notag \\ 
    &\quad +\EE\left[\left|\intT \int_D \left[\SSS(\nabla u):\nabla u \right]dxdt \right|^r\right] + \EE\left[\left|\intT \int_{\pD} g |u| d\Gamma dt \right|^{r}\right] \lesssim  c(r), \label{eq:7.7} \\ 
    &\EE\left[\norm{\rho}_{L_t^\infty L_x^\gamma}^{\gamma r}\right] + \EE\left[\delta^r \norm{\rho}_{L_t^\infty L_x^\Gamma}^{\Gamma r}\right] \lesssim c(r), \label{eq:7.8}\\ 
    & \EE\left[\norm{\rho |u|^2}_{L_t^\infty L_x^1}^{r}\right] + \EE\left[\norm{\rho u}_{L_t^\infty L_x^{\frac{2\gamma}{\gamma +1}}}^{\frac{2\gamma}{\gamma +1}r}\right] \lesssim c(r), \label{eq:7.9} \\
    &\EE\left[\left| \intT \norm{u}_{W_x^{1,2}}^2 dt \right|^{r/2} \right] \lesssim  c(r) \label{eq:7.10}
\end{align}
uniformly in $\delta$, where 
$$
    c(r) = \EE\left[\left|\int_D \left[\frac 1 2 \rho_0 |u_0|^2 + P_\delta(\rho_0)\right]dx \right|^r\right] + 1,\quad r\geq 4.
$$
Finally, it follows from \eqref{eq:approx_CE_ep} that 
\begin{align}
    \label{eq:7.11}
    \norm{\rho(\tau)}_{L_x^1} = \norm{\rho_0}_{L_x^1} \leq \orho,\quad \tau\in [0,T].
\end{align}
Next, in order to derive refined estimates of the pressure, we apply a method analogous to Section \ref{sec:ep:Uniform estimates}. We consider 
$$
    \Phi = \Bog \left[b(\rho) - (b(\rho))_D\right]
$$
as test function in the momentum equation \eqref{eq:approx_ME_ep}, where $b(\rho) = \rho^\beta$ and $0<\beta<1/3$ will be chosen below. Since $\rho$ satisfies the renormalized equation of continuity (see Remark \ref{rem:ep:renormalized CE}), we have 
\begin{align}
    \label{eq:7.12}
    d \Phi = \Bog \left[ - \rdiv \left(b(\rho) u\right) + (b(\rho)-b'(\rho)\rho)\rdiv u - \left( (b(\rho)-b'(\rho)\rho)\rdiv u\right)_D\right] dt.
\end{align}
Exactly as in Section \ref{sec:ep:Uniform estimates}, we deduce 
\begin{align}
    \inttau \int_D p_\delta(\rho)(b(\rho) - (b(\rho))_D) dx dt &= \left[\int_D \rho u \cdot \Phi dx \right]_{t=0}^{t=\tau} - \inttau \int_D \left[\rho u \otimes u - \SSS(\nabla u) \right]:\nabla \Phi dx dt \notag \\ 
    &\quad + \inttau \int_D \rho u\cdot \pp_t \Phi dx dt - \inttau \int_D \GG(\rho,\rho u)\cdot \Phi dx dW\notag \\ 
    &=: \sum_{j=1}^{4}I_j. \label{eq:7.13}
\end{align}
By estimating each term on the right-hand side in \eqref{eq:7.13}, we will find the desired uniform bound:
\begin{align}
    \label{eq:7.14}
    \EE\left[\left|\intT \int_D \left(p(\rho) + \delta (\rho + \rho^\Gamma)\right)\rho^\beta dx dt \right|^{r/2}\right]\lesssim c(\orho)c(r),
\end{align}
for a certain $\beta>0$. To this end, we first note that 
$$
    \norm{\Phi}_{L_x^\infty} \lesssim \norm{\rho^\beta}_{L_x^q},\quad q>3,
$$
and \eqref{eq:7.11}. Consequently, we have 
\begin{align*}
    \left|I_1 \right|   &\lesssim \sup_{\tau\in [0,T]}\norm{\sqrt{\rho}}_{L_x^2} \norm{\sqrt{\rho} u}_{L_x^2}\norm{\Phi}_{L_x^\infty} \\ 
    &\lesssim c(\orho) \sup_{\tau\in [0,T]} \norm{\rho |u|^2}_{L_x^1}^{1/2} \norm{\rho^\beta}_{L_x^\Gamma},
\end{align*}
and 
\begin{align*}
    &\left|I_2 \right|   \lesssim \sup_{\tau\in [0,T]} \norm{\rho^\beta}_{L_x^q}^4 + \left(\intT \norm{u}_{W_x^{1,2}}dt \right)^2.
\end{align*}
We also have
\begin{align*}
    \norm{\rho u\cdot \Bog \left(\rdiv\left(\rho^\beta u\right)\right)}_{L^1_{t,x}} &\lesssim \sup_{\tau\in [0,T]}\norm{\rho}_{L_x^\gamma} \intT \norm{u}_{L_x^q}\norm{\Bog\left(\rdiv\left(\rho^\beta u\right)\right)}_{L_x^q}dt \\ 
    &\lesssim \sup_{\tau\in [0,T]} \norm{\rho}_{L_x^\gamma} \sup_{\tau\in [0,T]}\norm{\rho^\beta}_{L_x^p} \intT \norm{u}_{L_x^q}\norm{u}_{L_x^r}dt,
\end{align*}
where 
$$
    \frac1\gamma +\frac2q = 1,\quad \frac1q =\frac1p + \frac1r.
$$
As $\gamma >3/2$, we can choose $q<6$ and $r\leq 6$. Similarly, 
\begin{align*}
    &\norm{\rho u\cdot \Bog\left((b(\rho)-b'(\rho)\rho)\rdiv u -\left((b(\rho)-b'(\rho)\rho)\rdiv u\right)_D\right)}_{L^1_{t,x}} \\ 
    &\quad \lesssim \sup_{\tau\in [0,T]}  \norm{\rho}_{L_x^\gamma}\intT \norm{u}_{L_x^p} \norm{\Bog \left((b(\rho)-b'(\rho)\rho)\rdiv u -\left((b(\rho)-b'(\rho)\rho)\rdiv u\right)_D\right)}_{L_x^q}dt,\quad \frac1\gamma + \frac1p +\frac1q =1.
\end{align*}
We use the embedding $W^{1,r}(D)\hookrightarrow L^q(D)$, $q\leq \frac{3r}{3-r}$ if $r<3$, to obtain 
\begin{align*}
    &\norm{\Bog \left((b(\rho)-b'(\rho)\rho)\rdiv u -\left((b(\rho)-b'(\rho)\rho)\rdiv u\right)_D\right)}_{L_x^q} \\ 
    &\quad \lesssim \norm{(b(\rho)-b'(\rho)\rho)\rdiv u -\left((b(\rho)-b'(\rho)\rho)\rdiv u\right)_D}_{L_x^r} \\ 
    &\quad \lesssim \norm{\rho^\beta}_{L_x^s}\norm{\nabla u}_{L_x^2},\quad \frac1s + \frac12 = \frac1r.
\end{align*}
Note that we can choose $q<6,p\leq 6$, and $r<2$. Therefore, by using \eqref{eq:7.8}--\eqref{eq:7.11}, we obtain 
\begin{align*}
    &\EE\left[\left|I_j \right|^{r/2}\right]  \lesssim c(r)c(\orho),\quad j=1,2,3.
\end{align*}
Finally, by using the Burkholder--Davis--Gundy inequality, we have 
\begin{align*}
    &\EE\left[\left|I_4 \right|^r\right]  \lesssim \EE\left[\left(\inttau \sum_{k=1}^{\infty}\left|\int_D \rho F_k(\rho,u)\cdot \Phi dx \right|^2 dt\right)^{r/2}\right],
\end{align*}
where
\begin{align*}
    \left|\int_D \rho F_k(\rho,u)\cdot \Phi dx \right| &\lesssim f_k\norm{\Phi}_{L_x^\infty}\int_D \left(\rho + \rho |u|\right)dx \\ 
    &\lesssim c(\orho)f_k\int_D \left(\rho + \rho |u|\right)dx.
\end{align*}
We conclude 
\begin{align*}
    &\EE\left[\left|I_4 \right|^{r/2}\right]  \lesssim c(\orho)c(r).
\end{align*}

\subsection{Asymptotic limit}\label{sec:delta:Asymptotic limit}
Let $\Lambda$ be the initial law given by Theorem \ref{thm:delta:sol}, and let $(\rho_0,q_0)$ be a random variable having the law $\Lambda$. As in Section \ref{sec:ep:Asymptotic limit}, there exists a sequence $([\rho_{0,\delta},q_{0,\delta}])_{\delta\in (0,1)}$ such that 
\begin{align}
    &\rho_{0,\delta} \in L^\Gamma (D),\quad \rho_{0,\delta}>0,\quad 0<\frac{\urho}{2} \leq \int_D \rho_{0,\delta} dx \leq 2\orho\ \PP\text{-a.s.,} \notag \\ 
    &\rho_{0,\delta} \to \rho_0 \quad \text{in } L^\gamma (D),\quad q_{0,\delta} \to q_0 \quad \text{in } L^1(D) \  \PP\text{-a.s.,} \label{eq:7.15}\\
    &\EE\left[\left|\int_D \left[\frac12 \frac{|q_{0,\delta}|^2}{\rho_{0,\delta}} + P_\delta (\rho_{0,\delta})\right]dx \right|^r\right] \lesssim 1 \label{eq:7.16}
\end{align}
for some $r\geq 4$ uniformly in $\delta\to 0$. Also, 
\begin{align}
    \label{eq:7.17}
    \int_D \left[\frac12 \frac{|q_{0,\delta}|^2}{\rho_{0,\delta}} + P_\delta (\rho_{0,\delta})\right]dx \to \int_D \left[\frac12 \frac{|q_{0}|^2}{\rho_{0}} + P(\rho_{0})\right]dx,
\end{align}
as $\delta\to 0$ $\PP$-a.s. This implies 
\begin{align}
    \frac{q_{0,\delta}}{\sqrt{\rho_{0,\delta}}} \to \frac{q_0}{\sqrt{\rho_0}}\quad \text{in } L^2(D)\ \PP\text{-a.s.}\label{eq:7.18}
\end{align}
Therefore, applying Theorem \ref{thm:sol_ep} to the laws $\PP\circ (\rho_{0,\delta},q_{0,\delta})^{-1}$ on $L^1(D)\times L^1(D)$ yields for every $\delta\in (0,1)$ a multiplet
$$
    ((\Omega^\delta,\FFF^\delta,(\FFF_t)^\delta_{t\geq 0},\PP^\delta),\overline{\rho}_{0,\delta},\overline{u}_{0,\delta},\rho_\delta,u_\delta,W_\delta),
$$
which is a weak martingale solution in the sense of Definition \ref{def:sol_ep}. Moreover, in view of \eqref{eq:7.15}--\eqref{eq:7.18}, the laws $\Lambda_\delta := \PP^\delta\circ (\overline{\rho}_{0,\delta}, \overline{\rho}_{0,\delta} \overline{u}_{0,\delta})^{-1}$ on $L^1(D)\times L^1(D)$ satisfy \eqref{eq:7.5} and \eqref{eq:7.6} uniformly in $\delta$, and 
\begin{align}
    \label{eq:7.19}
    \Lambda_\delta \to \Lambda\quad \text{weakly on } L^1(D)\times L^1(D).
\end{align}
As in Section \ref{sec:m:Asymptotic limit}, we may assume without loss of generality that 
$$
    (\Omega^\delta,\FFF^\delta,\PP^\delta) = ([0,1],\overline{\fB([0,1])},\fL),\quad \delta\in (0,1)
$$
and that 
$$
    \FFF^\delta_t = \sigma \left(\sigma_t[\rho_\delta]\cup \sigma_t [u_\delta]\cup \sigma_t[W_\delta]\right),\quad t\in [0,T].
$$
We consider a path space analogous to Section \ref{sec:ep:Asymptotic limit}, that is 
$$
    \XX = \XX_{\rho_0}\times \XX_{q_0}\times \XX_{\frac{q_0}{\sqrt{\rho_0}}}\times \XX_\rho \times \XX_{\rho u}\times \XX_{u}\times \XX_W \times \XX_E \times \XX_\nu, 
$$
where 
\begin{align*}
    &\XX_{\rho_0} = L^\gamma(D),\quad \XX_{q_0} = L^1(D),\quad \XX_{\frac{q_0}{\sqrt{\rho_0}}} = L^2(D), \\ 
    &\XX_\rho = \left[L^{\gamma + \beta}((0,T)\times D),w\right]\cap C_w([0,T];L^\gamma(D)), \\ 
    &\XX_{\rho u} = C_w([0,T];L^{\frac{2\gamma }{\gamma +1}}(D)) \cap C([0,T];W^{-k,2}(D)), \\ 
    &\XX_u = \left[L^2(0,T;W^{1,2}_{\rn }(D)),w\right], \\ 
    &\XX_W = C([0,T];\fU_0), \\ 
    &\XX_E = \left[L^\infty(0,T;\MMM_b(D)),w^\ast\right], \\ 
    &\XX_\nu = (L_{w^\ast}^\infty((0,T)\times D;\PPP(\RR^{13})),w^\ast),
\end{align*}
for certain $k\in \NN$, where $\beta$ is chosen so that \eqref{eq:7.14} holds. The energy $E_\delta(\rho_\delta,u_\delta)$ is defined as 
$$
    E_\delta(\rho_\delta,u_\delta) (\tau) = \left[\frac12 \rho_\delta\left|u_\delta \right|^2 + P_\delta(\rho_\delta)\right] (\tau) + \inttau \left[\rho_\delta^\gamma + \delta \rho_\delta^\Gamma \right] \rho_\delta^\beta dt.
$$
\begin{proposition}
    \label{prop:delta:tightness}
    The family of probability measures
    $$
        \left\{ \LL\left[\overline{\rho}_{0,\delta}, \overline{\rho}_{0,\delta} \overline{u}_{0,\delta},\frac{\overline{\rho}_{0,\delta} \overline{u}_{0,\delta}}{\sqrt{\overline{\rho}_{0,\delta}}},\rho_\delta,,\rho_\delta u_\delta,u_\delta,W_\delta,E_\delta (\rho_\delta,u_\delta),\delta_{(\rho_\delta,u_\delta,\nabla u_\delta)} \right]:\delta\in (0,1) \right\}
    $$
    is tight on $\XX$. 
\end{proposition}
\begin{proof}
    The only change from Section \ref{sec:ep:Asymptotic limit} is the proof of tightness for $\left\{ \LL[\rho_\delta u_\delta]:\delta\in (0,1) \right\}$. We proceed as in Proposition \ref{prop:ep:tightness} and decompose $\rho_\delta u_\delta$ into two parts $\rho_\delta u_\delta(t) = Y^\delta(t) + Z^\delta(t)$, where 
    \begin{align*}
        &Y^\delta(\tau) = \ov{\rho}_{0,\delta}\ov{u}_{0,\delta} - \inttau \rdiv (\rho_\delta u_\delta\otimes u_\delta)dt + \inttau \rdiv \SSS(\nabla u_\delta )dt - \inttau \nabla p(\rho_\delta)dt + \inttau \GG(\rho_\delta,\rho_\delta u_\delta)dW, \\ 
        &Z^\delta(\tau) = - \delta \inttau \nabla(\rho_\delta + \rho_\delta^\Gamma)dt.   
    \end{align*}
    By the same way as in the proof of Proposition \ref{prop:ep:tightness}, we obtain the H\"older continuity of $Y^\delta$, that is, there exist $\kappa>0$ and $l>0$ such that 
    \begin{align*}
        &\EE\left[\norm{Y^\delta}_{C^\kappa([0,T];W^{-l,2}(D))}\right]  \leq C.
    \end{align*}
    Next, we show that the family of the laws $\left\{ \LL[Z^\delta]:\delta\in (0,1) \right\}$ is tight on 
    \begin{align*}
        &C([0,T];W^{-1,\frac{\Gamma + \beta}{\Gamma}}(D)),
    \end{align*}
    and the conclusion follows similarly to Proposition \ref{prop:ep:tightness}. Due to \eqref{eq:7.14}, we have (up to a subsequence)
    \begin{align*}
        &\delta(\rho_\delta + \rho_\delta^\Gamma)\to 0 \quad \text{in }L^{\frac{\Gamma + \beta}{\Gamma}}((0,T)\times D)\quad \PP\text{-a.s.}  
    \end{align*}
    Hence, we obtain  
    \begin{align*}
        &\delta\nabla(\rho_\delta + \rho_\delta^\Gamma)\to 0 \quad \text{in }L^{\frac{\Gamma + \beta}{\Gamma }}(0,T;W^{-1, \frac{\Gamma + \beta}{\Gamma}}(D))\quad \PP\text{-a.s.,}    
    \end{align*}
    and 
    \begin{align*}
        &Z^\delta \to 0 \quad \text{in }C([0,T];W^{-1, \frac{\Gamma + \beta}{\Gamma}}(D))\quad \PP\text{-a.s.}     
    \end{align*}
    This leads to the convergence in law 
    \begin{align*}
        &Z^\delta \overset{d}{\longrightarrow} 0 \quad \text{on }  C([0,T];W^{-1, \frac{\Gamma + \beta}{\Gamma}}(D)),
    \end{align*}
    and the claim follows. 
\end{proof}
Consequently, we may apply Jakubowski's theorem, Theorem \ref{thm:Jakubowski--Skorokhod} as well as \cite[Corollary 2.8.3]{BFHbook18} to obtain the following. 
\begin{proposition}
    \label{prop:delta:repr}
        There exists a complete probability space $\tprobsp$ with $\XX$-valued Borel measurable random variables $(\trho_{0,\delta},\tq_{0,\delta},\tk_{0,\delta},\trho_\delta,\tq_\delta,\tu_\delta,\tW_\delta,\tE_\delta,\tnu_\delta),\delta\in (0,1)$, as well as 
        $(\trho_{0},\tq_{0},\tk_{0},\trho,\tq,\tu,\tW,\tE,\tnu)$ such that (up to a subsequence):
        \begin{itemize}
            \item[(1)] the laws $\LL[(\trho_{0,\delta},\tq_{0,\delta},\tk_{0,\delta},\trho_\delta,\tq_\delta,\tu_\delta,\tW_\delta,\tE_\delta,\tnu_\delta)]$ and $\LL[\overline{\rho}_{0,\delta}, \overline{\rho}_{0,\delta} \overline{u}_{0,\delta},(\overline{\rho}_{0,\delta} \overline{u}_{0,\delta}) / \sqrt{\overline{\rho}_{0,\delta}}, \rho_\delta, \rho_\delta u_\delta, u_\delta, W_\delta, \linebreak E_\delta(\rho_\delta, u_\delta), \delta_{(\rho_\delta,u_\delta,\nabla u_\delta)} ]$ 
            coincide on $\XX$. In particular, 
            \begin{align*}
                &\trho_{0,\delta} = \trho_\delta(0),\quad \tq_{0,\delta} = \trho_\delta \tu_\delta(0),\quad \tk_{0,\delta} = \frac{\tq_{0,\delta}}{\sqrt{\trho_{0,\delta}}} = \frac{\trho_\delta \tu_\delta(0)}{\sqrt{\trho_\delta(0)}}, \\ 
                &\tq_\delta = \trho_\delta \tu_\delta,\quad \tE_\delta = E_\delta(\trho_\delta,\tu_\delta),\quad \tnu_\delta = \delta_{(\trho_\delta,\tu_\delta,\nabla \tu_\delta)},   
            \end{align*}
            $\tPP$-a.s., as well as 
            \begin{align}
                \EE\left[\left|\intT \int_D \left[\trho_\delta^\gamma + \delta \trho_\delta^\Gamma \right]\trho_\delta^\beta dx dt \right|^{r/2}\right] + \EE\left[\left|\sup_{t\in [0,T]}\int_D \left[\frac12 \trho_\delta |\tu_\delta|^2 + P_\delta(\trho_\delta)\right]dx \right|^{r} \right] \lesssim c(r);
            \end{align}
            \item[(2)] the law of $(\trho_{0},\tq_{0},\tk_{0},\trho,\tq,\tu,\tW,\tE,\tnu)$ on $\XX$ is a Radon measure;
            \item[(3)] $(\trho_{0,\delta},\tq_{0,\delta},\tk_{0,\delta},\trho_\delta,\tq_\delta,\tu_\delta,\tW_\delta,\tE_\delta,\tnu_\delta)$ converges in the topology of $\XX$ $\tPP$-a.s. to $(\trho_{0},\tq_{0},\tk_{0},\trho,\tq,\tu,\tW,\tE,\tnu)$, i.e., 
            \begin{align}
                \begin{aligned}
                    \trho_{0,\delta} &\to \trho_0 \quad \text{in } L^\gamma (D), \\ 
                    \tq_{0,\delta}&\to \tq_0 \quad \text{in } L^1(D), \\ 
                    \tk_{0,\delta} &\to \tk_0 \quad \text{in } L^2(D), \\ 
                    \trho_\delta&\to \trho \quad \text{in } C_w([0,T];L^\gamma(D)), \\ 
                    \trho_\delta&\weakarrow \trho \quad \text{in } L^{\gamma+\beta}((0,T)\times D), \\ 
                    \trho_\delta \tu_\delta &\to \tq \quad \text{in } C_w([0,T];L^{\frac{2\gamma }{\gamma +1}}(D)), \\ 
                    \tu_\delta &\weakarrow \tu\quad \text{in } L^2(0,T;W^{1,2}_{\rn }(D)), \\ 
                    \tW_\delta &\to \tW\quad \text{in } C([0,T];\fU_0), \\
                    E(\trho_\delta,\tu_\delta) &\weakstararrow \tE\quad \text{in } L^\infty(0,T;\MMM_b(D)), \\ 
                    \delta_{(\trho_\delta,\tu_\delta,\nabla \tu_\delta)} &\to \tnu \quad \text{in } (L_{w^\ast}^\infty((0,T)\times D;\PPP(\RR^{13})),w^\ast),
                \end{aligned}\label{eq:7.21}
            \end{align}
            as $\delta\to 0$ $\tPP$-a.s.;
            \item[(4)] for any Carath\'{e}odory function $H=H(t,x,\rho,v,V)$ where $(t,x)\in (0,T)\times D$, $(\rho,v,V)\in \RR^{13}$, satisfying for some $q_1,q_2>0$ the growth condition 
            $$
                \left|H(t,x,\rho,v,V) \right| \lesssim 1 + |\rho|^{q_1} + |v|^{q_2} +|V|^{q_2}
            $$
            uniformly in $(t,x)$, denote $\overline{H(\trho,\tu,\nabla \tu)}(t,x) = \ev{\tnu_{t,x},H}$. Then we have 
            \begin{align}
                H(\trho_\delta,\tu_\delta,\nabla \tu_\delta)\weakarrow \overline{H(\trho,\tu,\nabla \tu)}\quad \text{in } L^r((0,T)\times D)\quad \text{for all } 1<r\leq \frac{\gamma + \beta}{q_1}\wedge \frac{2}{q_2}, \label{eq:7.22}
            \end{align}
            as $\delta\to 0$ $\tPP$-a.s.
        \end{itemize}
\end{proposition}
\begin{remark}
    As stated in Remark \ref{rem:r.d.}, we may deduce that the filtration 
    $$
        \tFFF_t:= \sigma \left(\sigma_t[\trho] \cup \sigma_t[\tu]\cup \sigma_t[\tW]\right),\quad t\in [0,T],
    $$
    is non-anticipating with respect to $\tW= \sum_{k=1}^{\infty} e_k \tW_k$, which is a cylindrical $(\tFFF_t)$-Wiener process on $\fU$.    
\end{remark}
\begin{remark}
    In view of the convergence of the initial conditions \eqref{eq:7.19} and \eqref{eq:7.21}, we obtain that 
    $$
        \Lambda = \LL[\trho_0,\trho_0 \tu_0] \quad \text{on } L^1(D)\times L^1(D),
    $$
    where $\tu_0$ is defined as $\tu_0 = \tq_0/ \trho_0 1_{\left\{ \trho_0>0 \right\}}$.
\end{remark}
As a direct consequence of Proposition \ref{prop:delta:repr}, similarly to the proofs of Lemmas \ref{lem:alpha(q = rho u)} and \ref{lem:alpha(conv rho u otimes u)}, we obtain the following result. 
\begin{lemma}
    \label{lem:delta:conv rho u otimes u}
    We have $\tPP$-a.s., $\trho \tu = \tq$, and 
    \begin{align}
        \label{eq:7.23}
        \trho_\delta \tu_\delta\otimes \tu_\delta\weakarrow \trho \tu\otimes \tu\quad \text{in }L^1(0,T;L^1(D)).
    \end{align}
\end{lemma}
Similarly to Proposition \ref{prop:eq_of_conti1}, $(\trho,\tu)$ satisfies the equation of continuity \eqref{eq:7.2} on the new probability space.
\begin{proposition}
    The random distribution $(\trho,\tu)$ satisfies \eqref{eq:7.2} for all $\phi\in C_c^\infty(\ico{0}{T})$ and $\psi\in C_c^\infty(\RR^3)$ $\tPP$-a.s.
\end{proposition}
Next, we perform the limit $\delta\to 0$ in the interior momentum equation \eqref{eq:approx_ME_ep}. Similarly to Section \ref{sec:ep:Asymptotic limit}, at this stage of the proof, we are not able to identify the limit in the pressure or in the stochastic integral.
We apply Proposition \ref{prop:delta:repr} (4), to the compositions $p(\trho_\delta)$ and $\trho_\delta F_k(\trho_\delta, \tu_\delta), k\in\NN$. Specifically, we have 
\begin{align}
    &p(\trho_\delta) \weakarrow \overline{p(\trho)}\quad \text{in } L^q((0,T)\times D), \label{eq:7.24}\\ 
    &\trho_\delta F_k(\trho_\delta,\tu_\delta) \weakarrow \overline{\trho F_k(\trho,\tu)} \quad \text{in } L^q((0,T)\times D), \label{eq:7.25}
\end{align}
for some $q>1$ $\tPP$-a.s. Let us now define a Hilbert--Schmidt operator $\overline{\GG(\trho,\trho \tu)}= \overline{\trho \FF(\trho,\tu)}$ by 
$$
    \overline{\trho \FF(\trho,\tu)}e_k := \overline{\trho F_k(\trho,\tu)},\quad k\in \NN.
$$
Exactly as in the proof of \cite[Proposition 4.5.8]{BFHbook18}, we obtain the following result.
\begin{proposition}
    \label{prop:delta:IME}
    The random distribution $(\trho,\tu,\tW)$ satisfies 
    \begin{align}
        &-\intT \pp_t \phi \int_D \trho \tu \cdot \bphi dxdt - \phi(0) \int_D \trho_0 \tu_0 \cdot \bphi dx \notag\\ 
        &\quad = \intT \phi \int_D [\trho \tu\otimes \tu:\nabla \bphi + \overline{p(\trho)}\rdiv \bphi]dxdt \notag \\ 
        &\quad \quad - \intT \phi\int_D \SSS(\nabla \tu):\nabla \bphi dx dt + \intT \phi \int_D \overline{\trho \FF(\trho,\tu)} \cdot \bphi dx d\tW,\label{eq:7.26}
    \end{align}
    for all $\phi\in C_c^\infty(\ico{0}{T})$ and all $\bphi\in C_c^\infty(D)$ $\tPP$-a.s.
\end{proposition}
\begin{proof}
    Similarly to Proposition \ref{prop:ep:IME}, $(\trho_\delta,\tu_\delta,\tW_\delta)$ solves the approximate momentum equation \ref{eq:approx_ME_ep}. Thus, it remains to show 
    \begin{align}
        \label{eq:7.27}
        \trho_\delta \FF(\trho_\delta,\tu_\delta) &\to \ov{\trho \FF(\trho,\tu)}\quad \text{in } L^2(0,T;L_2(\fU,(W^{l,2}(D))^\ast)) \quad \tPP\text{-a.s.,}
    \end{align}
    where $l>\tfrac32$. Since 
    \begin{align*}
        &\tEE\left[\intT \norm{\trho_\delta \FF(\trho_\delta,\tu_\delta)}_{L_2(\fU,(W_x^{l,2})^\ast)}^2\right]  \lesssim \tEE\left[\intT (\trho_\delta)_D \int_D \left(\trho_\delta + \trho_\delta |\tu_\delta|^2\right)dx dt\right] \lesssim c(r),
    \end{align*}    
    due to Proposition \ref{prop:delta:repr} and Lemma \ref{lem:delta:conv rho u otimes u}, the convergence \eqref{eq:7.27} follows from 
    \begin{align*}
        &\trho_\delta F_k(\trho_\delta,\tu) \to \ov{\trho F_k(\trho,\tu)}\quad \text{in }L^2(0,T;(W^{l,2}(D))^\ast)\quad \tPP\text{-a.s.,}
    \end{align*}
    for any $k\in\NN$, where $l>\tfrac32$. For details, we refer to the proof of \eqref{eq:6.25}. Similarly to the proof of \eqref{eq:6.28}, we also have 
    \begin{align}
        \label{eq:7.28}
        &b(\trho_\delta) \to \ov{b(\trho)}  \quad \text{in }C_w([0,T];L^2(D))\quad \text{as }\delta\to 0,
    \end{align}
    for any $b\in C^1(\ico{0}{\infty})$, $b'(\rho) = 0$ for $\rho \gg 1$. As in the proof of \eqref{eq:6.26}, by combining \eqref{eq:7.28} with the renormalized equation of continuity (see Remark \ref{rem:ep:renormalized CE}), one can see the desired convergence
    \begin{align*}
        &\trho_\delta F_k(\trho_\delta,\tu) \to \ov{\trho F_k(\trho,\tu)}\quad \text{in }L^2(0,T;(W^{l,2}(D))^\ast)\quad \tPP\text{-a.s.,}
    \end{align*}
    for any $k\in\NN$, where $l>\tfrac32$. Combining this with the convergence of $\tW_\delta$ from Proposition \ref{prop:delta:repr}, we may apply Lemma \ref{lem:conv_of_stoch_int} to pass to the limit in the stochastic integral and hence complete the proof.
\end{proof}
\begin{remark}
    \label{rem:delta:conv of each term of ME}
    From the same reason as Remark \ref{rem:alpha(conv of each term of ME)}, the convergence to each term in the momentum equation \eqref{eq:7.26} actually holds for any $\bphi\in C^\infty(\oD)$, but the equality holds only for $\bphi\in C_c^\infty(D)$.
\end{remark}
\subsection{Strong convergence of the densities} \label{sec:delta:Strong convergence of the densities}
As in Section \ref{sec:ep:Strong convergence of the densities}, our ultimate task is to show strong convergence of the densities $\trho_\delta$. We introduce the \textit{oscillation defect measure}:
$$
    \osc_\alpha\left[\trho_\delta \to \trho\right]\left((0,T)\times D\right) = \sup_{k\geq 1}\left(\limsup_{\delta\to 0}\tEE\left[\intT \int_D \left|T_k(\trho_\delta)-T_k(\trho) \right|^\alpha dxdt\right]\right),
$$
where $\alpha\geq 1$, and $T_k$ is a family of cut-off functions defined for $k\in \NN$ as 
\begin{align*}
    &T_k(r) =kT\left(\frac{r}{k}\right),\quad T\in C^\infty(\ico{0}{\infty}),\quad 
    \begin{cases}
    T(r) =r &\text{for }0\leq r\leq 1, \\
    T'(r)\leq 0 &\text{for }r\in (1,3), \\ 
    T(r) = 2  &\text{for }r\geq 3.
    \end{cases}
\end{align*}
In the following we will show 
$$
    \osc_{\gamma + 1}\left[\trho_\delta \to \trho\right]\left((0,T)\times D\right)<\infty.
$$
We proceed as follows.
\begin{itemize}
    \item[1.] We establish a variant of the effective viscous flux identity as in Section \ref{sec:ep:Strong convergence of the densities}.
    \item[2.] We show that $\osc_{\alpha}\left[\trho_\delta \to \trho\right]\left((0,T)\times D\right)<\infty$ for $\alpha = \gamma + 1$. In particular, this implies the limit $(\trho,\tu)$ satisfies the renormalized equation of continuity.
    \item[3.] As in Section \ref{sec:ep:Strong convergence of the densities}, we prove the strong convergence of $\trho_\delta$. 
\end{itemize}
First, extend $\trho_\delta$ to be zero outside $D$ and apply an extension theorem for Sobolev spaces to $\tu_\delta$ such that $\tu_\delta\in L^2(0,T;W^{1,2}(\RR^3))$. Then, we use the quantity 
$$
    \phi \cA(1_D b(\trho_\delta)) = \phi\invdiv{1_D b(\trho_\delta)}
$$
as a test function in \eqref{eq:approx_ME_ep}. More precisely, approximating $1_{[0,\tau]}$ by smooth test functions in renormalized equation of continuity (see Remark \ref{rem:ep:renormalized CE}), we have 
$$
    1_D b(\trho_\delta)(\tau) = 1_D b(\trho_0) -  \inttau \rdiv \left(1_D b(\trho_\delta)\tu_\delta\right)dt - \inttau 1_D (b'(\trho_\delta)\trho_\delta-b(\trho_\delta))\rdiv \tu_\delta dt \quad \text{in } W^{-1,p}(\RR^3),
$$
for all $\tau\in [0,T]$ $\tPP$-a.s., all $b\in C^1(\ico{0}{\infty})$, $b'(\rho) = 0 $ for $\rho \gg 1$, and for some $p>1$. Thus, we have 
\begin{align*}
    &\phi\invdiv{1_D b(\trho_\delta)}(\tau) \\ 
    &\quad = \phi \invdiv{1_D b(\trho_0)} -  \inttau \phi\RRR\left(1_D b(\trho_\delta)\tu_\delta\right)dt - \inttau \phi \invdiv{1_D (b'(\trho_\delta)\trho_\delta-b(\trho_\delta))\rdiv \tu_\delta} dt,
\end{align*}
a.e. in $\RR^3$, for all $\tau\in [0,T]$ $\tPP$-a.s. By the same argument as in Section \ref{sec:ep:Uniform estimates} and Section \ref{sec:ep:Strong convergence of the densities}, applying It\^{o} product rule yields 
\begin{align}
    \label{eq:7.29}
    \inttau \int_D \phi\left[p_\delta(\trho_\delta) b(\trho_\delta) - \SSS(\nabla \tu_\delta):\RRR(1_D b(\trho_\delta)) \right]dxdt = \sum_{j=1}^{5} I_{j,\delta},
\end{align}
where 
\begin{align*}
    &I_{1,\delta} = \left[\int_D \phi \trho_\delta \tu_\delta \cdot \invdiv{1_D b(\trho_\delta)}dx\right]_{t=0}^{t=\tau}, \\ 
    &I_{2,\delta} = \inttau \int_D \phi \trho_\delta \tu_\delta \cdot\left[\RRR (1_D b(\trho_\delta)\tu_\delta) + \invdiv{1_D (b'(\trho_\delta)\trho_\delta-b(\trho_\delta))\rdiv \tu_\delta}\right]dx dt, \\ 
    &I_{3,\delta} = - \inttau \int_D \phi \trho_\delta \tu_\delta \otimes \tu_\delta :\RRR(1_D b(\trho_\delta))dxdt, \\ 
    &I_{4,\delta} = - \inttau \int_D \left[\trho_\delta \tu_\delta \otimes \tu_\delta + p_\delta(\trho_\delta)\II - \SSS(\nabla \tu_\delta)\right]:\nabla \phi\otimes \invdiv{1_D b(\trho_\delta)}dxdt, \\ 
    &I_{5,\delta} = - \inttau \int_D \phi \GG(\trho_\delta,\trho_\delta \tu_\delta)\cdot\invdiv{1_D b(\trho_\delta)} dx d\tW_\delta.
\end{align*}
On the other hand, as $\delta\to 0$ in the renormalized equation of continuity (see Remark \ref{rem:ep:renormalized CE}), we have 
\begin{align}
    \label{eq:7.30}
    &\intT \pp_t \phi  \int_D \ov{b(\trho)}\psi dxdt = \phi(0) \int_D b(\trho_0)\psi dx + \intT \phi \int_D \ov{b(\trho)}\tu \cdot \nabla \psi dxdt -\intT \phi \int_D \ov{(b'(\trho)\trho -b(\trho))\rdiv \tu}\psi dxdt
\end{align}
for all $\phi\in C_c^\infty(\ico{0}{T})$, and all $\psi\in C_c^\infty(\RR^3)$. Here, note that we may use \eqref{eq:7.28} to get 
$$
    \ov{b(\trho)\tu} = \ov{b(\trho)}\tu.
$$
Therefore, as above, we have 
\begin{align*}
    &\phi\invdiv{1_D \ov{b(\trho)}}(\tau) \\ 
    &\quad = \phi \invdiv{1_D b(\trho_0)} - \inttau \phi \RRR(1_D \ov{b(\trho)}\tu) dt - \inttau \phi \invdiv{1_D \ov{b'(\trho)\trho-b(\trho)\rdiv \tu}}dt   \quad \text{a.e. in }\RR^3
\end{align*}
for all $\tau\in [0,T]$ $\tPP$-a.s., and testing the equation \eqref{eq:7.26} by this quantity yields 
\begin{align}
    \label{eq:7.31}
    \inttau \int_D \phi\left[\ov{p(\trho)} \ \ov{b(\trho)} - \SSS(\nabla \tu):\Riesz{1_D \ov{b(\trho)}} \right]dxdt = \sum_{j=1}^{5} I_{j},
\end{align}
where 
\begin{align*}
    &I_{1} = \left[\int_D \phi \trho \tu \cdot \invdiv{1_D \ov{b(\trho)}}dx\right]_{t=0}^{t=\tau}, \\ 
    &I_{2} = \inttau \int_D \phi \trho \tu \cdot\left[\Riesz{1_D \ov{b(\trho)}\tu} + \invdiv{1_D \ov{(b'(\trho)\trho-b(\trho))\rdiv \tu}}\right]dx dt, \\ 
    &I_{3} = - \inttau \int_D \phi \trho \tu \otimes \tu :\Riesz{1_D \ov{b(\trho)}}dxdt, \\ 
    &I_{4} = - \inttau \int_D \left[\trho \tu \otimes \tu + \ov{p(\trho)}\II - \SSS(\nabla \tu)\right]:\nabla \phi\otimes \invdiv{1_D \ov{b(\trho)}}dxdt, \\ 
    &I_{5} = - \inttau \int_D \phi \ov{\GG(\trho,\trho \tu)}\cdot\invdiv{1_D \ov{b(\trho)}}dx d\tW.
\end{align*}
As $\delta\to 0$, by using the convergences \eqref{eq:7.21}, \eqref{eq:7.24}, \eqref{eq:7.27}, and \eqref{eq:7.28}, together with a variant of the Div-Curl lemma (see \cite[Lemma 3.5]{FN17}), 
we can verify that the right-hand side of \eqref{eq:7.29} converges to the right-hand side of \eqref{eq:7.31}. 
Indeed, as in the proof of \eqref{eq:6.33}, it follows from \eqref{eq:7.28}, the assumptions on $b$, and the equation for $\phi \nabla \Delta^{-1}(1_D b(\trho_\delta))$ that 
\begin{align*}
    &\phi \nabla \Delta^{-1}\left(1_D b(\trho_\delta)\right) \to \phi \nabla \Delta^{-1}\left(1_D \ov{b(\trho)}\right)\quad \text{in }C([0,T]\times \RR^3).
\end{align*}
Therefore, from Proposition \ref{prop:delta:repr}, and \eqref{eq:7.23}--\eqref{eq:7.28}, we have 
\begin{align*}
    &I_{1,\delta}\to I_1  
\end{align*}
and 
\begin{align*}
    &I_{j,\delta}\to I_j,\quad j=4, 5,  
\end{align*}
for all $\tau\in [0,T]$, and all $\phi\in C_c^\infty$ $\tPP$-a.s.
Moreover, by the embedding $W^{1,2}(D)\cptarrow L^p(D),p<6$, we have 
\begin{align*}
    &\inttau \int_D \phi \trho_\delta \tu_\delta \cdot  \invdiv{1_D (b'(\trho_\delta)\trho_\delta-b(\trho_\delta))\rdiv \tu_\delta} dx dt \\ 
    &\quad \to \inttau \int_D \phi \trho \tu \cdot \invdiv{1_D \ov{(b'(\trho)\trho-b(\trho))\rdiv \tu}}dx dt.
\end{align*}
Finally, we will show that 
\begin{align*}
    &\inttau \int_D \phi \tu_\delta \cdot \left[\trho_\delta \RRR (1_D b(\trho_\delta)\tu_\delta) - \trho_\delta \tu_\delta \RRR(1_D b(\trho_\delta))\right]dx dt \notag  \\ 
    &\quad \to \inttau \int_D \phi \tu \cdot \left[\trho \RRR (1_D \ov{b(\trho)}\tu) - \trho \tu \RRR(1_D \ov{b(\trho)})\right]dx dt,
\end{align*}
for all $\tau\in [0,T]$, and all $\phi\in C_c^\infty$ $\tPP$-a.s. Note that due to the symmetry of $\RRR$, 
\begin{align*}
    &\inttau \int_D \phi \tu_\delta \cdot \trho_\delta \RRR (1_D b(\trho_\delta)\tu_\delta) dxdt = \inttau \int_D \tu_\delta \cdot b(\trho_\delta)\Riesz{1_D \phi \trho_\delta \tu_\delta}dxdt.
\end{align*}
Thus, as in the proof of \eqref{eq:6.33}, using a variant of Div-Curl lemma with $1_D b(\trho_\delta)$ and $1_D\phi \trho_\delta \tu_\delta$ yields 
\begin{align*}
    &b(\trho_\delta)\Riesz{1_D \phi \trho_\delta \tu_\delta} - \phi \trho \tu \Riesz{1_D b(\trho_\delta)} \\
    &\quad  \to \ov{b(\trho)}\Riesz{1_D \phi \trho \tu} - \phi \trho \tu \Riesz{1_D \ov{b(\trho)}}\quad \text{strongly in } L^2(0,T;(W^{1,2}(D))^\ast). 
\end{align*}
This implies 
\begin{align*}
    &I_{2,\delta} + I_{3,\delta} \to I_2 + I_3,
\end{align*}
for all $\tau\in [0,T]$, and all $\phi\in C_c^\infty$ $\tPP$-a.s.
Therefore, we obtain the identity 
$$
    \inttau\int_D \phi\left[\ov{p(\trho)b(\trho)} - \ov{\SSS(\nabla \tu):\Riesz{1_D b(\trho)}}\right]dx dt = \inttau \int_D \phi \left[\ov{p(\trho)}\ \ov{b(\trho)} - \SSS(\nabla \tu):\Riesz{1_D \ov{b(\trho)}}\right]dx dt 
$$
for all $\tau\in [0,T]$, and all $\phi\in C_c^\infty(D)$ $\tPP$-a.s. Similarly to the proof of \eqref{eq:6.34}, this identity can be rewritten as 
$$
    \inttau\int_D \phi\left[\ov{p(\trho)b(\trho)} - \ov{b(\trho)\rdiv \tu}\right]dx dt = \inttau \int_D \phi \left[\ov{p(\trho)}\ \ov{b(\trho)} - \ov{b(\trho)}\rdiv \tu\right]dx dt, 
$$
and approximating $1_D$ by smooth functions $\phi\in C_c^\infty(D)$ yields 
\begin{align}
    \label{eq:7.32}
    \inttau\int_D \left[\ov{p(\trho)b(\trho)} - \ov{b(\trho)\rdiv \tu}\right]dx dt = \inttau \int_D \left[\ov{p(\trho)}\ \ov{b(\trho)} - \ov{b(\trho)}\rdiv \tu\right]dx dt,
\end{align}
for any $\tau\in [0,T]$ $\tPP$-a.s., which is the desired form of the effective viscous flux. Choosing $b=T_k$ in \eqref{eq:7.32} yields 
\begin{align}
    \label{eq:7.33}
    \inttau\int_D \left[\ov{p(\trho)T_k(\trho)} - \ov{p(\trho)}\ \ov{T_k(\trho)}\right]dx dt = \inttau \int_D \left[ \ov{T_k(\trho)\rdiv \tu} - \ov{T_k(\trho)}\rdiv \tu\right]dx dt.
\end{align}
Using \eqref{eq:7.21}, the integral on the right-hand side can be estimated as 
\begin{align}
    &\tEE\left[\inttau \int_D \left[\ov{T_k(\trho)\rdiv \tu} - \ov{T_k(\trho)}\rdiv \tu\right]dx dt\right] \notag\\ 
    &\quad = \lim_{\delta\to 0} \tEE\left[\inttau \int_D \left[T_k(\trho_\delta) - T_k(\trho)\right]\rdiv \tu_\delta dx dt \right] + \lim_{\delta\to 0} \tEE\left[\inttau \int_D \left[T_k(\trho) - T_k(\trho_\delta)\right]\rdiv \tu dx dt \right]\notag \\ 
    &\quad \lesssim \limsup_{\delta\to 0} \norm{T_k(\trho_\delta)-T_k(\trho)}_{L^2(\tOmega\times (0,\tau)\times D)} \label{eq:7.34}
\end{align}
uniformly in $k$. Since $p$ is convex, nonnegative and non-decreasing, we have 
$$
    \left(p(a)-p(b)\right)\left(T_k(a)- T_k(b)\right) \gtrsim \left|T_k(a) - T_k(b) \right|^{\gamma + 1},\quad \text{for any }a,b\geq 0.
$$
Therefore, we have 
\begin{align}
    &\inttau\int_D \left[\ov{p(\trho)T_k(\trho)} - \ov{p(\trho)}\ \ov{T_k(\trho)}\right]dx dt \notag \\ 
    &\quad = \lim_{\delta \to 0} \inttau \int_D \left[p(\trho_\delta)T_k(\trho_\delta) - \ov{p(\trho)}\ \ov{T_k(\trho)}\right]dx dt \notag \\ 
    &\quad = \lim_{\delta\to 0}\left[\inttau \int_D \left(p(\trho_\delta)-p(\trho)\right)\left(T_k(\trho_\delta)- T_k(\trho)\right)dxdt + \inttau \int_D \left(\ov{p(\trho)}-p(\trho)\right)\left(T_k(\trho_\delta)- T_k(\trho)\right)dxdt\right] \notag \\ 
    &\quad \geq  \lim_{\delta\to 0}\left[\inttau \int_D \left(p(\trho_\delta)-p(\trho)\right)\left(T_k(\trho_\delta)- T_k(\trho)\right)dxdt \right] \notag \\ 
    &\quad \gtrsim \limsup_{\delta\to 0}\inttau \int_D \left|T_k(\trho_\delta)- T_k(\trho) \right|^{\gamma + 1}dx dt. \label{eq:7.35}
\end{align}
Thus, combining the relations \eqref{eq:7.33}--\eqref{eq:7.35}, we obtain the desired conclusion:
\begin{align*}
    \limsup_{\delta\to 0}\tEE\left[\inttau \int_D \left|T_k(\trho_\delta)- T_k(\trho) \right|^{\gamma + 1}dx dt\right]\leq c
\end{align*}
uniformly in $k$ and $\tau$. This implies 
\begin{align}
    \osc_{\gamma + 1}\left[\trho_\delta \to \trho\right]\left((0,T)\times D\right)<\infty. \label{eq:7.36}
\end{align}
Hereafter, we will proceed exactly as in \cite[Section 4.5.4]{BFHbook18} to obtain the strong convergence 
\begin{align}
    \label{eq:7.37}
    \trho_\delta \to \trho\quad \text{in } L^1((0,T)\times D) \ \tPP\text{-a.s.}
\end{align}
Based on \eqref{eq:7.36}, we are going to pass to the limit $k\to \infty$ in the following equation:
\begin{align}
    \label{eq:7.38}
    \pp_t b(\ov{T_k(\trho)}) + \rdiv (b(\ov{T_k(\trho)})\tu) + (b'(\ov{T_k(\trho)})\ov{T_k(\trho)} - b(\ov{T_k(\trho)}))\rdiv \tu = - b'(\ov{T_k(\trho)})\ov{(T_k'(\trho)\trho - T_k(\trho))\rdiv \tu},
\end{align}
which holds in the sense of distributions due to \eqref{eq:7.30}. Since, for all $p\in (1,\gamma)$,
\begin{align*}
    \tEE\left[\norm{\ov{T_k(\trho)} - \trho}_{L^p_{t,x}}^p\right]  &\leq \liminf_{\delta\to 0}\tEE\left[\norm{T_k(\trho_\delta) - \trho_\delta}_{L^p_{t,x}}^p\right] \\ 
    &\leq 2^{p}\liminf_{\delta\to 0} \tEE\left[\int_{\left\{ |\trho_\delta|\geq k \right\}} |\trho_\delta|^p dx dt\right]\\ 
    &\leq 2^p k^{p-\gamma} \liminf_{\delta\to 0}\tEE \left[\intT \int_D |\trho_\delta|^\gamma dx dt\right] \to 0,\quad k\to \infty,
\end{align*}
we have 
\begin{align}
    \label{eq:7.39}
    \ov{T_k(\trho)} \to \trho\quad \text{in }L^p(\tOmega\times (0,T)\times D).
\end{align}
In order to pass to the limit in \eqref{eq:7.38}, we have to show 
\begin{align}
    \label{eq:7.40}
    b'(\ov{T_k(\trho)})\ov{(T_k'(\trho)\trho - T_k(\trho))\rdiv \tu}\to 0\quad \text{in } L^1(\tOmega\times (0,T)\times D).
\end{align}
Recall that $b$ satisfies $b'(z)=0$ for all $z\geq M$ for some $M=M(b)$. We define 
\begin{align*}
    &Q_{k,M} := \left\{ (\omega,t,x)\in \tOmega\times (0,T)\times D: \ov{T_k(\trho)} \leq M \right\}
\end{align*}
and obtain 
\begin{align*}
    &\tEE\left[\intT\int_D \left|b'(\ov{T_k(\trho)})\ov{(T_k'(\trho)\trho - T_k(\trho))\rdiv \tu} \right|dx dt\right]   \\ 
    &\quad \leq \sup_{z\leq M}|b'(z)| \tEE\left[\intT\int_D \chi_{Q_{k,M}}\left|\ov{(T_k'(\trho)\trho - T_k(\trho))\rdiv \tu} \right|dx dt\right]  \\ 
    &\quad \leq C \liminf_{\delta \to 0}  \tEE\left[\intT\int_D \chi_{Q_{k,M}}\left|(T_k'(\trho_\delta)\trho_\delta - T_k(\trho_\delta))\rdiv \tu_\delta \right|dx dt\right]  \\
    &\quad \leq \sup_{\delta}\norm{\rdiv \tu_\delta}_{L^2_{\omega,t,x}} \liminf_{\delta\to 0}\norm{T_k'(\trho_\delta)\trho_\delta - T_k(\trho_\delta)}_{L^2(Q_{k,M})}.
\end{align*}
It follows from interpolation that 
\begin{align}
    \label{eq:7.41}
    \norm{T_k'(\trho_\delta)\trho_\delta - T_k(\trho_\delta)}_{L^2(Q_{k,M})}^2 \leq \norm{T_k'(\trho_\delta)\trho_\delta - T_k(\trho_\delta)}_{L^1_{\omega,t,x}}^\alpha \norm{T_k'(\trho_\delta)\trho_\delta - T_k(\trho_\delta)}_{L^{\gamma + 1}(Q_{k,M})}^{(1-\alpha)(\gamma + 1)},
\end{align}
where $\alpha = \tfrac{\gamma - 1 }{\gamma}$. Moreover, similarly to the proof of \eqref{eq:7.39}, we have 
\begin{align}
    \label{eq:7.42}
    &\norm{T_k'(\trho_\delta)\trho_\delta - T_k(\trho_\delta)}_{L^1_{\omega,t,x}}\leq Ck^{1-\gamma} \sup_{\delta} \tEE\left[\intT \int_D |\trho_\delta|^\gamma dx dt\right] \to 0,\quad k\to \infty,
\end{align}
so it suffices to prove 
\begin{align}
    \label{eq:7.43}
    \sup_{\delta}\norm{T_k'(\trho_\delta)\trho_\delta - T_k(\trho_\delta)}_{L^{\gamma + 1}(Q_{k,M})}\leq C,
\end{align}
independently of $k$. As $T_k'(z)z \leq T_k(z)$, we have, by the definition of $Q_{k,M}$, 
\begin{align*}
    &\norm{T_k'(\trho_\delta)\trho_\delta - T_k(\trho_\delta)}_{L^{\gamma + 1}(Q_{k,M})} \\ 
    &\quad \leq 2(\norm{T_k(\trho_\delta)- T_k(\trho)}_{L^{\gamma +1}_{\omega,t,x}} + \norm{T_k(\trho)}_{L^{\gamma + 1}(Q_{k,M})}) \\ 
    &\quad \leq 2(\norm{T_k(\trho_\delta)- T_k(\trho)}_{L^{\gamma +1}_{\omega,t,x}} + \norm{T_k(\trho)- \ov{T_k(\trho)}}_{L^{\gamma + 1}_{\omega,t,x}} + \norm{\ov{T_k(\trho)}}_{L^{\gamma + 1}(Q_{k,M})}) \\ 
    &\quad \leq 2(\norm{T_k(\trho_\delta)- T_k(\trho)}_{L^{\gamma +1}_{\omega,t,x}} + \norm{T_k(\trho)- \ov{T_k(\trho)}}_{L^{\gamma + 1}_{\omega,t,x}}) + CM.
\end{align*}
Now \eqref{eq:7.36}, \eqref{eq:7.28} with $b=T_k$ and the weak lower semicontinuity of the $L^{\gamma + 1}$-norm imply \eqref{eq:7.43}.
On the other hand, \eqref{eq:7.41}--\eqref{eq:7.43} imply \eqref{eq:7.40}. So we pass to the limit in \eqref{eq:7.38} and obtain 
\begin{align}
    \label{eq:7.44}
    \pp_t b(\trho) + \rdiv (b(\trho)\tu) + (b'(\trho)\trho - b(\trho))\rdiv \tu = 0,
\end{align}
in the sense of distributions. Combining this with \eqref{eq:7.2},  we have 
\begin{align}
    \label{eq:7.45}
    &\int_D L_k(\trho) (\tau) dx + \inttau \int_D T_k(\trho)\rdiv \tu dx dt = \int_D  L_k(\trho_0)dx,
\end{align}
for any $\tau\in [0,T]$ $\tPP$-a.s., where 
\begin{align*}
    &L_k(\rho) = 
    \begin{cases}
    \rho \log(\rho) &\text{if }0\leq \rho<k \\
    \rho\log(k) + \rho \int_k^\rho \frac{T_k(z)}{z^2}dz &\text{if } \rho\geq k, 
    \end{cases}  
\end{align*}
and note that $L_k$ can be written as 
\begin{align*}
    &L_k(\rho) = \beta_k \rho + b_k(\rho) 
\end{align*}
where $b_k$ satisfies $|b_k'(\rho)|= 0$ for $\rho \gg 1$.
Similarly, it follows from \eqref{eq:7.30} that 
\begin{align}
    \label{eq:7.46}
    &\int_D \ov{L_k(\trho)} (\tau) dx + \inttau \int_D \ov{T_k(\trho)\rdiv \tu} dx dt = \int_D  L_k(\trho_0)dx,
\end{align}
for any $\tau\in [0,T]$ $\tPP$-a.s.
Moreover, as $p$ is non-decreasing, we deduce from \eqref{eq:7.32} 
\begin{align}
    \label{eq:7.47}
    \inttau \int_D \ov{T_k(\trho)\rdiv \tu}dx dt \geq \inttau \int_D \ov{T_k(\trho)}\rdiv \tu dx dt,
\end{align}
for any $\tau\in [0,T]$ $\tPP$-a.s.
Summing up \eqref{eq:7.45}--\eqref{eq:7.47} we obtain 
\begin{align}
    \label{eq:7.48}
    \int_D \left[\ov{ L_k(\trho)} - L_k(\trho)\right](\tau)dx \leq \inttau \int_D (\ov{T_k(\trho)} - T_k(\trho))\rdiv \tu dx dt,
\end{align}
for any $\tau\in [0,T]$ $\tPP$-a.s. 

Finally, by the convergence from Proposition \ref{prop:ep:repr}, \eqref{eq:7.39}, and an interpolation argument using \eqref{eq:7.36}, we observe that 
\begin{align*}
    & \inttau \int_D (\ov{T_k(\trho)} - T_k(\trho))\rdiv \tu dx dt\to 0\quad \text{as }k\to \infty,
\end{align*}
for any $\tau\in [0,T]$ $\tPP$-a.s. Since $L_k$ is convex, by Jensen's inequality, we have 
\begin{align*}
    &\ov{L_k(\trho)} \geq L_k(\trho).  
\end{align*}
Thus, letting $k\to \infty$ in \eqref{eq:7.48} and applying Fatou's lemma, we obtain 
\begin{align*}
    &\int_D \left[\ov{\trho \log(\trho)} - \trho \log(\trho)\right](\tau)dx \leq 0  
\end{align*}
for any $\tau\in [0,T]$ $\tPP$-a.s. As the function $\trho\mapsto \trho\log(\trho)$ is strictly convex, we again use Jensen's inequality to obtain 
\begin{align*}
    &\ov{\trho \log(\trho)} - \trho \log(\trho) \geq 0,
\end{align*}
which implies $\ov{\trho \log(\trho)} = \trho \log(\trho)$ and the marginal of the Young measure $\tnu$ with respect to the $\rho$-variable is the Dirac measure $\delta_{\trho}$. 
Consequently, we obtain the desired conclusion 
\begin{align}
    \label{eq:7.49}
    \trho_\delta \to \trho\quad \text{in } L^1((0,T)\times D) \ \tPP\text{-a.s.}
\end{align}

As in Section \ref{sec:ep:Strong convergence of the densities}, this implies 
\begin{align*}
    &\ov{p(\trho)} = p(\trho),\quad \ov{\GG(\trho,\trho \tu)} = \GG(\trho,\trho \tu),
\end{align*}
and we can pass to the limit in the momentum and energy inequality \eqref{eq:approx_MEI_ep}, to obtain \eqref{eq:7.4}. 
Therefore, we obtain a solution to \eqref{eq:1.1}--\eqref{eq:1.5} in the sense of Definition \ref{def:delta:sol}. 
The proof of Theorem \ref{thm:delta:sol} is hereby complete.

\appendix
\section{Appendix}
\subsection{Bogovskii operator}\label{sec:Bogovskii}
We refer to the properties of the Bogovskii operator $\Bog$, which is a sort of ``inverse divergence'' as proven in \cite[Theorem 11.17]{FN17}.
\begin{theorem}
    \label{thm:Bog}
    Let $D\subset \RR^N$ be a bounded Lipschitz domain. Then there exists a linear mapping $\Bog$
    $$
        \Bog:\left\{ f\in C_c^\infty(D): (f)_D :=  \frac{1}{|D|} \int_D f dx = 0\right\} \to C_c^\infty(D;\RR^N),
    $$
    such that
    \begin{itemize}
        \item[(i)] $\rdiv(\Bog[f]) = f$;
        \item[(ii)] we have 
        $$
            \norm{\Bog[f]}_{W^{k+1,p}(D;\RR^N)}\leq c \norm{f}_{W^{k,p}(D)}\quad \text{for any }1<p<\infty,\ k=0,1,\ldots,
        $$
        in particular, $\Bog$ can be extended uniquely as a bounded linear operator 
        $$
            \Bog:\left\{ f\in L^p(D): (f)_D = 0\right\} \to W_0^{1,p}(D;\RR^N);
        $$
        \item[(iii)] if $f\in L^p(D)$, $(f)_D = 0$, and, in addition, $f= \rdiv g$, where 
        $$
            g\in \left\{ u\in L^q(D;\RR^N):\rdiv u\in L^p(D) \right\},\quad 1<q<\infty,
        $$
        then 
        $$
            \norm{\Bog[f]}_{L^q(D;\RR^N)} \leq c\norm{g}_{L^q}(D;\RR^N);
        $$
        \item[(iv)] $\Bog$ can be uniquely extended as a bounded linear operator 
        $$
            \Bog:\left\{ f\in (W^{1,p'})^\ast: \ev{f,1} = 0 \right\} \to L^p(D;\RR^N)
        $$
        in such a way that 
        \begin{align*}
            &-\int_D \Bog[f]\cdot \nabla v dx = \ev{f,v}\quad \text{for all } v\in W^{1,p'}(D),
        \end{align*}
        where $p'$ is the H\"older conjugate exponent of $p$.
    \end{itemize}
\end{theorem}
\subsection{Properties of pseudo-differential operators}\label{sec:Properties of pseudo-differential operators}
In this section, we refer to the properties of \textit{pseudo-differential operators} used in this paper, which are identified by the Fourier symbols (see \cite[Section 11.17]{FN17}). Specifically, we consider the following:
\begin{itemize}
    \item the ``double'' Riesz transform:
    \begin{align*}
        &\RRR= \Delta^{-1}\nabla \otimes \nabla = (\RRR_{ij})_{i,j=1}^3,\quad \RRR_{ij} \approx \frac{\xi_i \xi_j}{|\xi|^2},\quad i,j= 1,2,3,
    \end{align*}
    meaning that 
    \begin{align*}
        &\RRR_{ij}(v) := \cF_{\xi\to x}^{-1}\left[\frac{\xi_i \xi_j}{|\xi|^2} \cF_{x\to \xi}(v)\right],\quad v\in \cS(\RR^3),\quad i,j= 1,2,3;
    \end{align*}
    \item the inverse divergence:
    \begin{align*}
        &\cA = \nabla \Delta^{-1} = (\cA_j)_{j=1}^3 = (\pp_j \Delta^{-1})_{j=1}^3,\quad \pp_j \Delta^{-1}\approx -\frac{i \xi_j}{|\xi|^2},\quad j=1,2,3,
    \end{align*}
    meaning that 
    \begin{align*}
        &\pp_j \Delta^{-1}(v) := - \cF_{\xi\to x}^{-1}\left[\frac{i \xi_j}{|\xi|^2} \cF_{x\to \xi}(v)\right],\quad v\in \cS(\RR^3),\quad j=1,2,3,
    \end{align*}
\end{itemize}
where $\cS(\RR^3)$ denotes the space of smooth rapidly decreasing functions, $\cF$ is the Fourier transform on the space of tempered distributions $\cS'(\RR^3)$, and the values of the above mappings make sense as elements of $\cS'(\RR^3)$.

The following is a classical result for operators associated with Fourier symbols (see \cite[Chapter IV, Theorem 3]{Ste70}).
\begin{theorem}[H\"olmander--Mikhlin]
    Consider an operator $\cL$ defined by means of a Fourier symbol $m=m(\xi)$, 
    \begin{align*}
        &\cL[v](x) = \cF_{\xi\to x}^{-1}\left[m(\xi)\cF_{x\to \xi}(v)\right],\quad v\in \cS(\RR^N),
    \end{align*}
    where $m\in L^\infty(\RR^N)$ has classical derivatives up to order $\left[\frac{N}{2} \right] + 1$ in $\RR^N\backslash \left\{ 0 \right\}$ and satisfies 
    \begin{align*}
        &\left|\pp_\xi^\alpha m(\xi) \right| \leq c_\alpha \left|\xi \right|^{-\left|\alpha \right|},\quad \xi\neq 0,
    \end{align*}
    for any multi-index $\alpha$ such that $\left|\alpha \right|\leq \left[\frac{N}{2} \right] + 1$, where $[\cdot]$ denotes the integer part. Then $\cL$ is a bounded linear operator on $L^p(\RR^N)$ for any $1<p<\infty$.
\end{theorem}
The following theorem is an immediate consequence of the H\"olmander--Mikhlin theorem.
\begin{theorem}
    \label{thm:conti_of_Riesz}
    The operators $(\RRR_{ij})_{i,j=1}^3$ are continuous linear operators mapping $L^p(\RR^3)$ into $L^p(\RR^3)$ for any $1<p<\infty$. In particular, the following estimate holds true:
    \begin{align*}
        &\norm{\RRR_{ij}(v)}_{L^p(\RR^3)} \leq c(p) \norm{v}_{L^p(\RR^3)}\quad \text{for all }v\in L^p(\RR^3),\quad i,j=1,2,3.  
    \end{align*}
\end{theorem}
Next, we refer to the result for the continuity properties of $\cA$, which was proven in \cite[Theorem 11.33]{FN17}.
\begin{theorem}
    \label{thm:conti_of_invdiv}
    \begin{itemize}
        \item[(i)] The operators $(\cA_j)_{j=1}^3$ are continuous linear operators mapping $L^1(\RR^3)\cap L^2(\RR^3)$ into $L^2(\RR^3)+L^\infty(\RR^3)$, and $L^p(\RR^3)$ into $L^{\frac{3p}{3-p}}(\RR^3)$ for any $1<p<3$.
        \item[(ii)] In particular, 
        \begin{align*}
            &\norm{\cA_j(v)}_{L^\infty(\RR^3)+L^2(\RR^3)} \leq c \norm{v}_{L^1(\RR^3)\cap L^2(\RR^3)},\quad \text{for all }v\in L^1(\RR^3)\cap L^2(\RR^3),  
        \end{align*}
        and 
        \begin{align*}
            &\norm{\cA_j(v)}_{L^{\frac{3p}{3-p}}(\RR^3)}\leq c(p)\norm{v}_{L^p(\RR^3)}\quad \text{for all }v\in L^p(\RR^3),\quad 1<p<3.  
        \end{align*}
        \item[(iii)] If $v,\pdv{v}{t}\in L^p(I\times \RR^3)$, where $I$ is an open interval, then 
        \begin{align*}
            &\pdv{\cA_j(v)}{t}(t,x) = \cA_j\left(\pdv{v}{t}\right)(t,x)\quad \text{for a.a. }(t,x)\in I\times \RR^3,\quad 1<p<3.  
        \end{align*}
    \end{itemize}
\end{theorem}
Finally, the following formulas
\begin{align*}
    &\RRR_{ij}(f) = \pp_i \cA_j (f),\quad \sum_{j=1}^{3}\RRR_{jj}(f) =f,\quad  \int_{\RR^3}\RRR_{ij}(f) g dx = \int_{\RR^3} f \RRR_{ij}(g)dx 
\end{align*}
hold for all (real-valued) $f,g\in \cS(\RR^3)$, and can be extended by density in accordance with Theorems \ref{thm:conti_of_Riesz}, \ref{thm:conti_of_invdiv} to $f,g\in L^p(\RR^3)$, $1<p<\infty$, whenever the both sides make sense.
Combining this with Theorems \ref{thm:conti_of_Riesz}, \ref{thm:conti_of_invdiv} yields the following assertion.
\begin{corollary}
        \label{cor:conti_of_invdiv}
    Let $D$ be a bounded domain. Then 
    \begin{align*}
        &\norm{\cA_j (1_D v)}_{W^{1,p}(\RR^3)} \leq c(p,D) \norm{v}_{L^p(D)},\quad \text{for all }v\in L^p(\RR^3),\quad   1<p<\infty, \quad j=1,2,3.
    \end{align*}
\end{corollary}
\subsection{Proof of Theorem 1.8}\label{sec:proof of Theorem 1.8}
The differences from the proof of Theorem \ref{thm:1.3} are as follows. We first note that, in the Galerkin approximation constructed in Section \ref{sec:Approximation system}, it is not necessary to include the technical subsequence $(V_{0,m})$ in $(V_m)$. 
Moreover, there is no need to introduce the convex approximation layer. Next, we replace the basic approximation problem \eqref{eq:3.11}--\eqref{eq:3.12} by the following:
\begin{align*}
    &d\rho + \rdiv (\rho[u]_R)dt = \varepsilon \Delta \rho dt,\quad \nabla \rho \cdot \rn |_{\pD} =0,\quad \text{in}\ (0,T)\times D,  \\
    \int_D \rho u\cdot \bphi(\tau) dx &= \int_D \rho_0 u_0 \cdot \bphi dx  + \inttau \int_D \left[(\rho[u]_R \otimes u):\nabla\bphi + \chi\left(\norm{u}_{V_m}-R\right)p_\delta (\rho)\rdiv \bphi\right]dxdt  \notag \\ 
    &\quad - \inttau \int_D \left[\SSS(\nabla u):\nabla\bphi -\varepsilon \rho u\cdot \Delta \bphi \right]dxdt - \inttau \int_{\pD}g u\cdot \bphi d\Gamma dt \notag \\ 
    &\quad + \inttau\int_D \GG_\varepsilon(\rho,\rho u)\cdot \bphi dxdW, \quad \ \tau\in [0,T],\  \bphi\in V_m.
\end{align*}
Specifically, in the iteration scheme in Section \ref{sec:Iteration scheme}, we redefine the operator $\NNN$ in \eqref{eq:3.27} as follows:
\begin{align*}
    \NNN[\rho](v) &= \left[-\rdiv (\rho[v]_R\otimes v) - \chi(\norm{v}_{V_m}-R)\nabla p_\delta(\rho)+\varepsilon \Delta (\rho v) +\rdiv \SSS(\nabla v)\right]_m^\ast  \notag \\
    &\quad - [\varepsilon \rho(\nabla v)\rn]_m^{\ast \pp} + [\varepsilon \rho v \cdot \nabla_\rn]_m^{\ast \pp} - [\SSS(\nabla v)\rn]_m^{\ast \pp} - [g u]_m^{\ast \pp}.
\end{align*}
By exactly the same argument as in Section \ref{sec:Approximation system}, we obtain a pathwise unique solution $(\rho,u)$, and the energy balance is given in the following form (see Proposition \ref{prop:energy_balance}):
\begin{align*}
    &\int_D \left[\frac 1 2 \rho |u|^2 + P_\delta(\rho)\right](\tau)dx + \inttau \int_D \left[\SSS(\nabla u):\nabla u + \varepsilon \rho |\nabla u|^2 + \varepsilon P_\delta''(\rho)|\nabla \rho|^2\right]dxdt +\inttau \int_{\pD} g \left|u \right|^2 d\Gamma dt \notag\\ 
    &\quad = \int_D \left[\frac 1 2 \rho_0 |u_0|^2 + P(\rho_0)\right]dx +\frac 1 2 \sum_{k=1}^{\infty} \inttau \int_D G_{k,\varepsilon}(\rho,\rho u) \cdot \MMM^{-1}[\rho] ([G_{k,\varepsilon}(\rho,\rho u)]_m^\ast) dxdt \notag \\ 
    &\quad \quad + \inttau \int_D \GG(\rho,\rho u) \cdot u dx dW.
\end{align*}
As already mentioned above, the discussion in Section \ref{sec:convex_approx} is no longer needed.
In Section \ref{sec:The limit in the Galerkin approximation scheme}, we extend the class of test functions for the momentum equation \eqref{eq:approx_ME_m} to $\bphi\in C^\infty(\oD)$ with $\bphi\cdot\rn|_{\pD} = 0$, and replace the momentum and energy inequality \eqref{eq:approx_MEI_m} by the standard energy inequality of the form \eqref{eq:1.23}. 
For the a priori estimates in Section \ref{sec:m:Uniform estimates}, it suffices to choose a dense sequence in $L^{\frac{2\Gamma}{\Gamma - 1}}(D)$ from $(V_m)_{m\in\NN}$. 
For passing to the limit, note that 
\begin{align*}
    &\intT\phi\int_{\pD}g u_n\cdot \bphi d\Gamma dt \to \intT\phi\int_{\pD}g u\cdot \bphi d\Gamma dt, \\ 
    &\intT \phi \int_{\pD} g|u|^2 d\Gamma dt \leq \liminf_{n\to\infty}\intT \phi \int_{\pD} g|u_n|^2 d\Gamma dt,
\end{align*}
whenever $u_n\weakarrow u$ in $L^2(0,T;W^{1,2}(D))$, for all $\phi\in C^\infty_c(\ico{0}{T})$, and all $\bphi\in C^\infty(\oD)$.

The discussion from Section \ref{sec:Vanishing viscosity limit} onward can be carried out in exactly the same manner, taking the above observations into consideration.
This completes the proof of Theorem \ref{thm:1.8}.

\section*{Acknowledgements}
The author would like to thank Associate Professor Yushi Hamaguchi, Professor Seiichiro Kusuoka and Professor Senjo Shimizu for carefully reading the manuscript and providing insightful comments.

\bibliographystyle{alpha} 
\bibliography{subfile/reference.bib}
\end{document}